\documentclass[11pt,twoside]{amsart}

\usepackage[T1]{fontenc}
\usepackage[latin1]{inputenc}
\usepackage[english]{babel}%

\usepackage{amsmath,amsthm,amssymb,amsfonts,amscd}
\usepackage{url}
\usepackage{color}
\usepackage{setspace}
\usepackage{enumerate}
\usepackage{enumitem}
\usepackage[all]{xy}

\usepackage{stmaryrd}

\usepackage{graphicx}
\usepackage{wrapfig}
\usepackage{mathrsfs}
\usepackage[outercaption]{sidecap}

\definecolor{light-gray}{gray}{0.60}

\usepackage{epsfig}
\usepackage{graphicx}
\usepackage{wrapfig}
\usepackage{mathrsfs}
\usepackage{pinlabel}
\usepackage[outercaption]{sidecap}

\definecolor{light-gray}{gray}{0.60}

\theoremstyle{plain}
\newtheorem{theorem}{Theorem}
\newtheorem{proposition}[theorem]{Proposition}

\newtheorem{corollary}[theorem]{Corollary}

\newtheorem{lemma}[theorem]{Lemma}
\newtheorem{fact}[theorem]{Fact}

\newtheoremstyle{theoremwithref}{}{}{\itshape}{}{\bfseries}{.}{.5em}{#1 #2 #3}
\theoremstyle{theoremwithref}

\theoremstyle{definition}
\newtheorem{definition}[theorem]{Definition}
\newtheorem{example}[theorem]{Example}
\newtheorem{examples}[theorem]{Examples}
\newtheorem{remark}[theorem]{Remark}

\newtheorem{question}[theorem]{Question}

\newtheorem{notation}[theorem]{Notation}

\numberwithin{theorem}{section}
\numberwithin{equation}{section}

\newcommand{\C}{\mathcal C}

\newcommand{\NN}{\mathbb{N}}
\newcommand{\ZZ}{\mathbb{Z}}

\newcommand{\RR}{\mathbb{R}}
\newcommand{\CC}{\mathbb{C}}
\newcommand{\HH}{\mathbb{H}}
\newcommand{\PP}{\mathbb{P}}
\newcommand{\RP}{\mathbb{RP}}
\renewcommand{\SS}{\mathbb{S}}

\newcommand{\SL}{\mathrm{SL}}
\newcommand{\GL}{\mathrm{GL}}
\newcommand{\SO}{\mathrm{SO}}

\newcommand{\PO}{\mathrm{PO}}

\newcommand{\PSL}{\mathrm{PSL}}
\newcommand{\PGL}{\mathrm{PGL}}
\newcommand{\Sp}{\mathrm{Sp}}
\newcommand{\SU}{\mathrm{SU}}

\newcommand{\g}{\mathfrak{g}}
\newcommand{\ssl}{\mathfrak{sl}}

\newcommand{\uu}{\mathfrak{u}}
\newcommand{\aaa}{\mathfrak{a}}

\newcommand{\Ad}{\operatorname{Ad}}

\newcommand{\ie}{i.e.\ }
\newcommand{\eg}{e.g.\ }
\newcommand{\resp}{resp.\ }

\newcommand{\Fr}{\mathrm{Fr}}
\newcommand{\Int}[1]{{#1^\circ}}
\newcommand{\Conv}[1]{\langle \! \langle #1 \rangle \! \rangle}

\newcommand{\partiali}{\partial_{\boldsymbol{\mathrm{i}}}}
\newcommand{\partialn}{\partial_{\boldsymbol{\mathrm{n}}}}
\newcommand{\Lambdao}{\Lambda^{\mathsf{orb}}}
\newcommand{\Ccore}{\C^{\mathsf{cor}}}

\newcommand{\cro}[4]{[#1 \!:\! #2 \!:\! #3 \!:\! #4]}

\title{Combination theorems in convex projective geometry}

\author{Jeffrey Danciger}
\address{Department of Mathematics, The University of Texas at Austin, 1 University Station C1200, Austin, TX 78712, USA}
\email{jdanciger@math.utexas.edu}

\author{Fran\c{c}ois Gu\'eritaud}
\address{CNRS and IRMA, Universit\'e de Strasbourg, 7 rue Ren\'e Descartes, 67084 Strasbourg Cedex, France}
\email{francois.gueritaud@unistra.fr}

\author{Fanny Kassel}
\address{CNRS and Laboratoire Alexander Grothendieck, Institut des Hautes \'Etudes Scientifiques, Universit\'e Paris-Saclay, 35 route de Chartres, 91440 Bures-sur-Yvette, France}
\email{kassel@ihes.fr}

\thanks{This project received funding from the European Research Council (ERC) under the European Union's Horizon 2020 research and innovation programme (ERC starting grant DiGGeS, grant agreement No 715982). J.D. was partially supported by an Alfred P. Sloan Foundation fellowship and by the National Science Foundation under grants DMS 1812216 and DMS 1945493. Part of this work was completed while F.K. was in residence at the Institute for Advanced Study in Princeton, supported by the National Science Foundation under grant DMS 1926686.}

\begin{document}

\begin{abstract}
We prove a general combination theorem for discrete subgroups of $\PGL(n,\RR)$ preserving properly convex open subsets in the projective space $\PP(\RR^n)$, in the spirit of Klein and Maskit.
We use it in particular to prove that a free product of two ($\ZZ$-)linear groups is again ($\ZZ$-)linear, and to construct Zariski-dense discrete subgroups of $\PGL(n,\RR)$ which are not lattices but contain a lattice of a smaller higher-rank simple Lie group.
We also establish a version of our combination theorem for discrete groups that are convex cocompact in $\PP(\RR^n)$ in the sense of \cite{dgk-proj-cc}.
In particular, we prove that a free product of two convex cocompact groups is convex cocompact, which implies that the free product of two Anosov groups is Anosov.
We also prove a virtual amalgamation theorem over convex cocompact subgroups generalizing work of Baker--Cooper.
\end{abstract}

\maketitle
\tableofcontents

\section{Introduction}

Combination theorems in group theory have a long history.
In pioneering work, Klein~\cite{kle83} gave sufficient conditions for when the group generated by two discrete subgroups of the Lie group $\PSL(2,\CC)$ is again discrete and isomorphic to the free product of the initial groups. The underlying idea of Klein's technique, nowadays referred to as ``ping-pong dynamics'', is profound yet simple. 
Considering the action of $\PSL(2,\CC)$ on $\mathbb{CP}^1$, each of the two discrete groups will attract points that are far away from its limit set into a neighborhood of its limit set. Hence if the limit sets of the two groups are, in a sense, far from each other, then there are two disjoint neighborhoods in $\mathbb{CP}^1$ such that the first group maps the second neighborhood into the first and vice versa, and one may argue on set-theoretic grounds that the group generated by the initial two groups has no new relations and is discrete.

Maskit (see~\cite{mas88, mas93}) generalized Klein's techniques to the case of amalgamated free products and HNN extensions.
He also gave sufficient conditions for the resulting discrete group to inherit certain geometric conditions, such as convex cocompactness or geometrical finiteness.
The main ideas of Klein and Maskit generalize well in the setting of discrete subgroups of any real semisimple Lie group of real rank one, and more generally in the context of groups acting on Gromov hyperbolic spaces. See \cite{mm22} for a survey of some generalizations, and also the recent paper~\cite{tw23}. 

On the other hand, Baker--Cooper~\cite{bc08} took a more geometric viewpoint on combination theorems. 
They gave a condition for when the gluing of two convex hyperbolic manifolds along isometric submanifolds may be thickened to a convex hyperbolic manifold.
Since convex hyperbolic manifolds have discrete and faithful holonomy representations, this convex gluing theorem gives a way to combine discrete subgroups of the isometry group of hyperbolic space (in any dimension).
The geometric notion of convexity can be checked locally along the boundary of a hyperbolic manifold.
Baker--Cooper also proved a virtual amalgamation theorem, showing that under separability assumptions, the hypotheses of their combination theorem are often satisfied after passing to finite-index subgroups of the initial groups.

In this paper, we take this circle of ideas into a broader setting.
Let $V$ be a real vector space of finite dimension $\dim V =:d\geq 3$. 
We study discrete subgroups of the projective general linear group $\PGL(V)$ that preserve some nonempty properly convex open subset in the projective space $\PP(V)$.
Here, an open subset of $\PP(V)$ is called \emph{properly convex} if it is contained, convex, and bounded in some affine chart.
We establish a general combination theorem for this class of discrete groups in the spirit of Klein and Maskit. We also establish combination theorems for discrete subgroups of $\PGL(V)$ satisfying several notions of convex cocompactness, studied in \cite{dgk-proj-cc}.
By considering the quotient convex real projective manifolds, our combination theorems may be viewed, in the same spirit as Baker--Cooper, as conditions for obtaining a new convex real projective manifold from old convex real projective manifolds glued together along isomorphic submanifolds. Following Baker--Cooper, we also give a virtual amalgamation theorem in this setting.

We emphasize that for the higher-rank semisimple Lie group $\PGL(V)$, the dynamics of the action on its various flag manifolds, such as real projective space $\PP(V)$, is in many ways more complicated than in the rank-one setting of Klein and Maskit's original arguments. 
Nonetheless, our approach generalizes the original idea of ping-pong dynamics, with the set containment arguments from Klein and Maskit replaced by a notion of properly convex set occultation: one properly convex set blocking others from viewing each other.
Crucially, this notion inducts well, leading to a local-to-global approach to convexity in projective space that replaces the local arguments for convexity of hyperbolic manifolds used by Baker--Cooper.

For comparison, recent work of Dey--Kapovich--Leeb \cite{dkl19} and Dey--Kapovich \cite{dk23,dk25} also establishes generalizations of the Klein--Maskit combination theorems in the setting of discrete subgroups of higher-rank semisimple Lie groups.
They work in the more general setting of any semisimple Lie group (not just $\PGL(V)$), but restrict to special classes of discrete subgroups, namely Anosov subgroups, which are Gromov hyperbolic.
In some cases a (projective) Anosov subgroup of $\PGL(V)$ does preserve a properly convex open set, so there is some overlap with our results.
However, we note that our main results apply to many discrete subgroups which are not Anosov, and in fact not necessarily Gromov hyperbolic nor relatively hyperbolic.

\subsection{Amalgams and HNN extensions of discrete groups} \label{sec:amalhnn}

Our main result, Theorem~\ref{thm:general-gog}, is a general combination theorem in the context of a graph of groups for which the vertex groups act on properly convex open sets.
The statement assumes knowledge of graphs of groups and the construction of their fundamental groups, so we do not give it until later in the paper.
Here we state two important special cases, that of amalgamated free products and of HNN extensions.
These cases may be the most familiar to the reader, although their proofs are not significantly simpler than the general case.
We use the following notation.

\begin{notation} \label{not:intro}
For any subset $X$ of $\PP(V)$, we denote by $\overline{X}$ its closure in $\PP(V)$ and by $X^\circ$ its interior.
If $X$ is connected and contained in an affine chart of $\PP(V)$, we denote by $\Conv{X}$ the convex hull of $X$ in this affine chart (it does not depend on the choice of affine chart containing~$X$).

For any properly convex open subset $\Omega$ of $\PP(V)$, we denote its dual by
$$\Omega^* := \left\{ H \in \PP(V^*) ~|~ H \cap \overline{\Omega} = \varnothing \right\},$$
where we view the dual projective space $\PP(V^*)$ as the space of projective hyperplanes of $\PP(V)$.
The set $\Omega^*$ is properly convex and open in $\PP(V^*)$.

Finally, inside any group~$\Gamma$, we denote by $\langle t\rangle$ the cyclic subgroup generated by an element~$t$.
\end{notation}

To state our theorems, let us first introduce a key notion of the paper.

\begin{definition} \label{def:occult}
A triple $(A, B, C)$ of properly convex open subsets of $\PP(V)$ is \emph{in occultation position} if the following three conditions are satisfied:
\begin{itemize}
  \item $A \not\subset B$ and $C \not\subset B$,
  \item $A \cap B \neq \varnothing$ and $B \cap C \neq \varnothing$, but $A \cap C = \varnothing$,
  \item $\overline{B^*} \subset A^* \cup C^*$.
\end{itemize}
\end{definition}

By Lemma~\ref{lem:occult} below, the last condition is equivalent to requiring that any projective line of $\PP(V)$ passing through both $\overline{A}$ and $\overline{C}$ must pass through $B$; in other words $B$ blocks the sets $A$, $C$ (and even their closures) from ``seeing'' one another: see Figure~\ref{fig:invisible-configuration}.
This optical analogy justifies the term \emph{occultation position}, which we resisted renaming \emph{Mask-it} position.

\begin{figure}[h]
\labellist
\small\hair 2pt
\pinlabel {$A$} at 	1.2	3 
\pinlabel {$B$} at 	4.5	3 
\pinlabel {$C$} at 	8.5	3 
\endlabellist
\includegraphics[width = 0.33\textwidth]{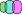}
\caption{\label{fig:invisible-configuration} Three properly convex open sets in occultation position}
\end{figure}

Theorem~\ref{thm:general-gog} interpreted in the case of amalgamated free products is the following.

\begin{theorem}[Amalgamated Free Products] \label{thm:amalgam}
Suppose $\dim \PP(V) \geq 2$. 
For $i=0,1$, let $\Gamma_i$ be a nontrivial discrete subgroup of $\PGL(V)$ preserving a properly convex open subset $\Omega_i$ of $\PP(V)$.
Set $\Delta := \Gamma_0 \cap \Gamma_1$ and $\Gamma := \Gamma_0 *_\Delta \Gamma_1$.
Suppose that for each $i\in\{ 0,1\}$ and each $\gamma \in \Gamma_i \smallsetminus \Delta$, the triple of properly convex open sets $(\Omega_{1-i}, \ \Omega_i, \ \gamma \cdot \Omega_{1-i})$ is in occultation position, and that there is at least one such triple for some~$i$ (\ie $\Gamma_0 \neq \Gamma_1$).
Then the representation $\rho:  \Gamma \to \PGL(V)$ induced by the inclusions $\Gamma_0, \Gamma_1 \hookrightarrow \PGL(V)$ is discrete and faithful, and the $\rho(\Gamma)$-invariant open set 
$$\Omega := \bigcup_{\gamma \in \Gamma} \rho(\gamma) \cdot \Conv{\Omega_0 \cup \Omega_1}$$
is well defined and properly convex.
\end{theorem}

Theorem~\ref{thm:general-gog} interpreted in the case of HNN extensions is the following.

\begin{theorem}[HNN extensions] \label{thm:HNN}
Suppose $\dim \PP(V) \geq 2$. 
Let $\Gamma_0$ be a nontrivial discrete subgroup of $\PGL(V)$ preserving a properly convex open subset $\Omega_0$ of $\PP(V)$.
Let $t \in \PGL(V)$, let $\Delta := \Gamma_0 \cap t \Gamma_0 t^{-1}$, and let $\Gamma := (\Gamma_0) *_t$ be the HNN extension of~$\Gamma_0$ determined by the conjugation $c_t: t^{-1}\Delta t \to \Delta$.
Suppose that 
\begin{enumerate}
  \item for every $\gamma \in \Gamma_0$, the triple of properly convex open sets $(t \cdot \Omega_0, \Omega_0, \gamma t^{-1} \cdot \Omega_0)$ is in occultation position, and
  \item for every $\varepsilon \in \{1,-1\}$ and every $\gamma' \in \Gamma_0 \smallsetminus t^{\varepsilon} \Gamma_0 t^{-\varepsilon}$, the triple  of properly convex open sets $(t^{\varepsilon} \cdot \Omega_0, \ \Omega_0, \gamma' t^{\varepsilon} \cdot \Omega_0)$ is in occultation position.
\end{enumerate}
Then the representation $\rho: \Gamma \to \PGL(V)$ induced by the inclusions $\Gamma_0, \langle t \rangle \hookrightarrow \PGL(V)$ is discrete and faithful, and the $\rho(\Gamma)$-invariant open set
$$\Omega := \bigcup_{\gamma \in \Gamma} \rho(\gamma) \cdot \Conv{\Omega_0 \cup t \cdot \Omega_0}$$
is well defined and properly convex.
\end{theorem}

Theorems~\ref{thm:amalgam} and~\ref{thm:HNN} are special cases of Theorem~\ref{thm:general-gog}, which concerns general graphs of groups.
Let us give a sense of the proof strategy.

An important object in the theory of graphs of groups is the \emph{Bass--Serre tree}.  
 For a free product with amalgamation  $\Gamma =\Gamma_0 *_\Delta \Gamma_1 $ (resp. an HNN extension $\Gamma := (\Gamma_0) *_t$) the Bass--Serre tree is a certain simplicial tree $T_\Gamma$ endowed with a $\Gamma$-action such that the vertex stabilizers are conjugate to $\Gamma_0$ and $\Gamma_1$ (\resp to $\Gamma_0$). 
From the assumptions of Theorem~\ref{thm:amalgam} (resp. Theorem~\ref{thm:HNN}), we build a $\Gamma$-equivariant realization of $T_\Gamma$ as a \emph{tree of properly convex open sets}, namely as the assignment, to each vertex of $T_\Gamma$, of a properly convex open subset of $\PP(V)$ which is a $\Gamma$-translate of $\Omega_0$ or $\Omega_1$ (\resp of~$\Omega_0$), in such a way that any triple of adjacent vertices in the tree maps to a triple of properly convex open sets in occultation position.
Edges of the tree correspond to the pair $(\Omega_0, \Omega_1)$ (\resp the pair $(\Omega_0, t \Omega_0)$) and its $\Gamma$-translates.

A key fact about triples $(A,B,C)$ in occultation position is Lemma~\ref{lem:invisible}, which says that the convex hull $\Conv{A \cup B \cup C}$ of $A,B,C$ is well defined and equal to the union $\Conv{A \cup B} \cup \Conv{B \cup C}$.
Therefore, for any two-edge path connecting three vertices in the Bass--Serre tree, the convex hull of the three associated convex sets may be taken one edge at a time. 
The next key idea is that occultation position propagates through the tree: namely, if an edge of the tree is contracted, replacing the two convex sets associated to the vertices with their convex hull, the result is a new tree of convex sets which again has the occultation position property for adjacent triples.
We then argue inductively that the convex hull of the union of all of the properly convex sets of the entire tree may be taken locally, one edge at a time.
The result is a $\Gamma$-invariant properly convex set $\Omega$ containing all of the $\Gamma$-translates of $\Omega_0$ and $\Omega_1$ (\resp of $\Omega_0)$, organized in such a way that the adjacency graph precisely recovers the Bass--Serre tree.
From this we deduce that the action of $\Gamma$ is properly discontinuous and faithful.

\subsection{Applications: free products of discrete groups} \label{subsec:intro-applic-free-prod}

In the sequel, we use the following basic definitions from \cite{dgk-proj-cc}.

\begin{definition} \label{def:limcore}
Given a subgroup $\Gamma$ of $\PGL(V)$ preserving a properly convex open subset $\Omega$ of $\PP(V)$, the \emph{full orbital limit set} of $\Gamma$ in~$\Omega$ is the set $\Lambdao_{\Omega}(\Gamma) \subset \partial\Omega$ of all accumulation points in $\partial\Omega$ of all $\Gamma$-orbits of~$\Omega$; it is $\Gamma$-invariant.
The \emph{convex core} $\Ccore_{\Omega}(\Gamma)$ of $\Gamma$ in~$\Omega$ is the intersection of $\Omega$ with the convex hull of $\overline{\Lambdao_{\Omega}(\Gamma)}$ in~$\overline{\Omega}$; it is a $\Gamma$-invariant closed convex subset of~$\Omega$.
\end{definition}

In the case of free products (without amalgamation), we use a variant of Theorem~\ref{thm:amalgam} with trivial~$\Delta$ to establish the following (see Section~\ref{subsec:remind-prox} for the notion of proximal limit set at the end of the statement).

\begin{theorem} \label{thm:free-product-fd}
For $i = 0,1$, let $\Gamma_i$ be a discrete subgroup of $\PGL(V)$ preserving a nonempty properly convex open subset $\Omega_i$ of $\PP(V)$.
Suppose that $\Omega_i \neq \Ccore_{\Omega_i}(\Gamma_i)$.
Then there exists $g\in\PGL(V)$ such that the representation $\rho : \Gamma_0 * g\Gamma_1 g^{-1} \to \PGL(V)$ induced by the inclusions $\Gamma_0, g\Gamma_1 g^{-1} \hookrightarrow \PGL(V)$ is discrete and faithful, and such that the image of this representation preserves a nonempty properly convex open subset of $\PP(V)$.
Moreover, we can take $g$ in any given lattice of $\PGL(V)$ (causing $\rho$ to take values in that lattice if it contains $\Gamma_0, \Gamma_1$), or more generally in any given Zariski-dense subgroup of $\PGL(V)$ whose proximal limit sets in $\PP(V)$ and in $\PP(V^*)$ are everything.
\end{theorem}

\begin{remark}
It follows from the proof (see Section~\ref{subsec:proof-free-product-fd}) that if $g\in\PGL(V)$ satisfies the conclusion of Theorem~\ref{thm:free-product-fd}, then there is a neighborhood $\mathcal{U}$ of $g$ in $\PGL(V)$ such that any element of~$\mathcal{U}$ still satisfies this conclusion.
\end{remark}

Theorem~\ref{thm:free-product-fd} allows us to give a negative answer to the following question.

\begin{question} \label{question-Nori}
Let $G$ be a real simple algebraic group, and $H$ a simple algebraic subgroup of~$G$ of real rank at least two.
Let $\Gamma$ be a Zariski-dense discrete subgroup of~$G$ such that $\Gamma \cap H$ is a lattice in~$H$.
Is $\Gamma$ a lattice in~$G$?
\end{question}

Question~\ref{question-Nori} was asked by Chatterji--Venkataramana \cite[Q.\,2]{cv-Nori}, as a special case of a more general problem posed by Nori in 1983 and asking, for a real algebraic subgroup $H$ of a real semisimple algebraic group~$G$, to find sufficient conditions on $H$ and~$G$ such that any Zariski-dense discrete subgroup of~$G$ which intersects $H$ in a lattice of~$H$, is itself a lattice in~$G$ (see \cite[Prob.\,1]{cv-Nori} and \cite[Q.\,8.3]{fis-survey-margulis}).

Nori's problem was settled completely in the case that $G$ has real rank $\geq 2$ and $H$ is a horospherical subgroup of~$G$ by Oh, Benoist--Oh, and Benoist--Miquel (see \cite[\S\,8]{fis-survey-margulis}): in that case, a Zariski-dense discrete subgroup of~$G$ which intersects $H$ in a lattice of~$H$, is always a lattice in~$G$.

For simple~$H$ of real rank $\geq 2$, using Margulis's superrigidity theorem, Chatterji--Venkatara\-mana \cite{cv-Nori} showed that Question~\ref{question-Nori} has a positive answer in many cases, for instance for all pairs $(G,H) = (\SL(n,\RR),\SL(m,\RR))$ with $2<m<n$ (where $\SL(m,\RR)$ is embedded in the upper left corner of $\SL(n,\RR)$).
More precisely, they proved that Question~\ref{question-Nori} has a positive answer as soon as there is no $H$-orbit with compact stabilizer in a flag manifold of~$G$.

In contrast to these positive answers, here we prove that the answer to Question~\ref{question-Nori} is negative in general, by constructing explicit counterexamples.
For any $d\geq 2$, we set
$$n_d := \frac{d(d+1)}{2}.$$
We denote by $\tau_d : \SL(d,\RR)\to\SL(S^2\RR^d)\simeq\SL(n_d,\RR)$ the second symmetric power of the standard representation of $\SL(d,\RR)$, and by $\tau'_d := \tau_d \oplus \mathbf{1} : \SL(d,\RR) \to \SL(S^2\RR^d \oplus \RR) \simeq \SL(n_d+1,\RR)$ the sum of $\tau_d$ and of the one-dimensional trivial representation of $\SL(d,\RR)$.

\begin{corollary} \label{cor:counterex-nori}
For any $d\geq 2$, there exists a Zariski-dense discrete subgroup of $G =\linebreak \SL(S^2\RR^d\oplus\RR) \simeq \SL(n_d+1,\RR)$ which is \emph{not} a lattice of~$G$, but which contains a lattice of $H := \tau'_d(\SL(d,\RR))$.
This lattice of~$H$ can be taken to be uniform or not.
\end{corollary}

The construction (see Section~\ref{subsec:nori}) starts from a lattice of~$H$, seen as acting on the projective model of the symmetric space of $\SL(d,\RR)$; we can then adjoin elements of $G$ in a manner analogous to the ping-pong phenomenon in Kleinian groups: see Section~\ref{subsec:nori}.
A similar construction works over the complex numbers and over the quaternions, see Remark~\ref{rem:Nori-over-C-H}.

Theorem~\ref{thm:free-product-fd} also allows us to construct free products of linear groups.
More precisely, it is known \cite{nis40,weh73,sha79} that the free product of two subgroups of $\SL(n,\RR)$ can be realized as a subgroup of some $\SL(N,\RR)$.
However, the question \cite[Question\,D]{bn22} of whether the free product of two \emph{discrete} subgroups of $\SL(n,\RR)$ can be realized as a \emph{discrete} subgroup of some $\SL(N,\RR)$ seems to have been open until now, and it seems to have been unknown whether the free product of $\SL(n,\ZZ)$ with itself, for $n\geq 3$, can be realized inside some $\SL(N,\ZZ)$.
We give affirmative answers to both these questions, by establishing the following as a consequence of Theorem~\ref{thm:free-product-fd}.

\begin{corollary} \label{cor:combine-discrete-groups}
Let $\mathbb{A}$ be a subring of $\RR$ (\eg $\mathbb A = \mathbb Z$ or $\mathbb A = \mathbb R$).
For any $d\geq 2$ and any discrete subgroups $\Gamma_0$ and~$\Gamma_1$ of $\SL(d,\mathbb{A})$, the free product $\Gamma_0*\Gamma_1$ embeds as a discrete subgroup of $\SL(n_d+1,\mathbb{A})$.
In particular, a free product of two $\ZZ$-linear groups is $\ZZ$-linear.
\end{corollary}

We also prove a variant of Theorem~\ref{thm:free-product-fd} for semisimple subgroups of $\PGL(V)$ (Theorem~\ref{thm:free-product-in-G}), which implies the following.
We say that a point of $G/P$ is \emph{uniformly transverse} to a subset of~$G/P$ if it is transverse to all points in the closure of the subset (and therefore also to all points in a neighborhood of that closure).

\begin{theorem} \label{thm:free-product-G/P}
Let $G$ be a noncompact connected real linear semisimple Lie group, $P$ a self-opposite parabolic subgroup of~$G$, and $\Gamma_0$ and~$\Gamma_1$ discrete subgroups of~$G$.
Suppose that for each $i\in\{0,1\}$ there is a point $x_i \in G/P$ which is uniformly transverse to $(\Gamma_i\smallsetminus\{1\}) \cdot x_i$.
Then there exists $g \in G$ such that the representation $\rho : \Gamma_0 * g\Gamma_1 g^{-1} \to G$ induced by the inclusions $\Gamma_0, g\Gamma_1 g^{-1} \hookrightarrow G$ is discrete and faithful.

Moreover, for any Zariski-dense subgroup $\Gamma$ of~$G$, if we can choose the points $x_0, x_1 \in G/P$ inside~$\Lambda_{\Gamma}^{G/P}$, then we can choose $g$ inside~$\Gamma$.
\end{theorem}

Here $\Lambda_{\Gamma}^{G/P}$ denotes the proximal limit set of $\Gamma$ in $G/P$, \ie the closure in $G/P$ of the set of all attracting fixed points of elements of~$\Gamma$ (see Section~\ref{subsec:remind-prox}).

As an application, we improve a result of Soifer \cite{soi12} for $d=3$ by showing that for any~$d$, the group $G = \SL(d,\RR)$ contains a discrete subgroup isomorphic to the free product $\ZZ^{d-1} * \ZZ^{d-1}$ (Corollary~\ref{cor:Soifer}).
This provides examples of Zariski-dense discrete subgroups of $\SL(d,\RR)$ which are not lattices but contain diagonalizable subgroups isomorphic to~$\ZZ^{d-1}$ --- another open case of Nori's problem discussed in Fisher's survey \cite[\S\,8]{fis-survey-margulis}.

\subsection{Combination theorems for (naively) convex cocompact actions}

We prove that our combination results behave well with respect to the following notions of convex cocompactness introduced in~\cite{dgk-proj-cc} and subsequently studied, for example, in~\cite{gm21, iz23, wei23, bf24} (see also~\cite{kas-notes} for more on this notion and its role in the larger study of discrete subgroups of Lie groups).

\begin{definition} \label{def:cc-group}
Let $\Gamma$ be an infinite discrete subgroup of $\PGL(V)$ preserving a properly convex open subset $\Omega$ of $\PP(V)$.
The action of $\Gamma$ on~$\Omega$~is
\begin{enumerate}[label=(\roman*)]
  \item\label{item:def-naiv-cc} \emph{naively convex cocompact} if $\Gamma$ preserves and acts with compact quotient on some non\-empty $\Gamma$-invariant closed convex subset $\C$ of~$\Omega$;
  \item\label{item:def-cc} \emph{convex cocompact} if it is naively convex cocompact as in \ref{item:def-naiv-cc} and if the set $\C$ can be taken ``large enough'', in the sense that the closure of $\C$ in $\PP(V)$ contains the full orbital limit set $\Lambdao_{\Omega}(\Gamma)$.
\end{enumerate}
We say that the group $\Gamma$ is \emph{naively convex cocompact in $\PP(V)$} (\resp is \emph{convex cocompact in $\PP(V)$}) if it preserves and acts naively convex cocompactly (\resp convex cocompactly) on some properly convex open subset of $\PP(V)$.
\end{definition}

Given two properly convex sets $\C \subset \Omega \subset \PP(V)$ where $\Omega$ is open and $\C$ is closed in $\Omega$ with nonempty interior $\Int{\C}$, we say that $\C$ has \emph{strictly convex nonideal boundary in~$\Omega$} if any nontrivial projective segment contained in $\overline{\C} \smallsetminus \Int{\C}$ (see Notation~\ref{not:intro}) is in fact contained in $\partial\Omega$.

The two theorems below say, informally, that if the convex open sets $\Omega_0, \Omega_1$ of Theorems~\ref{thm:amalgam} and~\ref{thm:HNN} come equipped with $\Gamma_i$-invariant cocompact subsets $\C_i$ whose interiors satisfy the same assumptions as the $\Omega_i$, then under natural cocompactness assumptions we can combine the $\C_i$ into a cocompact $\Gamma$-invariant convex subset $\C$ of~$\Omega$.

\begin{theorem} \label{thm:amalgam-cc}
In the setting of Theorem~\ref{thm:amalgam}, suppose that there exist $\Gamma_i$-invariant closed convex subsets $\C_i$ of~$\Omega_i$ such that for any $i\in\{ 0,1\}$ and any $\gamma \in \Gamma_i \smallsetminus \Delta$, the triple of properly convex open sets $(\Int{\C_{1-i}}, \ \Int{\C_i}, \ \gamma \cdot \Int{\C_{1-i}})$ is in occultation position.
Suppose in addition that $\C_i$ has compact quotient by~$\Gamma_i$ for both~$i$, and that the closure of $\Conv{\C_0\cup\C_1} \smallsetminus (\C_0\cup\C_1)$ in $\Conv{\Omega_0\cup\Omega_1}$ has compact quotient by~$\Delta$.
Then the $\Gamma$-invariant set
$$\C := \bigcup_{\gamma \in \Gamma} \rho(\gamma) \cdot \Conv{\C_0 \cup \C_1}$$
is convex and closed in~$\Omega$, and has compact quotient by $\Gamma = \Gamma_0 *_{\Delta} \Gamma_1$.
Moreover, if for each $i\in\{0,1\}$ the convex set $\C_i$ has strictly convex nonideal boundary in~$\Omega$, then the action of $\Gamma$ on~$\Omega$ is convex cocompact.
\end{theorem}

\begin{theorem} \label{thm:HNN-cc}
In the setting of Theorem~\ref{thm:HNN}, suppose that there exists a $\Gamma_0$-invariant closed convex subset $\C_0$ of~$\Omega_0$ such that
\begin{enumerate}
  \item for every $\gamma \in \Gamma_0$, the triple of properly convex open sets $(t \cdot \Int{\C_0}, \Int{\C_0}, \gamma t^{-1} \cdot\nolinebreak \Int{\C_0})$ is in occultation position, and
  \item for every $\varepsilon \in \{1,-1\}$ and every $\gamma' \in \Gamma_0 \smallsetminus t^{\varepsilon} \Gamma_0 t^{-\varepsilon}$, the triple of properly convex open sets $(t^{\varepsilon} \cdot \Int{\C_0}, \ \Int{\C_0}, \gamma' t^{\varepsilon} \cdot \Int{\C_0})$ is in occultation position.
\end{enumerate}
Suppose in addition that $\C_0$ has compact quotient by~$\Gamma_0$, and that the closure of $\Conv{\C_0 \cup t\cdot\C_0} \smallsetminus (\C_0\cup t\cdot\C_0)$ in $\Conv{\Omega_0 \cup t\cdot\Omega_0}$ has compact quotient by~$\Delta$.
Then the $\Gamma$-invariant set
$$\C := \bigcup_{\gamma \in \Gamma} \rho(\gamma) \cdot \Conv{\C_0 \cup t\C_0}$$
is convex and closed in~$\Omega$, and has compact quotient by $\Gamma = (\Gamma_0)*_t$.
Moreover, if $\C_0$ has strictly convex nonideal boundary in~$\Omega$, then the action of $\Gamma$ on $\Omega$ is convex cocompact.
\end{theorem}

\subsection{Applications: free products of (naively) convex cocompact groups}

In the case of free products (without amalgamation), we use a variant of Theorem~\ref{thm:amalgam-cc} with trivial~$\Delta$ to establish the following refinement of Theorem~\ref{thm:free-product-fd}.

\begin{theorem} \label{thm:free-product-cc}
For $i=0,1$, let $\Gamma_i$ be an infinite discrete subgroup of $\PGL(V)$ preserving a properly convex open subset $\Omega_i$ of $\PP(V)$, such that the action of $\Gamma_i$ on~$\Omega_i$ is naively convex cocompact (\resp is convex cocompact) and $\Omega_i \neq \Ccore_{\Omega_i}(\Gamma_i)$.
Then there exists $g\in\PGL(V)$ such that the representation $\Gamma_0 * g\Gamma_1 g^{-1} \to \PGL(V)$ induced by the inclusions $\Gamma_0, g\Gamma_1 g^{-1} \hookrightarrow \PGL(V)$ is faithful with a discrete image which is naively convex cocompact (\resp is convex cocompact) in $\PP(V)$.
Moreover, we can take $g$ in any given lattice of $\PGL(V)$ (causing $\rho$ to take values in that lattice if it contains $\Gamma_0, \Gamma_1$), or more generally in any given Zariski-dense subgroup of $\PGL(V)$ whose proximal limit sets in $\PP(V)$ and in $\PP(V^*)$ are everything.
\end{theorem}

Note that the condition $\Omega_i \neq \Ccore_{\Omega_i}(\Gamma_i)$ is automatic if the action of $\Gamma_i$ on $\Omega_i$ is convex cocompact but $\Gamma_i$ does not \emph{divide} (\ie act properly discontinuously and cocompactly on) $\Omega_i$.
Thus, here is an immediate consequence of Theorem~\ref{thm:free-product-cc}, which was announced in \cite[Prop.\,12.4]{dgk-proj-cc}.

\begin{corollary} \label{cor:free-prod-cc}
Let $\Gamma_0$ and~$\Gamma_1$ be infinite discrete subgroups of $\PGL(V)$ which are convex cocompact in $\PP(V)$ but do not divide any nonempty properly convex open subset of $\PP(V)$.
Then there exists $g\in\PGL(V)$ such that the group generated by $\Gamma_0$ and $g\Gamma_1g^{-1}$ is isomorphic to the free product $\Gamma_0\ast\Gamma_1$ and is convex cocompact in $\PP(V)$.
\end{corollary}

\begin{remark}
It follows from the proof (see Section~\ref{subsec:proof-free-product-cc}) that if $g\in\PGL(V)$ satisfies the conclusion of Theorem~\ref{thm:free-product-cc} (\resp Corollary~\ref{cor:free-prod-cc}), then there is a neighborhood $\mathcal{U}$ of $g$ in $\PGL(V)$ such that any element of~$\mathcal{U}$ still satisfies this conclusion.
\end{remark}

In~\cite{dgk-proj-cc}, we gave an example (recalled here in Example~\ref{ex:not-irreducible}) of a group acting naively convex cocompactly in projective space, but not convex cocompactly.
However, in this example the group does not act irreducibly on projective space.
Here, we use Theorem~\ref{thm:free-product-cc} to construct an example of a group which is naively convex cocompact, but not convex cocompact, in $\PP(V)$, and which acts \emph{strongly irreducibly} on $\PP(V)$, in the sense that it does not preserve any finite union of proper projective subspaces of $\PP(V)$.

\begin{theorem} \label{thm:naive-cc-not-cc-irred}
For any $d \geq 3$, there exists a discrete subgroup $\Gamma$ of $\PGL(d+\nolinebreak 1,\RR)$ which acts strongly irreducibly on $\PP(\RR^{d+1}) = \RP^d$ and which is naively convex cocompact in $\PP(\RR^{d+1})$, but not convex cocompact in $\PP(\RR^{d+1})$. 
\end{theorem}

This was announced in \cite[\S\,4.1]{dgk-proj-cc}.
The construction is given in Example~\ref{ex:irreducible}.

We also use a variant of Theorem~\ref{thm:amalgam-cc} with trivial~$\Delta$ to show that any lattice of $\PGL(V)$ (or more generally any subgroup of $\PGL(V)$ with full proximal limit sets in $\PP(V)$ and $\PP(V^*)$) contains convex cocompact free subgroups preserving properly convex sets approaching any given convex shape, with full orbital limit set arbitrarily dense in the boundary (Proposition~\ref{prop:lattice-pingpong-subgroup}).

\subsection{Another application: convex gluing of convex projective manifolds with boundary}

We use Theorem~\ref{thm:amalgam-cc} to obtain a convex gluing theorem for convex cocompact projective manifolds with totally geodesic boundary. This generalizes the gluing construction used by Goldman~\cite{gol90} to construct closed convex projective surfaces from convex pairs of pants, the gluing construction used by Ballas--Danciger--Lee \cite{bdl18} to produce examples of closed convex projective 3-manifolds with embedded totally geodesic tori, and the gluing construction used by Blayac--Viaggi \cite{bv} to produce divisible convex sets with properly embedded cones in all dimensions.

\begin{theorem}[Convex gluing] \label{thm:gluing}
For $i = 0,1$, let $\Gamma_i$ be an infinite discrete subgroup of $\PGL(V)$ acting convex cocompactly on some properly convex open subset $\Omega_i$ of $\PP(V)$.
Suppose that there is a projective hyperplane $X$ of $\PP(V)$ such that
\begin{itemize}
  \item $\Ccore_{\Omega_0}(\Gamma_0) \cap  \Ccore_{\Omega_1}(\Gamma_1) = X \cap \Omega_0 = X \cap \Omega_1 \neq \varnothing$;
  \item the stabilizer $\Delta$ of $X \cap \Omega_0$ in~$\Gamma_0$ is equal to the stabilizer of $X \cap \Omega_0$ in~$\Gamma_1$;
  \item $\Delta$ divides $X \cap \Omega_0$ and fixes a point $z \in \PP(V) \smallsetminus (X \cup \Lambdao_{\Omega_0}(\Gamma_0) \cup \Lambdao_{\Omega_1}(\Gamma_1))$.
\end{itemize}
Then the representation $\Gamma_0 *_\Delta \Gamma_1 \to \PGL(V)$ induced by the inclusions $\Gamma_0, \Gamma_1 \hookrightarrow \PGL(V)$ is discrete and faithful, and its image is convex cocompact in $\PP(V)$.
\end{theorem}

As in \cite{bdl18} and \cite{bv}, we apply this to give a doubling construction for convex cocompact projective manifolds. 

\begin{corollary}[Doubling theorem] \label{cor:double}
Let $\Gamma$ be an infinite discrete subgroup of $\PGL(V)$ acting convex cocompactly on some properly convex open subset $\Omega$ of $\PP(V)$.
Suppose that there is a projective hyperplane $X$ of $\PP(V)$ such that $X \cap \Omega$ is nonempty and contained in $\partialn \Ccore_\Omega(\Gamma)$, and that $\Delta := \mathrm{stab}_{\Gamma}(X)$ divides $X \cap \Omega$ and fixes a point $z \in \PP(V) \smallsetminus \Omega$.
Then the double of $\Gamma$ over $\Delta$ admits a discrete and faithful representation $\Gamma *_\Delta \Gamma \to \PGL(V)$ whose image is convex cocompact in $\PP(V)$. 
\end{corollary}

Observe that Corollary~\ref{cor:double} may be applied to recover the following fact, which also follows from~\cite{gol90}: if $\Gamma$ is a torsion-free infinite discrete subgroup of $\SL(3,\RR)$ which is convex cocompact in $\PP(\RR^3)$, then $\Gamma$ is contained in the image of some $\SL(3,\RR)$-Hitchin representation of a closed surface group.

\subsection{Virtual amalgamation over convex cocompact subgroups} \label{subsec:intro-virt-amalgam}

In Section~\ref{sec:virtual}, we use the above Theorems \ref{thm:amalgam} and~\ref{thm:amalgam-cc} to prove the following general \emph{virtual amalgamation theorem} in the spirit of Baker--Cooper~\cite{bc08}.
Recall that a subgroup $\Delta$ of a group $\Gamma$ is called \emph{separable} if for any $\gamma \in \Gamma \smallsetminus \Delta$, there exists a finite-index subgroup $H$ of~$\Gamma$ which contains~$\Delta$ but not~$\gamma$.
We refer to Section~\ref{subsec:P1-div} for the notion of $P_1$-divergence.

\begin{theorem} \label{thm:hopeful}
For $i=0,1$, let $\Gamma_i$ be a discrete subgroup of $\PGL(V)$ preserving a nonempty properly convex open subset $\Omega_i$ of $\PP(V)$, with $\Omega_{01} := \Omega_0 \cap \Omega_1 \neq \varnothing$.
Suppose that
\begin{enumerate}[label=(\alph*)]
  \item\label{item:hopeful-separable} $\Delta := \Gamma_0 \cap \Gamma_1$ is separable in both $\Gamma_0$ and~$\Gamma_1$,
  \item\label{item:hopeful-CC-Omega01} if $\Delta$ is infinite, then it acts convex cocompactly on $\Omega_{01} = \Omega_0 \cap \Omega_1$,
  \item\label{item:hopeful-lim-set} $\varnothing \neq \overline{\Ccore_{\Omega_i}(\Gamma_i)} \smallsetminus \Lambdao_{\Omega_{01}}(\Delta) \subset \Omega_{1-i}$ for both $i=0,1$,
  \item\label{item:hopeful-cases} for both $i=0,1$, the ideal boundary $\partiali\Ccore_{\Omega_i}(\Gamma_i) = \overline{\Ccore_{\Omega_i}(\Gamma_i)} \smallsetminus \Ccore_{\Omega_i}(\Gamma_i)$ of $\Ccore_{\Omega_i}(\Gamma_i)$ is a union of open faces of $\partial\Omega_i$, and any ray in $\Omega_i$ with endpoint in $\partiali\Ccore_{\Omega_i}(\Gamma_i)$ eventually enters any uniform neighborhood of $\Ccore_{\Omega_i}(\Gamma_i)$ in $(\Omega_i,d_{\Omega_i})$.
\end{enumerate}
Then there exist finite-index subgroups $\Gamma'_0$ of~$\Gamma_0$ and $\Gamma'_1$ of~$\Gamma_1$, each containing~$\Delta$, such that
\begin{enumerate}
  \item\label{item:hopeful-1} the representation $\rho : \Gamma'_0 *_{\Delta} \Gamma'_1 \to \PGL(V)$ induced by the inclusions $\Gamma'_0, \Gamma'_1 \hookrightarrow \PGL(V)$ is discrete and faithful, and its image preserves a nonempty properly convex open subset $\Omega$ of $\PP(V)$ containing $\Omega_{01} \cup \Ccore_{\Omega_0}(\Gamma_0) \cup \Ccore_{\Omega_1}(\Gamma_1)$;
  \item\label{item:hopeful-2} if $\Gamma_i$ is $P_1$-divergent for both $i=0,1$, then so is $\rho(\Gamma'_0 *_{\Delta} \Gamma'_1)$;
  \item\label{item:hopeful-3} if the action of $\Gamma_i$ on~$\Omega_i$ is convex cocompact  for both $i=0,1$, then so is the action of $\Gamma'_0 *_{\Delta} \Gamma'_1$ on $\Omega$ via~$\rho$.
\end{enumerate}
Moreover, for any neighborhood $\mathcal{U}$ of $\partiali\Ccore_{\Omega_0}(\Gamma_0) \cup \partiali\Ccore_{\Omega_1}(\Gamma_1)$ in $\PP(V)$, we can choose $\Gamma'_0$ and~$\Gamma'_1$ to satisfy \eqref{item:hopeful-1}, \eqref{item:hopeful-2}, \eqref{item:hopeful-3}, and
\begin{enumerate}\setcounter{enumi}{3}
  \item\label{item:hopeful-4} $\Lambdao_{\Omega}(\rho(\Gamma'_0 *_{\Delta} \Gamma'_1)) \subset \mathcal{U}$.
\end{enumerate}
\end{theorem}

Assumption~\ref{item:hopeful-cases} is satisfied in many situations, for instance if for each $i\in\{0,1\}$, the action of $\Gamma_i$ on~$\Omega_i$ is convex cocompact or every point of $\Lambdao_{\Omega_i}(\Gamma_i)$ is extremal and $C^1$ in~$\overline{\Omega_i}$ (see Lemma~\ref{lem:hopeful-cases-satisfied} below). In both such cases, the ideal boundary $\partiali \Ccore_{\Omega_i}(\Gamma_i)$ of the convex core is equal to the full orbital limit set $\Lambdao_{\Omega_i}(\Gamma_i)$.

Theorem~\ref{thm:hopeful} applies in particular to trivial~$\Delta$.
On the other hand, in the setting of Theorem~\ref{thm:hopeful}, when $\Delta$ is infinite, its action on $\Omega$ via~$\rho$ is still convex cocompact (see Fact~\ref{fact:cc-subsets}).

We use Theorem~\ref{thm:hopeful} to establish the following.

\begin{proposition} \label{prop:amalgam-word-hyperbolic}
For $i = 0,1$, let $d_i\geq 3$ be an integer, let $\Gamma_i$ be an infinite discrete subgroup of $\SL(d_i, \RR)$, and let $\Delta_i$ be a subgroup of~$\Gamma_i$ such that
\begin{itemize}
  \item $\Gamma_i$ is convex cocompact in $\PP(\RR^{d_i})$ and Gromov hyperbolic,
  \item $\Delta_i$ is separable in $\Gamma_i$,
  \item $\Delta_i$ preserves a nontrivial decomposition $\RR^{d_i} = V_i \oplus W_i$, such that $\Delta_i$ divides a nonempty properly convex open subset of $\PP(V_i)$ and acts trivially on~$W_i$,
  \item there is a linear isomorphism $\varphi: V_0 \to V_1$ inducing an isomorphism $\varphi_\Delta: \Delta_0 \to \Delta_1$.
\end{itemize}
Let $V$ be the finite-dimensional real vector space $V_0 \oplus W_0 \oplus W_1$.
Then there exist, for each $i \in \{0,1\}$, a finite-index subgroup $\Gamma_i'$ of~$\Gamma_i$, containing $\Delta_i$, and a discrete and faithful representation $\rho : \Gamma_0' *_{\varphi_\Delta} \Gamma_1' \to \SL(V)$ of the amalgamated free product whose image is convex cocompact in $\PP(V)$.
\end{proposition}

We note that if $\Gamma_i$ is torsion-free and $\Delta_i$ is a maximal cyclic subgroup of~$\Gamma_i$, then the separability assumption in Proposition~\ref{prop:amalgam-word-hyperbolic} is always satisfied (see \cite[Prop.\,3.2]{tt25}), and the amalgamated free product $\Gamma_0' *_{\varphi_\Delta} \Gamma_1'$ is Gromov hyperbolic (see Fact~\ref{fact:amalgam-Gromov-hyp}).

Proposition~\ref{prop:amalgam-word-hyperbolic} applies, for example, in the case that $d_0 = d_1 = 3$, that $\Gamma_0$ and~$\Gamma_1$ are surface groups dividing properly convex open subsets of $\PP(\RR^3)$ (\ie $\Gamma_0$ and~$\Gamma_1$ are images of Hitchin representations of surface groups into $\SL(3,\RR)$, by \cite{cg93}), and that $\Delta_0$ and $\Delta_1$ are conjugate maximal cyclic subgroups with symmetric eigenvalues; then $V = V_0 \oplus W_0 \oplus W_1 \simeq \RR^4$.
See Example~\ref{ex:Hitchin}.

\subsection{Applications to Anosov representations}

Anosov representations of Gromov hyperbolic groups into noncompact real semisimple Lie groups are representations with finite kernel, discrete image, and strong dynamical properties, which generalize convex cocompact representations into semisimple Lie groups of real rank one (see in particular \cite{lab06,gw12,klp17,klp-survey,ggkw17,bps19}).
They have been much studied recently, especially in the setting of higher Teichm\"uller theory.
Given a noncompact real semisimple Lie group~$G$, there are various possible types of Anosov representations, corresponding to the various possible conjugacy classes of proper parabolic subgroups $P$ of~$G$.

We shall not recall the definition of Anosov representations in this paper, nor assume any technical knowledge of them.
We shall just use the following relation with convex cocompactness proved in \cite{dgk-proj-cc}, where $P_1$ denotes the stabilizer of a line of~$V$ (a parabolic subgroup of $\PGL(V)$).

\begin{fact}[{\cite[Th.\,1.15]{dgk-proj-cc}}] \label{fact:Anosov}
Let $\Gamma$ be an infinite discrete subgroup of $\PGL(V)$ preserving some nonempty properly convex open subset of $\PP(V)$.
Suppose that $\Gamma$ is Gromov hyperbolic.
Then the following are equivalent:
\begin{enumerate}
  \item $\Gamma$ is convex cocompact in $\PP(V)$ (Definition~\ref{def:cc-group}),
  \item the natural inclusion $\Gamma \hookrightarrow \PGL(V)$ is $P_1$-Anosov.
\end{enumerate}
\end{fact}

We can use this relation to translate Theorems \ref{thm:amalgam-cc}--\ref{thm:HNN-cc}--\ref{thm:free-product-cc}, Corollary~\ref{cor:free-prod-cc}, Theorem~\ref{thm:hopeful}, and Proposition~\ref{prop:amalgam-word-hyperbolic} into statements about Anosov representations.
For instance, putting together Corollary~\ref{cor:free-prod-cc}, Fact~\ref{fact:Anosov}, and the fact that the free product of two Gromov hyperbolic groups is Gromov hyperbolic, yields the following.

\begin{corollary} \label{cor:free-prod-Ano}
For $i=0,1$, let $\Gamma_i$ be an infinite discrete subgroup of $\PGL(V)$ which preserves, but does not divide, a nonempty properly convex open subset of $\PP(V)$; suppose that $\Gamma_i$ is Gromov hyperbolic and that the natural inclusion $\Gamma_i \hookrightarrow \PGL(V)$ is $P_1$-Anosov.
Then there exists $g\in\PGL(V)$ such that the representation $\rho:  \Gamma_0 * g\Gamma_1 g^{-1} \to \PGL(V)$ induced by the natural inclusions of $\Gamma_0$ and $g\Gamma_1 g^{-1}$ preserves a nonempty properly convex open subset of $\PP(V)$ and is $P_1$-Anosov.
\end{corollary}

\begin{remark} \label{rem:vcd}
If $\Gamma_i$ preserves a nonempty properly convex open subset $\Omega_i$ of $\PP(V)$, then $\mathrm{vcd}(\Gamma_i) \leq \dim \PP(V)$, and the property that $\Gamma_i$ does \emph{not} divide~$\Omega_i$ is equivalent to the strict inequality $\mathrm{vcd}(\Gamma_i) < \dim \PP(V)$, where $\mathrm{vcd}$ is the \emph{virtual cohomological dimension} (see \eg \cite[Fait\,2.6]{ben05}).
\end{remark}

Using Corollary~\ref{cor:free-prod-Ano} and appropriate linear representations, we obtain the following.

\begin{corollary} \label{cor:free-prod-Ano-general}
For $i=0,1$, let $G_i$ be a noncompact linear real semisimple Lie group and $P_{\theta_i}$ a proper parabolic subgroup of~$G_i$.
Then there exist a finite-dimensional real vector space~$V$ and, for each $i=0,1$, a representation $\tau_i : G_i\to\GL(V)$, with the following property: for any Gromov hyperbolic group $\Gamma_i$ and any $P_{\theta_i}$-Anosov representation $\rho_i : \Gamma\to G_i$, there exists $g\in\PGL(V)$ such that the representation $\rho:  \Gamma_0 * g\Gamma_1 g^{-1} \to \PGL(V)$ induced by $\tau_0\circ\rho_0$ and $\tau_1\circ\rho_1$ is $P_1$-Anosov.
\end{corollary}

We can actually construct free products of Anosov representations without passing to a linear representation, as follows.

\begin{theorem} \label{thm:Anosov-free-product-G/P}
Let $G$ be a noncompact connected real linear semisimple Lie group, $P$ a self-opposite parabolic subgroup of~$G$, and $\Gamma_0$ and~$\Gamma_1$ torsion-free discrete subgroups of~$G$ which are Gromov hyperbolic and such that the natural inclusion $\Gamma_i \hookrightarrow G$ is $P$-Anosov.
Suppose that for each $i\in\{0,1\}$ there is a point $x_i \in G/P$ which is transverse to all points of the proximal limit set $\Lambda_{\Gamma_i}^{G/P}$.
Then there exists $g \in G$ such that the representation $\rho : \Gamma_0 * g\Gamma_1 g^{-1} \to G$ induced by the inclusions $\Gamma_0, g\Gamma_1 g^{-1} \hookrightarrow G$ is $P$-Anosov.

Moreover, for any Zariski-dense subgroup $\Gamma$ of~$G$, if we can choose the points $x_0, x_1 \in G/P$ inside~$\Lambda_{\Gamma}^{G/P}$, then we can choose $g$ inside~$\Gamma$.
\end{theorem}

Note that the assumption that there be a point of $G/P$ transverse to all points of $\Lambda_{\Gamma_i}^{G/P}$ (\ie that the proximal limit set $\Lambda_{\Gamma_i}^{G/P}$ not ``fill'' $G/P$) is necessary: otherwise $\Lambda_{\Gamma_i}^{G/P}$ could not be part of a larger limit set of a $P$-Anosov representation.

We can also construct amalgamated free products of Anosov representations: see for instance Corollary~\ref{cor:virt-amalgam-Ano} below, which is an immediate consequence of Theorem~\ref{thm:hopeful} and of a refinement of Fact~\ref{fact:Anosov} (namely Proposition~\ref{prop:Ano-cc-P1-div}).
By putting together Proposition~\ref{prop:amalgam-word-hyperbolic} and Fact~\ref{fact:Anosov}, we also obtain the following.

\begin{corollary} \label{cor:amalgam-cyclic-Ano}
For $i = 0,1$, let $d_i\geq 3$ be an integer and let $\Gamma_i$ be an infinite Gromov hyperbolic subgroup of $\SL(d_i,\RR)$ preserving a nonempty properly convex open subset of $\PP(\RR^{d_i})$, such that the natural inclusion $\Gamma_i \hookrightarrow \SL(d_i,\RR)$ is $P_1$-Anosov.
Let $\langle\gamma_i\rangle$ be a maximal cyclic subgroup of~$\Gamma_i$ such that $\gamma_i$ has $1$ as an eigenvalue with multiplicity $d_i-2$ and such that the centralizer of $\langle \gamma_i \rangle$ in $\Gamma_i$ is $\langle \gamma_i\rangle$ (this is automatic if $\Gamma_i$ is torsion-free).
Suppose that $\gamma_0$ and~$\gamma_1$ have the same eigenvalues $\neq 1$.
Then there exist a finite-index subgroup $\Gamma'_i$ of~$\Gamma_i$, containing~$\langle\gamma_i\rangle$, and a $P_1$-Anosov representation $\rho : \Gamma'_0 *_{\langle\gamma_0\rangle = \langle\gamma_1\rangle} \Gamma'_1 \to \SL(d_0+d_1-2,\RR)$.
\end{corollary}

We note that the amalgamated free product $\Gamma'_0 *_{\langle\gamma_0\rangle = \langle\gamma_1\rangle} \Gamma'_1$ is Gromov hyperbolic by Fact~\ref{fact:amalgam-Gromov-hyp}.

Tholozan--Tsouvalas conjectured \cite[Conj.\,1.10]{tt23} that for any Anosov subgroup $\Gamma$ of a noncompact linear real semisimple Lie group~$G$ and for any maximal cyclic subgroup $\langle\gamma\rangle$ of~$\Gamma$, there exists a finite-index subgroup $\Gamma'$ of~$\Gamma$, containing $\langle\gamma\rangle$, such that the amalgamated free product $\Gamma' *_{\langle\gamma\rangle} \Gamma'$ admits an Anosov representation into some semisimple Lie group.
Corollary~\ref{cor:amalgam-cyclic-Ano} confirms this conjecture in a special case.

\subsection{Organization of the paper}

In Section~\ref{sec:proofs} we study families of properly convex open subsets of $\PP(V)$ in mutual occultation position, and give a natural construction of properly convex open set from a tree of such sets.
In Section~\ref{sec:gog} we upgrade this construction to involve group actions (Theorem~\ref{thm:tree-with-group-actions}), and prove the main Theorem~\ref{thm:general-gog} for graphs of groups, which contains Theorems \ref{thm:amalgam} and~\ref{thm:HNN} as special cases.
In Section~\ref{sec:free-prod} we prove two basic results on thickenings of convex sets, which are useful for various applications of Theorems \ref{thm:amalgam}, \ref{thm:HNN}, \ref{thm:amalgam-cc}, \ref{thm:HNN-cc}, and their generalizations; we then prove Theorems \ref{thm:free-product-fd} and~\ref{thm:free-product-G/P} and Corollaries \ref{cor:counterex-nori} and~\ref{cor:combine-discrete-groups} on free products of discrete groups.
In Section~\ref{sec:remind} we recall and establish some basic facts on properly convex sets and convex cocompactness.
In Section~\ref{sec:gog-cc} we prove the main Theorem~\ref{thm:gog-cc} concerning (naive) convex cocompactness for graphs of groups, which contains Theorems \ref{thm:amalgam-cc} and~\ref{thm:HNN-cc} as special cases; we deduce Theorems~\ref{thm:free-product-cc} and~\ref{thm:naive-cc-not-cc-irred}.
In Section~\ref{sec:gluing} we establish Theorem~\ref{thm:gluing} and Corollary~\ref{cor:double} on convex gluing of convex projective manifolds with boundary.
In Section~\ref{sec:virtual} we prove Theorem~\ref{thm:hopeful} about virtual amalgamation over convex cocompact subgroups, and in Section~\ref{sec:virt-amalg-applic} we give several applications, including 
Proposition~\ref{prop:amalgam-cc-general} which implies Proposition~\ref{prop:amalgam-word-hyperbolic}.
Finally, in Section~\ref{sec:Anosov} we prove Corollaries \ref{cor:free-prod-Ano}, \ref{cor:free-prod-Ano-general}, \ref{cor:amalgam-cyclic-Ano} and Theorem~\ref{thm:Anosov-free-product-G/P} on Anosov representations.

\subsection*{Acknowledgements}

This work started in 2017.
Theorem~\ref{thm:free-product-cc}, Corollary \ref{cor:free-prod-cc}, Theorem~\ref{thm:naive-cc-not-cc-irred}, and Example~\ref{ex:irreducible} were announced in our previous paper \cite{dgk-proj-cc}, initially under the title \emph{Examples and non-examples of convex cocompact groups in projective space}.
In \cite{dgk-proj-cc} the link with Anosov subgroups was also made, yielding Corollaries \ref{cor:free-prod-Ano} and~\ref{cor:free-prod-Ano-general}.

Part of the work was completed while JD and FG were in residence at the IHES in Bures-sur-Yvette in 2023 and while FK was in residence at the IAS in Princeton in 2024.
We thank these two institutions for their hospitality and support.

FK would like to thank Sara Maloni for inspiring discussions several years ago around Maskit's combination theorems and generalizations.
JD would like to thank Daryl Cooper for enlightening discussions from 2017 which shaped some of the key ideas in this paper. 
We are grateful to Sami Douba for many valuable questions and suggestions which helped improve our results.
We also thank Oliver Guichard for many useful comments on a previous version of this paper, Blandine Galiay for interesting discussions on Nagano spaces in relation to Corollary~\ref{cor:Anosov-free-product-proper-domain}, and Subhadip Dey for pointing out a subtlety in Corollary~\ref{cor:amalgam-cyclic-Ano}.

\section{Trees of convex sets} \label{sec:proofs}

Throughout this section, we consider a finite-dimensional real vector space $V = \RR^d$.
In Section~\ref{subsec:basic-occult} we investigate the notion of \emph{occultation position} from Definition~\ref{def:occult}, as well as some weaker version of it.
In Section~\ref{subsec:propagate-occult} we establish the key Lemmas \ref{lem:line-up} and~\ref{lem:valence-3}, about propagating occultation.
In Section~\ref{subsec:trees} we use these lemmas to establish Theorem~\ref{thm:tree}, which gives a natural construction of a properly convex open set from a tree of properly convex open sets in mutual occultation position.

\subsection{Basic tool: occultation position} \label{subsec:basic-occult}

The following elementary observation justifies the terminology of \emph{occultation}, when we have three convex sets $A,B,C$ such that $B$ prevents $A$ and $C$ from ``seeing'' each other.

\begin{lemma} \label{lem:occult}
Let $(A, B, C)$ be a triple of properly convex open subsets of $\PP(V)$ satisfying
\begin{enumerate}[label=(O\arabic*)]
  \item $A \not\subset B$ and $C \not\subset B$,
  \item $A \cap B \neq \varnothing$ and $B \cap C \neq \varnothing$, but $A \cap C = \varnothing$.
\end{enumerate}
Then the following are equivalent:
\begin{enumerate}[label=(O3-\roman*)]
  \item\label{item:occult-3-i} $(A, B, C)$ is in occultation position (Definition~\ref{def:occult}), \ie $\overline{B^*} \subset A^* \cup C^*$, or in other words, any projective hyperplane of $\PP(V)$ meeting both $\overline{A}$ and~$\overline{C}$ must meet~$B$;
  \item\label{item:occult-3-ii} any projective line of $\PP(V)$ meeting both $\overline{A}$ and~$\overline{C}$ must meet~$B$.
 \end{enumerate}
\end{lemma}

\begin{proof}
The implication \ref{item:occult-3-ii}~$\Rightarrow$~\ref{item:occult-3-i} follows from the fact that any projective hyperplane of $\PP(V)$ meeting both $\overline{A}$ and~$\overline{C}$ contains a projective line with the same property.
For the implication \ref{item:occult-3-i}~$\Rightarrow$~\ref{item:occult-3-ii}, observe that, since $B$ is convex, any projective line of $\PP(V)$ missing~$B$ is contained in a projective hyperplane of $\PP(V)$ missing~$B$; if the line meets both $\overline{A}$ and~$\overline{C}$, then so does the hyperplane.
\end{proof}

In the important Lemma~\ref{lem:invisible} below, we do not need the full strength of occultation position: we can replace the condition $\overline{B^*} \subset A^* \cup C^*$ by the weaker condition $B^* \subset A^* \cup C^*$, yielding a notion of \emph{weak occultation}.
The following lemma is proved in the same way as Lemma~\ref{lem:occult}.

\begin{lemma} \label{lem:weak-occult}
Let $(A, B, C)$ be a triple of properly convex open subsets of $\PP(V)$ satisfying
\begin{enumerate}[label=(O\arabic*)]
  \item\label{item:occult-1} $A \not\subset B$ and $C \not\subset B$,
  \item\label{item:occult-2} $A \cap B \neq \varnothing$ and $B \cap C \neq \varnothing$, but $A \cap C = \varnothing$.
\end{enumerate}
Then the following are equivalent:
\begin{enumerate}[label=(o3-\roman*)]
  \item\label{item:weak-occult-3-i} $B^* \subset A^* \cup C^*$, \ie any projective hyperplane of $\PP(V)$ meeting both $\overline{A}$ and~$\overline{C}$ must meet~$\overline{B}$;
  \item\label{item:weak-occult-3-ii} $\overline{B^*} \subset \overline{A^*} \cup \overline{C^*}$, \ie any projective hyperplane of $\PP(V)$ meeting both $A$ and~$C$ must meet~$B$;
  \item\label{item:weak-occult-3-iii} any projective line of $\PP(V)$ meeting both $\overline{A}$ and~$\overline{C}$ must meet~$\overline{B}$;
  \item\label{item:weak-occult-3-iv} any projective line of $\PP(V)$ meeting both $A$ and~$C$ must meet~$B$.
\end{enumerate}
\end{lemma}

\begin{definition} \label{def:weak-occult}
A triple $(A, B, C)$ of properly convex open subsets of $\PP(V)$ is \emph{in weak occultation position} if it satisfies conditions \ref{item:occult-1}, \ref{item:occult-2}, and \ref{item:weak-occult-3-i}--\ref{item:weak-occult-3-iv} of Lemma~\ref{lem:weak-occult}.
\end{definition}

\begin{lemma}[Invisibility Lemma] \label{lem:invisible}
Let $A, B, C \subset \PP(V)$ be properly convex sets such that the triple $(A,B,C)$ is in weak occultation position.
Then
\begin{enumerate}
  \item \label{invis:1} $A, B, C$ all lie in a common affine chart, $A\cup B\cup C$ is connected, and the convex hull $\Conv{A \cup B \cup C}$ is well defined and properly convex,
  \item \label{invis:2} $\Conv{A \cup B} \cup \Conv{B \cup C} = \Conv{A \cup B \cup C}$,
  \item \label{invis:3} $\Conv{A \cup B} \cap \Conv{B \cup C} = B$,
  \item \label{invis:4} $A \not \subset \Conv{B\cup C}$ and $C \not \subset \Conv{A\cup B}$.
\end{enumerate}
\end{lemma}

\begin{proof}
\eqref{invis:1} By condition~\ref{item:weak-occult-3-i} of Lemma~\ref{lem:weak-occult}, we have $B^* = (B^* \cap A^*) \cup (B^* \cap C^*)$. 
By condition~\ref{item:occult-1} we have $B^* \not\subset C^*$, hence $B^* \cap A^*\neq \varnothing$. 
Similarly, $B^* \cap C^*\neq \varnothing$. 
Since $B^*$ is connected, the open sets $B^* \cap A^*$ and $B^* \cap C^*$ are not disjoint, hence their intersection $A^* \cap B^* \cap C^*$ contains a point $X\in \PP(V^*)$.
If we view $X$ as a projective hyperplane in $\PP(V)$, then the convex sets $A,B,C$ are contained and bounded in the affine chart $\PP(V)\smallsetminus X$.
Moreover, their union is connected since $A\cap B$ and $B\cap C$ are both nonempty by condition~\ref{item:occult-2}.
Thus the convex hull $\Conv{A \cup B \cup C}$ is well defined and properly convex.

\eqref{invis:2} We work in an affine chart containing $\Conv{A \cup B \cup C}$, as given by~\eqref{invis:1}.
The inclusion $\Conv{A \cup B} \cup \Conv{B \cup C} \subset \Conv{A \cup B \cup C}$ is clear.
Let us prove the reverse inclusion.
Consider a point $z \in \Conv{A \cup B \cup C}$.
It is a barycentric combination of a point $x\in \Conv{A \cup C}$ and a point $y\in B$.
In order to prove that $z \in \Conv{A \cup B} \cup \Conv{B \cup C}$, it is enough to prove that $x$ belongs to $\Conv{A \cup B}$ or $\Conv{B \cup C}$ (since $y\in B$ belongs to both these convex sets).
The point $x \in \Conv{A \cup C}$ is contained in a closed interval $[a,c]$ with $a\in A$ and $c\in C$.
By condition~\ref{item:weak-occult-3-iv} the projective line through $a$ and~$c$ contains a point $b\in B$.
In our affine chart, we have $x \in  [a, b] \subset \Conv{A \cup B}$ or $x \in [b, c] \subset \Conv{B \cup C}$.
This shows that $\Conv{A \cup B \cup C} \subset \Conv{A \cup B} \cup \Conv{B \cup C}$.

\eqref{invis:3} We work again in an affine chart containing $\Conv{A \cup B \cup C}$, as given by~\eqref{invis:1}.
First, recall from condition~\ref{item:weak-occult-3-iv} that every projective line through two points $a \in A$ and $c \in C$ meets~$B$.
We claim that in fact, the interval $(a,c)$ meets~$B$.
Indeed, suppose by contradiction that this is not the case.
Then, up to switching the roles of $a$ and $c$, we can find an affine hyperplane $H$ through $c$, transverse to $[a,c]$, separating $a$ from~$B$.
Consider any point $a_0\in A\cap B$ (as given by condition~\ref{item:occult-2}).
The segment $[a,a_0]$ intersects $H$ at a point $a'\in A$, hence $H$ intersects both $A$ and~$C$, but misses~$B$: contradiction with condition~\ref{item:weak-occult-3-ii}.

Consider a point $x \in \Conv{A \cup B} \cap \Conv{B \cup C}$.
There exist $(a,b,b',c) \in A\times B^2\times C$ such that $x \in [a,b] \cap [b',c]$. 
By the claim above, there also exists a point $b'' \in (a,c) \cap B$. 
Observe that $x \in \Conv{\{b,b',b''\}}$.
This shows that $\Conv{A \cup B} \cap \Conv{B \cup C} \subset B$, hence $\Conv{A \cup B} \cap \Conv{B \cup C} = B$. 

\eqref{invis:4} This is a direct consequence of~\eqref{invis:3}: if $A$ were contained in $\Conv{B\cup C}$, then the left-hand side of~\eqref{invis:3} would be $\Conv{B\cup C}$, hence~\eqref{invis:3} would state $C\subset B$, contradicting condition~\ref{item:occult-1}.
Thus $A \not \subset \Conv{B\cup C}$.
Similarly, $C \not \subset \Conv{A\cup B}$.
\end{proof}

\subsection{Propagating occultation position} \label{subsec:propagate-occult}

In this section we establish the key Lemmas \ref{lem:line-up} and~\ref{lem:valence-3}, dealing with four convex sets $A,B,C,D$ satisfying certain mutual occultations: we show that triples of the form $(A,\Conv{B,C},D)$ or $(A,B,\Conv{C,D})$ are then themselves in occultation position.
This will be the bootstrap for an induction in the next section.
We note that in Lemmas \ref{lem:line-up} and~\ref{lem:valence-3} weak occultation is not enough: we need the full strength of occultation (see Figure~\ref{fig:4sets}).

\begin{lemma}[Lining Up] \label{lem:line-up}
Suppose $\dim \PP(V) \geq 2$. 
Let $A,B,C,D$ be properly convex open subsets of $\PP(V)$ such that the triples $(A,B,C)$ and $(B,C,D)$ are in occultation position.
Then
\begin{enumerate}
  \item \label{line-up:1} the set $\Conv{A \cup B \cup C \cup D}$ is well defined and properly convex;
  \item \label{line-up:ABcupCD} the triple $(A, \Conv{B\cup C}, D)$ is in occultation position.
  \item \label{line-up:ABCcupD} the triple $(A, B, \Conv{C\cup D})$ is in occultation position (and similarly for the triple\linebreak $(\Conv{A\cup B},C,D)$) .
\end{enumerate}
\end{lemma}

\begin{figure}[h]
\labellist
\small\hair 2pt
\pinlabel {\ref{lem:line-up}.\eqref{line-up:1}} at 				12	100 
\pinlabel {\ref{lem:line-up}.\eqref{line-up:ABcupCD}} at 			145	100 
\pinlabel {\ref{lem:line-up}.\eqref{line-up:ABCcupD}} at 		276	100 
\pinlabel {\ref{lem:valence-3}.\eqref{valence-three:1}} at 		12	69 
\pinlabel {\ref{lem:valence-3}.\eqref{valence-three:234}} at 	133	69 
\pinlabel {$A$} at 				15	132 
\pinlabel {$B$} at 				45	132 
\pinlabel {$C$} at 				80	132 
\pinlabel {$D$} at 				112	132 
\pinlabel {$A$} at 				145	132 
\pinlabel {$\Conv{B\cup C}$} at 		195	132 
\pinlabel {$D$} at 				243	132 
\pinlabel {$A$} at 				275	132 
\pinlabel {$B$} at 				305	132 
\pinlabel {$\Conv{C\cup D}$} at 	355	132 
\pinlabel {$A$} at 				13	14 
\pinlabel {$B$} at 				53	36 
\pinlabel {$C$} at 				54	75 
\pinlabel {$D$} at 				94	14 
\pinlabel {$A$} at 				134	14 
\pinlabel {$\Conv{B\cup C}$} at 		175	40 
\pinlabel {$D$} at 				215	14 
\pinlabel {${}_{A'}$} at 				272	73 
\pinlabel {${}_{B'}$} at 				273	52 
\pinlabel {${}_{C'}$} at 				282	35 
\pinlabel {${}_{D'}$} at 				287	18 
\pinlabel {${}_{A'}$} at 				360	75 
\pinlabel {${}_{B'}$} at 				339	57 
\pinlabel {${}_{C'}$} at 				337	33 
\pinlabel {${}_{D'}$} at 				360	16 
\endlabellist
\includegraphics[width = 0.99 \textwidth]{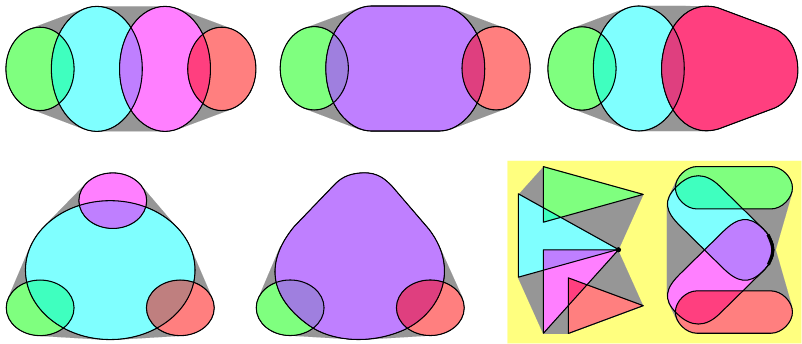}
\caption{Illustrations of Lemmas~\ref{lem:line-up} (top) and~\ref{lem:valence-3} (bottom).
Yellow box: configurations showing how Lemma~\ref{lem:line-up} can fail if we replace the occultation condition $\overline{B^*} \subset A^*\cup C^*$ by the weaker one $B^* \subset A^*\cup C^*$.}
\label{fig:4sets}
\end{figure}

\begin{proof}
\eqref{line-up:1} Since $(A,B,C)$ is in occultation position, Lemma~\ref{lem:invisible}.\eqref{invis:1} states that $\Conv{B \cup C}$ is a well-defined properly convex open set. 
Its dual $\Conv{B \cup C}^* = B^* \cap C^*$ has boundary
\begin{equation} \label{eqn:bound-B*-cap-C*}
\partial(B^* \cap C^*) = (\partial B^* \cap C^*) \sqcup (B^* \cap \partial C^*) \sqcup (\partial B^* \cap \partial C^*)
\end{equation}
(see Figure~\ref{fig:4sets}, top left).
Since $B \not\subset C$ and $C \not \subset B$, we have $C^* \not\subset B^*$ and $B^* \not \subset C^*$, hence the first two sets in this disjoint union are nonempty. 
These two sets are also open in the space $\partial(B^* \cap C^*)$, which is connected (it is a topological sphere of dimension $\dim \PP(V)-1 \geq 1$). 
Therefore the third set in the disjoint union is also nonempty, \ie there exists $X \in \partial B^* \cap \partial C^*$. 
Since $(A,B,C)$ is in occultation position, we have $\overline{B^*} \subset A^* \cup C^*$; since $X \in \overline{B^*} \smallsetminus C^*$, we deduce $X \in A^*$.
Similarly, since $(B,C,D)$ is in occultation position, we have $\overline{C^*} \subset B^* \cup D^*$; since $X \in \overline{C^*} \smallsetminus B^*$, we deduce $X \in D^*$.
Thus $X$ does not intersect $\overline{A}$, $B$, $C$, nor $\overline{D}$.
In an affine chart where the properly convex set $\Conv{B\cup C}$ is bounded, $X$ is tangent to both $B$ and $C$.
A small affine translation in a direction transverse to~$X$, away from the side that contains $\Conv{B\cup C}$, takes $X$ to a hyperplane $X'$ still disjoint from $\overline{A}$ and~$\overline{D}$, but also from $\overline{B}$ and~$\overline{C}$. 
The set $A\cup B \cup C \cup D$ is contained and bounded in the affine chart $\PP(V)\smallsetminus X'$,  and connected.

We prove \eqref{line-up:ABcupCD} that the triple $(A, \Conv{B\cup C}, D)$ is in occultation position  (see Figure~\ref{fig:4sets}, top middle), \ie that
\begin{enumerate}[label=(\alph*)]
  \item \label{line-up:2} $A \not\subset \Conv{B \cup C}$ and $D \not\subset \Conv{B \cup C}$;
  \item \label{line-up:3} $\Conv{B \cup C}$ is properly convex and $\overline{\Conv{B \cup C}^*}  \subset A^* \cup D^*$;
  \item \label{line-up:4} $A \cap \Conv{B \cup C} \neq \varnothing$, $D \cap \Conv{B \cup C} \neq \varnothing$, but $A \cap D = \varnothing$
\end{enumerate}
and~\eqref{line-up:ABCcupD} that the triple $(A, B, \Conv{C\cup D})$ is in occultation position  (see Figure~\ref{fig:4sets}, top right), \ie that
\begin{enumerate}[label=(\alph*')]
  \item \label{line-up:2prime} $A \not\subset B$ and $\Conv{C \cup D} \not \subset B$;
  \item \label{line-up:3prime} $\Conv{C \cup D}$ is properly convex and $\overline{B^*}  \subset A^* \cup \Conv{C \cup D}^*$;
  \item \label{line-up:4prime} $A \cap B  \neq \varnothing$, $\Conv{C \cup D} \cap B \neq \varnothing$, but $A \cap \Conv{C \cup D} = \varnothing$.
\end{enumerate}

Items~\ref{line-up:2} and~\ref{line-up:2prime} follow immediately from Lemma~\ref{lem:invisible}.\eqref{invis:4}, applied to $(A,B,C)$ and $(B,C,D)$ respectively.

We now check item~\ref{line-up:3}.
The fact that $\Conv{B\cup C}$ is properly convex follows from Lemma~\ref{lem:invisible}.\eqref{invis:1}. 
In order to show that $\overline{\Conv{B\cup C}^*} \subset A^*\cup D^*$, we first check that $\partial \Conv{B \cup C}^* \subset A^{*} \cup D^{*}$.
Note that $\partial \Conv{B\cup C}^* \subset (\overline{B^*} \smallsetminus C^*) \cup (\overline{C^*} \smallsetminus B^*)$. Indeed, use~\eqref{eqn:bound-B*-cap-C*} and observe that  $\partial B^* \cap C^*$ is contained in $\overline{C^*} \smallsetminus B^*$,  that $B^* \cap \partial C^*$ is contained in $\overline{B^*} \smallsetminus C^*$, and that $\partial B^* \cap \partial C^*$ is contained in both.
Moreover, $\overline{B^*} \smallsetminus C^* \subset A^*$ by the occultation assumption on $(A,B,C)$, and $\overline{C^*} \smallsetminus B^* \subset D^*$ by the occultation assumption on $(B,C,D)$.
Therefore $\partial \Conv{B \cup C}^* \subset A^{*} \cup D^{*}$.
We now check that $\Conv{B\cup C}^* \subset A^*\cup D^*$.
Consider any point $X \in \Conv{B\cup C}^*$.
The above covering of $\partial \Conv{B \cup C}^*$ by two open sets $A^*$ and~$D^*$ induces a covering of the sphere of visual directions at any point $p\in \Conv{B\cup C}^*$.
One of the two terms of the covering must have spherical area larger than one-half of the full visual sphere, hence contain a pair of opposite directions; in other words, there exist two distinct points $X_1,X_2 \in \partial \Conv{B \cup C}^*$, with $X_1,X,X_2$ aligned in $\PP(V^*)$, such that $X_1$ and~$X_2$ both belong to~$A^*$ or both belong to~$D^*$.
Suppose they both belong to~$A^*$.
Viewing $X_1$ and~$X_2$ as projective hyperplanes of $\PP(V)$, the complement $\PP(V) \smallsetminus (X_1 \cup X_2)$ has two connected components, and the convex sets $\overline{A}$ and $\Conv{B \cup C}$ are both contained in one of them, in fact in the same component since $A \cap B \neq \varnothing$.
The hyperplane $X \in \Conv{B \cup C}^*$ does not intersect this component, hence it does not intersect~$\overline{A}$; thus $X \in A^*$.
Similarly, if $X_1$ and~$X_2$ both belong to~$D^*$, then so does~$X$.
This completes the proof that $\overline{\Conv{B\cup C}^*} \subset A^*\cup D^*$.

We now check item~\ref{line-up:3prime}.
Again, the fact that $\Conv{C\cup D}$ is properly convex follows from Lemma~\ref{lem:invisible}.\eqref{invis:1}.
As in the proof of~\ref{line-up:3}, in order to show that $\overline{B^*} \subset A^* \cup \Conv{C\cup D}^*$, it suffices to check that $\partial B^* \subset A^* \cup \Conv{C\cup D}^{*} = A^* \cup (C^* \cap D^*)$.
For this, observe that $\partial B^* \smallsetminus A^* \subset C^*$, since $(A,B,C)$ is in occultation position, and $C^* \subset \overline{C^*} \subset B^* \cup D^*$, since $(B,C,D)$ is in occultation position.
Since $\partial B^* \cap B^* = \varnothing$, we have $\partial B^* \smallsetminus A^* \subset C^* \cap D^*$. 

Finally, we check items \ref{line-up:4} and~\ref{line-up:4prime}.
The inclusions $A \cap \Conv{B \cup C} \neq \varnothing$, $D \cap \Conv{B \cup C} \neq \varnothing$, $A \cap B  \neq \varnothing$, and $\Conv{C \cup D} \cap B \neq \varnothing$ are immediate consequences of the fact that the triples $(A,B,C)$ and $(B,C,D)$ are in occultation position.
We now prove $A \cap \Conv{C \cup D} = \varnothing$, which implies in particular $A \cap D = \varnothing$. 
By~\eqref{line-up:1}, we can work in the affine chart containing all four sets. 
Since $A$ and $C$ are disjoint convex sets, there is a hyperplane $X$ separating $A$ from~$C$. 
Since $A \cup B \cup C$ is connected, $X$ must pass through $B$, in particular $X \notin B^*$. 
Since $X \in \overline{C^*} \smallsetminus B^*$, we must have $X \in D^*$ by the occultation assumption on $(B,C,D)$. 
Hence $D$ lies entirely on one side of~$X$, and it must be the same side as~$C$ since $C \cap D \neq \varnothing$, again by the occultation assumption on $(B,C,D)$. 
This side of~$X$ therefore also contains $\Conv{C \cup D}$, hence $\Conv{C \cup D}$ is disjoint from~$A$. 
\end{proof}

\begin{lemma}[Valence Three] \label{lem:valence-3}
Suppose $\dim \PP(V) \geq 2$. 
Let $A,B,C,D \subset \PP(V)$ be properly convex open subsets of $\PP(V)$ such that the triples $(A,B,C)$, $(A,B,D)$, and $(C,B,D)$ are all in occultation position.
Then
\begin{enumerate}
  \item \label{valence-three:1} the set $\Conv{A \cup B \cup C \cup D}$ is well defined and properly convex; 
  \item \label{valence-three:234} the triple $(A, \Conv{B \cup C}, D)$ is in occultation position. 
\end{enumerate}
\end{lemma}

We refer to Figure~\ref{fig:4sets}, bottom, for some illustrations.

\begin{proof}
\eqref{valence-three:1} It follows from Lemma~\ref{lem:invisible}.\eqref{invis:1} that $\Conv{B \cup C}$ is a well-defined properly convex open set. 
As in the proof of Lemma~\ref{lem:line-up}.\eqref{line-up:1}, using \eqref{eqn:bound-B*-cap-C*} we see that there exists $X \in \partial B^* \cap \partial C^*$.
Since $(A,B,C)$ and $(C,B,D)$ are in occultation position, we have $X \in \overline{B^*} \smallsetminus C^* \subset A^* \cap D^*$. 
Thus $X$ does not intersect $\overline{A}$, $B$, $C$, nor $\overline{D}$. 
As in the proof of Lemma~\ref{lem:line-up}.\eqref{line-up:1}, we can perturb $X$ slightly to a hyperplane $X'$ disjoint from $\overline{A} \cup \overline{B} \cup \overline{C} \cup \overline{D}$.
This shows that the set $\Conv{A \cup B \cup C \cup D}$ is well defined and properly convex.

\eqref{valence-three:234} We need to prove that:
\begin{enumerate}[label=(\alph*)]
  \item \label{valence-three:2} $\Conv{B \cup C}$ is properly convex and $\overline{\Conv{B \cup C}^*} \subset A^* \cup D^*$;
  \item \label{valence-three:3} $A \not\subset \Conv{B \cup C}$ and $D \not\subset \Conv{B \cup C}$;
  \item \label{valence-three:4} $A \cap \Conv{B \cup C} \neq \varnothing$, $D \cap \Conv{B \cup C} \neq \varnothing$, but $A \cap D = \varnothing$.
\end{enumerate}
Item~\ref{valence-three:2} follows from the inclusions $\overline{\Conv{B \cup C}^*} \subset \overline{B^*} \subset A^* \cup D^*$, where the second inclusion is due to $(A,B,D)$ being in occultation position. 
Item~\ref{valence-three:3} follows from Lemma~\ref{lem:invisible}.\eqref{invis:4} applied to $(A,B,C)$ and $(C,B,D)$. 
Item~\ref{valence-three:4} follows immediately from the fact that the triples $(A,B,C)$, $(A,B,D)$, and $(C,B,D)$ are in occultation position.
\end{proof}

\subsection{Trees of properly convex open subsets} \label{subsec:trees}

In this section we consider a countable tree $\mathscr T = (\mathscr V, \mathscr E)$ together with an assignment $v \mapsto \Omega(v)$ of properly convex open set $\Omega(v)$ to each vertex $v \in \mathscr V$ of the tree.
We call this data a \emph{tree of properly convex open sets}.

Here $\mathscr{E}$ is the set of edges of~$\mathscr{T}$; they are unoriented by convention.
We say that two vertices $v,w \in \mathscr{V}$ are \emph{adjacent} if $\{v,w\} \in \mathscr{E}$, and that three (ordered) vertices $u,v,w \in \mathscr{V}$ are \emph{consecutive} if $u,v$ are adjacent, $v,w$ are adjacent, and $u\neq w$.

\begin{figure}[h]
\includegraphics[width = 0.8\textwidth]{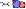}
\caption{\label{fig:tree} Left: A combinatorial tree. Right: A tree of convex sets in mutual occultation position as in Theorem~\ref{thm:tree}.}
\end{figure}

The following theorem will be central in the proofs of Theorems \ref{thm:amalgam} and~\ref{thm:HNN} and their more general version for graphs of groups (Theorem~\ref{thm:general-gog}) in Section~\ref{sec:gog}.

\begin{theorem} \label{thm:tree}
Suppose $\dim \PP(V) \geq 2$. 
Let $\mathscr T = (\mathscr V,\mathscr E)$ be a countable tree, with $|\mathscr V| \geq 3$, together with an assignment $v \mapsto \Omega(v)$ of a properly convex open set to each vertex.
Suppose that for any consecutive vertices $u,v,w \in \mathscr V$, the triple of properly convex open sets $(\Omega(u),\Omega(v),\Omega(w))$ is in occultation position.
Then for every $e = \{v, w\} \in \mathscr E$, the set $\Omega(e):= \Conv{\Omega(v) \cup \Omega(w)}$ is well defined and properly convex.
Further, 
\begin{enumerate}
  \item \label{treethm:1} the union
  $$\Omega := \bigcup_{e \in \mathscr E} \Omega(e)$$
  is a properly convex open subset of $\PP(V)$;
  \item \label{treethm:2} for any distinct $v, w \in \mathscr V$, we have $\Omega(v) \cap \Omega(w) \neq \varnothing$ if and only if $v,w$ are adjacent;
  \item \label{treethm:3} for any distinct $e,e'\in \mathscr E$, we have
  $$\Omega(e) \cap \Omega(e') = \left \{ \begin{array}{ll} \Omega(v) & \text{if $e\cap e'$ is reduced to a vertex $v$;} \\
  \varnothing & \text{otherwise.} \end{array} \right .$$
\end{enumerate}
\end{theorem}

The fact that $\Omega(e)$ is well defined and properly convex for every $e\in\mathscr{E}$ follows from Lemma \ref{lem:invisible}.\eqref{invis:1}, applied to a pair of consecutive edges containing~$e$.
In the sequel, we will use Lemmas \ref{lem:invisible}, \ref{lem:line-up}, and~\ref{lem:valence-3} inductively to prove Theorem~\ref{thm:tree}.
We first introduce some terminology.

Recall that by definition, a tree is path connected.
A nonempty subset $\mathscr V'$ of~$\mathscr V$ determines a subgraph $\mathscr T'$ of $\mathscr T$ by the rule that two vertices of $\mathscr V'$ are connected by an edge in $\mathscr T'$ if they are connected by an edge in $\mathscr T$. 
If the subgraph $\mathscr T'$ is connected, then it is itself a tree and we call it a \emph{subtree} of~$\mathscr T$. 
We say that two subtrees of~$\mathscr T$ are \emph{disjoint} if their vertex sets are disjoint, and that they are \emph{adjacent} if they are disjoint and the union of their vertex sets determines a connected subgraph (\ie a subtree).
We say that three subtrees $\mathscr T_A$, $\mathscr T_B$, $\mathscr T_C$ of~$\mathscr T$ are \emph{consecutive} if $\mathscr T_A, \mathscr T_B$ are adjacent, $\mathscr T_B, \mathscr T_C$ are adjacent, and $\mathscr T_A, \mathscr T_C$ are disjoint (and, automatically, not adjacent).
For any subtree $\mathscr{T}'=(\mathscr{V}', \mathscr{E}')$ of~$\mathscr{T}$, we define
$$\Omega(\mathscr{T'}) := \bigcup_{e\in \mathscr{E}'} \Omega(e) \cup \bigcup_{v \in \mathscr{V}'} \Omega(v).$$
Note that $\Omega(\mathscr{T'}) = \bigcup_{e\in \mathscr{E}'} \Omega(e)$ unless $\mathscr{V}'$ is reduced to a single vertex $\{v\}$ (\ie $\mathscr{E}'=\varnothing$), in which case $\Omega(\mathscr{T'}) =\Omega(v)$.

The following lemma shows that for finite subtrees $\mathscr T'$, the convex sets $\Omega(\mathscr T')$ behave much in the same way as the vertex convex sets $\Omega(v)$, $v \in \mathscr V$.
This will be used to induct on the number of vertices in the proof of the finite case of Theorem~\ref{thm:tree} (Corollary~\ref{cor:finite-tree}).

\begin{lemma} \label{lem:induction}
In the setting of Theorem~\ref{thm:tree}, suppose that $|\mathscr V| \geq 3$ is finite. 
Then for any consecutive subtrees $\mathscr T_{A} = (\mathscr V_A, \mathscr E_A)$, $\mathscr T_B = (\mathscr V_B, \mathscr E_B)$, and $\mathscr  T_C = (\mathscr V_C, \mathscr E_C)$ of~$\mathscr{T}$ with $\mathscr{V} = \mathscr V_A \sqcup \mathscr V_B \sqcup \mathscr V_C$,
$$\big ( \Omega(\mathscr{T}_A), \Omega(\mathscr{T}_B), \Omega(\mathscr{T}_C) \big )$$
is a triple of properly convex open sets  in occultation position.
\end{lemma}

\begin{proof}
We argue by induction on $n := |\mathscr V| \geq 3$.
To start with, observe that $\mathscr T_A$ and $\mathscr T_C$ play symmetric roles.
Recall that a \emph{leaf} of a combinatorial tree is a vertex of degree~$1$.

If $n=3$, then $|\mathscr{V}_A|=|\mathscr{V}_B|=|\mathscr{V}_C|=1$ and the result follows from the assumptions of Theorem~\ref{thm:tree} and from Lemma~\ref{lem:invisible}.
Given $N\geq 3$, suppose the result holds for all $n\leq N$, and let us prove it for $n=N+1$.
Since $n > 3$, at least one of the sets $\mathscr V_A$, $\mathscr V_B$, or $\mathscr V_C$ has at least two elements.

Suppose $|\mathscr V_B| \geq 2$. 
Choose a leaf $v \in \mathscr V_B$ of~$\mathscr T_B$.
Define $\mathscr V_{B_1} := \mathscr V_B \smallsetminus \{v\}$ and $\mathscr V_{B_2} := \{v\}$, let $\mathscr T_{B_1}$ and $\mathscr T_{B_2}$ be the corresponding subtrees of~$\mathscr T_B$, with respective edge sets $\mathscr E_{B_1}$ and $\varnothing$, and let $B_1 := \Omega(\mathscr{T}_{B_1})$ and $B_2 := \Omega(v)$ be the corresponding properly convex open subsets of $\PP(V)$. Let $A := \Omega(\mathscr{T}_{A})$ and $C := \Omega(\mathscr{T}_C)$. Note that $A, B_1, B_2, C$ are properly convex open sets by induction.
There are three possibilities to consider:
\begin{itemize}
  \item $v$ is adjacent in $\mathscr T$ to a vertex of exactly one of $\mathscr V_A$ or $\mathscr V_C$;
  \item $v$ is not adjacent in $\mathscr T$ to any vertex of $\mathscr V_A$ or $\mathscr V_C$;
  \item $v$ is adjacent in $\mathscr T$ to a vertex of $\mathscr V_A$ and to a vertex of $\mathscr V_C$.
\end{itemize}
If $v$ is adjacent in~$\mathscr T$ to some vertex in~$\mathscr V_C$ but to no vertex of $\mathscr V_A$, then the subtrees $\mathscr{T}_A$, $\mathscr{T}_{B_1}$, $\mathscr{T}_{B_2}$ of~$\mathscr{T}$ are consecutive, and so are the subtrees $\mathscr{T}_{B_1}$, $\mathscr{T}_{B_2}$, $\mathscr{T}_C$; by induction, the triples of properly convex sets $(A,B_1,B_2)$ and $(B_1,B_2,C)$ are in occultation position, and so the triple $(A,B,C) = (A,\Conv{B_1\cup B_2},C)$ is in occultation position by Lemma~\ref{lem:line-up}.\eqref{line-up:ABcupCD}.
Similarly, if $v$ is adjacent in~$\mathscr T$ to some vertex in~$\mathscr V_A$ but to no vertex of $\mathscr V_C$, then the triple $(A,B,C)$ is in occultation position.
Suppose $v$ is not adjacent in~$\mathscr{T}$ to any vertex in $\mathscr V_A$ nor~$\mathscr V_C$.
Then $v$ is a leaf of $\mathscr{T}$, and the subtrees $\mathscr{T}_A$, $\mathscr{T}_{B_1}$, $\mathscr{T}_{B_2}$ of~$\mathscr{T}$ are consecutive, and so are the subtrees $\mathscr{T}_A$, $\mathscr{T}_{B_1}$, $\mathscr{T}_C$ and the subtrees $\mathscr{T}_{B_2}$, $\mathscr{T}_{B_1}$, $\mathscr{T}_C$.
By induction, the triples of properly convex sets $(A,B_1,B_2)$, $(A,B_1,C)$, and $(B_2,B_1,C)$ are in occultation position, and so the triple $(A,B,C) = (A,\Conv{B_1\cup B_2},C)$ is in occultation position by Lemma~\ref{lem:valence-3}.\eqref{valence-three:234}.
Finally, suppose $v$ is adjacent to both a vertex of $\mathscr V_A$ and a vertex of $\mathscr V_C$.
Then the subtrees $\mathscr{T}_{A}, \mathscr{T}_{B_2}, \mathscr{T}_{B_1}$ are consecutive, and so are $\mathscr{T}_{C}, \mathscr{T}_{B_2}, \mathscr{T}_{B_1}$, and so are $\mathscr{T}_{A}, \mathscr{T}_{B_2}, \mathscr{T}_{C}$.
By induction, the triples of properly convex sets $(A,B_2,B_1)$, $(C,B_2,B_1)$, and $(A,B_2,C)$ are in occultation position, and so the triple $(A,B,C) = (A,\Conv{B_1\cup B_2},C)$ is in occultation position by Lemma~\ref{lem:valence-3}.\eqref{valence-three:234}.

Suppose $|\mathscr V_C| \geq 2$. 
There is only one vertex in $\mathscr V_C$ that is adjacent in~$\mathscr T$ to some vertex in~$\mathscr V_B$, hence there exists a leaf $v$ of $\mathscr T_C$ which is also a leaf of~$\mathscr T$.
Define $\mathscr V_{C_1} := \mathscr V_C \smallsetminus \{v\}$ and $\mathscr V_{C_2} := \{v\}$, let $\mathscr T_{C_1}$ and $\mathscr T_{C_2}$ be the corresponding subtrees of~$\mathscr T_C$, with respective edge sets $\mathscr E_{C_1}$ and $\varnothing$, and let $C_1 := \Omega({\mathscr{T}_{C_1}})$ and $C_2 := \Omega(v)$ be the corresponding properly convex open subsets of $\PP(V)$.
The subtrees $\mathscr{T}_A$, $\mathscr{T}_B$, $\mathscr{T}_{C_1}$ of~$\mathscr{T}$ are consecutive, and so are the subtrees $\mathscr{T}_B$, $\mathscr{T}_{C_1}$, $\mathscr{T}_{C_2}$; by induction, the triples of properly convex sets $(A,B_1,B_2)$ and $(B_1,B_2,C)$ are in occultation position, and so the triple $(A,B,C) = (A,B,\Conv{C_1\cup C_2})$ is in occultation position by Lemma~\ref{lem:line-up}.\eqref{line-up:ABCcupD}.
\end{proof}

\begin{corollary} \label{cor:finite-tree}
The conclusion of Theorem~\ref{thm:tree} holds when $|\mathscr V|$ is finite. 
\end{corollary}

\begin{proof}
Suppose $| \mathscr V| \geq 3$ is finite.
By Lemma~\ref{lem:invisible}.\eqref{invis:1}, the set $\Omega(e):= \Conv{\Omega(v) \cup\nolinebreak \Omega(w)}$ is well defined and properly convex for every edge $e = \{v, w\} \in\nolinebreak \mathscr E$.
We now check conclusions \eqref{treethm:1}, \eqref{treethm:2}, \eqref{treethm:3} of Theorem~\ref{thm:tree}.

\eqref{treethm:1} We argue by induction on $| \mathscr V|$.
Let  $\mathscr T_{A} = (\mathscr V_A, \mathscr E_A)$, $\mathscr T_B = (\mathscr V_B, \mathscr E_B)$, $\mathscr  T_C = (\mathscr V_C, \mathscr E_C)$ be any triple of consecutive subtrees of~$\mathscr{T}$ with $\mathscr{V} = \mathscr V_A \sqcup \mathscr V_B \sqcup \mathscr V_C$. 
We know from Lemma~\ref{lem:induction} that $(\Omega(\mathscr T_A), \Omega(\mathscr T_B), \Omega(\mathscr T_C))$ is a triple of properly convex open sets in occultation position.
By Lemma~\ref{lem:invisible}.\eqref{invis:1}--\eqref{invis:2},
$$\Conv{\Omega(\mathscr T_A) \cup \Omega(\mathscr T_B) \cup \Omega(\mathscr T_C)} = \Conv{\Omega(\mathscr T_A) \cup \Omega(\mathscr T_B)} \cup \Conv{\Omega(\mathscr T_B) \cup \Omega(\mathscr T_C)}$$
is a well-defined properly convex open set. 
Letting $\mathscr T_{A \cup B}$ denote the proper subtree of~$\mathscr T$ defined by $\mathscr V_A \cup \mathscr V_B$, we have by induction that $\Omega(\mathscr T_{A \cup B})$ is properly convex.
It then contains $\Conv{\Omega(\mathscr T_A) \cup \Omega(\mathscr T_B)}$, but it also is clearly contained in it; hence $\Omega(\mathscr T_{A \cup B}) = \Conv{\Omega(\mathscr T_A) \cup \Omega(\mathscr T_B)}$.
Similarly, $\Omega(\mathscr T_{B \cup C}) = \Conv{\Omega(\mathscr T_B) \cup \Omega(\mathscr T_C)}$.
By definition, $\Omega(\mathscr T) = \Omega(\mathscr T_{A \cup B}) \cup \Omega(\mathscr T_{B \cup C})$, and hence $\Omega(\mathscr T)$ is properly convex.

\eqref{treethm:2} Let $v,w \in \mathscr V$ be distinct vertices.
If they are adjacent, then there is a vertex $u \in \mathscr V$ such that $u,v,w$ are consecutive; by assumption, the triple $(\Omega(u), \Omega(v), \Omega(w))$ is in occultation position, hence $\Omega(v) \cap \Omega(w) \neq \varnothing$.
Suppose $v$ and~$w$ are not adjacent.
Let $\mathscr T_A$ and $\mathscr T_C$ be the subtrees of~$\mathscr T$ defined by $\{ v\}$ and~$\{ w\}$.
There is a subtree $\mathscr T_B$ of~$\mathscr T$ such that $\mathscr T_A,\mathscr T_B,\mathscr T_C$ are consecutive.
We know from Lemma~\ref{lem:induction} that $(\Omega(\mathscr T_A), \Omega(\mathscr T_B), \Omega(\mathscr T_C)) = (\Omega(v), \Omega(\mathscr T_B), \Omega(w))$ is a triple of properly convex open sets in occultation position; in particular, $\Omega(v) \cap \Omega(w) = \varnothing$.

\eqref{treethm:3} Let $e,e' \in \mathscr E$ be distinct edges.
If they share a vertex~$v$, then the occultation assumption and Lemma~\ref{lem:invisible}.\eqref{invis:3} imply $\Omega(e) \cap \Omega(e') = \Omega(v)$.
Suppose $e = \{ v,w\}$ and $e' = \{ v',w'\}$ do not share a vertex.
Let $\mathscr T_A$ and $\mathscr T_C$ be the subtrees of~$\mathscr T$ defined by $\{ v,w\}$ and~$\{ v',w'\}$.
There is a subtree $\mathscr T_B$ of~$\mathscr T$ such that $\mathscr T_A,\mathscr T_B,\mathscr T_C$ are consecutive.
We know from Lemma~\ref{lem:induction} that $(\Omega(\mathscr T_A), \Omega(\mathscr T_B), \Omega(\mathscr T_C)) = (\Omega(e), \Omega(\mathscr T_B), \Omega(e'))$ is a triple of properly convex open sets in occultation position; in particular, $\Omega(e) \cap \Omega(e') = \varnothing$.
\end{proof}

\begin{proof}[Proof of Theorem~\ref{thm:tree}]
The case that $|\mathscr V|$ is finite has already been established in Corollary~\ref{cor:finite-tree}, so we now consider the case that $|\mathscr V|$ is infinite (countable).
Conclusions~\eqref{treethm:2} and~\eqref{treethm:3} follow from Corollary~\ref{cor:finite-tree} as they can be checked inside a finite subtree of~$\mathscr{T}$.
We now check~\eqref{treethm:1}.
Choose an exhaustion of the countable tree $\mathscr T$ by an increasing sequence $(\mathscr T_n = (\mathscr V_n, \mathscr E_n))_{n\in\NN}$ of finite subtrees. 
By Corollary~\ref{cor:finite-tree}, Theorem~\ref{thm:tree} holds for each $\mathscr T_n$. 
This yields an increasing sequence of properly convex open sets $\Omega_n = \bigcup_{e\in\mathscr E_n} \Omega(e)$. 
Each $\Omega_n$ is contained in some affine chart $\PP(V) \smallsetminus X_n$, where $X_n$ is a projective hyperplane of $\PP(V)$. 
Up to passing to a subsequence, the hyperplanes $X_n$ converge to some projective hyperplane $X_\infty$. 
The increasing union $\Omega = \bigcup_{n\in\NN} \Omega_n$ is then contained in $\PP(V) \smallsetminus X_\infty$ and is convex in that affine chart.
It remains to prove that the convex open set $\Omega$ is properly convex.

For this, we first observe that if $v, w \in \mathscr V$ are adjacent vertices, then any hyperplane $X \in \partial \Omega(v)^* \cap \partial \Omega(w)^*$ supporting both $\Omega(v)$ and $\Omega(w)$ must support the entire convex set $\Omega$.
Indeed, suppose by contradiction that such a hyperplane $X$ intersects~$\Omega$. 
Then $X$ intersects some $\Omega(e) = \Conv{\Omega(v')\cup \Omega(w')}$ for $e = \{v',w'\} \in \mathscr{E}$, and hence $X$ intersects one of $\Omega(v'),\Omega(w')$, let us say~$\Omega(w')$. 
The vertices $v, w, w'$ lie on a unique path $\mathscr T'$ in the tree $\mathscr{T}$ with one endpoint $w'$ and the other either $v$ or $w$, let us say~$v$.
In particular, $\mathscr{T}'$ is a subtree of $\mathscr{T}$. 
The subtrees $\mathscr T_A, \mathscr T_B, \mathscr T_{C}$ of $\mathscr{T}'$ defined by $\mathscr V_A = \{v\}$, by $\mathscr V_B = \{w\}$, and by the remainder $\mathscr V_{C}$ of the vertices of~$\mathscr{T}'$ (including $w'$), are disjoint and cover the vertices of~$\mathscr T'$. 
By Lemma~\ref{lem:induction}, the corresponding triple of properly convex open sets $(A,B,C) = (\Omega(v),\Omega(w),\Omega(\mathscr T_{C}))$ is in occultation position.
In particular, $X \in \partial \Omega(v)^* \cap \partial \Omega(w)^*$ must lie in $C^*$, hence $X$ does not intersect $C \supset \Omega(w')$: contradiction.
Therefore, any $X\in \partial \Omega(v)^* \cap \partial \Omega(w)^*$ is a supporting hyperplane of $\Omega$.

We now prove that the convex open set $\Omega$ is properly convex.
Suppose by contradiction that this is not the case.
Then $\overline{\Omega}$ contains a projective subspace of $\PP(V)$.
Every supporting hyperplane to~$\Omega$ also contains this projective subspace, hence all supporting hyperplanes to $\Omega$ lie in the hyperplane $x^\perp$ of $\PP(V^*)$ consisting of all projective hyperplanes of $\PP(V)$ going through a common point~$x$.
Consider a consecutive triple $(u,v,w)$ of vertices of $\mathscr T$.
By the assumption, $\partial \Omega(u)^* \cap \partial \Omega(v)^*$ and $\partial \Omega(v)^* \cap \partial \Omega(w)^*$ are contained in $x^\perp$.
 This means that all hyperplanes tangent to both $\Omega(u)$ and $\Omega(v)$, or to both $\Omega(v)$ and $\Omega(w)$, run through the point $x \in \PP(V) \simeq \PP(\RR^d)$.
 Recall from the proof of Lemma~\ref{lem:line-up}.\eqref{line-up:1} that
$$\partial(\Omega(u)^* \cap \Omega(v)^*) = (\partial \Omega(u)^* \cap \Omega(v)^*) \sqcup (\Omega(u)^* \cap \partial \Omega(v)^*) \sqcup (\partial \Omega(u)^* \cap \partial \Omega(v)^*), $$
where all three sets on the right-hand side are nonempty, and the first two are open in $\partial(\Omega(u)^* \cap\nolinebreak \Omega(v)^*)$.
Hence $\partial \Omega(u)^* \cap \partial \Omega(v)^*$ is a closed subset of the topological sphere\linebreak $\partial(\Omega(u)^* \cap \Omega(v)^*)\cong \SS^{d-2}$ that separates it into two or more components. 
The intersection $x^\perp \cap \partial(\Omega(u)^* \cap \Omega(v)^*)$ separates $\partial(\Omega(u)^* \cap \Omega(v)^*)$ into (at most) two components, but removing any open set from $x^\perp \cap \partial(\Omega(u)^* \cap \Omega(v)^*)$ destroys this property. 
It follows that the closed subset $\partial \Omega(u)^* \cap \partial \Omega(v)^*$ is all of $x^\perp \cap \partial(\Omega(u)^* \cap \Omega(v)^*)$, and is homeomorphic to the sphere $\SS^{d-3}$ (it is the boundary of a properly convex $(d-2)$-ball). 
Dually, there exists an affine chart containing $\overline{\Omega(u)}, \overline{\Omega(v)}$ and for which $x$ is at infinity: the hyperplanes of $\partial \Omega(u)^* \cap \partial \Omega(v)^*$ cut out a \emph{tube} $U$, topologically $\SS^{d-3} \times \RR$, whose boundary is ruled by lines in the direction of $x$.
The tangent hyperplanes to $U$, which are tangent to both $\Omega(u)$ and $\Omega(v)$, sweep out the entire affine chart minus the tube~$U$, therefore $\Omega$ is contained in $U$ (and in the chart).
In particular, $\Omega(w) \subset U$. 
By the same argument, the hyperplanes of $\partial \Omega(v)^* \cap \partial \Omega(w)^*$ cut out a tube $U'$ whose boundary is ruled by lines in the direction of $x$. 
Since the first tube $U$ contains both $\Omega(v)$ and~$\Omega(w)$, we have $U' \subset U$. 
By symmetry we must also have $U \subset U'$, hence $U=U'$. 
Every supporting hyperplane $X$ to the tube $U=U'$ must intersect the boundaries of all three sets, that is $X \in \partial \Omega(u)^* \cap \partial \Omega(v)^* \cap \partial \Omega(w)^*$. 
But this contradicts occultation position. 
Therefore $\Omega$ is properly convex, i.e.~\eqref{treethm:1} holds.
\end{proof}

\section{Convex combination for graphs of groups} \label{sec:gog}

A graph of groups is a combinatorial graph whose vertices and edges are decorated with groups, together with inclusion maps from edge groups into vertex groups. 
The fundamental group of a graph of groups combines this data into one group, generalizing the construction of the fundamental group of a graph.
The goal of this section is to add extra data, namely proper actions on properly convex open sets, to a graph of groups, and then show (Theorem~\ref{thm:general-gog}) how to construct a properly convex open set on which the fundamental group of the graph of groups acts properly.

Our main tool is Theorem~\ref{thm:tree-with-group-actions}, established in Section~\ref{subsec:tree-conv-group} just below.
In Section~\ref{subsec:proof-amalgam-HNN} we use this theorem to give self-contained proofs of Theorems~\ref{thm:amalgam} and~\ref{thm:HNN}, concerning amalgamated free products and HNN extensions.
In Section~\ref{subsec:remind-gog} we recall the general definition of a graph of groups, and the construction of its fundamental group and its universal covering tree.
Finally, in Section~\ref{subsec:combin-gog} we prove our general combination result, Theorem~\ref{thm:general-gog}.

\subsection{Trees of convex sets with group actions} \label{subsec:tree-conv-group}

Here is a refinement of Theorem~\ref{thm:tree} in which each properly convex open set is endowed with a group action. 

\begin{theorem} \label{thm:tree-with-group-actions}
Suppose $\dim \PP(V) \geq 2$. 
Let $\mathscr T = (\mathscr V,\mathscr E)$ be a countable tree, with $|\mathscr V| \geq 3$, together with an assignment $v \mapsto \Omega(v)$ of a properly convex open set to each vertex.
Suppose that for any consecutive vertices $u,v,w \in \mathscr V$, the triple of properly convex open sets $(\Omega(u),\Omega(v),\Omega(w))$ is in occultation position, and let
$$\Omega := \bigcup_{e \in \mathscr{E}} \Omega(e)$$
be the properly convex open subset of $\PP(V)$ given by Theorem~\ref{thm:tree}.
Suppose further that a group $\Gamma$ acts faithfully on $\mathscr T$.
Let $\rho: \Gamma \to \PGL(V)$ be a representation such that
\begin{itemize}
  \item $\Omega(\gamma\cdot v) = \rho(\gamma) \cdot \Omega(v)$ for all $\gamma \in \Gamma$ and all $v \in \mathscr{V}$;
  \item for any $v \in \mathscr{V}$, the restriction of $\rho$ to the stabilizer $\Gamma(v)$ of $v$ in~$\Gamma$ is discrete and faithful.
\end{itemize}
Then $\rho$ is discrete and faithful and $\rho(\Gamma)$ preserves~$\Omega$. 
\end{theorem}

\begin{proof}
Since $\Gamma$ preserves~$\mathscr{E}$ and the assignment $v\mapsto\Omega(v)$ is $\rho(\Gamma)$-equivariant, the collection $\{ \Omega(e) \,|\, e\in\mathscr{E}\}$ is $\rho(\Gamma)$-invariant, and so is the set~$\Omega$.

Suppose by contradiction that $\rho$ is \emph{not} discrete and faithful.
Then there is a sequence $(\gamma_n)_{n\in\NN}$ of pairwise distinct elements of~$\Gamma$ such that $\rho(\gamma_n) \to 1 \in \PGL(V)$.
Fix a vertex $v_0 \in \mathscr{V}$.
For large enough~$n_0$, we have
$$\bigcap_{n\geq n_0} \Omega(\gamma_n\cdot v_0) = \bigcap_{n\geq n_0} \rho(\gamma_n) \cdot \Omega(v_0) \neq \varnothing.$$
By Theorem~\ref{thm:tree}.\eqref{treethm:2}, any two vertices $\gamma_m \cdot v_0$ and $\gamma_n \cdot v_0$, for $m,n\geq n_0$, are equal or adjacent. 
Since $\mathscr T$ is a tree, it has no triangles, and so in fact the tail of the sequence of vertices $(\gamma_n \cdot v_0)_{n\in \NN}$ takes at most two (adjacent) values. 
Up to passing to a subsequence, $(\gamma_n \cdot v_0)_{n\in \NN}$ is constant, equal to some vertex~$w$, and $\gamma_n \gamma_0^{-1} \in \Gamma(w)$ for all $n\in \NN$. 
But $\rho|_{\Gamma(w)}$ is discrete and faithful by assumption, contradicting $\rho(\gamma_n) \to 1$.
This shows that $\rho$ is discrete and faithful.
\end{proof}

\subsection{Application to amalgamated free products and HNN extension} \label{subsec:proof-amalgam-HNN}

We will apply Theorem~\ref{thm:tree-with-group-actions} in the context of a general graph of groups in Section~\ref{subsec:combin-gog} to prove Theorem~\ref{thm:general-gog}. 
But first, we give self-contained proofs of Theorems~\ref{thm:amalgam} and~\ref{thm:HNN}, concerning amalgamated free products and HNN extensions.
These are special cases of Theorem~\ref{thm:general-gog} (see Examples~\ref{ex:gog}).

\begin{proof}[Proof of Theorem~\ref{thm:amalgam}]
To the amalgamated free product $\Gamma =\Gamma_0 *_\Delta \Gamma_1$ is naturally associated a simplicial tree $\mathscr T = (\mathscr{V},\mathscr{E})$, the \emph{Bass--Serre tree}. 
The set $\mathscr{V} = \Gamma/\Gamma_0 \sqcup \Gamma/\Gamma_1$ of vertices is the collection of cosets of~$\Gamma_0$ and cosets of~$\Gamma_1$.
The set of edges is given by $\mathscr{E} = \big \{ \{g \Gamma_0, g \Gamma_1\} \,|\, g \in \Gamma \big \}$.
Our assumption that $\Gamma_0 \neq \Gamma_1$ means that $|\mathscr{V}|\geq 3$.

We define an assignment $v \mapsto \Omega(v)$ of a properly convex open set to each vertex: for each $g \in \Gamma$ and each $i \in \{0,1\}$, we set $\Omega(g \Gamma_i) := \rho(g) \cdot \Omega_i$.
This assignment is $\rho$-equivariant by construction.

One easily checks that any three consecutive vertices of~$\mathscr{T}$ are of the form $g\Gamma_{1-i},g\Gamma_i,g\gamma_i\Gamma_{1-i}$ where $g\in\Gamma$ and $i\in\{0,1\}$ and $\gamma_i\in\Gamma_i\smallsetminus\Delta$.
The corresponding triple
$$\big(\Omega(g\Gamma_{1-i}), \Omega(g\Gamma_i), \Omega(g\gamma_i\Gamma_{1-i})\big) = \rho(g) \cdot (\Omega_{1-i}, \Omega_i, \gamma_i\cdot\Omega_{1-i})$$
of properly convex open sets is in occultation position, because the triple $(\Omega_{1-i}, \Omega_i, \gamma_i\cdot\Omega_{1-i})$ is by assumption.

The stabilizer in~$\Gamma$ of a vertex $g \Gamma_i$ is $g \Gamma_i g^{-1}$, and the restriction $\rho|_{g \Gamma_i g^{-1}} : g\gamma g^{-1} \mapsto \rho(g) \gamma \rho(g)^{-1}$ 
is discrete and faithful by construction. 
By Theorem~\ref{thm:tree-with-group-actions}, the representation $\rho : \Gamma = \Gamma_0 *_{\Delta} \Gamma_1 \to \PGL(d,\RR)$ is discrete and faithful, and its image preserves the open set $\Omega = \bigcup_{\gamma \in \Gamma} \rho(\gamma) \cdot \Conv{\Omega_0 \cup \Omega_1}$, which is properly convex by Theorem~\ref{thm:tree}.
\end{proof}

\begin{proof}[Proof of Theorem~\ref{thm:HNN}]
To the HNN extension $\Gamma = (\Gamma_0) *_t$ is naturally associated a simplicial tree $\mathscr T = (\mathscr{V},\mathscr{E})$, the \emph{Bass--Serre tree}. 
The set $\mathscr{V} = \Gamma/\Gamma_0$ of vertices is the collection of cosets of~$\Gamma_0$.
The set of edges is given by $\mathscr{E} = \big \{ \{ g \Gamma_0, g t \Gamma_0\} \,|\, g \in \Gamma \big \}$.

We define an assignment $v \mapsto \Omega(v)$ of a properly convex open set to each vertex: for each $g \in \Gamma$, we set $\Omega(g \Gamma_0) := \rho(g) \cdot \Omega_0$.
This assignment is $\rho$-equivariant by construction.

It is an easy exercise to check that any three consecutive vertices of~$\mathscr{T}$ are, up to reversing order, either of the form 
$(gt \Gamma_0,g\Gamma_0,g\gamma t^{-1} \Gamma_0)$ where $g\in \Gamma$ and $\gamma \in \Gamma_0$, or of the form $(gt^\varepsilon \Gamma_0,g\Gamma_0,g\gamma' t^\varepsilon \Gamma_0)$ where $g\in\Gamma$ and $\varepsilon \in\{-1,1\}$ and $\gamma' \in \Gamma_0 \smallsetminus t^\varepsilon \Gamma_0 t^{-\varepsilon}$.
The corresponding triples
\begin{align*}
\big(\Omega(gt \Gamma_0), \Omega(g\Gamma_0), \Omega(g\gamma t^{-1} \Gamma_0)\big) & = \rho(g) \cdot (t \cdot \Omega_0, \Omega_0, \gamma t^{-1} \cdot \Omega_0) \\ 
\text{and }\ 
\big(\Omega(gt^\varepsilon \Gamma_0), \Omega(g\Gamma_0), \Omega(g\gamma' t^\varepsilon \Gamma_0)\big) &= \rho(g) \cdot (t^\varepsilon \cdot \Omega_0, \Omega_0, \gamma' t^\varepsilon \cdot \Omega_0)
\end{align*}
of properly convex open sets are in occultation position, because the triples $(t \cdot\Omega_0, \Omega_0, \gamma t^{-1} \cdot\Omega_0)$ and $(t^\varepsilon \cdot \Omega_0, \Omega_0, \gamma' t^\varepsilon \cdot\Omega_0)$ are by assumption.

The stabilizer in~$\Gamma$ of a vertex $g \Gamma_0$ is $g \Gamma_0 g^{-1}$, and the restriction $\rho|_{g \Gamma_0 g^{-1}} : g\gamma g^{-1} \mapsto \rho(g) \gamma \rho(g)^{-1}$ is discrete and faithful by construction.
By Theorem~\ref{thm:tree-with-group-actions}, the representation $\rho : \Gamma = (\Gamma_0) *_t \to \PGL(d,\RR)$ is discrete and faithful, and its image preserves the open set $\Omega = \bigcup_{\gamma \in \Gamma} \rho(\gamma) \cdot \Conv{\Omega_0 \cup t \cdot \Omega_0}$, which is properly convex by Theorem~\ref{thm:tree}.
\end{proof}

\subsection{Reminders: graphs of groups} \label{subsec:remind-gog}

Let us recall the notions of a graph of groups, of its fundamental group, and of its universal covering tree.
We follow Serre~\cite{ser80}, with some slight notational adjustments.

Given a graph $\mathsf{Y} = (\mathsf{V}, \mathsf{E})$, we denote by $\vec{\mathsf{E}}$ the collection of all edges with a choice of orientation. 
Given $\mathsf e \in \vec{\mathsf{E}}$, we denote by $\overline{\mathsf e} \in \vec{\mathsf{E}}$, the same edge but with opposite orientation, by $|\mathsf e| \in \mathsf{E}$ the corresponding unoriented edge, and by $\mathsf{o}(\mathsf e)$ (\resp $\mathsf{t}(\mathsf e)$) the initial (\resp terminal) vertex of~$\mathsf e$. 

A \emph{graph of groups} $(\boldsymbol{\Gamma},\mathsf{Y})$ is a connected graph $\mathsf{Y} = (\mathsf{V}, \mathsf{E})$ together with an assignment $\mathsf v \mapsto \Gamma_{\mathsf v}$ of group to each vertex, an assignment $|\mathsf e| \mapsto \Gamma_{|\mathsf e|}$ of group to each (unoriented) edge $|\mathsf e| \in \mathsf{E}$, and an injective group homomorphism $\iota_{\mathsf e} : \Gamma_{|\mathsf e|} \hookrightarrow \Gamma_{\mathsf{o}(\mathsf e)}$ for each oriented edge $\mathsf e \in \vec{\mathsf{E}}$.
For $\gamma \in \Gamma_{|\mathsf e|}$, we will use the notation $\gamma^{\mathsf e} := \iota_{\mathsf e}(\gamma)$. 
(Notations $\iota_{\mathsf e}$ and $\iota_{\overline{\mathsf e}}$ are switched compared to \cite{ser80}; the present convention is somewhat more natural for our purpose.)

Given a graph of groups $(\boldsymbol{\Gamma},\mathsf{Y})$, we denote by $F(\boldsymbol{\Gamma}, \mathsf{Y})$ the group freely generated by the vertex groups $\{\Gamma_{\mathsf v}\}_{\mathsf v \in \mathsf{V}}$ and by the oriented edges $\vec{\mathsf{E}}$, modulo the relations
\begin{itemize}
  \item $\mathsf e = \overline{\mathsf e}^{\, -1}$ for all $\mathsf e \in \vec{\mathsf{E}}$, and
  \item $\gamma^{\mathsf e} = \mathsf e \gamma^{\overline{\mathsf e}} \mathsf{e}^{-1}$ for all $\mathsf e \in \vec{\mathsf{E}}$ and all $\gamma \in \Gamma_{|\mathsf e|}$.
\end{itemize}
It is an important fact that each $\Gamma_{\mathsf v}$ injects naturally into $F(\boldsymbol{\Gamma},\mathsf{Y})$ (see \cite[\S\,5.2, Cor.\,1]{ser80}).
Thus we will think of elements of $\Gamma_{\mathsf v}$ as elements of $F(\boldsymbol{\Gamma},\mathsf{Y})$.

Given a path $\mathsf c = (\mathsf e_1, \ldots, \mathsf e_n)$ in $\mathsf{Y}$, where $\mathsf e_i \in \vec{\mathsf{E}}$ and $\mathsf{t}(\mathsf e_i) = \mathsf{o}(\mathsf e_{i+1}) := \mathsf v_i \in \mathsf{V}$ for all $1\leq i\leq n-1$, and given $\mu = (r_0, \ldots, r_n)$ where $r_i \in \Gamma_{\mathsf v_i}$ for all $0\leq i\leq n$ (where $\mathsf v_0 := \mathsf{o}(\mathsf e_1)$ and $\mathsf v_n := \mathsf{t}(\mathsf e_n)$), we set
\begin{align}\label{eqn:cmu}
[\mathsf c,\mu] := r_0 \mathsf e_1 r_1 \mathsf e_2 \ldots \mathsf e_n r_n \in F(\boldsymbol{\Gamma}, \mathsf{Y})
\end{align}
and call such an element a \emph{path in $F(\boldsymbol{\Gamma}, \mathsf{Y})$} starting at~$\mathsf v_0$ and ending at~$\mathsf v_n$.
We say that this path is a \emph{loop in $F(\boldsymbol{\Gamma}, \mathsf{Y})$} if $\mathsf v_0 = \mathsf v_n$.

Fix a base vertex $\mathsf v_0 \in \mathsf{V}$. 
Then the \emph{fundamental group} $\pi_1(\boldsymbol{\Gamma},\mathsf{Y},\mathsf v_0)$ of the graph of groups $(\boldsymbol{\Gamma},\mathsf{Y})$, based at $\mathsf v_0$, is the subgroup of $F(\boldsymbol{\Gamma},\mathsf{Y})$ consisting of all loops $[\mathsf c, \mu]$ in $F(\boldsymbol{\Gamma}, \mathsf{Y})$ with $\mathsf{o}(\mathsf c) = \mathsf{t}(\mathsf c) = \mathsf v_0$. 

\begin{examples} \label{ex:gog}
When $\Gamma_{\mathsf{v}}=\{1\}$ for all vertices $\mathsf{v} \in \mathsf{V}$, the group $\pi_1(\boldsymbol{\Gamma}, \mathsf{Y},\mathsf{v}_0)$ identifies with $\pi_1(\mathsf{Y},\mathsf{v}_0)$, the fundamental group in the usual sense.

When the graph $\mathsf{Y}$ consists of two vertices $\mathsf{v}_0,\mathsf{v}_1$ connected by an edge~$\mathsf{e}$, the group $\pi_1(\boldsymbol{\Gamma}, \mathsf{Y},\mathsf{v}_0)$ identifies with the amalgamated free product $\Gamma_{\mathsf v_0} *_{\Gamma_{|\mathsf e|}}\nolinebreak \Gamma_{\mathsf v_1}$.

When the graph $\mathsf{Y}$ consists of a single vertex $\mathsf{v}_0$ connected to itself by one edge $\mathsf e$, the group $\pi_1(\boldsymbol{\Gamma}, \mathsf{Y},\mathsf{v}_0)$ identifies with an HNN extension $(\Gamma_{\mathsf v_0})*_t$, where $t$ conjugates $\iota_{\mathsf{e}}(\Gamma_{|\mathsf e|})$ to $\iota_{\mathsf{\overline{e}}}(\Gamma_{|\mathsf e|})$.
\end{examples}

The \emph{universal covering tree} (or \emph{Bass--Serre tree}) ${\mathscr{T}}(\boldsymbol{\Gamma}, \mathsf{Y}, \mathsf v_0)$ of the graph of groups $(\boldsymbol{\Gamma}, \mathsf{Y})$, based at $\mathsf v_0$, is defined as follows. 
The set $\mathscr{V}(\boldsymbol{\Gamma}, \mathsf{Y}, \mathsf v_0)$ of vertices of $\mathscr{T}(\boldsymbol{\Gamma}, \mathsf{Y}, \mathsf v_0)$ is the collection of cosets in $F(\boldsymbol{\Gamma}, \mathsf{Y})$ of the form $v = [\mathsf c, \mu] \, \Gamma_{\mathsf v}$ where $\mathsf v \in \mathsf{V}$ and where $[\mathsf c, \mu]$ is a path in $F(\boldsymbol{\Gamma}, \mathsf{Y})$ starting at $\mathsf v_0$ and ending at~$\mathsf v$. 
Two vertices of $\mathscr{T}(\boldsymbol{\Gamma}, \mathsf{Y}, \mathsf v_0)$ are adjacent (\ie related by an edge) if they are of the form $[\mathsf c, \mu] \, \Gamma_{\mathsf v}$ and $[\mathsf c, \mu]\mathsf e \, \Gamma_{\mathsf v'}$ where $\mathsf v,\mathsf v' \in \mathsf{V}$, where $[\mathsf c, \mu]$ is a path in $F(\boldsymbol{\Gamma}, \mathsf{Y})$ starting at $\mathsf v_0$ and ending at~$\mathsf v$, and where $\mathsf e \in \vec{\mathsf{E}}$ is an oriented edge of~$\mathsf{Y}$ with $\mathsf{o}(\mathsf e) = \mathsf v$ and $\mathsf{t}(\mathsf e) = \mathsf v'$. 
Then $\mathscr{T}(\boldsymbol{\Gamma}, \mathsf{Y}, \mathsf v_0) = (\mathscr{V}(\boldsymbol{\Gamma}, \mathsf{Y}, \mathsf v_0),\mathscr{E}(\boldsymbol{\Gamma}, \mathsf{Y}, \mathsf v_0))$ is a tree, and the action of $F(\boldsymbol{\Gamma}, \mathsf{Y})$ on itself by left multiplication induces a natural simplicial action of the fundamental group $\pi_1(\boldsymbol{\Gamma},\mathsf{Y},\mathsf v_0)$ on the tree $\mathscr{T}(\boldsymbol{\Gamma}, \mathsf{Y}, \mathsf v_0)$.
There are natural projections 
$$ \begin{array}{ccccccc} \mathscr{V}(\boldsymbol{\Gamma}, \mathsf{Y}, \mathsf v_0) & \longrightarrow & \mathsf{V} &~\text{ and }~& {\mathscr{E}} (\boldsymbol{\Gamma}, \mathsf{Y}, \mathsf v_0) & \longrightarrow & {\mathsf{E}} \\ 
\, [\mathsf c, \mu] \, \Gamma_{\mathsf v} &\longmapsto & \mathsf{v} & & \{[\mathsf c, \mu] \, \Gamma_{\mathsf v}, [\mathsf c, \mu]\mathsf e \, \Gamma_{\mathsf v'}\} & \longmapsto & \ |\mathsf{e}|. \end{array} $$

\begin{remark}
There is an equivalent definition of the fundamental group and universal covering tree which replaces the basepoint $\mathsf v_0$ with a spanning tree in~$\mathsf{Y}$. 
The definition given here is more natural for our purposes.
\end{remark}

Given a target group $G$, a representation $\rho : \pi_1(\boldsymbol{\Gamma},\mathsf{Y}, \mathsf v_0) \to G$ may be defined as follows. 
Choose representations $\rho_{\mathsf v}: \Gamma_{\mathsf v} \to G$ for each $\mathsf v \in \mathsf{V}$, and elements $g_{\mathsf e} \in G$ for each $\mathsf e \in \vec{\mathsf{E}}$, such that
\begin{itemize}
  \item $g_{\mathsf e} = g_{\overline{\mathsf e}}^{-1}$ for all $\mathsf e \in \vec{\mathsf{E}}$, and
  \item $\rho_{\mathsf{o}(\mathsf e)}(\gamma^{\mathsf e}) = g_{\mathsf e} \, \rho_{\mathsf{t}(\mathsf e)}(\gamma^{\overline{\mathsf e}}) \, g_{\mathsf{e}}^{-1}$ for all $\mathsf e \in \vec{\mathsf{E}}$ and all $\gamma \in \Gamma_{|\mathsf e|}$.
\end{itemize}
Then
\begin{equation} \label{eqn:rhobar}
\begin{array}{crcc} \underline{\rho} : & F(\boldsymbol{\Gamma}, \mathsf{Y}) & \longrightarrow & G \\
& \Gamma_{\mathsf v} \ni r \hspace{0.15cm} & \longmapsto & \rho_{\mathsf v}(r) \\
& \vec{\mathsf{E}} \ni \mathsf{e} \hspace{0.15cm} & \longmapsto & g_{\mathsf e}  \end{array}
\end{equation}
is a representation, which restricts to a representation
\begin{equation} \label{eqn:rhonobar}
\begin{array}{cccc} \rho : & \pi_1(\boldsymbol{\Gamma},\mathsf{Y}, \mathsf v_0)  & \longrightarrow &  G \\
& \underbrace{r_0 \mathsf e_1r_1 \dots \mathsf e_n r_n}_{[\mathsf c,\mu]} & \longmapsto & \rho_{\mathsf v_0}(r_0) \, g_{\mathsf e_1}\,  \rho_{\mathsf v_1}(r_1) \dots \, g_{\mathsf e_n}\,  \rho_{\mathsf v_n}(r_n)~.   \end{array}
\end{equation}

\subsection{A general combination theorem for graphs of groups} \label{subsec:combin-gog}

Here is our main combination theorem.

\begin{theorem} \label{thm:general-gog}
Let $(\boldsymbol{\Gamma}, \mathsf{Y})$ be a graph of groups, where $\mathsf{Y} = (\mathsf{V}, \mathsf{E})$, and let $\mathsf v_0 \in \mathsf{V}$ be a base vertex.
Suppose $\mathsf{E} \neq \varnothing$.
For $d\geq 3$, consider discrete and faithful representations $\rho_{\mathsf v}: \Gamma_{\mathsf v} \to \PGL(d,\RR)=:G$ for each $\mathsf v \in \mathsf{V}$, and elements $g_{\mathsf e} \in \PGL(d,\RR)$ for each $\mathsf e \in \vec{\mathsf{E}}$, such that 
\begin{itemize}
  \item $g_{\mathsf e} = g_{\overline{\mathsf e}}^{-1}$ for all $\mathsf e \in \vec{\mathsf{E}}$, and
  \item $\rho_{\mathsf{o}(\mathsf e)}(\gamma^{\mathsf e}) = g_{\mathsf e} \, \rho_{\mathsf{t}(\mathsf e)}(\gamma^{\overline{\mathsf e}}) \, g_{\mathsf{e}}^{-1}$ for all $\mathsf e \in \vec{\mathsf{E}}$ and all $\gamma \in \Gamma_{|\mathsf e|}$.
\end{itemize}
Let $\underline{\rho} : F(\boldsymbol{\Gamma}, \mathsf{Y}) \to \PGL(d,\RR)$ and $\rho : \pi_1(\boldsymbol{\Gamma},\mathsf{Y},\mathsf v_0) \to \PGL(d,\RR)$ be the associated representations as in \eqref{eqn:rhobar}--\eqref{eqn:rhonobar}.
For each $\mathsf v \in \mathsf{V}$, let $\Omega_{\mathsf v}$ be a $\rho_{\mathsf v}(\Gamma_{\mathsf v})$-invariant properly convex open subset of $\PP(\RR^n)$. 
Suppose that for all oriented edges $\mathsf e, \mathsf e' \in \vec{\mathsf{E}}$ with $\mathsf{o}(\mathsf e) = \mathsf{o}(\mathsf e') =: \mathsf v$, and for each $\gamma \in \Gamma_{\mathsf v}$ such that either $\mathsf e\neq \mathsf e'$, or $\mathsf e= \mathsf e'$ and $\gamma \notin \Gamma_{|\mathsf{e}|}$, the triple of properly convex open sets
$$(g_{\mathsf e} \cdot \Omega_{\mathsf{t}(\mathsf e)}, \Omega_{\mathsf v}, \rho_{\mathsf v}(\gamma) g_{\mathsf e'} \cdot \Omega_{\mathsf{t}(\mathsf e')})$$
is in occultation position.
Then $\rho : \pi_1(\boldsymbol{\Gamma},\mathsf{Y},\mathsf v_0) \to \PGL(d,\RR)$ is discrete and faithful.
Moreover, for each oriented edge $\mathsf e\in \vec{\mathsf{E}}$, the set $\Omega_{\mathsf e}:=\Conv{\Omega_{\mathsf{o}(\mathsf e)} \cup\nolinebreak g_{\mathsf e} \cdot\nolinebreak \Omega_{\mathsf{t}(\mathsf e)}}$ is well defined, properly convex, 
and the $\rho(\pi_1(\boldsymbol{\Gamma},\mathsf{Y},\mathsf v_0))$-invariant open set
\begin{equation} \label{eqn:defomegaa}
\Omega  := \bigcup_{\substack{[\mathsf c,\mu]\,\Gamma_{\mathsf{v}} \in {\mathscr{V}}(\boldsymbol{\Gamma},\mathsf{Y}, \mathsf v_0)\\ \mathsf e \in \vec{\mathsf{E}} \text{ with $\mathsf{o}(\mathsf e) = \mathsf{v}$}}} \underline{\rho}([\mathsf c,\mu]) \cdot \Omega_{\mathsf e}\end{equation}
is properly convex.
\end{theorem}

We note that in the setting of Theorem~\ref{thm:general-gog} we have $\Omega_{\mathsf e} = g_{\mathsf e} \cdot \Omega_{\overline{\mathsf e}}$ for all oriented edges $\mathsf e \in \vec{\mathsf{E}}$.

\begin{proof}
Let ${\mathscr{T}}(\boldsymbol \Gamma, \mathsf Y, \mathsf v_0) = (\mathscr V(\boldsymbol \Gamma, \mathsf Y, \mathsf v_0), \mathscr E(\boldsymbol \Gamma, \mathsf Y, \mathsf v_0))$ be the universal covering tree of the graph of groups $(\boldsymbol{\Gamma}, \mathsf{Y})$, based at $\mathsf v_0 \in \mathsf{V}$.

To each vertex $v = [\mathsf c, \mu] \, \Gamma_{\mathsf v} \in \mathscr V(\boldsymbol \Gamma, \mathsf Y, \mathsf v_0)$ of ${\mathscr{T}}(\boldsymbol \Gamma, \mathsf Y, \mathsf v_0)$, we associate the properly convex open set
$$\Omega(v) = \Omega([\mathsf c, \mu] \, \Gamma_{\mathsf v}) := \underline{\rho}([\mathsf c,\mu]) \cdot \Omega_{\mathsf v}.$$
This assignment is $\rho$-equivariant: if $[\mathsf c',\mu'] \in \pi_1(\boldsymbol{\Gamma}, \mathsf{Y}, \mathsf v_0)$ and if $[\mathsf c,\mu] \, \Gamma_{\mathsf v}$ is a vertex of ${\mathscr{T}}(\boldsymbol \Gamma, \mathsf Y, \mathsf v_0)$ as above, then
\begin{align*}
\Omega([\mathsf c',\mu'] \cdot [\mathsf c,\mu] \, \Gamma_{\mathsf v}) & = \underline{\rho}([\mathsf c',\mu'][\mathsf c,\mu]) \cdot \Omega_{\mathsf v}\\
& = \rho([\mathsf c', \mu']) \cdot \Omega([\mathsf c, \mu] \, \Gamma_{\mathsf v}).
\end{align*}

One easily checks that any three consecutive vertices of ${\mathscr{T}}(\boldsymbol \Gamma, \mathsf Y, \mathsf v_0)$ are of the form $[\mathsf c, \mu] \mathsf e \, \Gamma_{\mathsf{t}(\mathsf e)}$, $[\mathsf c, \mu] \, \Gamma_{\mathsf v}$, $[\mathsf c, \mu] \gamma \mathsf e' \, \Gamma_{\mathsf{t}(\mathsf e')}$ where $\mathsf{v} \in \mathsf{V}$, where $[\mathsf c, \mu]$ is a path in $F(\boldsymbol{\Gamma}, \mathsf{Y})$ starting at $\mathsf v_0$ and ending at~$\mathsf v$, and where $\mathsf e, \mathsf e' \in \vec{\mathsf{E}}$ are oriented edges of~$\mathsf{Y}$ with $\mathsf{o}(\mathsf e) = \mathsf{o}(\mathsf e') = \mathsf v$. 
The corresponding triple
\begin{align*}
& \big(\Omega([\mathsf c, \mu] \mathsf e' \, \Gamma_{\mathsf{t}(\mathsf e)}), \Omega([\mathsf c, \mu] \, \Gamma_{\mathsf v}), \Omega([\mathsf c, \mu] \gamma \mathsf e \, \Gamma_{\mathsf{t}(\mathsf e')})\big)\\
& = \big(\underline{\rho}([\mathsf c,\mu]) g_{\mathsf e'} \cdot \Omega_{\mathsf{t}(\mathsf e')}, \underline{\rho}([\mathsf c,\mu]) \cdot \Omega_{\mathsf v}, \underline{\rho}([\mathsf c,\mu]) \rho_{\mathsf v}(\gamma) g_{\mathsf e} \cdot \Omega_{\mathsf{t}(\mathsf e)})
\end{align*}
of properly convex open sets is in occultation position, because the triple $(g_{\mathsf e} \cdot \Omega_{\mathsf{t}(\mathsf e)}, \Omega_{\mathsf v}, \rho_{\mathsf v}(\gamma) g_{\mathsf e'} \cdot \Omega_{\mathsf{t}(\mathsf e')})$ is by assumption.

The stabilizer in $\pi_1(\boldsymbol{\Gamma},\mathsf{Y},\mathsf v_0)$ of a vertex $[\mathsf c, \mu] \, \Gamma_{\mathsf v}$ is $[\mathsf c, \mu] \, \Gamma_{\mathsf v} \, [\mathsf c, \mu]^{-1}$, and the restriction $\rho|_{[\mathsf c, \mu] \, \Gamma_{\mathsf v} \, [\mathsf c, \mu]^{-1}} : [\mathsf c, \mu] \, \gamma \, [\mathsf c, \mu]^{-1} \mapsto \underline{\rho}([\mathsf c, \mu]) \gamma \underline{\rho}([\mathsf c, \mu])^{-1}$ 
is discrete and faithful by construction. 
By Theorem~\ref{thm:tree-with-group-actions}, the representation $\rho : \pi_1(\boldsymbol{\Gamma},\mathsf{Y},\mathsf v_0) \to \PGL(d,\RR)$ is discrete and faithful, and its image preserves the open set $\Omega$ of \eqref{eqn:defomegaa}, which is properly convex by Theorem~\ref{thm:tree}.
\end{proof}

\section{Applications: free products of discrete groups} \label{sec:free-prod}

In this section we prove Theorem~\ref{thm:free-product-fd} and Corollaries \ref{cor:counterex-nori} and~\ref{cor:combine-discrete-groups}.
This is done in Sections \ref{subsec:proof-free-product-fd} to~\ref{subsec:nori}.

Before this we start, in Section~\ref{subsec:thicken}, by proving two basic results on thickenings of convex sets, which are useful for various applications of Theorems \ref{thm:amalgam}, \ref{thm:HNN}, \ref{thm:tree-with-group-actions}, and later Theorems \ref{thm:amalgam-cc}, \ref{thm:HNN-cc}, \ref{thm:gog-cc} involving convex cocompactness.
We also give some reminders and basic facts about second symmetric powers, proximality, and limit sets in Section~\ref{subsec:remind-prox}.

\subsection{Reminders: the Hilbert metric} \label{subsec:remind-Hilbert}

Let $\Omega$ be a nonempty properly convex open subset of $\PP(V)$, with boundary~$\partial\Omega$.
Recall the \emph{Hilbert metric} $d_{\Omega}$ on~$\Omega$:
\begin{equation} \label{eqn:d-Omega}
d_{\Omega}(x,y) := \frac{1}{2} \log \, \cro{a}{x}{y}{b}
\end{equation}
for all distinct $x,y\in\Omega$, where $\cro{\,\underline{~~}}{\underline{~~}}{\underline{~~}}{\underline{~~}\,}$ is the cross-ratio on $\PP^1(\RR)$, normalized so that $\cro{0}{1}{t}{\infty}=t$, and where $a,b$ are the intersection points of $\partial\Omega$ with the projective line through $x$ and~$y$, with $a,x,y,b$ in this order.
The metric space $(\Omega,d_{\Omega})$ is proper (closed balls are compact) and complete, and the group
$$\mathrm{Aut}(\Omega) := \{g\in\PGL(V) ~|~ g\cdot\Omega=\Omega\}$$
acts on~$\Omega$ by isometries for~$d_{\Omega}$.
As a consequence, any discrete subgroup of $\mathrm{Aut}(\Omega)$ acts properly discontinuously on~$\Omega$.

\begin{remark} \label{rem:Hilb-metric-include}
It follows from the definition that if $\Omega_1 \subset \Omega_2$ are nonempty properly convex open subsets of $\PP(V)$, then the corresponding Hilbert metrics satisfy $d_{\Omega_1}(x,y) \geq d_{\Omega_2}(x,y)$ for all $x,y\in\Omega_1$.
\end{remark}

\subsection{Preliminaries: thickening convex sets} \label{subsec:thicken}

We shall use the following notation.

\begin{notation} \label{notn:fat}
Suppose our finite-dimensional real vector space $V$ is the direct sum of two linear subspaces $U$ and~$U'$.
For any nonempty properly convex open subset $\mathcal{O}$ of $\PP(U)$, we set
$$\mathcal{O} \times U' := \PP\big(\big\{ u+u' \,|\, u\in\widetilde{\mathcal{O}},\ u'\in U'\big\}\big),$$
where $\widetilde{\mathcal{O}}$ is any properly convex open cone of~$U$ lifting~$\mathcal{O}$.
The set $\mathcal{O} \times U'$ is open and convex, and its intersection with $\PP(U)$ is $\mathcal{O} \times \{0\} = \mathcal O$.
\end{notation}

In general, every nonempty convex open subset $\mathbb{O}$ of $\PP(V)$ is of the form $\mathbb{O}= \mathcal{O} \times U'$ for some decomposition $V = U\oplus U'$ and some nonempty properly convex open subset $\mathcal{O}$ of $\PP(U)$.
The subspace $U'$ is uniquely defined: $\PP(U')$ is the intersection of all projective hyperplanes of $\PP(V)$ disjoint from~$\mathbb{O}$.
The set $\mathbb{O}$ is properly convex if and only if $U' = \{0\}$.

\begin{proposition} \label{prop:thicken-add-dim}
Let $\Gamma$ be a subgroup of $\SL^{\pm}(V)$ preserving a direct sum decomposition $V = U \oplus U'$.
Suppose that $\Gamma$ preserves a nonempty properly convex open subset $\mathcal{O}$ of $\PP(U)$ and acts trivially on~$U'$.
Then there is a $\Gamma$-invariant properly convex open subset $\Omega$ of $\PP(V)$ such that $\mathcal{O} \subset \Omega \subset \mathcal{O} \times U'$ and $\overline{\Omega} \cap \PP(U') = \varnothing$.
Moreover, for any $\Gamma$-invariant properly convex open subset $\Omega$ of $\PP(V)$ containing~$\mathcal{O}$ we have $\Lambdao_{\Omega}(\Gamma) = \Lambdao_{\mathcal{O}}(\Gamma)$; in particular, the action of $\Gamma$ on~$\Omega$ is convex cocompact if and only if the action of $\Gamma$ on~$\mathcal{O}$ is.
\end{proposition}

\begin{proof}
Let $\widetilde{\mathcal{O}}$ be a properly convex open cone of $U\smallsetminus\{0\}$ whose projection to $\PP(U)$ is~$\mathcal{O}$.
Inside~$\widetilde{\mathcal{O}}$, there is a $\Gamma$-invariant hypersurface $S$ which intersects every ray in $\widetilde{\mathcal{O}}$ exactly once, and which is convex in the sense that $[1,\infty)S$ is a convex subset of~$V$; indeed, a canonical such hypersurface is described in the classical works \cite{koe57,vin63}, for instance.
Choose a bounded convex open neighborhood $\mathcal A$ of the origin in~$U'$, and consider the $\Gamma$-invariant open cone
$$\widetilde{\Omega} := \RR_{>0} \, (S + \mathcal A)$$
of~$V$.
By construction, $\widetilde{\mathcal{O}} \subset \widetilde{\Omega} \subset \widetilde{\mathcal{O}} + U'$, hence the projection $\Omega$ of $\widetilde{\Omega}$ to $\PP(V)$ satisfies $\mathcal{O} \subset \Omega \subset \mathcal{O} \times U'$.

Let us check that $\widetilde{\Omega}$ is convex, which implies that $\Omega$ is convex.
For this it is sufficient to check that for any $s_0, s_1 \in S$, any $a_0, a_1 \in \mathcal A$, and any $t \in [0,1]$, the point $(1-t)(s_0, a_0) + t (s_1, a_1) \in U \oplus U'$ lies in $\widetilde{\Omega}$.
Since $S$ is a convex hypersurface in~$U$, the point $(1-t)s_0 + t s_1$ lies above (or on)~$S$, hence there exists $\lambda_t \in (0,1]$ such that $\lambda_t((1-t)s_0 + t s_1) \in S$.
Since the subset $\mathcal A$ of~$U'$ is convex and contains the origin, we have $\lambda_t((1-t)a_0 + t a_1) \in \mathcal A$.
Therefore $(1-t)(s_0, a_0) + t (s_1, a_1)$ lies in $\widetilde{\Omega}$, showing that $\widetilde{\Omega}$ is convex.

The set $\overline{\Omega} = \overline{\{[(s,a)]: s \in S, a \in \mathcal A\}}$ is disjoint from $\PP(U')$ since the norm of any $s \in S$ is bounded below, and the norm of any $a \in \mathcal A$ is bounded above.

Since $\mathcal{O}$ is properly convex, there is a hyperplane $H$ of~$U$ such that $\PP(H)$ avoids~$\overline{\mathcal{O}}$.
Since $\overline{\Omega} \cap \PP(U') = \varnothing$, the projective hyperplane $\PP(H \oplus U')$ of $\PP(V)$ avoids~$\overline{\Omega}$, and so the convex set $\Omega$ is properly convex.

Now let $\Omega$ be any $\Gamma$-invariant properly convex open subset of $\PP(V)$ containing~$\mathcal{O}$.
By \cite[Lem.\,10.5.(1)]{dgk-proj-cc}, we have $\Lambdao_{\Omega}(\Gamma) = \Lambdao_{\mathcal{O}}(\Gamma)$, and so $\Ccore_{\Omega}(\Gamma) = \Ccore_{\mathcal{O}}(\Gamma)$.
In particular, the action of $\Gamma$ on~$\Omega$ is convex cocompact if and only if the action of $\Gamma$ on~$\mathcal{O}$ is.
\end{proof}

Here is a variant of Proposition~\ref{prop:thicken-add-dim} where the action of $\Gamma$ on~$U'$ is allowed to be nontrivial (but dominated by the action on~$U$).

\begin{proposition} \label{prop:thicken-add-dim-general}
Let $\Gamma$ be a subgroup of $\SL^{\pm}(V)$ preserving a direct sum decomposition $V = U \oplus U'$.
Suppose that $\Gamma$ acts convex cocompactly on some properly convex open subset $\mathcal{O}$ of $\PP(U)$ and that any accumulation point of any $\Gamma$-orbit of $(\mathcal{O} \times U') \smallsetminus \PP(U')$ is contained in $\PP(U)$.
Then $\Gamma$ acts convex cocompactly on some properly convex open subset $\Omega$ of $\PP(V)$, contained in $\mathcal{O} \times U'$, such that $\overline{\Omega} \cap \PP(U') = \varnothing$.
\end{proposition}

\begin{proof}
Since $\Gamma$ acts convex cocompactly on~$\mathcal{O}$, there is a compact subset $\mathcal{D}$ of $\PP(U)$ such that $\Ccore_{\mathcal{O}}(\Gamma) = \bigcup_{\gamma\in\Gamma} \gamma \cdot \mathcal{D}$.
Choose $R>0$ and let $\C$ be the closed uniform $R$-neighborhood of $\Ccore_{\mathcal{O}}(\Gamma)$ in $(\mathcal O ,d_{\mathcal O})$.
Then $\C = \bigcup_{\gamma\in\Gamma} \gamma \cdot \mathcal{D}'$ where $\mathcal{D}'$ is the closed uniform $R$-neighborhood of $\mathcal{D}$ in $(\mathcal O ,d_{\mathcal O})$.

We can lift $\mathcal{O}$ to a properly convex open cone $\widetilde{\mathcal{O}}$ of~$U$, and $\mathcal{D}'$ to a closed subcone $\widetilde{\mathcal{D}'}$ of the closure of~$\widetilde{\mathcal{O}}$.
As in the proof of Proposition~\ref{prop:thicken-add-dim}, inside~$\widetilde{\mathcal{O}}$, there is a $\Gamma$-invariant hypersurface $S$ which intersects every ray in $\widetilde{\mathcal{O}}$ exactly once.
Choose a bounded neighborhood $\mathcal A$ of the origin in~$U'$, and consider the $\Gamma$-invariant cone
$$\widetilde{\mathcal B} := \RR_{>0} \, \bigcup_{\gamma\in\Gamma} \gamma \cdot \big((S \cap \widetilde{\mathcal{D}'}) + \overline{\mathcal{A}}\big)$$
of~$V$.
It is contained in $\widetilde{\mathcal{O}} + U'$, hence its projection $\mathcal B$ to $\PP(V)$ is contained in $\mathcal{O} \times U'$.
We have $\mathcal{B} \cap \PP(U') = \varnothing$.

We claim that $\overline{\mathcal{B}} \cap \PP(U') = \varnothing$.
Indeed, consider a sequence $(b_n)_{n\in\NN}$ of points of $\widetilde{\mathcal B}$ whose projections $[b_n]$ converge to a point $x \in \PP(V)$.
Let us check that $x \notin \PP(U')$.
For each~$n$, up to scaling we can write $b_n = \gamma_n \cdot (s_n + a_n)$ where $\gamma_n \in \Gamma$, where $s_n \in S \cap \widetilde{\mathcal{D}'}$, and where $a_n \in \overline{\mathcal{A}}$.
Up to passing to a subsequence, we may assume that either $(\gamma_n)_{n\in\NN}$ is constant, or $\gamma_n \to \infty$ in~$\Gamma$.
If $(\gamma_n)_{n\in\NN}$ is constant equal to~$\gamma$, then $x$ belongs to the projection of the compact set $\gamma \cdot ((S \cap \widetilde{\mathcal{D}'}) + \overline{\mathcal{A}})$ to $\PP(V)$, which lies in $(\mathcal{O} \times U') \smallsetminus \PP(U')$.
Suppose then that $\gamma_n \to \infty$ in~$\Gamma$.
Then $\gamma_n \to \infty$ in $\SL^{\pm}(V)$ (as $\Gamma$ is discrete in $\SL^{\pm}(V)$).
Decompose each $\gamma_n = g_n \oplus g_n'$ into its restrictions to $U$ and~$U'$.
Choose auxiliary norms $\|\cdot\|$ on both $U$ and~$U'$, and denote also by $\|\cdot\|$ the corresponding operator norms on $\GL(U)$ and $\GL(U')$.
Since all accumulation points of $\gamma_n\cdot p$ lie in $\PP(U)$ for an open set's worth of $p \in \PP(V)$, we have $\|g_n\|/\|g_n'\| \to +\infty$.
It easily follows from the Cartan decomposition in $\GL(U)$ that, after replacing $(\gamma_n)_{n\in\NN}$ by a subsequence, there is a nonzero linear subspace $U^-$ of~$U$ with the following property: for any compact subset $\mathcal{K}$ of $U \smallsetminus U^-$, there exists $c > 0$ such that $\| g_n\cdot u\| \geq c\| g_n\|$ for all $n\in\NN$ and all $u\in K$.
Further, we must have $\PP(U^-) \cap \mathcal O = \varnothing$.
Applying this fact with $\mathcal{K} = S \cap \widetilde{\mathcal{D}'}$, and setting $M := \max \{ \| a\| \,|\, a \in \overline{\mathcal A}\}$, we see that
$$\frac{\|g_n \cdot s_n\|}{\| g_n' \cdot a_n\|} \geq \frac{c \|g_n\|}{M\| g_n'\|} \longrightarrow +\infty, $$
which implies $x \in \PP(U)$, hence $x \notin \PP(U')$.
This shows that $\overline{\mathcal{B}} \cap \PP(U') = \varnothing$.

Let $\Omega$ be the interior of the convex hull of $\mathcal B$ taken inside the convex open set $\mathcal O \times U'$.
Since $\mathcal B$ and $\mathcal O \times U'$ are $\Gamma$-invariant, so is~$\Omega$. 

Since $\mathcal{O}$ is properly convex, there is a hyperplane $H$ of~$U$ such that $\PP(H)$ avoids~$\overline{\mathcal{O}}$.
Since $\overline{\mathcal{B}} \cap \PP(U') = \varnothing$, the projective hyperplane $\PP(H \oplus U')$ of $\PP(V)$ avoids~$\overline{\mathcal{B}}$, hence also~$\overline{\Omega}$.
This shows that $\Omega$ is properly convex.

By assumption $\Lambdao_{\Omega}(\Gamma) \subset \PP(U)$, hence $\Lambdao_{\Omega}(\Gamma)$ is a closed subset of $\Lambdao_{\mathcal{O}}(\Gamma)$ by looking at the first factor.
Since the action of $\Gamma$ on $\Ccore_{\mathcal{O}}(\Gamma)$ is cocompact, so is the action of $\Gamma$ on $\Ccore_{\Omega}(\Gamma)$; thus the action of $\Gamma$ on~$\Omega$ is convex cocompact.
\end{proof}

Here is another result about thickening properly convex sets, inside of convex (but possibly not properly convex) sets. 
For a nonempty convex open subset $\Omega$ of $\PP(V)$, we let $\vec{\Omega} \subset \PP(V)$ denote the intersection of all projective hyperplanes of $\PP(V)$ disjoint from $\Omega$.
Equivalently, $\vec{\Omega}$ is the projectivization of the largest subspace entirely contained in the closure of a convex open cone of~$V$ lifting $\Omega$.
One has $\vec{\Omega} = \varnothing$ if and only if $\Omega$ is properly convex.

\begin{proposition} \label{prop:thicken-not-proper}
Let $\Omega \subset \mathbb{O}$ be two nonempty convex open subsets of $\PP(V)$.
Suppose $\overline{\Omega} \cap \vec{\mathbb{O}}=\varnothing$.
Then there is a family $(\Omega_t)_{t\geq 0}$ of properly convex open subsets of $\mathbb{O}$ such that $\Omega_0 = \Omega$ and $\bigcup_{t\geq 0} \Omega_t = \mathbb{O}$, with $\Omega_s \subset \Omega_t$ for all $s<t$.
Furthermore, if the sets $\mathbb{O}$ and $\Omega$ are preserved by a subgroup $\Gamma$ of $\PGL(V)$, then we may arrange that $\Omega_t$ is also preserved by~$\Gamma$ for all $t \geq 0$.
\end{proposition}

\begin{proof}
If $\mathbb{O}$ is properly convex, we may take $\Omega_t$ to be the uniform $t$-neighbor\-hood of $\Omega$ for the Hilbert metric on $\mathbb{O}$. 
If not, we first notice that $\Omega$ is properly convex: indeed $\vec{\Omega}$ is both contained in and disjoint from $\vec{\mathbb{O}}$, so it is empty.
The dual convex sets in $\PP(V^*)$ satisfy $\mathbb{O}^* \subset \Omega^*$, where $\mathbb{O}^* \subset \PP(V^*)$ is defined as the set of projective hyperplanes $X$ of $\PP(V)$ such that $X \cap \overline{\mathbb{O}} = \vec{\mathbb{O}}$.
For any $t>0$, let $\mathcal{O}_t \subset \PP(V^*)$ be the open uniform $(1/t)$-neighborhood of $\mathbb{O}^*$ in $\Omega^*$ for the Hilbert metric, and let $\Omega_t:=(\mathcal{O}_t)^* \subset \PP(V)$.
If $s<t$, then $\mathcal{O}_s \supset \mathcal{O}_t$, hence $\Omega_s \subset \Omega_t$.
Since $\mathcal{O}_t$ is open (and properly convex), $\Omega_t$ is properly convex (and open).
Since $\bigcap_{t>0} \mathcal{O}_t = \mathbb{O}^*$ (possibly not open), $\bigcup_{t>0} \Omega_t = \mathbb{O}$.
All constructions are compatible with projective actions preserving $(\Omega, \mathbb{O})$.
\end{proof}

\subsection{Reminders: second symmetric powers, proximality, and limit sets} \label{subsec:remind-prox}

In the sequel we shall use several times the following elementary fact.

\begin{fact} \label{fact:Omega-sym}
Let $\tau_d : \SL(d,\RR)\to\SL(S^2\RR^d)$ be the second symmetric power of the standard representation.
Then $\SL(d,\RR)$ acts properly and transitively via~$\tau_d$ on a nonempty properly convex open subset $\Omega_d$ of $\PP(S^2\RR^d)$; if we identify $S^2\RR^d$ with the space of symmetric $d\times d$ real matrices (on which $\SL(d,\RR)$ acts by $g\cdot M = gMg^t$), then $\Omega_d$ is the projectivization of those symmetric matrices that are positive definite.
As an $\SL(d,\RR)$-homogeneous space, $\Omega_d$ identifies with $\SL(d,\RR)/\SO(d)$.
\end{fact}

Recall that an element $g\in\GL(V)$ is said to be \emph{proximal in $\PP(V)$} if it has a unique eigenvalue of maximal modulus; this eigenvalue is then necessarily real.
The corresponding eigenline is then an attracting fixed point $x_g^+ \in \PP(V)$ for the action of $g$, and the sum of the generalized eigenspaces for the other eigenvalues of~$g$ is then an element $X_g^-$ of $\PP(V^*)$ (a projective hyperplane) such that for any $x\in\PP(V)\smallsetminus X_g^-$ we have $g^n\cdot x\to x_g^+$ as $n\to +\infty$.
We say that an element of $\PGL(V)$ is proximal if it is the image of a proximal element of $\GL(V)$.
We say that $g$ is \emph{biproximal in $\PP(V)$} if both $g$ and~$g^{-1}$ are proximal in $\PP(V)$.

For any subgroup $\Gamma$ of $\PGL(V)$, the \emph{proximal limit set} $\Lambda^{\PP(V)}_{\Gamma}$ (\resp $\Lambda^{\PP(V^*)}_{\Gamma}$) of $\Gamma$ in $\PP(V)$ (\resp $\PP(V^*)$) is the closure in $\PP(V)$ (\resp $\PP(V^*)$) of the set of fixed points $x_{\gamma}^+ \in \PP(V)$ (\resp $X_{\gamma}^- \in \PP(V^*)$) for elements $\gamma \in \Gamma$ that are proximal in $\PP(V)$.

Let $\mathcal{F}$ be the space of partial flags $(x,X)$ with $x \in \PP(V)$ and $X \in \PP(V^*)$, such that $x$ belongs to~$X$ seen as a projective hyperplane of $\PP(V)$.
For any subgroup $\Gamma$ of $\PGL(V)$, the \emph{proximal limit set} $\Lambda^{\mathcal{F}}_{\Gamma}$ of $\Gamma$ in~$\mathcal{F}$ is the closure in~$\mathcal{F}$ of the set of fixed flags $F_{\gamma}^+ = (x_{\gamma}^+,X_{\gamma^{-1}}^-)$ for elements $\gamma \in \Gamma$ that are biproximal in $\PP(V)$.
We shall use the following two facts.

\begin{fact}[{Benoist \cite[Lem.\,3.6.(iv)]{ben97}}] \label{fact:double-lim-set}
Let $\Gamma$ be a Zariski-dense subgroup of $\PGL(V)$.
Then the set $\{(F^+_{\gamma}, F^+_{\gamma^{-1}}) \,|\, \gamma\in\Gamma \text{ biproximal in }\PP(V)\}$ is dense in $\Lambda_{\Gamma}^{\mathcal{F}} \times\nolinebreak \Lambda_{\Gamma}^{\mathcal{F}}$.
\end{fact}

\begin{fact}[{Borel \cite{bor60}, Mostow \cite[Lem.\,8.5]{mos73}}] \label{fact:lattices}
Let $\Gamma$ be a lattice of $\PGL(V)$.
Then $\Gamma$ is Zariski-dense in $\PGL(V)$, and $\Lambda^{\mathcal{F}}_{\Gamma} = \mathcal{F}$.
\end{fact}

We note that there exist Zariski-dense discrete subgroups of $\PGL(V)$ which satisfy $\Lambda^{\mathcal{F}}_{\Gamma} = \mathcal{F}$, but which are not lattices in $\PGL(V)$: see \cite{dh}.

In Section~\ref{subsec:free-product-in-G} we shall work in the following more general setting.
Let $G$ be a noncompact connected real linear semisimple Lie group and $P$ a self-opposite parabolic subgroup of~$G$.
We then say that an element $g\in G$ is \emph{proximal in $G/P$} if it has an attracting fixed point $x_g^+ \in G/P$; in that case, $x_g^+$ is unique, and $g$ also a unique repelling fixed point $x_g^- \in G/P$; moreover, $g^n \cdot x \to x_g^+$ as $n \to +\infty$ for all $x \in G/P$ transverse to~$x_g^-$ (see \cite[\S\,2.4]{ggkw17}).
For any subgroup $\Gamma$ of $\PGL(V)$, the \emph{proximal limit set} $\Lambda_{\Gamma}^{G/P}$ of $\Gamma$ in $G/P$ is the closure in $G/P$ of the set of fixed points $x_{\gamma}^+$ for elements $\gamma \in \Gamma$ that are proximal in $G/P$.
The following holds similarly to Fact~\ref{fact:double-lim-set}.

\begin{fact}[{Benoist \cite[Lem.\,3.6.(iv)]{ben97}}] \label{fact:double-lim-set-G/P}
Let $G$ be a noncompact connected real linear semisimple Lie group, $P$ a self-opposite parabolic subgroup of~$G$, and $\Gamma$ a Zariski-dense subgroup of~$G$.
Then $\{(x^+_{\gamma}, x^+_{\gamma^{-1}}) \,|\, \gamma\in\Gamma \text{ proximal in }G/P\}$ is dense in $\Lambda_{\Gamma}^{G/P} \times\nolinebreak \Lambda_{\Gamma}^{G/P}$.
\end{fact}

\subsection{Proof of Theorem~\ref{thm:free-product-fd}} \label{subsec:proof-free-product-fd}

We first note that for each $i\in\{0,1\}$, there exist a nonempty $\Gamma_i$-invariant properly convex open subset $\Omega'_i$ of~$\Omega_i$, an extremal point $x_i$ of $\overline{\Omega'_i}$, 
and a neighborhood $\mathcal{U}_i$ of $x_i$ such that
\begin{itemize}
  \item $\mathcal{U}_i \subset \Omega_i \smallsetminus \Ccore_{\Omega_i}(\Gamma_i)$,
  \item $\mathcal{U}_i \cap \gamma\cdot \mathcal{U}_i = \varnothing$ for all $\gamma \in \Gamma \smallsetminus \{1\}$, and
  \item $\mathcal{U}_i$ contains no other extremal points of $\overline{\Omega'_i}$ than~$x_i$.
\end{itemize}
Indeed, by assumption $\Omega_i \smallsetminus \Ccore_{\Omega_i}(\Gamma_i) \neq \varnothing$.
Since the action of $\Gamma_i$ on~$\Omega_i$ is properly discontinuous, we can find a nonempty open subset $\mathcal{U}'_i$ of $\Omega_i \smallsetminus \Ccore_{\Omega_i}(\Gamma_i)$ such that $\mathcal{U}'_i \cap \gamma\cdot \mathcal{U}'_i = \varnothing$ for all $\gamma \in \Gamma \smallsetminus \{1\}$.
Let $\Sigma_i \subset \mathcal{U}'_i$ be a projective simplex with vertex set $\mathbf{x}_i$ and nonempty interior.
Let $\Omega'_i$ be the interior of the convex hull of $\Gamma_i \cdot \Sigma_i$ in~$\Omega_i$.
The extremal points of $\overline{\Omega'_i}$ are contained in $\overline {\Gamma_i \cdot \mathbf{x}_i} \subset \Lambdao_{\Omega_i}(\Gamma_i) \cup \Gamma_i \cdot \mathbf{x}_i$. 
Since $\mathbf{x}_i$ is disjoint from the convex hull $\Ccore_{\Omega_i}(\Gamma_i)$ of $\Lambdao_{\Omega_i}(\Gamma_i)$, it follows that some point $x_i \in \mathbf{x}_i$ is extremal for~$\overline{\Omega'_i}$, and there is a neighborhood $\mathcal{U}_i\subset \mathcal{U}'_i$ of $x_i$ containing no other extremal point.

Since $x_i$ is an isolated extremal point of~$\overline{\Omega'_i}$, we can find a projective hyperplane $X_i$ of $\PP(V)$ such that $\varnothing \subsetneq X_i \cap \Omega'_i \subset \mathcal{U}_i$.
Let $\Delta_i$ be the connected component of $\Omega'_i \smallsetminus X_i$ containing $x_i$ in its boundary: it is a small open pyramid with tip~$x_i$.
See Figure~\ref{fig:taille}.

\begin{figure}[h]
\labellist
\small\hair 2pt
\pinlabel {$x_i$} at 						7.85	3.75
\pinlabel {$x$} at 						6.8	3.75
\pinlabel {$g_i \! \cdot \! \Omega$} at 		7.7	4.5
\pinlabel {$\Omega''_i(x)$} at 				4.7	5.6
\pinlabel {$\gamma g_i \! \cdot \! \Omega$} at 	5.3	7.2
\pinlabel {$\Omega'_i$} at 				5.5	0.3
\pinlabel {$X_i$} at 						5.9	1.3
\pinlabel {$\overbrace{~\hspace{2cm}~}^{\textstyle{\Delta_i}~}$} at 	6.8	6.4
\endlabellist
\includegraphics[width = 8cm]{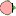}
\caption{Illustration of the proof of Theorem~\ref{thm:free-product-fd}: the triple \eqref{eq:repyr2} is in occultation position}
\label{fig:taille}
\end{figure}

Fix a closed, properly convex subset $B$ of $\PP(V)$, with nonempty interior and with smooth, strictly convex boundary (\eg an ellipsoid).
For each $i\in\{0,1\}$, we can find an element $g_i \in \PGL(V)$ taking $B$ into~$\Delta_i$ (indeed, take a large power of any proximal element whose attracting fixed point lies in~$\Delta_i$ and whose repelling hyperplane misses~$B$, see Section~\ref{subsec:remind-prox}).
By Fact~\ref{fact:double-lim-set}, we may require $g_i$ to lie in any given Zariski-dense subgroup $\Gamma$ of $\PGL(V)$ satisfying $\Lambda^{\PP(V)}_{\Gamma} = \PP(V)$ and $\Lambda^{\PP(V^*)}_{\Gamma} = \PP(V^*)$, for instance any lattice of $\PGL(V)$ (see Fact~\ref{fact:lattices}).

The compact set $B$ has a neighborhood $\mathcal{B}$ such that $g_i \cdot \overline{\mathcal{B}} \subset \Delta_i$ for both $i\in \{0,1\}$.
Hence, for any nonempty convex open subset $\Omega$ of~$\mathcal{B}$ and any point $x\in g_i\cdot\Omega$, 
if $\Delta'_i(x) \subset \Delta_i$ denotes the open pyramid with the same basis $X_i \cap \Omega'_i$ as $\Delta_i$ but with tip~$x$,
then the set
\begin{equation} \label{eq:repyr}
\Omega''_i(x):= \Big ( \Omega'_i \smallsetminus \bigsqcup_{\gamma\in \Gamma_i} \gamma \cdot \Delta_i \Big ) 
\cup \bigsqcup_{\gamma \in \Gamma_i} \gamma \cdot \Delta'_i(x) 
\quad \subset \Omega'_i
\end{equation}
is convex, $\Gamma_i$-invariant, and has the property that for any $\gamma \in \Gamma_i \smallsetminus \{1\}$, the triple
\begin{equation} \label{eq:repyr2} 
\left ( g_i \cdot \Omega,~ \Omega''_i(x),~ \gamma g_i \cdot \Omega \right )
\end{equation}
is in occultation position (Figure~\ref{fig:taille}).
We will choose appropriate $\Omega$ and $x$ below.

\begin{figure}[h]
\labellist
\small\hair 2pt
\pinlabel {$\color{red}\beta^{\texttt{-}N} g_0^{\texttt{-}1}=:g_{\mathsf{e}_0}$} at 	13		15
\pinlabel {$\color[rgb]{0,0.65,0.25} g_0$} at 						16.5		25.3
\pinlabel {$x_0$} at 											11.5		19.6
\pinlabel {$X_0$} at 											3.2		14.1
\pinlabel {${}_{x'_0}$} at 										9.4		19.9 
\pinlabel {${}_{X'_0}$} at 										9.1		17.2
\pinlabel {${}_{x''_0}$} at 										6.6		19.7 
\pinlabel {${}_{g_0 \cdot \Omega}$} at 							6.7		21 
\pinlabel {$\overbrace{~\hspace{2.2cm}~}^{~\textstyle{\Delta_0}}$} at 		7.7		25.6
\pinlabel {${}_{\Delta'_0}$} at 										3.75		18
\pinlabel {$\color{red}\Omega''_0$} at 							1.2		19.8
\pinlabel {$\underbrace{~\hspace{3cm}~}_{~\textstyle{\Omega'_0}}$} at 	6		12.4
\pinlabel {$\color{blue} g_{\mathsf{e}_1}:=\beta^{N} g_1^{\texttt{-}1}$} at 			32.8		15
\pinlabel {$\color[rgb]{0,0.65,0.25} g_1$} at 						30		25.4 
\pinlabel {$x_1$} at 											35.0		19.6
\pinlabel {$X_1$} at 											43.2		14.2
\pinlabel {${}_{x'_1}$} at 										37.0		19.9 
\pinlabel {${}_{X'_1}$} at 										37.4		17.2
\pinlabel {${}_{x''_1}$} at 										39.6		19.9 
\pinlabel {${}_{g_1 \!\cdot \Omega}$} at 							39.4		21.2 
\pinlabel {$\overbrace{~\hspace{2.2cm}~}^{~\textstyle{\Delta_1}}$} at 		38.5		25.7
\pinlabel {${}_{\Delta'_1}$} at 										42.4		18
\pinlabel {$\color{blue} \Omega''_1$} at 							45		19.8
\pinlabel {$\underbrace{~\hspace{3cm}~}_{\textstyle{\Omega'_1}~}$} at 	40		12.4
\pinlabel {$\mathcal{B}$} at 						21.8		14.7
\pinlabel {$B$} at 								23.0		15.7
\pinlabel {$\Omega$} at 							23.0		22.7
\pinlabel {$\beta$} at 							23.0		20.9
\pinlabel {${}_{y_1}$} at 							17.7		20.1
\pinlabel {${}_{y_{\!\smash \beta}^{\texttt{-}}}$} at 		19.35	19.7
\pinlabel {${}_{Y_1}$} at 							18.15	17.3
\pinlabel {${}_{Y_{\! \beta}^{\texttt{-}}}$} at 				19.5		17.2
\pinlabel {${}_{y_0}$} at 							28.5		20.1
\pinlabel {${}_{y_{\!\smash \beta}^{\texttt{+}}}$} at 		26.93	19.67
\pinlabel {${}_{Y_0}$} at 							28.3		17.4
\pinlabel {${}_{Y_{\! \beta}^{\texttt{+}}}$} at 			26.8		17.3
\pinlabel {$\color{red} g_{\mathsf{e}_0} \!\!\cdot\! \Omega''_0$} at 	15.3		18.8
\pinlabel {$\color{blue} g_{\mathsf{e}_1} \!\!\cdot\! \Omega''_1$} at 	31.1		18.7
\pinlabel {$x_0$} at 					12		4.5
\pinlabel {${}_{x''_0}$} at 				7.5		5.2
\pinlabel {${}_{x'''_0}$} at 				5.2		5.1
\pinlabel {$\color{red} \Omega''_0$} at 	-0.9		5.1
\pinlabel{${}_{\C_0^+}$} at 			9.6		5.3
\pinlabel{${}_{\Delta''_0}$} at 				3.8		3.2
\pinlabel{$\color{red} \C_0''$} at 		1.5		5.1
\pinlabel{$\color{red} g_{\mathsf{e}_0}$} at 		8.9		0.2
\pinlabel {$x_1$} at 					34		4.5
\pinlabel {${}_{x''_1}$} at 				38.4		5.4
\pinlabel {${}_{x'''_1}$} at 				40.5		5.4
\pinlabel {$\color{blue} \Omega''_1$} at 	46.5		5.1
\pinlabel{${}_{\C_1^+}$} at 			36.1		5.4
\pinlabel{${}_{\Delta''_1}$} at 				41.9		3.5
\pinlabel{$\color{blue} \C_1''$} at 		44.2		5.1
\pinlabel{$\color{blue} g_{\mathsf{e}_1}$} at 		38		0.1
\pinlabel{$\mathcal{C}$} at 	20.6		7.3
\pinlabel {$\Omega$} at 		23		9.45
\pinlabel {$\beta$} at 		23.0		6.3
\endlabellist
\includegraphics[width = \textwidth]{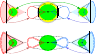}
\caption{Illustrations of the proofs of Theorems~\ref{thm:free-product-fd} (top) and~\ref{thm:free-product-cc} (bottom)}
\label{fig:free}
\end{figure}

For the rest of this proof, we refer to Figure~\ref{fig:free}, top panel.
For each $i\in\{0,1\}$, there exists $x'_i \in g_i \cdot \partial B$ such that the supporting hyperplane $X'_i$ to $g_i\cdot B$ at $x'_i$ misses $\overline{\Omega'_i \smallsetminus \Delta_i}$ and separates $g_i\cdot B^\circ$ from the tip $x_i$ in the pyramid~$\overline{\Delta_i}$.
Since $\partial B$ is smooth, up  to a small perturbation we may assume that the points $y_i := g_i^{-1}\cdot x'_i \in \partial B$ satisfy $y_0 \neq y_1$. 
Let $Y_i$ be the supporting hyperplane to $B$ at~$y_i$. 
Since $\overline{B}$ is strictly convex, the flags $(y_0, Y_0)$ and $(y_1, Y_1)$ are transverse.
We can find a biproximal element $\beta \in \PGL(V)$ whose attracting and repelling eigenflags, $(y_\beta^+, Y_\beta^+)$ and $(y_\beta^-,Y_\beta^-)$, are arbitrarily close to $(y_0, Y_0)$ and $(y_1, Y_1)$ respectively.
By Fact~\ref{fact:double-lim-set}, we may require $\beta$ to lie in any given Zariski-dense subgroup $\Gamma$ of $\PGL(V)$ satisfying $\Lambda^{\PP(V)}_{\Gamma} = \PP(V)$ and $\Lambda^{\PP(V^*)}_{\Gamma} = \PP(V^*)$, for instance any lattice of $\PGL(V)$ (see Fact~\ref{fact:lattices}).
Up to squaring $\beta$, we may assume that it preserves (hence acts cocompactly on) the open segment $(y_\beta^-, y_\beta^+)$; by Proposition~\ref{prop:thicken-add-dim-general}, this segment is then contained in some $\beta$-invariant properly convex open subset $\Omega$ of $\PP(V)$.
Up to replacing $\Omega$ by some open uniform neighborhood of $(y_\beta^-, y_\beta^+)$ in $(\Omega,d_{\Omega})$, we may assume that $\Omega$ is contained in~$\mathcal{B}$ (though not necessarily in~$B$).
There are properly convex neighborhoods $\mathcal{U}^\pm$ of $y_\beta^{\pm}$ in $\PP(V)$ such that $(\mathcal{U}^-, \Omega, \mathcal{U}^+)$ is in occultation position.

By construction, $g_0 \cdot Y_\beta^+$ misses $\overline{\Omega'_0 \smallsetminus \Delta_0}$, because $g_0\cdot Y_0=X'_0$ does.
Therefore, if we choose a point $x''_0 \in g_0 \cdot \Omega$, then the set $\Omega''_0:=\Omega''_0(x''_0)$ defined by~\eqref{eq:repyr} satisfies $\overline{\Omega''_0} \cap g_0 \cdot Y_\beta^+ =\varnothing$, hence for large $N\geq 0$ the set $\beta^{-N} g_0^{-1} \cdot \Omega''_0$ shrinks into~$\mathcal{U}^-$.

Similarly, $g_1 \cdot Y_\beta^-$ misses $\overline{\Omega'_1 \smallsetminus \Delta_1}$, because $g_1 \cdot Y_1 = X'_1$ does.
Choosing $x''_1 \in g_1 \cdot \Omega$, the set $\Omega''_1:=\Omega''_1(x''_1)$ satisfies $\overline{\Omega''_1} \cap g_1 \cdot Y_\beta^- =\varnothing$, hence for large $N$ the set $\beta^{N} g_1^{-1} \cdot \Omega''_1$ shrinks into~$\mathcal{U}^+$.

It follows that the triple
\begin{equation} \label{eq:betaN} 
\left ( \beta^{-N} g_0^{-1} \cdot \Omega''_0 ,~ \Omega,~ \beta^{N} g_1^{-1} \cdot \Omega''_1 \right )
\end{equation}
is in occultation position. 
Since $\Omega$ is $\langle \beta \rangle$-invariant and the triples \eqref{eq:repyr2} are in occultation position, so is the triple
\begin{equation} \label{eq:betaN2} 
\left (g_0 \beta^N \cdot \Omega, \Omega''_0, \gamma g_0 \beta^N \cdot \Omega \right ) \quad \text{ (\resp } \left (g_1 \beta^{-N} \cdot \Omega, \Omega''_1, \gamma g_1 \beta^{-N} \cdot \Omega \right )\text{)}
\end{equation}
for all $\gamma \in \Gamma_0 \smallsetminus \{1\}$ (\resp for all $\gamma \in \Gamma_1 \smallsetminus \{1\}$).

Let $\mathsf{Y} = (\mathsf{V},\mathsf{E})$ be the graph with three vertices $\mathsf{V} = \{\mathsf{v}_0, \mathsf{w} , \mathsf{v}_1\}$ and two edges $(\mathsf{w},\mathsf{v}_0)$ and $(\mathsf{w},\mathsf{v}_1)$. 
For $i \in \{0,1\}$, write $\mathsf{e}_i \in \vec{\mathsf{E}}$ for the directed edge from $\mathsf{w}$ to $\mathsf{v}_i$. 
Consider the graph of groups $(\boldsymbol{\Gamma}, \mathsf{Y})$ with $\Gamma_{\mathsf{v}_i} = \Gamma_i$ and $\Gamma_{\mathsf{w}}\simeq \Gamma_{|\mathsf{e}_i|} \simeq \{1\}$. 
Let $\rho_{\mathsf{v}_i}:\Gamma_i \rightarrow \PGL(V)$ be the inclusion, let $(g_{\mathsf{e}_{0}}, g_{\mathsf{e}_{1}}) := (\beta^{-N} g_0^{-1}, \beta^N g_1^{-1})$, and 
let $(\Omega_{\mathsf{v}_0}, \Omega_\mathsf{w}, \Omega_{\mathsf{v}_1}) := (\Omega''_0,  \Omega, \Omega''_1)$.

Theorem~\ref{thm:general-gog} applies to this datum, as the triples from \eqref{eq:betaN} and \eqref{eq:betaN2} are in occultation position. 
Since $\pi_1(\boldsymbol{\Gamma},\mathsf{Y},\mathsf{w})$ is naturally isomorphic to $\Gamma_0 * \Gamma_1$, this proves Theorem~\ref{thm:free-product-fd} with $g=g_{\mathsf{e}_0}^{-1} g_{\mathsf{e}_1} = g_0 \beta^{2N} g_1^{-1}$ (whose factors can, as remarked above, be chosen in any given Zariski-dense subgroup $\Gamma$ of $\PGL(V)$ satisfying $\Lambda^{\PP(V)}_{\Gamma} = \PP(V)$ and $\Lambda^{\PP(V^*)}_{\Gamma} = \PP(V^*)$, for instance any lattice of $\PGL(V)$).

\subsection{An application of Theorem~\ref{thm:free-product-fd} and Proposition~\ref{prop:thicken-add-dim}}

The following can be used \eg to construct discrete subgroups isomorphic to $\ZZ^{d-2} * \ZZ^{d-2}$ inside $\SL(d,\RR)$; see Corollary~\ref{cor:Soifer} below for an improvement.

\begin{corollary} \label{cor:free-product-fd-subspaces}
For $i\in\{0,1\}$, let $\Gamma_i$ be a discrete subgroup of $\SL^{\pm}(V)$ preserving a nontrivial decomposition $V = V_i \oplus V'_i$, preserving a nonempty properly convex open subset $\mathcal{O}_i$ of $\PP(V_i)$, and acting trivially on~$V'_i$.
Then there exists $g\in\PGL(V)$ such that the representation $\rho:  \Gamma_0 * g\Gamma_1 g^{-1} \to \PGL(V)$ induced by the inclusions $\Gamma_0, g\Gamma_1 g^{-1} \hookrightarrow \PGL(V)$ is discrete and faithful, and such that the image of this representation preserves a nonempty properly convex open subset of $\PP(V)$.
Moreover, we can take $g$ in any given lattice of $\PGL(V)$ (causing $\rho$ to take values in that lattice if it contains $\Gamma_0, \Gamma_1$), or more generally in any given Zariski-dense subgroup of $\PGL(V)$ whose proximal limit sets in $\PP(V)$ and $\PP(V^*)$ are everything.
\end{corollary}

\begin{proof}
By Proposition~\ref{prop:thicken-add-dim}, the group $\Gamma_i$ preserves a nonempty properly convex open subset $\Omega_i$ of $\PP(V)$ with $\Ccore_{\Omega_i}(\Gamma_i) = \Ccore_{\mathcal{O}_i}(\Gamma_i) \subset \Omega_i \cap \PP(V_i) \neq \Omega_i$.
We conclude by applying Theorem~\ref{thm:free-product-fd}.
\end{proof}

\subsection{Proof of Corollary~\ref{cor:counterex-nori}: counterexample in the setting of Nori's problem} \label{subsec:nori}

Fix $d\geq\nolinebreak 2$.
Let $\tau_d : \SL(d,\RR)\to\SL(S^2\RR^d)$ be the second symmetric power of the standard representation.
By Fact~\ref{fact:Omega-sym}, the group $\SL(d,\RR)$ acts properly and transitively, via~$\tau_d$, on some nonempty properly convex open subset $\Omega_d$ of $\PP(S^2\RR^d)$.
Let $V := S^2\RR^d\oplus\RR \simeq \RR^{n_d+1}$ and let $\tau'_d = \tau_d\oplus\mathbf{1} : \SL(d,\RR) \to \SL(V)$ be the sum of $\tau_d$ and of the trivial representation; it induces an embedding of $\SL(d,\RR)$ into $\PSL(V)$, which we still denote by~$\tau'_d$.

By Proposition~\ref{prop:thicken-add-dim}, the group $\SL(d,\RR)$, acting faithfully on $\PP(V)$ via~$\tau'_d$, preserves a properly convex open subset $\Omega'_d$ of $\PP(V)$ containing~$\Omega_d$ (where we see $\PP(S^2\RR^d)$ as a projective hyperplane of $\PP(V) = \PP(S^2\RR^d\oplus\RR)$), and $\Lambdao_{\Omega'_d}(\tau'_d(\SL(d,\RR)) = \Lambdao_{\Omega_d}(\tau_d(\SL(d,\RR)) \subset \PP(S^2\RR^d)$; in particular, we have $\Omega'_d \supsetneq \Ccore_{\Omega'_d}(\tau'_d(\SL(d,\RR))$.

Let $\Gamma_0$ be a lattice of $\SL(d,\RR)$.
Then $\Gamma'_0 := \tau'_d(\Gamma_0) < \PSL(S^2\RR^d\oplus\RR)$ preserves~$\Omega'_d$ and $\Omega'_d \supsetneq \Ccore_{\Omega'_d}(\tau'_d(\SL(d,\RR)) \supset \Ccore_{\Omega'_d}(\Gamma'_0)$.
Let $\Gamma_1$ be a Zariski-dense discrete subgroup of $\PGL(S^2\RR^d\oplus\RR)$ acting convex cocompactly on some properly convex open subset $\Omega_1$ of $\PP(S^2\RR^d\oplus\RR)$ such that $\Omega_1 \neq \Ccore_{\Omega_1}(\Gamma_1)$ (there exist many such groups, \eg free groups, as given \eg by Lemma~\ref{lem:free-groups} below).
By Theorem~\ref{thm:free-product-fd}, there exists $g\in\PGL(V)$ such that the subgroup $\Gamma$ of $\PGL(V)$ generated by $\Gamma'_0$ and $g\Gamma_1 g^{-1}$ is discrete in $\PGL(V)$, is isomorphic to the amalgamated free product $\Gamma_0 * g\Gamma_1 g^{-1}$, and preserves a nonempty properly convex open subset $\Omega$ of $\PP(V)$.
This group $\Gamma$ is Zariski-dense in $G := \PGL(V)$ because its subgroup $\Gamma_1$ is.
The intersection of $\Gamma$ with $H := \tau'_d(\SL(d,\RR))$ is the lattice $\Gamma'_0$ of~$H$.
The group $\Gamma$ is not itself a lattice of~$G$ since the attracting fixed points of proximal elements in $\Gamma$ all lie in $\partial\Omega$, hence we do \emph{not} have $\Lambda^{\PP(V)}_{\Gamma} = \PP(V)$ as in Fact~\ref{fact:lattices}. 
Corollary~\ref{cor:counterex-nori} is proved.

\begin{remark}
Similarly, in the construction above, taking for $\Gamma_0$ a lattice of any \emph{subgroup} $S$ of $\SL(d,\RR)$ yields a discrete subgroup $\Gamma$ of $G=\PGL(V)$ which is \emph{not} a lattice of~$G$ but contains a lattice $\tau'_d(\Gamma_0)$ of $H:=\tau'_d(S)$.
\end{remark}

\begin{remark}
As in Fact~\ref{fact:Omega-sym}, we can identify $S^2\RR^d$ with the space of symmetric $d\times d$ real matrices, and $\Omega_d$ with the projectivization of those symmetric matrices that are positive definite.
Using the fact that $\det(\cdot)^\frac{1}{d+1}$ is concave on positive definite symmetric matrices, we can then take the following explicit properly convex open set for~$\Omega'_d$:
$$\Omega'_d = \PP \left \{(M,t) \in \Omega_d \times \RR~ \middle | ~ |t| < \det(M)^\frac{1}{d+1} \right \} \subset \PP(V).$$
\end{remark}

\begin{remark} \label{rem:Nori-over-C-H}
The construction above also works over $\mathbb{K} = \CC$ or the ring $\HH$ of quaternions.
Indeed, let $U$ be the space of Hermitian $(d\times d)$ matrices over~$\mathbb{K}$; it is a real vector space of dimension $n_d := d^2$ if $\mathbb{K} = \CC$, and $n_d := 2d^2-d$ if $\mathbb{K} = \HH$.
The group $\SL(d,\mathbb{K})$ acts on~$U$ by $g\cdot X = g X \overline{g}^t$, which defines a linear representation $\tau_{d,\mathbb{K}} : \SL(d,\mathbb{K}) \to \SL(U)$.
It preserves a nonempty properly convex open subset $\Omega_{d,\mathbb{K}}$ of $\PP(U)$, namely the projectivization of the set of positive definite Hermitian matrices. 
Let $V := U\oplus\RR$ and let $\tau'_{d,\mathbb{K}} = \tau_{d,\mathbb{K}}\oplus\mathbf{1} : \SL(d,\mathbb{K}) \to \SL(V)$ be the sum of $\tau_{d,\mathbb{K}}$ and of the trivial representation; it induces an embedding of $\SL(d,\mathbb{K})$ into $\PSL(V)$, which we still denote by~$\tau'_{d,\mathbb{K}}$.
Arguing as above, we can construct a discrete subgroup $\Gamma$ of $G = \PGL(V) \simeq \PGL(n_d+1,\RR)$ which meets $H = \tau'_{d,\mathbb{K}}(\SL(d,\mathbb{K}))$ in a lattice of~$H$, providing a negative answer to Question~\ref{question-Nori}.
\end{remark}

\begin{lemma} \label{lem:free-groups}
Let $n \geq 2$ and let $F_2$ denote the free group on two generators.
Then there exists a representation $\rho : F_2 \to \SL(n+1,\RR)$ whose image is both convex cocompact in $\PP(\RR^{n+1})$ and Zariski-dense in $\SL(n+1,\RR)$. 
\end{lemma}

\begin{proof}
Start with a convex cocompact representation into $\SO(n,1)$.
Any small enough deformation into $\SL(n+1,\RR)$ remains convex cocompact in $\PP(\RR^{n+1})$.
There are arbitrarily small deformations which become Zariski-dense in $\SL(n+1,\RR)$.
\end{proof}

\subsection{Proof of Corollary~\ref{cor:combine-discrete-groups}: free products of arbitrary discrete groups} \label{subsec:proof-combine-discrete-groups}

Fix $d\geq 2$ and let $\Gamma_0$ and~$\Gamma_1$ be two discrete subgroups of $\SL(d,\mathbb{A}) \subset \SL(d,\RR)$.
Let $\tau'_d = \tau\oplus\mathbf{1} : \SL(d,\RR) \hookrightarrow \PSL(S^2\RR^d\oplus\RR) \simeq \PSL(n_d+1,\RR)$ be as in Section~\ref{subsec:nori} just above. Note that $\tau'_d(\SL(d,\mathbb A)) \subset \PSL(n_d+1,\mathbb A)$.
By Fact~\ref{fact:Omega-sym}, the groups $\Gamma'_0 := \tau'_d(\Gamma_0)$ and $\Gamma'_1 := \tau'_d(\Gamma_1)$ preserve the nonempty properly convex open subset $\Omega_d$ of $\PP(S^2\RR^d)$ and act trivially on the $\RR$ factor of $V = S^2\RR^d \oplus \RR$.
By Corollary~\ref{cor:free-product-fd-subspaces}, there exists $g\in\PGL(n_d+1,\mathbb{Z}) \subset \PGL(n_d+1,\mathbb{A})$ such that the representation $\rho:  \Gamma_0 * g\Gamma_1 g^{-1} \to \PGL(n_d+1,\mathbb{A})$ induced by the inclusions $\Gamma_0, g\Gamma_1 g^{-1} \hookrightarrow \PGL(n_d+1,\mathbb{A})$ is discrete and faithful, and preserves a nonempty properly convex open subset of $\PP(\RR^{n_d+1})$.

\subsection{A variant of Theorem~\ref{thm:free-product-fd} in a semisimple Lie group} \label{subsec:free-product-in-G}

As in Section~\ref{subsec:remind-prox}, let $\mathcal{F}$ be the space of partial flags $(x,X)$ with $x \in \PP(V)$ and $X \in \PP(V^*)$, and for any subgroup $\Gamma$ of $\PGL(V)$ let $\Lambda_{\Gamma}^{\mathcal{F}}$ be the proximal limit set of~$\Gamma$ in~$\mathcal{F}$.
As in Section~\ref{subsec:intro-applic-free-prod}, we say that a point of $\mathcal{F}$ is \emph{uniformly transverse} to a subset of~$\mathcal{F}$ if it is transverse to all points in the closure of the subset (and therefore also to all points in a neighborhood of that closure).
Theorem~\ref{thm:free-product-G/P} is a consequence of the following.

\begin{theorem} \label{thm:free-product-in-G}
Let $G$ be a noncompact closed connected subgroup of $\PGL(V)$ which is semisimple, acts irreducibly on $\PP(\mathrm{span}(\Lambda_G^{\PP(V)}))$, and preserves a nonempty properly convex open subset $\Omega$ of $\PP(V)$.
Let $\Gamma_0$ and~$\Gamma_1$ be discrete subgroups of~$G$, and let $\Gamma$ be a Zariski-dense subgroup of~$G$ (\eg $G$ itself).
Suppose that for each $i\in\{0,1\}$ there is a flag $F_i \in \Lambda_{\Gamma}^{\mathcal{F}}$ which is uniformly transverse to $(\Gamma_i\smallsetminus\{1\}) \cdot F_i$.
Then
\begin{enumerate}
  \item \label{item:unitrans} there is a neighborhood $\mathcal{W}_i$ of $F_i$ in~$\mathcal{F}$ such that $F'_i$ is uniformly transverse to\linebreak $(\Gamma_i\smallsetminus\{1\}) \cdot F''_i$ for all $F'_i,F''_i \in \mathcal{W}_i$;
  \item there exists $g\in\Gamma$ such that the representation $\rho : \Gamma_0 * g\Gamma_1 g^{-1} \to G$ induced by the inclusions $\Gamma_0, g\Gamma_1 g^{-1} \hookrightarrow G$ is discrete and faithful.
\end{enumerate}
\end{theorem}

\begin{proof}
For $i\in \{0,1\}$, we write the flag $F_i \in \Lambda_{\Gamma}^{\mathcal{F}}$ as $(x_i,X_i)  \in \PP(V) \times \PP(V^*)$.
We break up the proof in several steps and refer to Figure~\ref{fig:chapeau} for illustration.

\smallskip
\noindent
$\bullet$ \textbf{Step 1: find a suitable $\Gamma_i$-invariant properly convex open set.}
Since $X_i$ belongs to~$\Lambda_G^{\PP(V^*)}$, it is a supporting hyperplane to $\Omega$ at~$x_i$.
Let $\Lambda_i^* \subset \PP(V^*)$ be the set of accumulation points of the $\Gamma_i$-orbit of~$X_i$; it consists again of supporting hyperplanes to~$\Omega$.
Let $\Omega_i^{\max}$ be the connected component  of
$$\PP(V) \smallsetminus \bigcup_{H\in\Lambda_i^*} H$$
containing~$\Omega$; it is a $\Gamma_i$-invariant convex open subset of $\PP(V)$.
By assumption $x_i$ is uniformly transverse to $(\Gamma_i\smallsetminus\{1\}) \cdot X_i$, hence $x_i \in \Omega_i^{\max}$.
By Proposition~\ref{prop:thicken-not-proper}, there is a $\Gamma_i$-invariant properly convex open subset $\Omega_i$ of $\PP(V)$ (inside $\Omega_i^{\max}$) containing $\Omega \cup \{ x_i\}$.
Similarly, there is a $\Gamma_i$-invariant properly convex open subset $\mathcal{O}_i$ of $\PP(V^*)$ containing $\Omega^* \cup \{ X_i\}$.

\smallskip
\noindent
$\bullet$ \textbf{Step 2: uniform transversality is stable under small deformations.}
Let us prove~\eqref{item:unitrans}: by contradiction, suppose~\eqref{item:unitrans} fails.
Then there exist sequences $(F'_n)_{n\in\NN}, (F''_n)_{n\in\NN} \in \mathcal{F}^{\NN}$ converging to~$F_i$ such that for each~$n$, the flag $F'_n$ is \emph{not} uniformly transverse to $(\Gamma_i \smallsetminus \{1\}) \cdot F''_n$: there is a sequence $(\gamma_{n,k})_{k\in\NN}$ of elements of $\Gamma_i \smallsetminus \{1\}$, such that $(\gamma_{n,k}\cdot F''_n)_{k\in\NN}$ converges to some element of the set $Z_{F'_n} \subset \mathcal{F}$ of flags nontransverse to~$F'_n$.
Let $d_{\mathcal{F}}$ be a Riemannian distance function on~$\mathcal{F}$.
For each $n\geq 1$ there exists $\varphi(n)\in\NN$ such that $d_{\mathcal{F}}(\gamma_{n,\varphi(n)}\cdot F''_n, Z_{F'_n}) \leq 1/n$.
Up to passing to a subsequence, we may assume that $(\gamma_{n,\varphi(n)}\cdot F''_n)_{n\in\NN}$ converges to some $F''\in\PP(V)$.
On the other hand, since $F'_n \to F_i$, the compact sets $Z_{F'_n}$ converge to~$Z_{F_i}$.
Thus $F''$ is nontransverse to~$F_i$ by passing to the limit.
On the other hand, since $F''_n\to F_i$, we have $F'' = \lim_n \gamma_{n,\varphi(n)}\cdot F''_n = \lim_n \gamma_{n,\varphi(n)}\cdot F_i$ (if $\gamma_{n,\varphi(n)} \to \infty$ this follows from~\cite[Cor.\,3.6]{dgk-proj-cc} applied both to~$\Omega_i$ and to~$\mathcal{O}_i$). 
This contradicts uniform transversality of $F_i$ to $(\Gamma_i\smallsetminus\{1\}) \cdot F_i$.
Therefore \eqref{item:unitrans} holds: there is neighborhood $\mathcal{W}_i$ of $F_i$ in~$\mathcal{F}$ such that $F'_i$ is uniformly transverse to $(\Gamma_i\smallsetminus\{1\}) \cdot F''_i$ for all $F'_i,F''_i \in \mathcal{W}_i$.

We may and shall assume that this neighborhood $\mathcal{W}_i$ is of the form $\mathcal{F} \cap (\mathcal{U}_i \times \mathcal{V}_i)$ for some open subsets $\mathcal{U}_i$ of~$\Omega_i$ and $\mathcal{V}_i$ of~$\mathcal{O}_i$.

\begin{figure}[h]
\labellist
\small\hair 2pt
\pinlabel {$\mathcal{U}_i$} at 				3 	8
\pinlabel {$\mathcal{U}'_i$} at 				7.2	5.55
\pinlabel {$\Delta_i$} at 					9 	7.3
\pinlabel {$\C_i$} at 						4.4	0.2
\pinlabel {$\C_i^\varepsilon$} at 			0.8	1
\pinlabel {$H_i$} at 						-0.3	2.2
\pinlabel {$X_i$} at 						3.5	5.2
\pinlabel {${}_{x_i}$} at 					9.1	5.5
\pinlabel {$\Lambda_\Gamma^{\PP(V)}$} at 	-0.3	3
\pinlabel {${}_{y^+}$} at 					6.4	4.8
\pinlabel {${}_{y^-}$} at 					11.8	4.8
\pinlabel {$B$} at 						7.4	4.65
\pinlabel {$\Delta_i^z$} at 					4.6	2.9
\pinlabel {${}_z$} at 						8.9	4.6
\endlabellist
\includegraphics[width = 0.8\textwidth]{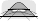}
\caption{Illustration of the proof of Theorem~\ref{thm:free-product-in-G}.}
\label{fig:chapeau}
\end{figure}

\smallskip
\noindent
$\bullet$ \textbf{Step 3: construct hats.}
Since $X_i$ belongs to~$\Lambda_G^{\PP(V^*)}$, it is a supporting hyperplane to $\Omega$ at~$x_i$; therefore, the closure of $(\Gamma_i\smallsetminus\{1\}) \cdot x_i$ lies entirely on one side of the projective hyperplane~$X_i$ inside the convex set~$\overline{\Omega_i}$, and in fact strictly on one side by the uniform transversality assumption.
Inside~$\Omega_i$, the convex hull $\C_i := \Conv{\Gamma_i \cdot x_i}$ admits $x_i$ as an extremal point; in fact it looks near $x_i$ like a properly convex cone (possibly of positive codimension) inside an affine chart of $\PP(V)$ containing~$\overline{\Omega_i}$.
Let $H_i$ be a projective hyperplane of $\PP(V)$ close to~$X_i$, not meeting $X_i$ in $\overline{\Omega_i}$, and chopping off a small conical neighborhood of $x_i$ in~$\C_i$ which is still contained in~$\mathcal{U}_i$.
Points of $\C_i \smallsetminus \mathcal{U}_i$ are a uniform $d_{\Omega_i}$-distance away from $H_i$, by uniform transversality of $F_i = (x_i, X_i)$ to its whole $(\Gamma_i\smallsetminus\{1\})$-orbit.
For $\varepsilon>0$, let $\C_i^\varepsilon$ be the closed uniform $\varepsilon$-neighborhood of $\C_i$ in $(\Omega_i,d_{\Omega_i})$.
If $\varepsilon$ is small enough, then $H_i$ chops off a small neighborhood (a ``hat'') $\Delta_i$ of $x_i$ in $\C_i^{\varepsilon}$ which is still contained in~$\mathcal{U}_i$, and $H_i \cap \C_i^{\varepsilon}$ (being contained in $H_i \cap \Omega_i$) does not meet~$X_i$.
Up to making the neighborhood $\mathcal{V}_i$ of $X_i$ smaller, we may assume that any hyperplane $X \in \mathcal{V}_i$ supporting $\Omega$ satisfies $X\cap \C_i^{\varepsilon} = X\cap \Delta_i \neq \varnothing$, and the base $H_i \cap \C_i^{\varepsilon}$ of the hat $\Delta_i$ lies entirely on the side of $X$ that contains the properly convex set~$\Omega$ (preserved by $G$ and supported by $X_i)$. 

There is a neighborhood $\mathcal{U}'_i \subset \mathcal{U}_i$ of $x_i$ in $\PP(V)$ whose closure is contained in the hat $\Delta_i$. 
Define $\mathcal{W}'_i := \mathcal{U}'_i \times \mathcal{V}_i \subset \mathcal{W}_i$.

\smallskip
\noindent
$\bullet$ \textbf{Step 4: conjugate appropriately.}
Having completed the first three steps separately for both values $i\in \{0,1\}$, we conclude as follows.

By assumption the group $G$ contains an element which is biproximal in $\PP(V)$.
Since $G$ acts irreducibly on $\PP(\mathrm{span}(\Lambda_G^{\PP(V)}))$, we may apply \cite[Prop.\,3.3]{ggkw17} to the restriction $G \to \PGL(\mathrm{span}(\Lambda_G^{\PP(V)}))$ of the natural inclusion $G \hookrightarrow \PGL(V)$: this yields a self-opposite parabolic subgroup $P$ of~$G$~such~that
\begin{itemize}
  \item for any $g\in G \subset \PGL(V)$, the element $g$ is proximal in $G/P$ if and only if it is biproximal in $\PP(V)$,
  \item the natural inclusion $G\hookrightarrow\PGL(V)$ induces a $G$-equivariant embedding $G/P \hookrightarrow \mathcal{F}$, with image $\Lambda_G^{\mathcal{F}}$, with the property that two points in $G/P$ are transverse if and only if their images in $\mathcal{F}$ are transverse.
\end{itemize}

Since $\Gamma$ is Zariski-dense in~$G$, there exists a flag in $\Lambda_{\Gamma}^{G/P} \simeq \Lambda_{\Gamma}^{\mathcal{F}}$ which is transverse to $F_1 = (x_1,X_1)$.
By Fact~\ref{fact:double-lim-set-G/P}, there is an element $\gamma\in\Gamma$ which is proximal in $G/P$ hence biproximal in $\PP(V)$, whose attracting flag $\gamma^+$ belongs to~$\mathcal{W}'_0$ and whose repelling flag $\gamma^-$ is transverse to~$F_1$.
Up to shrinking $\mathcal{W}'_1$ around~$F_1$, we may assume that~$\gamma^-$ is transverse to all flags in~$\mathcal{W}'_1$.
Then for all large enough~$n$, we have $\gamma^n \cdot \mathcal{W}'_1 \subset \mathcal{W}'_0$.
Therefore, up to replacing $\Gamma_1$  with $\gamma^n \Gamma_1 \gamma^{-n}$ for large~$n$, we may and shall assume $\mathcal{W}'_0 \cap \mathcal{W}'_1 \cap \Lambda_{\Gamma}^{\mathcal{F}} \neq \varnothing$.

By Fact~\ref{fact:double-lim-set-G/P}, there is an element $h\in\Gamma$ which is proximal in $G/P$, hence biproximal in $\PP(V)$, and whose attracting and repelling flags $(y^+, Y^+)$, $(y^-, Y^-)$ both belong to $\mathcal{W}'_0 \cap \mathcal{W}'_1$.
By Proposition~\ref{prop:thicken-add-dim-general}, the segment 
$$(y^-, y^+) ~\subset~ \mathcal{U}'_0 \cap \mathcal{U}'_1 \cap \Omega ~\subset~ \Delta_0 \cap \Delta_1 \cap \Omega$$ 
admits an $h$-invariant properly convex open neighborhood $B \subset \Delta_0 \cap \Delta_1\cap \Omega$.
Since $B$ is contained in~$\Omega$, it lies on the same side of $Y^\pm$ (inside~$\Omega_i$) as the base $H_i \cap \C_i^{\varepsilon}$ of the hat $\Delta_i$, for each $i\in \{0,1\}$.

Fix a point $z \in B$, and let $\Delta_i^z \subset \Delta_i$ be the open pyramid with vertex $z$ and base $H_i \cap \Int{(\C_i^{\varepsilon})}$.
By construction, the closure of $\Delta_i^z$ is disjoint from the hyperplanes~$Y^\pm$. 
For $i\in\{0,1\}$, define
\begin{equation} \label{eq:pyrchir}
 \Omega_i^z := \bigg(\Int{(\C_i^{\varepsilon})} \:\: \smallsetminus \:\: \bigsqcup _{\gamma \in \Gamma_i} (\gamma \cdot \Delta_i)\bigg) \:\: \cup \:\: \bigsqcup_{\gamma \in \Gamma_i} (\gamma \cdot \Delta_i^z) \quad \subset \:\:\Int{(\C_i^{\varepsilon})}.
 \end{equation}
Then $\Omega_i^z$ is a nonempty $\Gamma_i$-invariant properly convex open subset of $\PP(V)$, meeting $B$, and $\Omega_i^z$ is disjoint from neighborhoods of the hyperplanes $Y^\pm$, because $(y^\pm, Y^\pm)\in\mathcal{W}'_i$ implies $Y^\pm \cap \C_i^{\varepsilon} = Y^\pm \cap \Delta_i$.
Therefore, for large $N\in \NN$ and for all $\gamma \in \Gamma_i \smallsetminus \{1\}$, the triples
\begin{equation}\label{eq:oryx}
 (B~,~ \Omega_i^z~,~ \gamma \cdot B) \quad \text{ and } \quad (h^N \cdot \Omega_0^z~,~ B~,~ h^{-N} \cdot \Omega_1^z)
 \end{equation}
are in occultation position: the first because $B\subset \Delta_i$ and any segment from the hat $\Delta_i$ to $\gamma \cdot \Delta_i$ crosses $\Int{(\C_i^\varepsilon)} \smallsetminus \bigsqcup_{\gamma'\in \Gamma_i} \gamma'\cdot \Delta_i \subset \Omega_i^z$; the second because $h^{N} \cdot \Omega_{0}^z \to y^+$ and $h^{-N} \cdot \Omega_{1}^z \to y^-$ as $N \to +\infty$. 

We conclude as in the proof of Theorem~\ref{thm:free-product-fd}: by~\eqref{eq:oryx}, Theorem~\ref{thm:general-gog} applies to the graph of groups $\big [ (\Gamma_0) \overset{\mathsf{e}_0}\longleftarrow (1) \overset{\mathsf{e}_1}\longrightarrow (\Gamma_1) \big ] \simeq \Gamma_0 *_{\{1\}} \Gamma_1$, with $\Omega_{(\Gamma_i)}= \Omega_i^z$ and $\Omega_{(1)}=B$ and $(g_{\mathsf{e}_0}, g_{\mathsf{e}_1}) =(h^N, h^{-N})$. Theorem~\ref{thm:free-product-in-G} follows, with $g=g_{\mathsf{e}_0}^{-1} g_{\mathsf{e}_1} = h^{-2N}$.
\end{proof}

\begin{remark} \label{rem:move-g}
It follows from the proof that the element~$g$ can be chosen to admit a neighborhood $\mathcal{N}$ in~$G$ such that for any $g'\in\mathcal{N}$, the representation $\Gamma_0 * g'\Gamma_1 {g'}^{-1} \to G$ induced by the inclusions $\Gamma_0, g'\Gamma_1 {g'}^{-1} \hookrightarrow G$ is still discrete and faithful.
\end{remark}

\begin{proof}[Proof of Theorem~\ref{thm:free-product-G/P}]
By \cite[Lem.\,3.2 \& Prop.\,3.3]{ggkw17}, there is a finite-dimensional irreducible linear representation $\tau : G\to\GL(V)$ of~$G$ such that
\begin{itemize}
  \item for any $g\in G$, the element $g$ is proximal in $G/P$ if and only if $\tau(g)$ is biproximal in $\PP(V)$,
  \item there is a $\tau$-equivariant embedding $\iota : G/P \hookrightarrow \mathcal{F}$, with image $\Lambda_G^{\mathcal{F}}$, such that a pair of points of $G/P$ is transverse if and only if its image under~$\iota$ is transverse in~$\mathcal{F}$,
\end{itemize}
where $\mathcal{F}$ is the space of partial flags $(x,X)$ with $x \in \PP(V)$ and $X \in \PP(V^*)$.
Up to replacing $\tau$ by $\mathrm{Sym}^2(\tau)$ and $V$ by $\mathrm{Sym}^2(V)$ (see Fact~\ref{fact:Omega-sym}), we may furthermore assume that $\tau(G)$ preserves a nonempty properly convex open subset $\Omega$ of $\PP(V)$.
For each $i\in\{0,1\}$, since the point $x_i \in G/P$ is uniformly transverse to $(\Gamma_i\smallsetminus\{1\}) \cdot x_i$, the group $\Gamma_i$ acts faithfully on $G/P$; therefore $\Gamma_i$ also acts faithfully on $\Lambda_G^{\PP(V)} = \iota(G/P)$ via~$\tau$, and the restriction of $\tau$ to~$\Gamma_i$ is injective.
Moreover, $\iota(x_i) \in \Lambda_{\tau(\Gamma)}^{\mathcal{F}}$ is uniformly transverse to $(\tau(\Gamma_i)\smallsetminus\{1\}) \cdot x_i$.
We conclude by applying Theorem~\ref{thm:free-product-in-G}.
\end{proof}

\begin{corollary} \label{cor:Soifer}
For any $d\geq 2$, the group $G = \SL(d,\RR)$ contains a discrete subgroup isomorphic to the free product $\ZZ^{d-1} * \ZZ^{d-1}$, where each $\ZZ^{d-1}$ consists of matrices that are simultaneously diagonalizable over~$\RR$.
Such a discrete subgroup can be taken to be Zariski-dense in~$G$, but is not a lattice in~$G$.
\end{corollary}

The existence of a discrete subgroup of $\SL(d,\RR)$ isomorphic to $\ZZ^{d-1} * \ZZ^{d-1}$ was first proved by Soifer \cite{soi12} for $d=3$. 
Note that for any $k\geq 1$, the free product of $k+1$ copies of $\ZZ^{d-1}$ lives as an index-$k$ subgroup in $\ZZ^{d-1} * \ZZ^{d-1}$ (\eg the subgroup of elements whose abelianization belongs to $k\ZZ \times \ZZ^{2d-3}$).
Corollary~\ref{cor:Soifer} answers an open case of Nori's problem discussed in \cite[\S\,8]{fis-survey-margulis}.

\begin{proof}
For each $1\leq i\leq d-1$, let $v_i \in \ssl(d,\RR)$ be the diagonal matrix whose diagonal entries are all $-1$ except for the $i$-th entry which is $d-\nolinebreak 1$.
Let $\Gamma_0 = \exp(\ZZ\text{-span}(\{ v_1,\dots,v_{n-1}\}))$.
We claim that there is a flag $(x,X)$ (where $x\in\PP(\RR^d)$ and $X\in\PP((\RR^d)^*)$) which is uniformly transverse to $(\Gamma_0\smallsetminus\{1\}) \cdot (x,X)$.
Indeed, let $x := [(1,\dots,1)] \in \PP(\RR^d)$.
The orbit closure 
$$\overline{\Gamma_0 \cdot x} \subset \PP \big(\big\{ (y_1,\dots,y_d) \in \big ( \{\mathrm{e}^{-\nu}\}_{\nu\in\NN} \cup \{0\} \big )^d ~ \big | ~ \max(y_1,\dots,y_d) = 1 \big\}\big)$$ 
is compact and countable in $\PP(\RR^d)$, hence so is $\Xi:= \overline{\Gamma_0 \cdot x} \smallsetminus \{x\}$.
Therefore there exists a projective hyperplane $X$ of $\PP(\RR^d)$ containing~$x$ which is disjoint from~$\Xi$, and in fact from some uniform $\varepsilon$-neighborhood $\Xi_\varepsilon$ of $\Xi$ (for the spherical metric on $\PP(\RR^d)$). 
Thus, $\gamma \cdot x$ misses the $\varepsilon$-neighborhood $\mathcal{X}_\varepsilon$ of $X$ for all $\gamma\in \Gamma_0\smallsetminus \{1\}$.
Dually, there exists a neighborhood $\mathbf{x}_\varepsilon$ of $x$ such that $\gamma \cdot \mathbf{x}_\varepsilon \subset \Xi_\varepsilon$ for all $\gamma \in \Gamma_0 \smallsetminus \{1\}$; hence (applying $\gamma^{-1}$) this $\mathbf{x}_\varepsilon$ misses $\gamma^{-1}\cdot X$ for all $\gamma \in \Gamma_0 \smallsetminus \{1\}$. 
The flag $(x,X)$ thus has the desired uniform transversality property.

We apply Theorem~\ref{thm:free-product-in-G} with $\Gamma_0 = \Gamma_1$ as above: we obtain the existence of $g\in G$ such that the representation $\rho : \Gamma_0 * g\Gamma_1 g^{-1} \to G$ induced by the inclusions $\Gamma_0, g\Gamma_1 g^{-1} \hookrightarrow G$ is discrete and faithful.
Up to deforming~$g$ slightly using Remark~\ref{rem:move-g}, we may assume that $\Gamma_0 * g\Gamma_1 g^{-1}$ acts strongly irreducibly on~$\RR^d$.
Since its Zariski closure in~$G$ contains a Cartan subgroup of~$G$, it must be all of~$G$.
\end{proof}

\section{Reminders on properly convex sets and convex cocompactness} \label{sec:remind}

We recall some terminology and facts from~\cite{dgk-proj-cc} that will be needed in the rest of the paper.

\subsection{Boundaries of convex sets} \label{subsec:remind-boundaries}

We use the following terminology.

\begin{definition} \label{def:boundaries}
Let $\C$ be a nonempty convex subset of $\PP(V)$ (not necessarily open nor closed, possibly with empty interior).
\begin{itemize}
  \item The \emph{frontier} of~$\C$ is $\Fr(\C):=\overline{\C}\smallsetminus\Int{\C}$.
  \item A \emph{supporting hyperplane} of~$\C$ at a point $x\in\Fr(\C)$ is a projective hyperplane $X$ such that 
  $x \in X \cap \overline \C$ and $\overline \C \smallsetminus X$ is connected (possibly empty).
  \item A point $x\in\Fr(\C)$ is \emph{$C^1$ in $\overline{\C}$} if there is a unique supporting hyperplane of $\C$ at~$x$.
  \item The \emph{ideal boundary} of~$\C$ is $\partiali\C:=\overline{\C}\smallsetminus\C$.
  \item The \emph{nonideal boundary} of~$\C$ is $\partialn\C:= \C \smallsetminus \Int{\C} = \Fr(\C)\smallsetminus\partiali\C$.
  Note that if $\C$ is open, then $\partiali \C = \Fr(\C)$ and $\partialn \C = \varnothing$; in this case, it is common in the literature to denote $\partiali \C$ simply by $\partial \C$.
  \item The convex set~$\C$ has \emph{strictly convex nonideal boundary} if every point $x \in \partialn \C$ is an extremal point of~$\overline \C$.
  \item The convex set~$\C$ has \emph{$C^1$ nonideal boundary} if every point $x \in \partialn \C$ is $C^1$ in~$\overline{\C}$.
  \item \label{item:bisatdef} The convex set~$\C$ has \emph{bisaturated boundary} if for any supporting hyperplane $X$ of~$\C$, the set $X \cap \overline{\C} \subset \Fr(\C)$ is either fully contained in $\partiali \C$ or fully contained in $\partialn \C$.
  \end{itemize}
\end{definition}

Note that strictly convex nonideal boundary implies bisaturated boundary.
If $\C$ has bisaturated boundary and if $\C$ is not closed in $\PP(V)$, then $\C$ has nonempty interior.

\begin{remark} \label{rem:strict-conv-rays}
If $x\in\partial\Omega$ is $C^1$ in~$\overline{\Omega}$, then any two rays of~$\Omega$ with endpoint~$x$ are asymptotic, in the sense that their $d_{\Omega}$-distance tends to~$0$.
\end{remark}

Here $d_{\Omega}$ is the Hilbert metric of~$\Omega$ (see \eqref{eqn:d-Omega}).

\subsection{Basic facts on convex cocompact actions}

Recall the notions of full orbital limit set $\Lambdao_{\Omega}(\Gamma)$ and of (naively) convex cocompact actions on~$\Omega$ from Definitions \ref{def:limcore} and~\ref{def:cc-group}.
We shall use the following facts.

\begin{fact}[{\cite[Prop.\,4.18]{dgk-proj-cc}}] \label{fact:cc-subsets}
Let $\Gamma$ be an infinite discrete subgroup of $\PGL(V)$ preserving two nonempty properly convex open subsets $\Omega \subset \Omega'$ of $\PP(V)$.
\begin{itemize}
\item If $\Gamma$ acts convex cocompactly on~$\Omega$, then it also acts convex cocompactly on~$\Omega'$ and $\Lambdao_{\Omega'}(\Gamma) = \Lambdao_{\Omega}(\Gamma)$.
\item If $\Gamma$ acts convex cocompactly on $\Omega'$, then it acts convex cocompactly on $\Omega$ if and only if $\Ccore_{\Omega'}(\Gamma) \subset \Omega$. 
\end{itemize}
In particular, if $\Gamma$ acts convex cocompactly on~$\Omega$, then it also acts convex cocompactly on any properly convex open subset of $\PP(V)$ containing $\Ccore_{\Omega}(\Gamma)$.
\end{fact}

\begin{fact}[{\cite[Cor.\,4.10.(3) \& Lem.\,4.16]{dgk-proj-cc}}] \label{fact:ideal-bound-cc}
Let $\Gamma$ be an infinite discrete subgroup of $\PGL(V)$ acting convex cocompactly on a properly convex open subset $\Omega$ of $\PP(V)$.
Then any nonempty closed $\Gamma$-invariant convex subset $\C$ of~$\Omega$ contains $\Ccore_{\Omega}(\Gamma)$; moreover, if the action of $\Gamma$ on~$\C$ is cocompact (\eg if $\C = \Ccore_{\Omega}(\Gamma)$), then $\partiali \C = \Lambdao_{\Omega}(\Gamma)$.
\end{fact}

\begin{fact}[{\cite[Cor.\,5.4 \& Prop.\,5.10]{dgk-proj-cc}}] \label{fact:bisat-cc}
Let $\Gamma$ be an infinite discrete subgroup of $\PGL(V)$.
\begin{itemize}
  \item Suppose that $\Gamma$ acts properly discontinuously and cocompactly on a nonempty properly convex subset $\C$ of $\PP(V)$ with bisaturated boundary.
  Then $\Omega:=\Int\C$ is nonempty and $\Gamma$ acts convex cocompactly on both $\Omega$ and~$\Omega^*$.
  \item Suppose that $\Gamma$ acts convex cocompactly on some properly convex open subset $\Omega$ of $\PP(V)$ and also convex cocompactly on its dual~$\Omega^*$.
  Then $\C := \overline{\Omega} \smallsetminus \Lambdao_{\Omega}(\Gamma)$ is a nonempty properly convex subset of $\PP(V)$ with bisaturated boundary on which $\Gamma$ acts properly discontinuously and cocompactly.
\end{itemize}
\end{fact}

Following \cite[Def.\,4.11]{dgk-proj-cc}, given a properly convex open subset $\Omega$ of $\PP(V)$ and a point $z\in\partial\Omega$, we define the \emph{open face} of $\partial\Omega$ at~$z$ to be the union of $\{z\}$ and of all open segments of $\partial\Omega$ containing~$z$; in other words, it is the largest convex subset of $\partial\Omega$ containing~$z$ which is relatively open, in the sense that it is open in the projective subspace that it spans.
An \emph{open face} of $\partial\Omega$ is by definition the open face of some point $z\in\partial\Omega$.
The set $\partial\Omega$ is the disjoint union of its open faces.

\begin{fact}[{\cite[Cor.\,4.13]{dgk-proj-cc}}] \label{fact:cc-strata}
Let $\Gamma$ be a discrete subgroup of $\PGL(V)$ acting convex cocompactly on a properly convex open subset $\Omega$ of $\PP(V)$.
Let $F$ be an open face of $\partial \Omega$.
Then $\Lambdao_\Omega(\Gamma) \cap F$ is either empty or all of~$F$.
\end{fact}

Here is a consequence of Facts \ref{fact:bisat-cc} and~\ref{fact:cc-strata}.

\begin{corollary} \label{cor:cc-bisat-support}
Let $\Gamma$ be an infinite discrete subgroup of $\PGL(V)$ preserving a properly convex open subset $\Omega$ of $\PP(V)$.
Suppose that $\Gamma$ acts convex cocompactly on both $\Omega$ and~$\Omega^*$.
Then
\begin{enumerate}
  \item\label{item:cc-bisat-support-1} any element of $\Lambdao_{\Omega^*}(\Gamma)$, seen as a supporting hyperplane to~$\Omega$, meets $\partial\Omega$ in a subset of $\Lambdao_{\Omega}(\Gamma)$;
  \item\label{item:cc-bisat-support-2} any supporting hyperplane to~$\Omega$ at a point of $\Lambdao_{\Omega}(\Gamma)$ lies in $\Lambdao_{\Omega^*}(\Gamma)$.
\end{enumerate}
\end{corollary}

\begin{proof}
The two statements are dual to each other, so we only need to prove~\eqref{item:cc-bisat-support-1}.

First, we observe that any element $X$ of $\Lambdao_{\Omega^*}(\Gamma)$, seen as a supporting hyperplane to~$\Omega$, meets $\Lambdao_{\Omega}(\Gamma)$.
Indeed, we can write $X = \lim_n \gamma_n \cdot Y$ for some diverging sequence $(\gamma_n) \in \Gamma^{\NN}$ and some $Y \in \Omega^*$.
Lift $Y$ to a nonzero linear form $\tilde{Y} \in V^*$, and consider a nonzero vector $\tilde{y} \in V$ projecting to a point $y \in \Omega$.
Up to passing to a subsequence, we may assume that $\gamma_n \cdot y \to x$ for some $x \in \PP(V)$, which then belongs to $\Lambdao_{\Omega}(\Gamma)$.
We have $(\gamma_n\cdot\tilde{Y})(\gamma_n\cdot\tilde{y}) = (\gamma_0\cdot\tilde{Y})(\gamma_0\cdot\tilde{y}) \in \RR\smallsetminus\{0\}$ for all~$n$.
On the other hand, the norms of $\gamma_n\cdot\tilde{Y}$ and $\gamma_n\cdot\tilde{y}$ go to infinity by \cite[Lem.\,3.3]{dgk-proj-cc}.
Therefore $x \in X$ by passing to the limit.

By Fact~\ref{fact:bisat-cc}, the set $\C := \overline{\Omega} \smallsetminus \Lambdao_{\Omega}(\Gamma)$ is a nonempty properly convex subset $\C$ of $\PP(V)$ with bisaturated boundary on which $\Gamma$ acts properly discontinuously and cocompactly.
Consider an element $X \in \Lambdao_{\Omega^*}(\Gamma)$.
Then $X$, seen as a supporting hyperplane to~$\Omega$, meets $\overline{\C} = \overline{\Omega}$ in a subset of $\Fr(\C) = \partial\Omega$.
By the above observation, $X$ contains a point of $\Lambdao_{\Omega}(\Gamma) = \partiali \C$.
Since $\C$ has bisaturated boundary, $X$ must meet $\overline{\C} = \overline{\Omega}$ in a subset of $\partiali\C = \Lambdao_{\Omega}(\Gamma)$.
\end{proof}

\begin{fact}[{\cite[Cor.\,5.2]{dgk-proj-cc}}] \label{fact:cc-nbhd}
Let $\Gamma$ be an infinite discrete subgroup of $\PGL(V)$ acting convex cocompactly on a properly convex open subset $\Omega$ of $\PP(V)$.
Let $\C$ be a closed convex subset of~$\Omega$ on which $\Gamma$ acts cocompactly and which contains a neighborhood of $\Ccore_\Omega(\Gamma)$.
Then $\C$ is a properly convex subset of $\PP(V)$ with bisaturated boundary on which $\Gamma$ acts properly discontinuously and cocompactly.
\end{fact}

\begin{lemma} \label{lem:even-more-pure}
Let $\Gamma$ be an infinite discrete subgroup of $\PGL(V)$ acting convex cocompactly on a properly convex open subset $\Omega$ of $\PP(V)$.
Then the action of $\Gamma$ on $\overline{\Omega} \smallsetminus \Lambdao_{\Omega}(\Gamma)$ is cocompact (but not necessarily properly discontinuous).
In particular, for any sequence $(x_n)_{n\in\NN}$ of points of~$\Omega$ converging to a point of $\Lambdao_{\Omega}(\Gamma)$, there is a sequence $(\gamma_n)_{n\in\NN}$ of elements of~$\Gamma$ such that all accumulation points of the sequence $(\gamma_n\cdot x_n)_{n\in\NN}$ lie in $\overline{\Omega} \smallsetminus \Lambdao_{\Omega}(\Gamma)$. 
\end{lemma}

\begin{proof}
Let $\mathcal{D}$ be a compact fundamental domain for the action of $\Gamma$ on $\Ccore_{\Omega}(\Gamma)$.
Let $\mathcal{D}'$ be the closure in~$\overline{\Omega}$ of the set of points $x\in\Omega$ for which there exists $y\in\mathcal{D}$ with $d_{\Omega}(x,\Ccore_{\Omega}(\Gamma)) = d_{\Omega}(x,y)$.
Then $\mathcal{D}'$ is a compact subset of~$\overline{\Omega}$ and $\bigcup_{\gamma\in\Gamma} \gamma\cdot\mathcal{D}'$ contains~$\Omega$.

Let us check that $\mathcal{D}' \cap \Lambdao_{\Omega}(\Gamma) = \varnothing$.
Suppose by contradiction that there is a point $x_{\infty} \in \mathcal{D}' \cap \Lambdao_{\Omega}(\Gamma)$.
We can write $x_{\infty}$ as the limit of some sequence $(x_n)_{n\in\NN}$ of points of $\mathcal{D}' \cap \Omega$.
By definition of~$\mathcal{D}'$, for each~$n$ there exists $y_n \in \mathcal{D}$ such that $d_{\Omega}(x_n,\Ccore_{\Omega}(\Gamma)) = d_{\Omega}(x_n,y_n)$.
Up to passing to a subsequence, we may assume that $(y_n)_{n\in\NN}$ converges to some $y_{\infty} \in \mathcal{D} \subset \Ccore_{\Omega}(\Gamma)$.
The ray $[y_\infty, x_\infty)$ is contained in $\Ccore_{\Omega}(\Gamma)$.
Choose a point $w \in (y_{\infty}, x_{\infty})$.
We claim that
$$\varepsilon_n := |d_{\Omega}(x_n, y_n) - d_\Omega(x_n, w) - d_\Omega(w,y_n)| \longrightarrow 0.$$
Indeed, consider a projective hyperplane intersecting the ray $[y_\infty, x_\infty)$ transversely at~$w$.
For sufficiently large~$n$, it intersects the segment $[x_n,y_n]$ in a point~$w_n$, with $w_n \to w$.
Since $d_{\Omega}(x_n, y_n) = d_\Omega(x_n, w_n) + d_\Omega(w_n,y_n)$ and $d_{\Omega}(w_n,w) \to 0$, the triangle inequality gives $\varepsilon_n \to 0$.
We then have
$$d_{\Omega}(x_n,y_n) - d_{\Omega}(x_n, w) \geq d_{\Omega}(w,y_n) - \varepsilon_n \longrightarrow d_{\Omega}(w,y) > 0.$$
Thus $w$ is a point of $\Ccore_{\Omega}(\Gamma)$ which, for sufficiently large~$n$, is closer to $x_n$ than~$y_n$.
This contradicts the definition of~$y_n$.
Thus $\mathcal{D}' \cap \Lambdao_{\Omega}(\Gamma) = \varnothing$.

Finally, let us show that $\bigcup_{\gamma\in\Gamma} \gamma\cdot\mathcal{D}'$ contains (hence is equal to) $\overline{\Omega} \smallsetminus \Lambdao_{\Omega}(\Gamma)$.
Consider a point $x_{\infty} \in \partial \Omega \smallsetminus \Lambdao_{\Omega}(\Gamma)$.
We can write it as the limit of some sequence $(x_n)_{n\in\NN}$ of points of~$\Omega$.
For each~$n$, consider $y_n \in \Ccore_{\Omega}(\Gamma)$ such that $d_{\Omega}(x_n,\Ccore_{\Omega}(\Gamma)) = d_{\Omega}(x_n, y_n)$.
Up to passing to a subsequence, we may assume that $(y_n)_{n\in\NN}$ converges to some $y_{\infty} \in \overline{\Ccore_{\Omega}(\Gamma)}$.
We claim that $y_{\infty} \in \Ccore_{\Omega}(\Gamma)$.
Indeed, suppose by contradiction that $y_{\infty}$ belongs to $\partiali\Ccore_{\Omega}(\Gamma)$, which is equal to $\Lambdao_{\Omega}(\Gamma)$ by Fact~\ref{fact:ideal-bound-cc}.
Since $x_{\infty} \in \partial \Omega \smallsetminus \Lambdao_{\Omega}(\Gamma)$ and since $\Lambdao_{\Omega}(\Gamma)$ is a union of open faces of $\partial \Omega$ (Fact~\ref{fact:cc-strata}), the interval $(x_{\infty},y_{\infty})$ is contained in~$\Omega$.
Choose a point $w \in (x_{\infty},y_{\infty})$.
As above, we have
$$\varepsilon_n := |d_{\Omega}(x_n, y_n) - d_\Omega(x_n, w) - d_\Omega(w,y_n)| \longrightarrow 0.$$
Consider $p \in \Ccore_{\Omega}(\Gamma)$ such that $d_\Omega(w, \Ccore_{\Omega}(\Gamma)) = d_\Omega(w, p)$.
We have
$$d_\Omega(x_n, y_n) -  d_\Omega(x_n,  p) \geq d_\Omega(w,y_n) - \varepsilon_n - d_\Omega(w,  p) \longrightarrow +\infty. $$
Thus $p$ is a point of $\Ccore_{\Omega}(\Gamma)$ which, for sufficiently large~$n$, is closer to $x_n$ than~$y_n$.
This contradicts the definition of~$y_n$.
Thus it is not the case that $y_{\infty} \in \Lambdao_{\Omega}(\Gamma)$, and so we must have $y_{\infty} \in \Ccore_\Omega(\Gamma)$, proving the claim.
Since $y_n \to y_{\infty}$, up to passing to a subsequence we may assume that there exists $\gamma \in \Gamma$ such that $y_n \in \gamma \cdot \mathcal{D}$ for all~$n$.
We then have $x_n \in \gamma \cdot \mathcal{D}'$ for all~$n$, hence $x_{\infty} \in \gamma \cdot \mathcal{D}'$.
\end{proof}

\begin{lemma} \label{lem:enter-nbhd}
Let $\Gamma$ be an infinite discrete subgroup of $\PGL(V)$ acting convex cocompactly on some properly convex open subset $\Omega$ of $\PP(V)$.
For any $x\in\Omega$ and any $\xi\in\Lambdao_{\Omega}(\Gamma)$, the ray $[x,\xi)$ eventually enters any uniform neighborhood of $\Ccore_{\Omega}(\Gamma)$ in $(\Omega,d_{\Omega})$.
\end{lemma}

\begin{proof}
For $x\in\Omega$ and $\xi\in\Lambdao_{\Omega}(\Gamma)$, consider a sequence $(x_n)_{n\in\NN}$ of points on the ray $[x,\xi)$, with $x_n \to \xi$.
Let us check that $d_{\Omega}(x_n,\Ccore_{\Omega}(\Gamma)) \to 0$.

It follows from the definition \eqref{eqn:d-Omega} of the Hilbert metric that the ray $[x,\xi)$ remains at bounded distance in $(\Omega,d_{\Omega})$ from any other ray with endpoint~$\xi$.
In particular, there exists $R>0$ such that $[x,\xi)$ is contained in the uniform $R$-neighborhood $\C_R$ of $\Ccore_{\Omega}(\Gamma)$ in $(\Omega,d_{\Omega})$.
Since the action of $\Gamma$ on $\Ccore_{\Omega}(\Gamma)$ is cocompact by assumption, so is the action of $\Gamma$ on~$\C_R$.
Let $\mathcal{D}_R$ be a compact fundamental domain for the action of $\Gamma$ on~$\C_R$.
For any $n\in\NN$, there exists $\gamma_n\in\Gamma$ such that $\gamma_n\cdot x_n \in \mathcal{D}_R$.
Up to passing to a subsequence, we may assume that $(\gamma_n\cdot\xi)$, $(\gamma_n\cdot x_n)$, and $(\gamma_n\cdot x)$ converge respectively to some $\xi_{\infty} \in \Lambdao_{\Omega}(\Gamma)$, some $y \in \mathcal{D}_R$, and some $x_{\infty} \in \Lambdao_{\Omega}(\Gamma)$.
We have $y \in (\xi_{\infty}, x_{\infty})$, hence $y \in \Ccore_{\Omega}(\Gamma)$.
By $\Gamma$-invariance of the Hilbert metric, we obtain $d_{\Omega}(x_n, \Ccore_{\Omega}(\Gamma)) = d_{\Omega}(\gamma_n\cdot x_n, \Ccore_{\Omega}(\Gamma)) \leq d_{\Omega}(\gamma_n\cdot x_n, y) \to 0$.
\end{proof}

\begin{fact}[{\cite[Lem.\,9.2]{dgk-proj-cc}}] \label{fact:strict-C1-nbhd}
Let $\Gamma$ be an infinite discrete subgroup of $\PGL(V)$ acting convex cocompactly on some properly convex open subset $\Omega$ of $\PP(V)$.
Let $\C_{\mathsf{unif}}$ be a uniform neighborhood of $\Ccore_{\Omega}(\Gamma)$ in $(\Omega,d_{\Omega})$.
Then the convex core $\Ccore_{\Omega}(\Gamma)$ admits a $\Gamma$-invariant, properly convex, closed neighborhood $\C\subset\C_{\mathsf{unif}}$ in~$\Omega$ which has $C^1$, strictly convex nonideal boundary.
\end{fact}

\begin{lemma} \label{lem:Omega-max}
Let $\Gamma$ be an infinite discrete subgroup of $\PGL(V)$ acting convex cocompactly on some properly convex open subset $\Omega$ of $\PP(V)$.
Then $\Omega$ is contained in
$$\mathcal{U} := \PP(V) \smallsetminus \bigcup_{X \in \Lambdao_{\Omega^*}(\Gamma)} X,$$
where we see each $X \in \Lambdao_{\Omega^*}(\Gamma) \subset \PP(V^*)$ as a projective hyperplane in $\PP(V)$.
Moreover, the connected component $\Omega^{\max}$ of $\mathcal{U}$ containing~$\Omega$ is a $\Gamma$-invariant convex open subset of $\PP(V)$, and any $\Gamma$-invariant properly convex open subset of $\PP(V)$ intersecting~$\Omega$ is contained in $\Omega^{\max}$.
\end{lemma}

\begin{proof}
We have $\Lambdao_{\Omega^*}(\Gamma) \subset \partial \Omega^*$, hence any $X \in \Lambdao_{\Omega^*}(\Gamma)$ is a supporting hyperplane to $\Omega$ in $\PP(V)$, and $\Omega \subset \mathcal{U}$.
The connected component $\Omega^{\max}$ of $\mathcal{U}$ containing~$\Omega$ is clearly a $\Gamma$-invariant convex open subset of $\PP(V)$.

Since $\Gamma$ acts convex cocompactly on~$\Omega$, it contains an element which is proximal in $\PP(V^*)$ \cite[Prop.\,2.3.15]{bla-PhD}; in other words, the proximal limit set $\Lambda_{\Gamma}^{\PP(V^*)}$ of $\Gamma$ in $\PP(V^*)$ (see Section~\ref{subsec:remind-prox}) is nonempty.
Note that $\Lambda_{\Gamma}^{\PP(V^*)}$ is a closed $\Gamma$-invariant subset of $\Lambdao_{\Omega^*}(\Gamma)$.
By \cite[Lem.\,4.16]{dgk-proj-cc}, we have
$$\mathcal{U} = \PP(V) \smallsetminus \bigcup_{X \in \Lambda_{\Gamma}^{\PP(V^*)}} X.$$

Let us show that any nonempty $\Gamma$-invariant properly convex open subset $\Omega' $ of $\PP(V)$ which intersects $\Omega$ is contained in~$\mathcal{U}$.
For this it is sufficient to show that each $X \in \Lambda_{\Gamma}^{\PP(V^*)}$ is a supporting hyperplane to $\Omega'$ in $\PP(V)$.
Since the set of supporting hyperplanes to~$\Omega'$ is closed in $\PP(V^*)$, it is sufficient to check that for any element $\gamma \in \Gamma$ which is proximal in $\PP(V)$, its repelling hyperplane $X_{\gamma}^-$ (see Section~\ref{subsec:remind-prox}) is a supporting hyperplane to~$\Omega'$ in $\PP(V)$.
For this, consider any supporting hyperplane $X$ to~$\Omega'$ which does not contain the attracting fixed point $x_{\gamma}^+$ of $\gamma$ in $\PP(V)$; such an $X$ exists because $\Omega'$ is properly convex (but this could fail if $\Omega'$ were only convex).
Then $\gamma^{-n} \cdot X \to X_{\gamma}^-$ as $n\to +\infty$, and so $X_{\gamma}^-$ is a supporting hyperplane to~$\Omega'$.
This shows that $\Omega' \subset \mathcal{U}$.

We deduce that any $\Gamma$-invariant properly convex open subset of $\PP(V)$ intersecting~$\Omega$ is contained in $\Omega^{\max}$.
\end{proof}

\subsection{$P_1$-divergence} \label{subsec:P1-div}

We say that a sequence $(g_n)_{n\in\NN}$ of elements of $\PGL(V)$ is \emph{$P_1$-divergent} if the ratio of the first and second singular values of (a lift to $\GL(V)$ of) $g_n$ tends to $+\infty$ as $n\to +\infty$.
Here is an easy consequence of the Cartan decomposition (see \eg \cite[Lem.\,7.5]{dgk-proj-cc} and \cite[Lem.\,4.5.16]{kas-notes}).

\begin{fact} \label{fact:P1-div}
For a sequence $(g_n)_{n\in\NN}$ of elements of $\PGL(V)$, consider the following conditions:
\begin{enumerate}
  \item\label{item:P1-div-1} there exist a point $x^+ \in \PP(V)$ and a projective hyperplane $X^- \subset \PP(V)$ such that\linebreak $g_n \cdot x \to x^+$ for all $x \in \PP(V) \smallsetminus X^-$;
  \item\label{item:P1-div-2} there exist a point $x^+ \in \PP(V)$ and a nonempty open subset $\mathcal{U}$ of $\PP(V)$ such that\linebreak $g_n \cdot x \to x^+$ for all $x \in \mathcal{U}$;
  \item\label{item:P1-div-3} $(g_n)_{n\in\NN}$ is $P_1$-divergent.
\end{enumerate}
Then \eqref{item:P1-div-1}~$\Rightarrow$~\eqref{item:P1-div-2}~$\Rightarrow$~\eqref{item:P1-div-3}.
Moreover, if $(g_n)$ satisfies~\eqref{item:P1-div-3}, then some subsequence satisfies~\eqref{item:P1-div-1}.
\end{fact}

We say that a subgroup of $\PGL(V)$ is \emph{$P_1$-divergent} if any sequence of pairwise distinct elements of~$\Gamma$ is $P_1$-divergent; in particular, such a subgroup is discrete in $\PGL(V)$.
The following characterization is a complement to Fact~\ref{fact:Anosov}; it immediately follows from \cite{dgk-proj-cc}.

\begin{proposition} \label{prop:Ano-cc-P1-div}
Let $\Gamma$ be an infinite discrete subgroup of $\PGL(V)$ preserving a nonempty properly convex open subset of $\PP(V)$.
Then the following are equivalent:
\begin{enumerate}
  \item\label{item:P1-div-cc} $\Gamma$ is $P_1$-divergent and convex cocompact in $\PP(V)$ (Definition~\ref{def:cc-group}),
  \item\label{item:hyp-cc} $\Gamma$ is Gromov hyperbolic and convex cocompact in $\PP(V)$,
  \item\label{item:Ano} $\Gamma$ is Gromov hyperbolic and the natural inclusion $\Gamma\hookrightarrow\PGL(V)$ is $P_1$-Anosov.
\end{enumerate}
\end{proposition}

\begin{proof}
The equivalence \eqref{item:hyp-cc}~$\Leftrightarrow$~\eqref{item:Ano} is contained in \cite[Th.\,1.15]{dgk-proj-cc}.
Together with the classical fact that $P_1$-Anosov implies $P_1$-divergent (see \eg \cite{ggkw17}), it gives \eqref{item:Ano}~$\Rightarrow$~\eqref{item:P1-div-cc}.

Let us check \eqref{item:P1-div-cc}~$\Rightarrow$~\eqref{item:hyp-cc}.
Suppose that $\Gamma$ is $P_1$-divergent and acts convex cocompactly on some properly convex open subset $\Omega$ of $\PP(V)$.
By Fact~\ref{fact:cc-strata}, if $F$ is an open face of $\partial \Omega$ meeting $\Lambdao_\Omega(\Gamma)$, then $F \subset \Lambdao_\Omega(\Gamma)$.
By \cite[Th.\,1.15]{dgk-proj-cc}, in order to prove that $\Gamma$ is Gromov hyperbolic, it is enough to check that if $F$ is any open face of $\partial \Omega$ containing a point $x \in \Lambdao_\Omega(\Gamma)$, then $F$ is a singleton.
Let us check this.
By \cite[Cor.\,4.10.(2)]{dgk-proj-cc}, we can write $x = \lim_n \gamma_n \cdot y$ for some $(\gamma_n) \in \Gamma^{\NN}$ and $y \in \Omega$ with $(d_{\Omega}(\gamma_n \cdot y, [y,x)))_{n\in\NN}$ bounded.
By \cite[Lem.\,B.2]{dgk-proj-cc}, if $B$ is a compact subset of~$\Omega$ with nonempty interior, then any accumulation point of $(\gamma_n\cdot B)_{n\in\NN}$ (for the Hausdorff topology) is a compact subset of~$F$ with nonempty interior.
Since the sequence $(\gamma_n)_{n\in\NN}$ is $P_1$-divergent, some subsequence satisfies Condition~\eqref{item:P1-div-1} of Fact~\ref{fact:P1-div}.
Therefore the open face~$F$ must be a singleton.
\end{proof}

\subsection{When the ideal boundary of the convex core is a union of faces}

\begin{fact}[{\cite[Lem.\,4.14]{dgk-proj-cc}}] \label{fact:unif-neighb-face}
Let $\Omega$ be a nonempty properly convex open subset of $\PP(V)$ and let $R>0$.
\begin{enumerate}
  \item \label{item:distance-goes-down} Let $(x_n)_{n\in\NN}$ and $(x'_n)_{n\in\NN}$ be sequences of points of~$\Omega$ converging respectively to points $z$ and~$z'$ of $\partial\Omega$.
  If $d_{\Omega}(x_n,x'_n)\leq R$ for all $n\in\NN$, then $z$ and~$z'$ belong to the same open face $F$ of $\partial\Omega$ and $d_F(z,z') \leq R$.
  \item \label{item:R-nbhd-faces} Let $\C$ be a nonempty closed convex subset of~$\Omega$ and let $\C_R$ be the closed uniform $R$-neighborhood of $\C$ in $(\Omega,d_{\Omega})$.
  Then for any open face $F$ of $\partial\Omega$, the set $\partiali \C_R \cap F$ is equal to the closed uniform $R$-neighborhood of $\partiali \C \cap F$ in $(F,d_F)$. 
\end{enumerate}
\end{fact}

The following easy consequence of Fact~\ref{fact:unif-neighb-face}.\eqref{item:distance-goes-down} will be important in the proof of Theorem~\ref{thm:hopeful} in Section~\ref{sec:virtual}.

\begin{lemma} \label{lem:partiali-Ccore-union-of-faces}
Let $\Gamma$ be a discrete subgroup of $\PGL(V)$ preserving a nonempty properly convex open subset $\Omega$ of $\PP(V)$.
Suppose that $\partiali \Ccore_{\Omega}(\Gamma)$ is a union of open faces of $\partial\Omega$.
Let $\C$ be the closed uniform $R$-neighborhood of $\Ccore_{\Omega}(\Gamma)$ in $(\Omega,d_{\Omega})$, for some $R>0$.
Then $\partiali \C = \partiali \Ccore_{\Omega}(\Gamma)$.
\end{lemma}

\begin{proof}
The inclusion $\partiali \C \supset \partiali \Ccore_{\Omega}(\Gamma)$ is clear.
On the other hand, by Fact~\ref{fact:unif-neighb-face}.\eqref{item:distance-goes-down} each point of $\partiali \C$ belongs to the open face of some point of $\partiali \Ccore_{\Omega}(\Gamma)$; since by assumption $\partiali \Ccore_{\Omega}(\Gamma)$ is a union of open faces of $\partial\Omega$, we have $\partiali \C \subset \partiali \Ccore_{\Omega}(\Gamma)$.
\end{proof}

\section{Combination theorems for (naively) convex cocompact actions} \label{sec:gog-cc}

In this section we establish a strengthening of Theorem~\ref{thm:tree-with-group-actions} in the context of (naively) convex cocompact actions, namely Theorem~\ref{thm:gog-cc}.
We use it to prove Theorem~\ref{thm:free-product-cc}, and also to show that lattices of $\PGL(V)$ (or more generally subgroups with full proximal limit sets in $\PP(V)$ and $\PP(V^*)$) contain convex cocompact (Anosov) subgroups preserving properly convex sets approaching any given convex shape, with full orbital limit set arbitrarily dense in the boundary (Proposition~\ref{prop:lattice-pingpong-subgroup}).
Finally, we establish Theorem~\ref{thm:naive-cc-not-cc-irred} by giving the explicit Example~\ref{ex:irreducible} of a group acting strongly irreducibly and naively convex cocompactly, but not convex cocompactly, on $\PP(\RR^{d+1})$ for $d\geq 3$.

\subsection{A general combination theorem with (naively) convex cocompact actions}

The following is a strengthening of Theorem~\ref{thm:tree-with-group-actions} in the context of (naively) convex cocompact actions.
Theorem~\ref{thm:amalgam-cc} corresponds to the special case where the graph $\mathsf{Y}$ consists of two vertices related by an edge~$\mathsf{e}$ and $g_{\mathsf{e}}=1$, while Theorem~\ref{thm:HNN-cc} corresponds to the special case where the graph $\mathsf{Y}$ consists of one vertex connected to itself by an edge~$\mathsf{e}$ and $g_\mathsf{e}=t$: see Examples~\ref{ex:gog}.

\begin{theorem} \label{thm:gog-cc}
Let $(\boldsymbol{\Gamma}, \mathsf{Y})$ be a graph of groups, where $\mathsf{Y} = (\mathsf{V}, \mathsf{E})$, and let $\mathsf v_0 \in \mathsf{V}$ be a base vertex.
Suppose $\mathsf{E} \neq \varnothing$.
For $n\geq 3$, consider discrete and faithful representations $\rho_{\mathsf v}: \Gamma_{\mathsf v} \to \PGL(d,\RR)=:G$ for each $\mathsf v \in \mathsf{V}$, and elements $g_{\mathsf e} \in \PGL(d,\RR)$ for each $\mathsf e \in \vec{\mathsf{E}}$, such that 
\begin{itemize}
  \item $g_{\mathsf e} = g_{\overline{\mathsf e}}^{-1}$ for all $\mathsf e \in \vec{\mathsf{E}}$, and
  \item $\rho_{\mathsf{o}(\mathsf e)}(\gamma^{\mathsf e}) = g_{\mathsf e} \, \rho_{\mathsf{t}(\mathsf e)}(\gamma^{\overline{\mathsf e}}) \, g_{\mathsf{e}}^{-1}$ for all $\mathsf e \in \vec{\mathsf{E}}$ and all $\gamma \in \Gamma_{|\mathsf e|}$.
\end{itemize}
Let $\underline{\rho} : F(\boldsymbol{\Gamma}, \mathsf{Y}) \to \PGL(d,\RR)$ and $\rho : \pi_1(\boldsymbol{\Gamma},\mathsf{Y},\mathsf v_0) \to \PGL(d,\RR)$ be the associated representations as in \eqref{eqn:rhobar}--\eqref{eqn:rhonobar}.

For each $\mathsf{v} \in \mathsf{V}$, let $\C_\mathsf{v} \subset \Omega_\mathsf{v}$ be two $\rho_\mathsf{v}(\Gamma_\mathsf{v})$-invariant nonempty properly convex subsets of $\PP(\RR^d)$ with $\Omega_\mathsf{v}$ open in $\PP(\RR^d)$ and $\C_\mathsf{v}$ closed in $\Omega_\mathsf{v}$. 
Suppose that for all oriented edges $\mathsf e, \mathsf e' \in \vec{\mathsf{E}}$ with $\mathsf{o}(\mathsf e) = \mathsf{o}(\mathsf e') =: \mathsf v$, and for each $\gamma \in \Gamma_{\mathsf v}$ such that either $\mathsf e\neq \mathsf e'$, or $\mathsf e= \mathsf e'$ and $\gamma \notin \Gamma_{|\mathsf{e}|}$, the triples
$$\big (g_\mathsf{e} \cdot \Omega_{\mathsf{t}(\mathsf{e})}, \Omega_\mathsf{v}, \rho_\mathsf{v}(\gamma) g_{\mathsf{e}'} \cdot \Omega_{\mathsf{t}(\mathsf{e}')} \big ) \quad \text{and} \quad \big ( g_\mathsf{e} \cdot \Int{\C_{\mathsf{t}(\mathsf{e})}}, \Int{\C_\mathsf{v}}, \rho_\mathsf{v}(\gamma) g_{\mathsf{e}'} \cdot \Int{\C_{\mathsf{t}(\mathsf{e}')}} \big) $$
are in occultation position.
Then: 
\begin{enumerate}[leftmargin=0.5cm]
  \item \label{item:loong1} The representation $\rho: \pi_1(\boldsymbol{\Gamma}, \mathsf{Y}, \mathsf{v}_0) \to \PGL(d,\RR)$ is discrete and faithful and preserves the nonempty properly convex open set
$$\Omega := \bigcup_{\substack{[\mathsf{c},\mu] \, \Gamma_{\mathsf{v}} \in \mathscr{T}(\boldsymbol{\Gamma}, \mathsf{Y}, \mathsf{v}_0)\\ \mathsf{e} \in \vec{\mathsf{E}} \text{ with $\mathsf{o}(\mathsf{e}) = \mathsf{v}$}}} \underline{\rho}([\mathsf{c},\mu]) \cdot \Omega_\mathsf{e} \quad \text{where} \quad \Omega_\mathsf{e}:=\Conv{ \Omega_{\mathsf{o}(\mathsf{e})} \cup g_\mathsf{e} \cdot \Omega_{\mathsf{t}(\mathsf{e})}},$$
and its nonempty properly convex subset
$$\C := \bigcup_{\substack{[\mathsf{c},\mu] \, \Gamma_{\mathsf{v}} \in \mathscr{T}(\boldsymbol{\Gamma},\mathsf{Y}, \mathsf{v}_0)\\ \mathsf{e} \in \vec{\mathsf{E}} \text{ with $\mathsf{o}(\mathsf{e}) = \mathsf{v}$}}} \underline{\rho} ([\mathsf{c},\mu]) \cdot \C_\mathsf{e} \quad \text{where} \quad \C_\mathsf{e} :=\overline{\Conv{ \C_{\mathsf{o}(\mathsf{e})} \cup g_\mathsf{e} \cdot \C_{\mathsf{t}(\mathsf{e})}}} \cap \Omega_\mathsf{e},$$
which is closed in~$\Omega$.
  \item \label{item:loong2} Suppose moreover that the graph $\mathsf{Y} = (\mathsf{V}, \mathsf{E})$ is finite and that for each $\mathsf{v} \in \mathsf{V}$ the action of $\Gamma_{\mathsf{v}}$ on $\C_{\mathsf{v}}$ via $\rho_{\mathsf{v}}$ is cocompact, and for each $\mathsf{e} \in \vec{\mathsf{E}}$ the action of $\Gamma_{\mathsf{e}}$ on
$$\overline{\C_{\mathsf{e}} \smallsetminus \big( \C_{\mathsf{o}(\mathsf{e})} \cup g_{\mathsf{e}} \cdot \C_{\mathsf{t}(\mathsf{e})}\big)} \cap \Omega_{\mathsf{e}}$$
via $\rho_{\mathsf{o}(\mathsf{e})} \circ \iota_{\mathsf{e}}$ is cocompact.
Then the action of $\pi_1(\boldsymbol{\Gamma},\mathsf{Y},\mathsf v_0)$ on $\C$ via $\rho$ is cocompact, and so the action of $\pi_1(\boldsymbol{\Gamma},\mathsf{Y},\mathsf v_0)$ on $\Omega$ via $\rho$ is naively convex cocompact.
  \item \label{item:loong3} Suppose moreover that for every $\mathsf{v} \in \mathsf{V}$, the convex set $\C_\mathsf{v}$ has strictly convex nonideal boundary.
Then $\C$ has bisaturated boundary, hence (Facts \ref{fact:cc-subsets}--\ref{fact:bisat-cc}) the group $\pi_1(\boldsymbol{\Gamma}, \mathsf{Y}, \mathsf{v}_0)$ acts convex cocompactly on $\Omega$ via~$\rho$.
\end{enumerate}
\end{theorem}

We refer to Definition~\ref{def:boundaries} for the notions of strictly convex nonideal boundary and bisaturated boundary.

\subsection{Proof of Theorem~\ref{thm:gog-cc}}

By Theorem~\ref{thm:general-gog}, the representation $\rho: \pi_1(\boldsymbol{\Gamma}, \mathsf{Y}, \mathsf{v}_0) \to \PGL(d,\RR)$ is discrete and faithful and preserves the nonempty properly convex open set
$$\Omega = \bigcup_{\substack{[\mathsf{c},\mu] \, \Gamma_{\mathsf{v}} \in \mathscr{T}(\boldsymbol{\Gamma}, \mathsf{Y}, \mathsf{v}_0)\\ \mathsf{e} \in \vec{\mathsf{E}} \text{ with $\mathsf{o}(\mathsf{e}) = \mathsf{v}$}}} \underline{\rho}([\mathsf{c},\mu]) \cdot \Omega_\mathsf{e} \quad \text{where} \quad \Omega_\mathsf{e}:=\Conv{ \Omega_{\mathsf{o}(\mathsf{e})} \cup g_\mathsf{e} \cdot \Omega_{\mathsf{t}(\mathsf{e})}}.$$
Since the interiors $\Int{\C_{\mathsf{v}}}$ of the $\C_\mathsf{v}$ satisfy similar assumptions to the~$\Omega_{\mathsf{v}}$, Theorem~\ref{thm:general-gog} implies that $\rho$ preserves the nonempty properly convex open set
$$\Int{\C} := \bigcup_{\substack{[\mathsf{c},\mu] \, \Gamma_{\mathsf{v}} \in \mathscr{T}(\boldsymbol{\Gamma}, \mathsf{Y}, \mathsf{v}_0)\\ \mathsf{e} \in \vec{\mathsf{E}} \text{ with $\mathsf{o}(\mathsf{e}) = \mathsf{v}$}}} \underline{\rho}([\mathsf{c},\mu]) \cdot \Int{\C_\mathsf{e}} \quad \text{where} \quad \C^\circ_\mathsf{e} := \Conv{ \C^\circ_{\mathsf{o}(\mathsf{e})} \cup g_\mathsf{e} \cdot \C^\circ_{\mathsf{t}(\mathsf{e})}}.$$

The statements \eqref{item:loong1}, \eqref{item:loong2}, \eqref{item:loong3} of Theorem~\ref{thm:gog-cc} are now proved, respectively, in the three lemmas below.

\begin{lemma} \label{lem:C-closed}
In the setting of Theorem~\ref{thm:gog-cc}, the set
$$\C := \bigcup_{\substack{[\mathsf{c},\mu] \, \Gamma_{\mathsf{t}(\mathsf{c})} \in \mathscr{T}(\boldsymbol{\Gamma},\mathsf{Y},\mathsf{v}_0)\\ \mathsf{e} \in \vec{\mathsf{E}} \text{ with $\mathsf{o}(\mathsf{e}) = \mathsf{t}(\mathsf{c})$}}} \underline{\rho} ([\mathsf{c},\mu]) \cdot \C_\mathsf{e}
\quad \text{where} \quad \C_\mathsf{e} := \overline{\C^\circ_\mathsf{e}} \cap \Omega_\mathsf{e}\hspace{1.2cm}~ $$
is closed in~$\Omega$.
\end{lemma}

\begin{proof}
The set $\C$ is locally closed in~$\Omega$ since each $\underline{\rho}([\mathsf{c},\mu]) \cdot \C_\mathsf{e}$ is closed in the corresponding $\underline{\rho}([\mathsf{c},\mu]) \cdot \Omega_\mathsf{e}$.
Moreover, open sets of the latter form intersect only two at a time, and their union is $\Omega$.
\end{proof}

\begin{lemma}\label{lem:cocompact}
In the setting of Theorem~\ref{thm:gog-cc}, suppose that the graph $\mathsf{Y} = (\mathsf{V}, \mathsf{E})$ is finite and that for each $\mathsf{v} \in \mathsf{V}$ the action of $\Gamma_{\mathsf{v}}$ on $\C_{\mathsf{v}}$ via $\rho_{\mathsf{v}}$ is cocompact, and for each $\mathsf{e} \in \vec{\mathsf{E}}$ the action of $\Gamma_{\mathsf{e}}$ on $\overline{\C_{\mathsf{e}} \smallsetminus ( \C_{\mathsf{o}(\mathsf{e})} \cup g_{\mathsf{e}} \cdot \C_{\mathsf{t}(\mathsf{e})})} \cap \Omega_{\mathsf{e}}$ via $\rho_{\mathsf{o}(\mathsf{e})} \circ \iota_{\mathsf{e}}$ is cocompact.
Then the action of $\pi_1(\boldsymbol{\Gamma},\mathsf{Y},\mathsf v_0)$ on $\C$ via $\rho$ is cocompact.
\end{lemma}

\begin{proof}
The finitely many sets $\C_\mathsf{e}$, varying over all $\mathsf{e} \in \vec{\mathsf{E}}$, project down to subsets whose union is the entire quotient $\rho(\pi_1(\boldsymbol{\Gamma},\mathsf{Y},\mathsf v_0)) \backslash \C$.
For each $\mathsf{e} \in \vec{\mathsf{E}}$,
$$\C_\mathsf{e} = \C_{\mathsf{o}(\mathsf{e})} \cup g_\mathsf{e} \cdot \C_{\mathsf{t}(\mathsf{e})} \cup \left(\overline{\C_\mathsf{e} \smallsetminus \big( \C_{\mathsf{o}(\mathsf{e})} \cup g_\mathsf{e} \cdot \C_{\mathsf{t}(\mathsf{e})}\big)} \cap \Omega_\mathsf{e} \right).$$
All three sets in the union project down to compact sets in $\C/\rho$ by assumption. 
So the quotient is covered by $\leq 3 |\mathsf{E}|$ compact sets, hence is compact.
\end{proof}

\begin{lemma} \label{lem:cocobisat}
In the setting of Lemma~\ref{lem:cocompact}, suppose that for every $\mathsf{v} \in \mathsf{V}$, the convex set $\C_\mathsf{v}$ has strictly convex nonideal boundary.
Then $\C$ has bisaturated boundary.
\end{lemma}

\begin{proof}
Suppose by contradiction that $\C$ does \emph{not} have bisaturated boundary: this means that there is an infinite ray $[x,\xi)$ in $\partialn \C$ with $\xi \in \partiali \C$.
Recall (Lemma~\ref{lem:C-closed}) that $\C$ is a union, parametrized by the edges of the Bass--Serre tree $\mathscr T(\boldsymbol \Gamma, \mathsf Y, \mathsf v_0)$, of translates of sets of the form $\C_{\mathsf{e}} =\overline{\Conv{ \C_{\mathsf{o}(\mathsf{e})} \cup g_\mathsf{e} \cdot \C_{\mathsf{t}(\mathsf{e})}}} \cap \Omega_\mathsf{e}$.
Further, two such sets intersect if and only if the corresponding edges are adjacent, \ie share a vertex, in $\mathscr T(\boldsymbol \Gamma, \mathsf Y, \mathsf v_0)$. It follows that the ray $[x,\xi)$ is contained in the nonideal boundary of a single such set, or otherwise must cross the nonideal boundary of two adjacent such sets. We rule out both possibilities.

By assumption, each convex set $\C_\mathsf{v}$, for $\mathsf{v} \in \mathsf{V}$, has strictly convex nonideal boundary. Therefore whenever the ray $[x,\xi)$ meets a translate $\underline{\rho}([\mathsf{c},\mu]) \C_{\mathsf{e}}$ of some $\C_{\mathsf{e}}$, it in fact meets the subset:
\begin{equation} \label{eqn:set-of-form}
\underline{\rho}([\mathsf{c},\mu]) \  \overline{\C_\mathsf{e} \smallsetminus \big( \C_{\mathsf{o}(\mathsf{e})} \cup g_\mathsf{e} \cdot \C_{\mathsf{t}(\mathsf{e})}\big)} \cap \Omega_\mathsf{e}.
\end{equation}

We claim that the ray $[x,\xi)$ cannot meet two sets of the form \eqref{eqn:set-of-form} corresponding to neighboring edges in  $\mathscr T(\boldsymbol \Gamma, \mathsf Y, \mathsf v_0)$.
Indeed, observe that for any $\mathsf{e}, \mathsf{e}' \in \vec{\mathsf{E}}$, with $\mathsf{o}(\mathsf{e}) = \mathsf{o}(\mathsf{e}') =: \mathsf{v}$, and any $\gamma \in \Gamma_{\mathsf{v}}$, such that $\mathsf{e} \neq \mathsf{e}'$ or $\mathsf{e} = \mathsf{e}'$ and $\gamma \notin \Gamma_\mathsf{e}$, 
the intersection $$\partial(\C_{\mathsf{o}(\mathsf{e})}^\circ)^* \cap \partial(g_\mathsf{e} \cdot \C_{\mathsf{t}(\mathsf{e})}^\circ)^* \cap \partial(\rho_\mathsf{v}(\gamma) g_{\mathsf{e}'} \cdot \C_{\mathsf{t}(\mathsf{e}')}^\circ)^*$$
is empty, since $\partial(\C_{\mathsf{o}(\mathsf{e})}^\circ)^* \subset (g_\mathsf{e} \cdot \C_{\mathsf{t}(\mathsf{e})}^\circ)^* \cup (\rho_\mathsf{v}(\gamma) g_{\mathsf{e}'} \cdot \C_{\mathsf{t}(\mathsf{e}')}^\circ)^*$ by the occultation assumption on $(g_\mathsf{e} \cdot \Int{\C_{\mathsf{t}(\mathsf{e})}}, \Int{\C_{\mathsf{o}(\mathsf{e})}}, \rho_\mathsf{v}(\gamma) g_{\mathsf{e}'} \cdot \Int{\C_{\mathsf{t}(\mathsf{e}')}})$.
In other words, there is no common supporting hyperplane to all three sets $\C_{\mathsf{o}(\mathsf{e})}^\circ$, $g_\mathsf{e} \cdot \C_{\mathsf{t}(\mathsf{e})}^\circ$, $\rho_\mathsf{v}(\gamma) g_{\mathsf{e}'} \cdot \C_{\mathsf{t}(\mathsf{e}')}^\circ$. 
It follows that the ray $[x, \xi)$, which lives in the nonideal boundary of $\C$, cannot cross from one set of the form~\eqref{eqn:set-of-form} to a neighboring (for the tree structure) set of that form. 

Therefore, the only remaining possibility is that the ray $[x,\xi)$ is fully contained in \eqref{eqn:set-of-form} for some $\mathsf{e} \in \vec{\mathsf{E}}$.
Since the action of $\Gamma_{\mathsf{e}}$ on \eqref{eqn:set-of-form} is cocompact by assumption, the ray $[x,\xi)$ is contained in some uniform neighborhood $\C^u$ of $\C_{\mathsf{o}(\mathsf{e})}$ in $(\Omega,d_{\Omega})$.
We claim that $\partiali \C^u = \partiali \C_{\mathsf{o}(\mathsf{e})}$.
Indeed, since $\C_{\mathsf{o}(\mathsf{e})}$ has strictly convex nonideal boundary by assumption, it has bisaturated boundary, hence the action of $\Gamma_{\mathsf{e}}$ on $\Int{\C_{\mathsf{o}(\mathsf{e})}}$ is convex cocompact by Fact~\ref{fact:bisat-cc}, and the action of $\Gamma_{\mathsf{e}}$ on~$\Omega$ is also convex cocompact by Fact~\ref{fact:cc-subsets}; then Fact~\ref{fact:ideal-bound-cc} yields $\partiali \C^u = \partiali \C_{\mathsf{o}(\mathsf{e})}$.
Consider a sequence $(x_n)_{n\in\NN}$ of points of $[x,\xi)$ converging to~$\xi$.
Let $\mathcal{D}_\mathsf{e}$ be a compact fundamental domain for the action of $\Gamma_\mathsf{e}$ on the set \eqref{eqn:set-of-form}.
For each~$n$, there exists $\gamma_n \in \Gamma_\mathsf{e}$ such that $\gamma_n \cdot x_n \in \mathcal{D}_{\mathsf{e}}$.
Up to passing to a subsequence, we may assume that $(\gamma_n\cdot x_n)_{n\in\NN}$ converges to some $x'_{\infty} \in \mathcal{D}_{\mathsf{e}}$, that $(\gamma_n\cdot\xi)_{n\in\NN}$ converges to some $\xi_{\infty} \in \partial \Omega$, and that $(\gamma_n\cdot x)_{n\in\NN}$ converges to some $x_{\infty} \in \partial \Omega$, with $x_{\infty}, x'_{\infty}, \xi_{\infty}$ aligned in this order. 
The points $x_\infty$ and~$\xi_{\infty}$ both belong to $\overline{\C^u} \cap \partial \Omega = \partiali \C^u = \partiali\C_{\mathsf{o}(\mathsf{e})}$.
Since $x'_\infty \in \Omega$, the open segment $(x_\infty, \xi_\infty) \subset \overline{\C_{\mathsf{o}(\mathsf{e})}}$ must be contained in $\C_{\mathsf{o}(\mathsf{e})}$. 
But $\C_{\mathsf{o}(\mathsf{e})}$ has strictly convex nonideal boundary by assumption, hence the open segment $(x_\infty, \xi_\infty) \subset \overline{\C_\mathsf{e} \smallsetminus ( \C_{\mathsf{o}(\mathsf{e})} \cup g_\mathsf{e} \cdot \C_{\mathsf{t}(\mathsf{e})})}$ must be contained in  $\C^\circ_{\mathsf{o}(\mathsf{e})}$: contradiction since $[x,\xi)$ was assumed to be contained in \eqref{eqn:set-of-form}.
\end{proof}

\subsection{Proof of Theorem~\ref{thm:free-product-cc}} \label{subsec:proof-free-product-cc}

The strategy is similar to the proof of Theorem~\ref{thm:free-product-fd} in Section \ref{subsec:proof-free-product-fd}, applying Theorem~\ref{thm:gog-cc} instead of Theorem~\ref{thm:general-gog}.
It is illustrated in Figure~\ref{fig:free}, bottom panel.

We start with the case that the action of  $\Gamma_i$ on $\Omega_i$ is naively convex cocompact for both $i\in\{ 0,1\}$.
As in Section~\ref{subsec:proof-free-product-fd}, we construct 
$\Omega'_i \subset \Omega_i$, the pyramids $\Delta_i \subset \Omega'_i$, the elements $g_i, \beta, g_{\mathsf{e}_i} \in \PGL(V)$ and the $\beta$-invariant convex set $\Omega \subset \PP(V)$.
As in Section~\ref{subsec:proof-free-product-fd}, using Facts \ref{fact:double-lim-set} and~\ref{fact:lattices}, we may choose $g_0$, $g_1$, and~$\beta$ (hence $g_{\mathsf{e}_0}$ and~$g_{\mathsf{e}_1}$) inside any given lattice of $\PGL(V)$, or more generally inside any given Zariski-dense subgroup of $\PGL(V)$ whose proximal limit sets in $\PP(V)$ and $\PP(V^*)$ are everything.
We also define the smaller pyramid $\Delta'_i \subset \Delta_i$ with tip $x''_i\in g_i \cdot \Omega$, and the $\Gamma_i$-invariant set
$$\Omega''_i:= \Big ( \Omega'_i \smallsetminus \bigsqcup_{\gamma\in \Gamma_i} \gamma \cdot \Delta_i \Big ) \cup \bigsqcup_{\gamma\in \Gamma_i} \gamma \cdot \Delta'_i \quad \subset \Omega'_i$$
just as in the proof of Theorem~\ref{thm:free-product-fd}, so that the triple
\begin{equation}
\label{eq:bob}
(g_{\mathsf{e}_i}^{-1} \cdot \Omega,~ \Omega''_i,~ \gamma g_{\mathsf{e}_i}^{-1} \cdot \Omega)
\end{equation} 
of convex sets is in occultation position for all $\gamma \in \Gamma_i \smallsetminus \{1\}$, as in~\eqref{eq:betaN2}.

Next, fix a nonempty closed convex subset $\C'_i$  of $\Omega'_i$ such that $\Gamma_i \backslash \C'_i$ is compact.
We can find a uniform neighborhood $\C^+_i$ of $\C'_i$ in the Hilbert metric of $\Omega'_i$ such that $g_{\mathsf{e}_i}^{-1} \cdot \overline{\Omega} \subset \Int{(\C^+_i)}$. 
(Recall that uniform Hilbert neighborhoods of convex sets are convex, see \cite[(18.12)]{bus55}.)
We cone $\Int{(\C^+_i)} \cap H$ off to a point $x'''_i\in (g_{\mathsf{e}_i}^{-1} \cdot \Omega) \cap \Delta'_i$ to form a pyramid $\Delta''_i \subset \Delta'_i$, and define 
$$\C''_i:= \Big ( \C^+_i \smallsetminus \bigsqcup_{\gamma\in \Gamma_i} \gamma \cdot \Delta_i \Big ) \cup \bigsqcup_{\gamma\in \Gamma_i} \gamma \cdot \overline{\Delta''_i} \quad \subset \C^+_i \cap \Omega''_i$$
which is closed in $\Omega''_i$.
By construction, for any $\gamma \in \Gamma_i \smallsetminus \{1\}$, the convex sets 
\begin{equation} \label{eq:bcb}
(g_{\mathsf{e}_i}^{-1} \cdot \Omega,~ \Int{(\C''_i)}, \gamma g_{\mathsf{e}_i}^{-1} \cdot \Omega)
\end{equation} 
are still in occultation position.

We may find a closed strictly convex $\C \subset \Omega$ such that the triple
$$(g_{\mathsf{e}_0} \cdot \Omega''_0, \Int{\C}, g_{\mathsf{e}_1} \cdot \Omega''_1)$$
of properly convex open sets is in occultation position, and therefore so are the triples
\begin{equation} \label{eq:obo}
 (g_{\mathsf{e}_0} \cdot \Omega''_0, \Omega, g_{\mathsf{e}_1} \cdot \Omega''_1)  \quad \text{ and } \quad 
 (g_{\mathsf{e}_0} \cdot \Int{(\C''_0)}, \Int{\C}, g_{\mathsf{e}_1} \cdot \Int{(\C''_1)}),
\end{equation}
because $\Omega \supset \C$ and $\C''_i\subset \Omega''_i$.
Note finally that the triple \eqref{eq:bcb} is still in occultation position if we replace $g_{\mathsf{e}_i}^{-1} \cdot \Omega$ by any smaller convex open set meeting~$\C''_i$: in particular, the triple
\begin{equation} \label{eq:cCc} 
(g_{\mathsf{e}_i}^{-1} \cdot \Int{\C}, \Int{(\C''_i)}, \gamma g_{\mathsf{e}_i}^{-1} \cdot \Int{\C})
\end{equation}
is in occultation position.

Let $\mathsf{Y} = (\mathsf{V},\mathsf{E})$ be the graph with three vertices $\mathsf{V} = \{\mathsf{v}_0, \mathsf{w}, \mathsf{v}_1\}$ and two edges $(\mathsf{w}, \mathsf{v}_0)$ and $(\mathsf{w}, \mathsf{v}_1)$. 
For $i \in \{0,1\}$, write $\mathsf{e}_i \in \vec{\mathsf{E}}$ for the directed edge from $\mathsf{w}$ to $\mathsf{v}_i$. 
Consider the graph of groups $(\boldsymbol{\Gamma}, \mathsf{Y})$ with $\Gamma_{\mathsf{v}_i} = \Gamma_i$ and $\Gamma_{\mathsf{w}}\simeq \Gamma_{|\mathsf{e}_i|} \simeq \{1\}$. 
Let $\rho_{\mathsf{v}_i}:\Gamma_i \rightarrow \PGL(V)$ be the inclusion, 
let $g_{\mathsf{e}_i} \in \PGL(V)$ be as above, and 
let $(\Omega_{\mathsf{v}_0}, \Omega_\mathsf{w}, \Omega_{\mathsf{v}_1}):=(\Omega''_0, \Omega, \Omega''_1)$ and $(\C_{\mathsf{v}_0}, \C_\mathsf{w}, \C_{\mathsf{v}_1}):=(\C''_0, \C, \C''_1)$.

Theorem~\ref{thm:gog-cc} applies to this datum, due to the occultation position of the triples of \eqref{eq:bob}, \eqref{eq:obo}, and \eqref{eq:cCc}. 
Since the group $\pi_1(\boldsymbol{\Gamma}, \mathsf{Y}, \mathsf{w})$ is naturally isomorphic to $\Gamma_0 * \Gamma_1$, this proves Theorem~\ref{thm:free-product-cc} with 
$g=g_{\mathsf{e}_0}^{-1} g_{\mathsf{e}_1}$, in the naively convex cocompact case.

For the convex cocompact case, in order to apply Theorem~\ref{thm:gog-cc} we only need to ensure in addition that the $\C''_i$ can be taken with strictly convex nonideal boundary. 
Since the action of $\Gamma_i$ on $\Omega_i$ (and hence on $\Omega'_i$) is convex cocompact, we may assume that the original set $\C'_i \subset \Omega'_i$ has strictly convex nonideal boundary, hence so does its uniform neighborhood $\C^+_i$.
The only region where the final set $\C''_i$ defined above may have nontrivial segments in its nonideal boundary is thus the pyramid $\Delta''_i$ (and its orbit), and all such segments end at the tip $x'''_i$ (because the base $\C^+_i \cap H_i$ of $\Delta''_i$ is strictly convex). 
We can modify $\Delta''_i$ to get rid of these segments, \eg by viewing $\partial \Delta''_i  \smallsetminus H$ in some affine chart as the graph of a small positive scalar function defined over the base of $\Delta''_i$, and composing that function with $\tanh$ (or any strictly concave contracting monotonic map $\varphi:(\RR^+,0) \rightarrow (\RR^+,0)$ close enough to the identity near~$0$).

This completes the proof of Theorem~\ref{thm:free-product-cc}.

\subsection{A variant of Theorem~\ref{thm:free-product-cc} in a semisimple Lie group}

Here is a version of Theorem~\ref{thm:free-product-in-G} in which the groups $\Gamma_0, \Gamma_1$ (input) and their free product (output) all satisfy a convex cocompactness property.
This is akin to the relationship between Theorems~\ref{thm:free-product-fd} and~\ref{thm:free-product-cc} in Section~\ref{subsec:proof-free-product-cc} above, and will be used in Section~\ref{sec:appli-anosov} to prove Theorem~\ref{thm:Anosov-free-product-G/P}.

As in Theorem~\ref{thm:free-product-in-G}, we write $\mathcal{F}$ for the space of partial flags $(x,X)$ with $x \in \PP(V)$ and $X \in \PP(V^*)$.

\begin{theorem} \label{thm:free-product-cc-in-G}
Let $G$ be a noncompact closed connected subgroup of $\PGL(V)$ which is semisimple, acts irreducibly on $\PP(\mathrm{span}(\Lambda_G^{\PP(V)}))$, and preserves a nonempty properly convex open subset $\Omega$ of $\PP(V)$.
For $i\in\{0,1\}$, let $\Gamma_i$ be a discrete subgroup of~$G$ acting convex cocompactly on some properly convex open subset of $\PP(V)$ intersecting~$\Omega$, and let $\Gamma$ be a Zariski-dense subgroup of~$G$ (\eg $G$ itself).
Suppose that for each $i\in\{0,1\}$ there is a flag $F_i \in \Lambda_{\Gamma}^{\mathcal{F}}$ which is uniformly transverse to $(\Gamma_i \smallsetminus\{1\}) \cdot F_i$.
Then there exists $g\in\Gamma$ such that the representation $\rho : \Gamma_0 * g\Gamma_1 g^{-1} \to G$ induced by the inclusions $\Gamma_0, g\Gamma_1 g^{-1} \hookrightarrow G$ is faithful with a discrete image which is convex cocompact in $\PP(V)$.
\end{theorem}

\begin{proof}
For $i\in \{0,1\}$, we write the flag $F_i \in \Lambda_{\Gamma}^{\mathcal{F}}$ as $(x_i,X_i)  \in \PP(V) \times \PP(V^*)$.
We follow the four-step proof strategy of Theorem~\ref{thm:free-product-in-G} (Section~\ref{subsec:free-product-in-G}), highlighting mainly the relevant changes.

In Step~1, we now wish to find a $\Gamma_i$-invariant properly convex open subset $\Omega_i$ of $\PP(V)$ containing $\Omega$ and~$x_i$, \emph{on which $\Gamma_i$ acts convex cocompactly}.
By assumption, there is a $\Gamma_i$-invariant properly convex open subset $\Omega'_i$ of $\PP(V)$ intersecting~$\Omega$ on which $\Gamma_i$ acts convex cocompactly.
We know from Step~1 of the proof of Theorem~\ref{thm:free-product-in-G} that $X_i$ is contained in a $\Gamma_i$-invariant properly convex open subset of $\PP(V^*)$, and so the set $\Lambda_i^* \subset \PP(V^*)$ of accumulation points of the $\Gamma_i$-orbit of~$X_i$ contains the proximal limit set $\Lambda_{\Gamma_i}^{\PP(V^*)}$ of $\Gamma_i$ in $\PP(V^*)$, which is nonempty by \cite[Prop.\,2.3.15]{bla-PhD}.
Both $\Omega$, $x_i$ and~$\Omega'_i$ are then contained in a common connected component $\Omega_i^{\max}$ of
$$\PP(V) \smallsetminus \bigcup_{H\in\Lambda_{\Gamma_i}^{\PP(V^*)}} H.$$
This component $\Omega_i^{\max}$ is a $\Gamma_i$-invariant convex open subset of $\PP(V)$.
Let $\mathcal{D}_i$ be a compact fundamental domain for the action of $\Gamma_i$ on $\Ccore_{\Omega'_i}(\Gamma_i)$.
By Proposition~\ref{prop:thicken-not-proper}, there is a $\Gamma_i$-invariant properly convex open subset $\Omega_i$ of $\Omega_i^{\max}$ containing both $\Omega$ and $\{ x_i\} \cup \mathcal{D}_i$.
Then the action of $\Gamma_i$ on~$\Omega_i$ is convex cocompact by Fact~\ref{fact:cc-subsets}.

By \cite[Lem.\,9.2]{dgk-proj-cc}, up to replacing $\Omega_i$ by a smaller $\Gamma_i$-invariant properly convex open subset still containing $\Omega \cup \{x_i\}$ (and on which the action of $\Gamma_i$ is still convex cocompact), we may and shall assume that $\partial\Omega_i \smallsetminus \Lambdao_{\Omega_i}(\Gamma_i)$ does not contain any nontrivial projective segments.
  
Step~2 is unchanged: as per Theorem~\ref{thm:free-product-in-G}.(1), for each $i\in\{0,1\}$ there is a neighborhood $\mathcal{W}_i = \mathcal{F} \cap (\mathcal{U}_i \times \mathcal{V}_i)$ of $F_i$ in~$\mathcal{F}$ such that $F'_i$ is uniformly transverse to $(\Gamma_i\smallsetminus\{1\}) \cdot F''_i$ for all $F'_i,F''_i \in \mathcal{W}_i$.
 
In Step~3, recall the convex hull $\C_i := \Conv{\Gamma_i \cdot x_i} \subset \Omega_i$, its $\varepsilon$-neighborhood $\C_i^\varepsilon$ in $(\Omega_i, d_{\Omega_i})$, and the ``hat'' $\Delta_i \ni x_i$ that a hyperplane $H_i$ chops off $\C_i^\varepsilon$. 
Given a positive $\varepsilon'<\varepsilon$, define in the same manner a neighborhood $\C_i^{\varepsilon'}$ and hat $\Delta'_i \subset \C_i^{\varepsilon'}$ (chopped off by the same hyperplane $H_i$).
Both $\C_i^{\varepsilon}$ and $\C_i^{\varepsilon'}$ have strictly convex nonideal boundary, because $\partial\Omega_i \smallsetminus \Lambdao_{\Omega_i}(\Gamma_i)$ does not contain any nontrivial projective segments.
We may assume that any hyperplane $X \in \mathcal{V}_i$ satisfies both $X\cap \C_i^\varepsilon = X\cap \Delta_i \neq \varnothing$ and $X\cap \C_i^{\varepsilon'} = X\cap \Delta'_i \neq \varnothing$.
We also choose the neighborhood $\mathcal{U}'_i$ of $x_i$ small enough to fit in the smaller hat,~$\Delta'_i$.
 
In Step~4, the construction produces a proximal element $h\in \Gamma$ whose axis $(y^-, y^+)$ (and its convex $h$-invariant neighborhood $B$) is contained in $\Delta'_0 \cap \Delta'_1 \cap \Omega$.
Next, having defined the open pyramid $\Delta_i^z \subset \Delta_i$ for $z\in B$, let also ${\Delta'_i}^z \subset \Delta'_i$ be the pyramid with vertex $z$ and  base $H_i \cap (\C_i^{\varepsilon'})^\circ$.
We can construct a slightly smaller convex region ${\Delta''_i}^z \subset {\Delta'_i}^z$ sharing the same base ($\overline{{\Delta''_i}^z} \cap H_i = \overline{{\Delta'_i}^z} \cap H_i = \C_i^{\varepsilon'} \cap H_i$), still intersecting $B$ but \emph{not} containing $z$ in its closure, and with strictly convex boundary away from the base.

We now define
\begin{align*} 
 \Omega_i^z & := \Int{(\C_i^{\varepsilon})} \: \smallsetminus \: \bigsqcup _{\gamma \in \Gamma_i} (\gamma \cdot \Delta_i)  \: \cup \: \bigsqcup_{\gamma \in \Gamma_i} (\gamma \cdot \Delta_i^z) \:\: \subset \Int{(\C_i^{\varepsilon})} & \text{ as in \eqref{eq:pyrchir};} \\
 \C_i^z & := (\C_i^{\varepsilon'}) \: \smallsetminus \: \bigsqcup _{\gamma \in \Gamma_i} (\gamma \cdot \overline{\Delta'_i})  \: \cup \: \bigsqcup_{\gamma \in \Gamma_i} (\gamma \cdot \overline{{\Delta''_i}^z}) \:\:  \subset  \Omega_i^z.
 \end{align*}
 By construction, $\Omega_i^z$ is open, convex, and $\C_i^z$ is closed in $\Omega_i^z$, with strictly convex boundary, and cocompact modulo $\Gamma_i$. 
Choose a compact, strictly convex subset $C \subset B$ whose interior $C^\circ$ still intersects $(\C_0^z)^\circ$ and $(\C_1^z)^\circ$ (that is to say, ${\Delta''_0}^z$ and ${\Delta''_1}^z$).
 The triples
$$ \begin{array}{ll}
 (B~,~ \Omega_i^z~,~ \gamma \cdot B) & (h^N \cdot \Omega_0^z~,~ B~,~ h^{-N} \cdot \Omega_1^z) \\
 (C^\circ~,~ (\C_i^z)^\circ~,~ \gamma \cdot C^\circ) & (h^N \cdot (\C_0^z)^\circ~,~ C^\circ~,~ h^{-N} \cdot (\C_1^z)^\circ)
 \end{array} $$
are in occultation position. 
Theorem~\ref{thm:gog-cc}.\eqref{item:loong1} applies to this datum, for the graph of groups $\big [ (\Gamma_0) \overset{\mathsf{e}_0}\longleftarrow (1) \overset{\mathsf{e}_1}\longrightarrow (\Gamma_1) \big ] \simeq \Gamma_0 *_{\{1\}} \Gamma_1$ and $(g_{\mathsf{e}_0}, g_{\mathsf{e}_1})=(h^N, h^{-N})$.
So does Theorem~\ref{thm:gog-cc}.\eqref{item:loong2}, because $C$ and $\Gamma_i \backslash \C_i^z$ are compact, as are the $\{1\} \backslash (C\cap \C_i^z)$ 
(where $\{1\}$ stands for an edge group). 
So does Theorem~\ref{thm:gog-cc}.\eqref{item:loong3}, because $C$ and $\C_i^z$ have strictly convex nonideal boundary.

Therefore, the representation $\rho : \Gamma_0 * g\Gamma_1 g^{-1} \to G$ induced by the inclusions $\Gamma_0, g\Gamma_1 g^{-1} \hookrightarrow G$ is discrete and faithful and its image is convex cocompact in $\PP(V)$.
\end{proof}

\subsection{Convex cocompact subgroups of lattices}

Here is another consequence of Theorems~\ref{thm:general-gog} and~\ref{thm:gog-cc}.

\begin{proposition} \label{prop:lattice-pingpong-subgroup}
Let $\Gamma$ be any lattice of $\PGL(V)$, or more generally any Zariski-dense subgroup of $\PGL(V)$ satisfying $\Lambda^{\PP(V)}_{\Gamma} = \PP(V)$ and $\Lambda^{\PP(V^*)}_{\Gamma} = \PP(V^*)$.
Let $\Omega$ be any nonempty properly convex open subset of $\PP(V)$.
Then $\Gamma$ contains a nonabelian free subgroup $\Gamma'$ acting convex cocompactly on some properly convex open subset $\Omega'$ of $\PP(V)$,
which can be chosen as close as we like to~$\Omega$, with full orbital limit set $\Lambdao_{\Omega'}(\Gamma')$ arbitrarily close to $\partial \Omega$ (for the spherical Hausdorff distance).
\end{proposition}

Such a group $\Gamma'$ is $P_1$-Anosov by Fact~\ref{fact:Anosov}.

\begin{proof}
Without loss of generality, we may assume that $\Omega$ is strictly convex with $C^1$ boundary.
Consider distinct points $x_1, \dots, x_N \in \partial \Omega$.
Choose properly convex neighborhoods $U_i$ of the $x_i$ such that the triple $(U_i, \Omega, U_j)$ is in occultation position for all $1\leq i < j \leq N$.
Choose $x'_i \in \partial \Omega \cap U_i \smallsetminus \{x_i\}$. 
Let $X_i, X'_i \subset \PP(V)$ be the supporting hyperplanes to $\Omega$ at $x_i, x'_i$, so that $\Omega \subset \PP(V) \smallsetminus (X_i \cup X'_i)$.
We can find a properly convex open subset $\Omega'$ of $\PP(V)$ such that $\overline{\Omega'} \subset \Omega$ and $\Omega'$ intersects every segment $(x_i, x'_i) \subset \Omega$.
For example, take $\Omega$ to be a sufficiently large ball for the Hilbert metric $d_{\Omega}$ on~$\Omega$. 
Taking $\Omega'$ large enough, we may assume that the $(U_i, \Omega', U_j)$ are still in occultation position.

For every $1\leq i \leq N$, by Fact~\ref{fact:double-lim-set} (and Fact~\ref{fact:lattices} if $\Gamma$ is a lattice), we can find biproximal elements $g_i \in \Gamma$ whose attracting eigenflag $(y_i^+, Y_i^+)$ (\resp repelling eigenflag $(y_i^-, Y_i^-)$) is arbitrarily close to $(x_i, X_i)$ (\resp $(x'_i, X'_i)$). 
In particular, we can arrange the following for all $1\leq i \leq N$:
$$ 
\begin{array}{lll} 
\bullet \: [y_i^+, y_i^-] \subset U_i,
& \quad &
\bullet \: [ y_i^+, y_i^- ] \text{ intersects } \Omega',
\\ 
\bullet \: y_i^+, y_i^- \notin \overline{\Omega'},
&&
\bullet \: \overline{\Omega'} \subset \PP(V) \smallsetminus (Y_i^+ \cup Y_i^-).
\end{array}
$$
Up to squaring $g_i$, we may assume that its extremal eigenvalues are positive.
Then there exists a $\langle g_i \rangle$-invariant properly convex set $\Omega_i \subset U_i \smallsetminus (Y_i^+ \cup Y_i^-)$ containing the open segment $(y_i^+, y_i^-)$: for example, take the convex hull in $U_i$ of the $\langle g_i \rangle$-orbit
of a small enough open ball of $U_i \smallsetminus (Y_i^+ \cup Y_i^-)$ that intersects $[y_i^+, y_i^-]$.
There exists $k_0 \in \NN$ such that for all distinct $i , j \in \{ 1, \dots, N \}$ and all $k\in \ZZ$ with $|k| \geq k_0$, the triples 
$$(\Omega_i, \Omega', \Omega_j) \quad \text{ and } \quad (\Omega', \Omega_i, g_i^k \cdot \Omega') $$
are in occultation position.

Let $\mathsf{Y} = (\mathsf{V},\mathsf{E})$ be the graph with $N+1$ vertices $\mathsf{V} = \{\mathsf{v}_0, \dots, \mathsf{v}_N\}$ and $N$ edges $(\mathsf{v}_0, \mathsf{v}_i)$ for $1\leq i \leq N$. 
Write $\mathsf{e}_i \in \vec{\mathsf{E}}$ for the directed edge from $\mathsf{v}_0$ to~$\mathsf{v}_i$. 
Consider the graph of groups $(\boldsymbol{\Gamma}, \mathsf{Y})$ with $\Gamma_{\mathsf{v}_i} =: \langle \gamma_i \rangle \simeq \ZZ$ for $1\leq i \leq N$ and $\Gamma_{\mathsf{v}_0}\simeq \Gamma_{|\mathsf{e}_i|} \simeq \{1\}$.
Let $\rho_{\mathsf{v}_i}:\Gamma_{\mathsf{v}_i} \rightarrow \PGL(V)$ take $\gamma_i$ to $g_i^{k_0} \in \Gamma$, let $g_{\mathsf{e}_i}:=1 \in \PGL(V)$, and let $(\Omega_{\mathsf{v}_0}, \Omega_{\mathsf{v}_1}, \dots, \Omega_{\mathsf{v}_N}):=(\Omega', \Omega_1,\dots , \Omega_N)$.

Theorem~\ref{thm:general-gog} applies to this datum, with $\pi_1(\boldsymbol{\Gamma}, \mathsf{Y}, v_0) =:\Gamma_\mathsf{Y} \simeq \ZZ^{*N}$.
We obtain a discrete and faithful representation $\rho: \Gamma_\mathsf{Y} \rightarrow \Gamma < \PGL(V)$, preserving a properly convex set (call it $\Omega_\mathsf{Y} \supset \Omega'$) that contains all the $(y_i^\pm, Y_i^\pm)_{1\leq i \leq N}$ among its supporting flags. 
By choosing the original subset $\{x_1, \dots, x_N\} \subset \partial \Omega$ dense enough, it follows that $(\Omega_\mathsf{Y}, \Lambdao_{\Omega_\mathsf{Y}}(\rho(\Gamma_\mathsf{Y}))$ can be made arbitrarily close to~$(\Omega, \partial \Omega)$.

Finally, we show that $\Gamma_\mathsf{Y}$ acts convex cocompactly on $\Omega_{\mathsf{Y}}$. 
For this, we replace Theorem~\ref{thm:general-gog} with Theorem~\ref{thm:gog-cc} in the above argument, by endowing 
$\Omega'$ with a compact properly convex subset~$\C'$, 
and each $\Omega_i$ with a closed, $\langle g_i \rangle$-invariant, properly convex subset~$\C_i$ that is compact modulo~$\langle g_i \rangle$, in such a way that
(with $i,j,k$ as above) the triples 
$$(\Int{\C_i}, \Int{\C'}, \Int{\C_j}) \quad \text{ and } \quad (\Int{\C'}, \Int{\C_i}, g_i^k \cdot \Int{\C'}) $$
are still in occultation position. This can be ensured by taking the $\C$'s large enough.
\end{proof}

\subsection{Examples of naively convex cocompact actions which are not convex cocompact}

Let us recall an example from~\cite{dgk-proj-cc} of a discrete group which is naively convex cocompact in projective space, but not convex cocompact in projective space. 

\begin{example}[{\cite[Ex.\,4.4.(2)]{dgk-proj-cc}}] \label{ex:not-irreducible}
Let $\Gamma_0 := \langle h \rangle \subset \SL(3,\RR)$, where
$$h := \begin{pmatrix} 4 & & \\ & 1/2 & \\ & & 1/2 \end{pmatrix}.$$
Then $\Gamma_0$ is naively convex cocompact in $\PP(\RR^3)$.
Indeed, let $T$ be the triangle of $\PP(\RR^3)$ obtained by projecting the cone of positive linear combinations of the standard basis $(b_1, b_2, b_3)$ of~$\RR^3$. 
Let $I \subset (b_2, b_3)$ be a nontrivial interval in the repelling edge of~$T$, and let $\mathcal{T} = \Conv{\{b_1\} \cup I} $ denote the closed convex hull in $T$ of $b_1$ and $I$. 
Then $\Gamma_0$ acts cocompactly on~$\C$, hence the action of $\Gamma_0$ on~$T$ is naively convex cocompact. 
However, $\Gamma_0$ is \emph{not} convex cocompact in $\PP(\RR^3)$ because $\Gamma_0 \cong \mathbb Z$ is Gromov hyperbolic, so if it were convex cocompact, then it would be $P_1$-Anosov by Fact~\ref{fact:Anosov}, but this is not the case since $h^{-1}$ is not proximal in $\PP(\RR^3)$.
\end{example}

This example easily extends as follows.

\begin{example} \label{ex:naive-cc-still-not-irred}
Let $V = V_1 \oplus \dots \oplus V_m$ be a direct sum decomposition of~$V$.
For each~$i$, let $\widetilde{\C}_i \subset \widetilde{\Omega}_i$ be two nonempty properly convex cones of~$V_i$, with $\widetilde{\Omega}_i$ open and $\widetilde{\C}_i \cup \{0\}$ closed in~$V_i$.
Then $\Omega := \PP(\widetilde{\Omega}_1 + \dots + \widetilde{\Omega}_m)$ is a properly convex open subset of $\PP(V)$ and $\C := \Omega \cap \PP(\widetilde{\C}_1 + \dots + \widetilde{\C}_m)$ is a closed convex subset of~$\Omega$.
The group of elements of $\PGL(V)$ preserving~$\Omega$ and acting on each~$V_i$ as a homothety, is isomorphic to $\RR^{m-1}$; let $\Gamma_0 \simeq \ZZ^{m-1}$ be a lattice in it.
Then $\Gamma_0$ acts cocompactly on~$\C$, and so $\Gamma_0$ is naively convex cocompact in $\PP(V)$.
However, if $\dim(V_i) \geq 2$ for some~$i$, then $\Gamma_0$ is \emph{not} convex cocompact in $\PP(V)$, as $\Lambdao_{\Omega}(\Gamma)$ contains $\PP(\widetilde{\Omega}_i)$ which is not fully contained in~$\overline{\C}$.
\end{example}

\begin{remark}
If $m=2$ and $\dim(V_1) = \dim(V_2) = p$ in Example~\ref{ex:naive-cc-still-not-irred}, then $\Gamma_0 \simeq \ZZ$ is contained in $\PO(p,p)$ and we can take $\Omega$ to be contained in
$$\HH^{p,p-1} = \{ [v] \in \PP(\RR^{2p}) \,|\, v_1^2 + \dots + v_p^2 - v_{p+1}^2 - \dots - v_{2p}^2 < 0\}, $$
which is a pseudo-Riemannian analogue of the real hyperbolic space in signature $(p,p-1)$.
Then $\Gamma_0$ is naively convex cocompact in $\PP(V)$ but \emph{not} $\HH^{p,p-1}$-convex cocompact in the sense of \cite{dgk-Hpq-cc}.
By deforming slightly a generator of~$\Gamma_0$ to act on each~$V_i$, not as a homothety, but as a homothety composed with an irrational rotation, we see that there exist small deformations of $\Gamma_0$ in $\PGL(V)$ which do \emph{not} preserve any properly convex open subset of $\PP(V)$.
\end{remark}

Note that in Examples \ref{ex:not-irreducible} and~\ref{ex:naive-cc-still-not-irred}, the action of $\Gamma_0$ on~$\PP(V)$ is \emph{not} irreducible.
As promised in \cite[\S\,4.1]{dgk-proj-cc}, we now give a strongly irreducible example which is naively convex cocompact, but not convex cocompact, proving Theorem~\ref{thm:naive-cc-not-cc-irred}.
For this we combine Example~\ref{ex:not-irreducible} with a Zariski-dense free group supplied by Lemma~\ref{lem:free-groups}.

\begin{example} \label{ex:irreducible}
Building on Example~\ref{ex:not-irreducible}, let $N \geq 1$ and let 
$$h' = \begin{pmatrix} 4 & & & \\ & 1/2 & & \\ & & 1/2  & \\ & & & \mathrm{Id}_N \end{pmatrix}$$
be the image of $h$ under the inclusion $\iota: \SL(3,\RR) \hookrightarrow \SL(3+N,\RR)$, and let $\Gamma_0 = \langle h' \rangle$. 
Let $\mathcal{T}'$ and~$T'$ be the respective images of $\mathcal{T}$ and~$T$ under the corresponding inclusion $\PP(\RR^3) \hookrightarrow \PP(\RR^{3+N})$.
By Proposition~\ref{prop:thicken-add-dim}, there exists an $\iota(\mathrm{Aut}(T))$-invariant properly convex open subset $\Omega_0$ of $\PP(\RR^{3+N})$ containing~$T'$.
Let $\C'_0$ be a closed uniform neighborhood of $\mathcal{T}'$ in $(\Omega_0,d_{\Omega_0})$.
Then $\C'_0$ is $\Gamma_0$-invariant, and the fact that $\Gamma_0$ acts cocompactly on~$\mathcal{T}$ implies that $\Gamma_0$ acts cocompactly on~$\C'_0$.
Since $\Gamma_0$ is a cyclic group, it is easy to smooth out $\C'_0$ into a $\Gamma_0$-invariant closed convex subset of~$\Omega_0$ such that $\partialn \C_0$ is strictly convex everywhere except along $\partialn \C_0 \cap T'$. 

Let $\Gamma_1$ be any discrete subgroup of $\PGL(3+N,\RR)$ which is strongly irreducible, Gromov hyperbolic, convex cocompact in $\PP(\RR^{3+N})$, and does not divide any properly convex set.
Such a group is supplied, for example, by Lemma~\ref{lem:free-groups}.
Since $\Gamma_1$ is convex cocompact in $\PP(\RR^{3+N})$, by Fact~\ref{fact:strict-C1-nbhd} there exist a $\Gamma_1$-invariant properly convex open subset $\Omega_1$ of $\PP(\RR^{3+N})$ and a nonempty closed convex subset $\C_1$ of~$\Omega_1$ with strictly convex nonideal boundary on which $\Gamma_1$ acts cocompactly. 
Since $\Gamma_1$ does not divide~$\Omega_1$, we have $\Omega_1 \neq \Ccore_{\Omega_1}(\Gamma_1)$. 

Theorem~\ref{thm:free-product-cc} applied to $\Gamma_0$ and $\Gamma_1$ yields an element $g \in \PGL(3+N,\RR)$ and a representation $\rho_g : \Gamma_0 * \Gamma_1 \to \PGL(3+N,\RR)$ acting naively convex cocompactly on some properly convex open set $\Omega \subset \PP(\RR^{3+N})$ such that for each $i \in \{0,1\}$, the restriction of $\rho_g$ to $\Gamma_i$ is conjugate to the inclusion. 
Since the action of $\Gamma_1$ is already strongly irreducible, the action of $\Gamma_0 * \Gamma_1$ by $\rho_g$ is also strongly irreducible. 
Finally, observe that the representation $\rho_g$ of the Gromov hyperbolic group $\Gamma_0 * \Gamma_1$ is not $P_1$-Anosov because the element $\rho_g(h') = h' \in \PGL(3+N,\RR)$ is not proximal in $\PP(\RR^{3+N})$. 
By Fact~\ref{fact:Anosov}, the group $\rho_g(\Gamma_0 * \Gamma_1)$ is not convex cocompact in $\PP(\RR^{3+N})$.
\end{example}

\section{Convex gluing of convex projective manifolds with boundary} \label{sec:gluing}

In this section we prove Theorem~\ref{thm:gluing} using Theorem~\ref{thm:amalgam-cc}.
We then deduce Corollary~\ref{cor:double}.

\subsection{Proof of Theorem~\ref{thm:gluing}}

If $\Gamma_0$ or~$\Gamma_1$ is equal to~$\Delta$, then Theorem~\ref{thm:gluing} is trivially true.
So we assume that $\Gamma_0$ and~$\Gamma_1$ both strictly contain~$\Delta$.

By assumption, $z \notin X \cup \Lambdao_{\Omega_0}(\Gamma_0) \cup \Lambdao_{\Omega_1}(\Gamma_1)$.
Therefore, after replacing $\Omega_i$ by a uniform neighborhood of $\Ccore_{\Omega_i}(\Gamma_i)$ in $(\Omega_i,d_{\Omega_i})$ (which does not change the fact that $\Gamma_i$ acts convex cocompactly on~$\Omega_i$, see Fact~\ref{fact:cc-subsets}), we may and shall assume that 
\begin{equation} \label{eq:zout}
z \notin \overline{\Omega_0} \quad \text{ and } z \notin \overline{\Omega_1}.
\end{equation}

\begin{lemma} \label{lem:gluing-1}
For $i\in\{0,1\}$, the set $\Ccore_{\Omega_i}(\Gamma_i)$ has nonempty interior and lies entirely on one side of $X \cap \Omega_0$ in~$\overline{\Omega_i}$.
\end{lemma}

Here and in the sequel, by a \emph{side} of $X \cap \Omega_0$ in~$\overline{\Omega_i}$, we mean the \emph{closure} of one of the two connected components of $\Omega_i \smallsetminus X$; thus a side of $X \cap \Omega_0$ includes $X \cap \Omega_0 = X \cap \Omega_1$.
Since the latter set is nonempty, note that we also have $X \cap \partial \Omega_0 = X \cap \partial \Omega_1$, by convexity.

\begin{proof}
By assumption, $\Delta$ divides $X \cap \Omega_0$, hence
\begin{equation} \label{eqn:bound-divide}
X \cap \partial \Omega_0 \subset \Lambdao_{\Omega_i}(\Delta)
\end{equation}
(see \cite[Prop.\,3]{vey70} or \cite[Cor.\,4.10.(3)]{dgk-proj-cc}).
By assumption, $\Gamma_i$ strictly contains the stabilizer $\Delta$ of~$X$, \ie $\Gamma_i$ does not preserve~$X$.
The projective span of $\Lambdao_{\Omega_i}(\Gamma_i)$ is $\Gamma_i$-invariant and contains~$X$ by \eqref{eqn:bound-divide}; therefore, $\Lambdao_{\Omega_i}(\Gamma_i) \smallsetminus X \neq \varnothing$ and $\Ccore_{\Omega_i}(\Gamma_i)$ has nonempty interior.
Since $\Ccore_{\Omega_0}(\Gamma_0) \cap\nolinebreak \Ccore_{\Omega_1}(\Gamma_1)$ is contained in the projective hyperplane~$X$, it follows that\linebreak $\Ccore_{\Omega_0}(\Gamma_0) \cap \Ccore_{\Omega_1}(\Gamma_1) = \partialn \Ccore_{\Omega_0}(\Gamma_0) \cap \partialn \Ccore_{\Omega_1}(\Gamma_1)$. In particular, since $X \cap \Omega_i \subset \partialn \Ccore_{\Omega_i}(\Gamma_i)$, the set $\Ccore_{\Omega_i}(\Gamma_i)$ lies entirely on one side of $X \cap \Omega_i$.
The same holds for $\Lambdao_{\Omega_i}(\Gamma_i)$.
\end{proof}

\begin{lemma} \label{lem:gluing-2}
For $i\in\{0,1\}$, we have $\Lambdao_{\Omega_i}(\Delta) = X \cap \partial \Omega_0$.
\end{lemma}

\begin{proof}
By \eqref{eqn:bound-divide}, it is enough to prove the inclusion $\Lambdao_{\Omega_i}(\Delta) \subset X \cap \partial \Omega_0$.
For this, consider a point $p \in \Omega_i$ and a sequence $(\gamma_n) \in \Delta^{\NN}$ such that $(\gamma_n \cdot p)_{n\in\NN}$ converges to some $p_{\infty} \in \Lambdao_{\Omega_i}(\Delta) \subset \partial \Omega_i$.
Let us check that $p_{\infty} \in X$.

Note that $\Delta$, which preserves~$X$, also preserves each side of $X \cap \Omega_0$ in~$\overline{\Omega_i}$ since it preserves $\Ccore_{\Omega_i}(\Gamma_i)$ which lies entirely on one side of $X \cap \Omega_0$ by Lemma~\ref{lem:gluing-1}.
In particular, if $p$ lies on the side of $X \cap \Omega_0$ opposite to $\Ccore_{\Omega_i}(\Gamma_i)$, then we must have $p_{\infty} \in X$.

Suppose that $p$ lies on the same side of $X \cap \Omega_0$ as $\Ccore_{\Omega_i}(\Gamma_i)$.
Consider a point $q \in \Omega_i \smallsetminus X$ on the other side of $X \cap \Omega_0$.
The interval $[p,q]$ intersects $X$ in a point~$r$.
Up to passing to a subsequence, we may assume that $(\gamma_n \cdot q)_{n\in\NN}$ and $(\gamma_n \cdot r)_{n\in\NN}$ converge respectively to some $q_{\infty} \in \Lambdao_{\Omega_i}(\Delta)$ and $r_{\infty} \in \Lambdao_{\Omega_i}(\Delta)$, which belong to~$X$ by the above.

Suppose by contradiction that $p_\infty \notin X$.
Then $p_\infty \neq r_\infty$ and the line containing $p_\infty, r_\infty$, and $q_\infty$ is transverse to~$X$.
Since $q_\infty \in X$, we must have $q_\infty = r_\infty$.
This gives a contradiction by a standard cross-ratio argument as follows.
Let $a, b \in \partial \Omega_i$ be such that $a, p, r, q, b$ are lined up in that order.
Then $d_{\Omega_i}(\cdot,\cdot) = d_{(a,b)}(\cdot,\cdot)$ for any pair of points from the set $\{p,q,r\}$.
Up to passing again to a subsequence, we may assume that $(\gamma_n \cdot a)_{n\in\NN}$ and $(\gamma_n \cdot b)_{n\in\NN}$ converge respectively to some $a_{\infty}, b_{\infty} \in \partial\Omega_i$, with $a_{\infty}, p_{\infty}, r_{\infty}, q_{\infty}, b_{\infty}$ lined up in that order.
Since $p_\infty \neq q_\infty$, we must have that $a_\infty \neq p_\infty$ and $b_\infty \neq q_\infty$, hence $d_{(a_\infty, b_\infty)}(q_\infty, r_\infty) = d_{(a,b)}(q,r) \neq 0$.
This contradicts the fact that $q_\infty = r_\infty$. 
\end{proof}

\begin{corollary} \label{cor:cc-Omega-i-and-dual}
For $i\in\{0,1\}$, the group $\Delta$ acts convex cocompactly on both $\Omega_i$ and its dual~$\Omega_i^*$.
\end{corollary}

\begin{proof}
By Lemma~\ref{lem:gluing-2}, we have $\Ccore_{\Omega_i}(\Delta) = X \cap \Omega_0 \neq \varnothing$.
By assumption, $\Delta$ acts cocompactly on this set.
Therefore $\Delta$ acts convex cocompactly on~$\Omega_i$.

Since $z$ lies outside of $\overline{\Omega_i}$ by~\eqref{eq:zout}, the hyperplane $z^*$ dual to $z$ in the dual projective space $\PP(V^*)$ intersects $\Omega_i^*$ in a nonempty $\Delta$-invariant closed convex subset of~$\Omega_i^*$.
Note that $\Delta$ acts convex cocompactly on some properly convex superset $\mathcal{O}_i^* \supset \Omega_i^*$ (for instance, we can take $\mathcal{O}_i^*$ to be the dual of a uniform neighborhood $\mathcal{O}_i$ of $\Ccore_{\Omega_i}(\Delta)$ in $\Omega_i$, see Facts \ref{fact:bisat-cc} and~\ref{fact:cc-nbhd}).
By minimality of the convex core $\Ccore_{\mathcal{O}_i^*}(\Delta)$ (see Fact~\ref{fact:ideal-bound-cc}), we then have $\Ccore_{\mathcal{O}_i^*}(\Delta) \subset z^* \cap \Omega_i^* \subset \Omega_i^*$.
Thus $\Delta$ acts convex cocompactly on $\Omega_i^*$ by Fact~\ref{fact:cc-subsets}.
\end{proof}

Since it will be useful later in this proof, we note that $\Delta$ divides $z^* \cap \Omega_i^*$ by a standard cohomological dimension argument (see Remark~\ref{rem:vcd}).

For $i\in\{0,1\}$, let $\Omega_i^{\max}$ be the connected component of
$$\PP(V) \smallsetminus \bigcup_{X' \in \Lambdao_{\Omega_i^*}(\Gamma_i)} X'$$
containing~$\Omega_i$, where we see each $X' \in \Lambdao_{\Omega_i^*}(\Gamma_i) \subset \PP(V^*)$ as a projective hyperplane in $\PP(V)$.
We will use $\Omega_i^{\max}$ as an ``ambient'' convex set.

\begin{lemma}
For $i\in \{0,1\}$ and $\gamma \in \Gamma_i \smallsetminus \Delta = \Gamma_i \smallsetminus \mathrm{Stab}_{\Gamma_i}(X)$, 
the set $\gamma \cdot \Ccore_{\Omega_{1-i}}(\Gamma_{1-i})$ lies in the open half-space of $X$ opposite from $\Ccore_{\Omega_{1-i}}(\Gamma_{1-i})$ inside $\Omega_i^{\max}$.
\end{lemma}

\begin{proof}
We begin by showing that
\begin{equation} \label{eq:pangolin}
\Ccore_{\Omega_{1-i}}(\Gamma_{1-i}) \subset \Omega_i^{\max}.
\end{equation}
Observe first that $\Delta$ acts convex cocompactly both on $\Omega_0 \cap \Omega_1$, since it contains $X \cap \Omega_0 = \Ccore_{\Omega_0}(\Delta)$, and on $(\Omega_0 \cap \Omega_1)^*$ (using Fact~\ref{fact:cc-subsets} and that this set contains \eg $\Omega_0^*$). Hence, by Fact~\ref{fact:cc-subsets}, $\Lambdao_{\Omega_0^*}(\Delta) = \Lambdao_{\Omega_1^*}(\Delta) = \Lambdao_{(\Omega_0 \cap \Omega_1)^*}(\Delta)$. 

Consider the convex subset $\Omega^{\max}_\Delta$ of $\PP(V)$ which contains $X \cap \Omega_0$ and is bounded by the hyperplanes of $\Lambdao_{\Omega_i^*}(\Delta)$ (this is independent of $i \in \{0,1\}$).
It contains $\Omega_i^{\max}$ and hence $\Ccore_{\Omega_i}(\Gamma_i)$ for both $i \in \{0,1\}$.
Note that $\Omega^{\max}_\Delta$ is not properly convex. 

Now, let $S := \Ccore_{\Omega_{1-i}}(\Gamma_{1-i}) \cap \Omega_i^{\max}$.
Clearly, $S$ is an open subset of $\Ccore_{\Omega_{1-i}}(\Gamma_{1-i})$.
It is also nonempty, since $X \cap \Omega_0 = \Ccore_{\Omega_0}(\Gamma_0) \cap  \Ccore_{\Omega_1}(\Gamma_1) \subset S$. 
Let us show that $S$ is closed in $\Ccore_{\Omega_{1-i}}(\Gamma_{1-i})$. 
Let $p \in (\overline{S} \cap \Ccore_{\Omega_{1-i}}(\Gamma_{1-i})) \smallsetminus S$.
Then $p\in \overline{\Omega_i^{\max}}$ because $S \subset \Omega_i^{\max}$, but $p \notin \Omega_i^{\max}$ because that would imply $p\in S$.
Thus $p\in \partial \Omega_i^{\max}$, which by definition means that $p$ lies in some projective hyperplane $X' \in \Lambdao_{\Omega_i^*}(\Gamma_i)$.
The hyperplane $X'$ supports $\Omega_i$ at some point $w \in \Lambdao_{\Omega_i}(\Gamma_i)$.
Since $p \in \Ccore_{\Omega_{1-i}}(\Gamma_{1-i}) \smallsetminus X$, we can be sure $p \neq w$, since $p$ lies strictly on one side of $X$ in $\Omega_{\Delta}^{\max}$ and $w$ lies in the other side of $X$ in $\Omega_{\Delta}^{\max}$.
Further, the interval $[p,w]$ (which has at least $(p,w)$ inside of $\Omega_{\Delta}^{\max}$) must cross $X \cap \Omega_0$.
This is a contradiction because $X'$ contains $[p,w]$, but does not intersect $X \cap \Omega_0$.
Hence $S = \Ccore_{\Omega_{1-i}}(\Gamma_{1-i})$ and so $\Ccore_{\Omega_{1-i}}(\Gamma_{1-i}) \subset \Omega_i^{\max}$.
This proves~\eqref{eq:pangolin}

Inside of $\Omega_i^{\max}$, the set $\Ccore_{\Omega_i}(\Gamma_i)$ lies on one side of $X$ and $\Ccore_{\Omega_{1-i}}(\Gamma_{1-i})$ lies on the other. 
For any $\gamma \in \Gamma_i \smallsetminus \Delta = \Gamma_i \smallsetminus \mathrm{Stab}_{\Gamma_i}(X)$, the set $\gamma \cdot (X \cap \Omega_0)$ must be disjoint from $X \cap \Omega_0$, since otherwise $\gamma \cdot (X \cap \Omega_0)$ would transversely cross $X \cap \Omega_0$, contradicting that $X \cap \Omega_0$ is contained in $\partialn \Ccore_{\Omega_i}(\Gamma_i)$. Further, $\gamma \cdot \Ccore_{\Omega_{1-i}}(\Gamma_{1-i}) \cap \Ccore_{\Omega_{1-i}}(\Gamma_{1-i}) = \varnothing$;  more precisely, $\gamma \cdot \Ccore_{\Omega_{1-i}}(\Gamma_{1-i})$ lies in the open half-space of $X$ opposite from $\Ccore_{\Omega_{1-i}}(\Gamma_{1-i})$ in $\Omega_i^{\max}$.
\end{proof}

We define a $\Gamma_i$-invariant closed convex subset $\C_i$ of $\Omega_i$ as follows. Choose a point $p_i \in \Omega_i$ just on the other side of $X \cap \Omega_0$ from $\Ccore_{\Omega_i}(\Gamma_i)$. Choose $p_i$ close enough to $X \cap \Omega_0$ so that $p_i \in \Ccore_{\Omega_{1-i}}(\Gamma_{1-i})$. This is possible by Lemma~\ref{lem:gluing-1}.
Define $\C_i$ to be the convex hull of $\Gamma_i \cdot p_i$ in~$\Omega_i$.

The set $\C_i$ is a closed $\Gamma_i$-invariant convex subset of~$\Omega_i$, contained in some uniform neighborhood of $\Ccore_{\Omega_i}(\Gamma_i)$ in $(\Omega_i,d_{\Omega_i})$, hence $\C_i$ has compact quotient by~$\Gamma_i$.
Observe that $\C_0^\circ \cap \C_1^\circ$ is an open neighborhood of $X \cap \Omega_0$; it is the union of two halves, one contained in $\Ccore_{\Omega_0}(\Gamma_0)$ and the other in $\Ccore_{\Omega_1}(\Gamma_1)$. Further, by choosing each $p_i$ even closer to $X \cap \Omega_0$ we may arrange that $\C_i$ and $\gamma \cdot \C_i$ are disjoint for each $i \in \{0,1\}$ and each $\gamma \in \Gamma_{1-i}\smallsetminus \Delta$. For if not, then images of $X \cap \Omega_0$ under $\Gamma_{1-i} \smallsetminus \Delta$ come arbitrarily close to $X \cap \Omega_0$. However, the compact submanifold $\Delta \backslash X \cap \Omega_0$ embeds in the quotient $\Gamma_{1-i} \backslash \Omega_{1-i}$; it does not accumulate on itself.

\begin{lemma}
For each $i\in\{0,1\}$ and each $\gamma \in \Gamma_i \smallsetminus \Delta$, the triple of properly convex open sets $(\C_{1-i}^\circ, \C_i^\circ, \gamma \cdot \C_{1-i}^\circ)$ is in occultation position.
\end{lemma}

\begin{proof}
Fix $\gamma \in \Gamma_i \smallsetminus \Delta$.
We need to show that for any points $p \in \overline{\C_{1-i}}$ and $q \in \gamma \cdot \overline{\C_{1-i}}$, the line $\ell$ through $p$ and $q$ intersects $\C_i^\circ$. This is clearly true if one of $p$ or $q$ is in $\C_i^\circ$.
So let us consider points $p$ and $q$ outside of $\C_i^\circ$; by construction of $\C_0, \C_1$ such $p$ and $q$ are on the opposite side of $X$ and $\gamma \cdot X$ respectively from $\Ccore_{\Omega_i}(\Gamma_i)$ in $\Omega^{\max}_i$.
Let $S \subset  (\overline{\C_{1-i}}\smallsetminus \C_i^\circ) \times (\gamma \cdot \overline{\C_{1-i}}\smallsetminus \C_i^\circ)$ denote the set of pairs $(p,q)$ for which the line $\ell$ through $p, q$ intersects $\C_i^\circ$. $S$ is clearly open. 
Let us show that it is closed. 
Suppose by contradiction $(p,q) \in \overline{S} \smallsetminus S$. 
The line $\ell$ through $p$ and $q$ must cross $\overline{X \cap \Omega_0}$ in some point $r$ and then must cross  $\gamma \cdot \overline{X \cap \Omega_0}$ in some point $s$. 
Since $X \cap \Omega_0 \subset \C_i^\circ$ and $(p,q) \notin S$, we must have $r \in X \cap \partial \Omega_0 = \Lambdao_{\Omega_i}(\Delta)$, and similarly we must have $s \in \gamma \cdot  (X \cap \partial \Omega_0) = \Lambdao_{\Omega_i}(\gamma \Delta \gamma^{-1})$.
The entire segment $[r,s]$ lies in $\Fr(\C_i) = \overline{\C_i} \smallsetminus \Int{\C_i}$.
To complete the proof, we must show that no such segment can exist.

Note that if $r = s$, then $\gamma \cdot (X \cap \Omega_0)$ comes arbitrarily close to $X \cap \Omega_0$ by Lemma~\ref{lem:enter-nbhd}, which is impossible since $\Delta \backslash (X \cap \Omega_0)$ is a compact submanifold of $\Gamma_i \backslash \Omega^{\max}_i$; it does not accumulate on itself. So $r \neq s$.  Let $X'$ be a hyperplane supporting $\Omega_i$ along the interval $[r,s]$. Since $\Delta$ acts convex cocompactly on both $\Omega_i$ and $\Omega_i^*$ (Corollary~\ref{cor:cc-Omega-i-and-dual}), and since $X'$ contains a point $r$ of $\Lambdao_{\Omega_i}(\Gamma_i)$, we have $X' \in \Lambdao_{\Omega_i^*}(\Delta)$ by Corollary~\ref{cor:cc-bisat-support}.
Similarly, $X' \in \Lambdao_{\Omega_i^*}(\gamma \Delta \gamma^{-1})$.
But this implies, again by Lemma~\ref{lem:enter-nbhd}, that $z^* \cap \Omega_i^*$ and $\gamma \cdot z^* \cap \Omega_i^*$ are arbitrarily close with respect to $d_{\Omega_i^*}$ (here $z^*$ still refers to the dual hyperplane to the fixed point $z$ of $\Delta$), which is again impossible since $\Delta \backslash (z^* \cap \Omega_i^*)$ is a compact submanifold of $\Delta \backslash \Omega_i^*$. 
This completes the proof that $S$ is closed and hence the triple $(\C_{1-i}^\circ, \C_i^\circ, \gamma \cdot \C_{1-i}^\circ)$ of convex sets is in occultation position.
\end{proof}

Finally, let $\C_i^\varepsilon$ denote the open uniform $\varepsilon$-neighborhood of $\C_i$ in $(\Omega_i,d_{\Omega_i})$. We may choose $\varepsilon> 0$ small enough that $\C_i^\varepsilon$ does not intersect $\gamma \cdot \C_i^\varepsilon$ for all $\gamma \in \Gamma_{1-i} \smallsetminus \Delta$
(this is again for the reason that $\Delta \backslash (X\cap \Omega_i)$ is a compact submanifold embedded in $\Gamma_{1-i} \backslash \Omega_{1-i}$).
Up to shrinking $\varepsilon > 0$ further, we may assume that the part of $\C_i^\varepsilon$ opposite $\Ccore_{\Omega_i}(\Gamma_i)$ with respect to $X$ is completely contained in $\Ccore_{\Omega_{1-i}}(\Gamma_{1-i})$.
The above paragraph then shows that the triple $((\C_{1-i}^\varepsilon)^\circ, (\C_i^\varepsilon)^\circ, \gamma \cdot (\C_{1-i}^\varepsilon)^\circ)$ is in occultation position.

We may now apply Theorem~\ref{thm:amalgam-cc} to the groups $\Gamma_0, \Gamma_1$ acting on the properly convex open sets ($\C_0^\varepsilon)^\circ, (\C_1^\varepsilon)^\circ$ and the closed convex subsets $\C_0, \C_1$.
This completes the proof of Theorem~\ref{thm:gluing}.

\subsection{Proof of Corollary~\ref{cor:double}}

Let $\sigma \in \PGL(V)$ denote the reflection fixing $X$ and~$z$.
Then $\Gamma_0 = \Gamma$ and $\Gamma_1 = \sigma \Gamma \sigma^{-1}$ satisfy the hypotheses of Theorem~\ref{thm:gluing}, giving the result.

\section{Virtual amalgamation over convex cocompact subgroups} \label{sec:virtual}

This section focusses on Theorem~\ref{thm:hopeful}.
Before proving it, let us observe that assumption~\ref{item:hopeful-cases} is satisfied in at least two important situations.
Recall that, given a nonempty properly convex open subset $\Omega$ of $\PP(V)$, a point $x\in\partial\Omega$ is said to be \emph{extremal} (\resp \emph{$C^1$}) in~$\overline{\Omega}$ if there is no open projective segment in $\partial\Omega$ containing~$x$ (\resp there is a unique supporting hyperplane to $\Omega$ at~$x$).

\begin{lemma} \label{lem:hopeful-cases-satisfied}
For $i=0,1$, let $\Gamma_i$ be a discrete subgroup of $\PGL(V)$ preserving a nonempty properly convex open subset $\Omega_i$ of $\PP(V)$; suppose that the action of $\Gamma_i$ on~$\Omega_i$ is convex cocompact, or that every point of $\Lambdao_{\Omega_i}(\Gamma_i)$ is extremal and $C^1$ in~$\overline{\Omega_i}$.
Then assumption \ref{item:hopeful-cases} of Theorem~\ref{thm:hopeful} holds.
\end{lemma}

\begin{proof}[Proof of Lemma~\ref{lem:hopeful-cases-satisfied}]
If the action of $\Gamma_i$ on~$\Omega_i$ is convex cocompact, we use Fact~\ref{fact:cc-strata} and Lemma~\ref{lem:enter-nbhd}.
If every point of $\Lambdao_{\Omega_i}(\Gamma_i)$ is extremal and $C^1$ in~$\overline{\Omega_i}$, we use the definition of extremality and Remark~\ref{rem:strict-conv-rays}.
\end{proof}

We also state the following interesting immediate consequence of Theorem~\ref{thm:hopeful}, Proposition~\ref{prop:Ano-cc-P1-div}, and Lemma~\ref{lem:hopeful-cases-satisfied}.

\begin{corollary} \label{cor:virt-amalgam-Ano}
For $i=0,1$, let $\Gamma_i$ be an infinite Gromov hyperbolic subgroup of $\PGL(V)$ preserving a properly convex open subset $\Omega_i$ of $\PP(V)$, such that the natural inclusion $\Gamma_i \hookrightarrow \PGL(V)$ is $P_1$-Anosov.
Suppose that $\Omega_{01} := \Omega_0 \cap \Omega_1 \neq \varnothing$ and that
\begin{enumerate}[label=(\alph*)]
  \item\label{item:hopeful-separable} $\Delta := \Gamma_0 \cap \Gamma_1$ is separable in both $\Gamma_0$ and~$\Gamma_1$,
  \item\label{item:hopeful-CC-Omega01} if $\Delta$ is infinite, then it acts convex cocompactly on $\Omega_{01} = \Omega_0 \cap \Omega_1$,
  \item\label{item:hopeful-lim-set} $\varnothing \neq \overline{\Ccore_{\Omega_i}(\Gamma_i)} \smallsetminus \Lambdao_{\Omega_{01}}(\Delta) \subset \Omega_{1-i}$ for both $i=0,1$,
\end{enumerate} 
Then there exist finite-index subgroups $\Gamma'_0$ of~$\Gamma_0$ and $\Gamma'_1$ of~$\Gamma_1$, each containing~$\Delta$, such that the amalgamated free product $\Gamma'_0 *_{\Delta} \Gamma'_1$ is Gromov hyperbolic and the representation $\rho : \Gamma'_0 *_{\Delta} \Gamma'_1 \to \PGL(V)$ induced by the inclusions $\Gamma'_0, \Gamma'_1 \hookrightarrow \PGL(V)$ is $P_1$-Anosov and preserves a nonempty properly convex open subset of $\PP(V)$.
\end{corollary}

We now give a proof of Theorem~\ref{thm:hopeful}.

For $i=0,1$, let us consider a discrete subgroup $\Gamma_i$ of $\PGL(V)$ preserving a nonempty properly convex open subset $\Omega_i$ of $\PP(V)$, with $\Omega_{01} := \Omega_0 \cap \Omega_1 \neq \varnothing$, such that conditions \ref{item:hopeful-separable}, \ref{item:hopeful-CC-Omega01}, \ref{item:hopeful-lim-set}, \ref{item:hopeful-cases} of Theorem~\ref{thm:hopeful} are satisfied.
We also fix a neighborhood $\mathcal{U}$ of $\Lambdao_{\Omega_0}(\Gamma_0) \cup \Lambdao_{\Omega_1}(\Gamma_1)$ in $\PP(V)$.
Our goal is to find finite-index subgroups $\Gamma'_0$ of~$\Gamma_0$ and $\Gamma'_1$ of~$\Gamma_1$, each containing $\Delta$, satisfying the conclusions of Theorem~\ref{thm:hopeful}.
We may assume that $\Delta$ has infinite index in both $\Gamma_0$ and~$\Gamma_1$ (if not, then Theorem~\ref{thm:hopeful} is trivial).

The proof will be broken into several steps.

\subsection{Step 1: find suitable convex sets}

We start with the following observation.

\begin{lemma} \label{lem:ias1}
The set $\overline{\Ccore_{\Omega_i}(\Gamma_i)} \smallsetminus \Lambdao_{\Omega_{01}}(\Delta)$ is a nonempty, closed convex subset of~$\Omega_{1-i}$ on which $\Delta$ acts properly discontinuously and cocompactly.
\end{lemma}

\begin{proof}
By assumption~\ref{item:hopeful-lim-set}, the $\Delta$-invariant set $\overline{\Ccore_{\Omega_i}(\Gamma_i)} \smallsetminus \Lambdao_{\Omega_{01}}(\Delta)$ is nonempty and contained in~$\Omega_{1-i}$.
Since the action of $\Delta$ on~$\Omega_{1-i}$ is properly discontinuous (see Section~\ref{subsec:remind-Hilbert}), we have $\Lambdao_{\Omega_{01}}(\Delta) \subset \partial \Omega_{1-i}$, hence $\overline{\Ccore_{\Omega_i}(\Gamma_i)} \smallsetminus \Lambdao_{\Omega_{01}}(\Delta)$ is closed in~$\Omega_{1-i}$ and the action of~$\Delta$ on it is properly discontinuous.
Let us check that this action is cocompact.

If $\Delta$ is finite, then $\Lambdao_{\Omega_{01}}(\Delta) = \varnothing$; since $\overline{\Ccore_{\Omega_i}(\Gamma_i)}$ is compact, it is cocompact under the action of the finite group~$\Delta$.

If $\Delta$ is infinite, then our assumption~\ref{item:hopeful-CC-Omega01} is that $\Delta$ acts convex cocompactly on~$\Omega_{01}$.
By Fact~\ref{fact:cc-subsets}, the group $\Delta$ also acts convex cocompactly on any open uniform neighborhood of $\Ccore_{\Omega_i}(\Gamma_i)$ in $(\Omega_i,d_{\Omega_i})$.
By Lemma~\ref{lem:even-more-pure}, this implies that $\overline{\Ccore_{\Omega_i}(\Gamma_i)} \smallsetminus \Lambdao_{\Omega_{01}}(\Delta)$ is cocompact under~$\Delta$.
\end{proof}

\begin{lemma} \label{lem:ias2-new}
If $\Delta$ is infinite, then for each $i = 0,1$, there exist $\Gamma_i$-invariant closed convex subsets $\C_i$ of~$\Omega_i$, for $i=0,1$, such that
\begin{itemize}
  \item $\C_i$ is contained in a uniform neighborhood of $\Ccore_{\Omega_i}(\Gamma_i)$ in $(\Omega_i,d_{\Omega_i})$, and contains a smaller such uniform neighborhood,
  \item  $\overline{\C_i} \smallsetminus \Lambdao_{\Omega_{01}}(\Delta) \subset \Omega_{1-i}$,
  \item the action of $\Delta$ on $\overline{\C_i} \smallsetminus \Lambdao_{\Omega_{01}}(\Delta)$ is properly discontinuous and cocompact,
 \item  $\partiali\C_i = \partiali\Ccore_{\Omega_i}(\Gamma_i)$,
  \item $\C_0^{\circ} \cap \C_1^{\circ}$ is nonempty and contains a neighborhood of $\Ccore_{\Omega_{01}}(\Delta)$ in~$\Omega_{01}$.
\end{itemize}
If $\Delta$ is finite, then the same conclusion holds after possibly replacing each $\Gamma_i$ by a finite-index subgroup.
\end{lemma}

\begin{proof}
By Lemma~\ref{lem:ias1}, the group $\Delta$ acts properly discontinuously and cocompactly on the closed convex subset $\overline{\Ccore_{\Omega_i}(\Gamma_i)} \smallsetminus \Lambdao_{\Omega_{01}}(\Delta)$ of~$\Omega_{1-i}$.
Let $\mathcal{D}_i$ be a compact fundamental domain for this action.
For $\varepsilon>0$ small enough, the closed uniform $\varepsilon$-neighborhood $\mathcal{D}_i^{\varepsilon}$ of $\mathcal{D}_i \cap \Omega_i$ in $(\Omega_i,d_{\Omega_i})$ satisfies $\overline{\mathcal{D}_i^{\varepsilon}} \subset \Omega_{1-i}$.
Let $\C_i^{\varepsilon}$ be the closed uniform $\varepsilon$-neighborhood of $\Ccore_{\Omega_i}(\Gamma_i)$ in $(\Omega_i,d_{\Omega_i})$.
The family $(\gamma\cdot\mathcal{D}_i^{\varepsilon})_{\gamma\in\Delta}$ covers~$\C_i^{\varepsilon}$, hence $\C_i^{\varepsilon} \subset \Omega_{1-i}$.

We claim that $\overline{\C_i^{\varepsilon}} \smallsetminus \Lambdao_{\Omega_{01}}(\Delta) \subset \Omega_{1-i}$.
Indeed, consider a point $x \in \overline{\C_i^{\varepsilon}}$.
We can write it as the limit of some sequence $(x_n)_{n\in\NN}$ of points of~$\C_i^{\varepsilon}$.
For each~$n$ there exists $\delta_n \in \Delta$ such that $y_n := \delta_n \cdot x_n$ belongs to $\mathcal{D}_i^{\varepsilon}$.
Up to passing to a subsequence, we may assume that either $(\delta_n)_{n\in\NN}$ is constant, or $\Delta$ is infinite and $\delta_n \to \infty$ in~$\Delta$.
In the first case, $x = \lim_n x_n$ lies in a $\Delta$-translate of $\overline{\mathcal D_i^\varepsilon}$, hence is contained in $\Omega_{1-i}$.
In the second case, since $\overline{\mathcal D_i^\varepsilon}$ is compact, up to passing to a subsequence, we may assume that $\delta_n \cdot \overline{\mathcal D_i^\varepsilon}$ converges to a compact subset of $\Lambdao_{\Omega_{1-i}}(\Delta)$, and so $x$ belongs to $\Lambdao_{\Omega_{1-i}}(\Delta)$, which is equal to $\Lambdao_{\Omega_{01}}(\Delta)$ by Fact~\ref{fact:cc-subsets} and assumption~\ref{item:hopeful-CC-Omega01}.
This proves the claim.
In particular, the action of $\Delta$ on $\overline{\C_i} \smallsetminus \Lambdao_{\Omega_{01}}(\Delta)$ is properly discontinuous (see Section~\ref{subsec:remind-Hilbert}).

If $\Delta$ is infinite, then our assumption~\ref{item:hopeful-CC-Omega01} is that $\Delta$ acts convex cocompactly on~$\Omega_{01}$.
Then $\Ccore_{\Omega_0}(\Gamma_0) \cap \Ccore_{\Omega_1}(\Gamma_1) \supset \Ccore_{\Omega_{01}}(\Delta)$ is nonempty.
By Lemma~\ref{lem:even-more-pure}, the action of $\Delta$ on $\overline{\C_i} \smallsetminus \Lambdao_{\Omega_{01}}(\Delta)$ is cocompact.
We have $\partiali \C_i = \partiali\Ccore_{\Omega_i}(\Gamma_i)$ by assumption~\ref{item:hopeful-cases} and Lemma~\ref{lem:partiali-Ccore-union-of-faces}.
We can take $\C_i = \C_i^{\varepsilon}$ as above.

If $\Delta$ is finite, then it admits a global fixed point $x_0$ in the nonempty properly convex open set~$\Omega_{01}$.
By assumption~\ref{item:hopeful-lim-set}, the set $\Lambdao_{\Omega_i}(\Gamma_i) \smallsetminus \Lambdao_{\Omega_{01}}(\Delta)$, which in this case is simply $\Lambdao_{\Omega_i}(\Gamma_i)$, is contained in $\Omega_{1-i}$, and so the set $F_i := \{ \gamma\in\Gamma_i \,|\, \gamma\cdot x_0 \notin \Omega_{1-i}\}$ is finite.
By our separability assumption~\ref{item:hopeful-separable}, there is a finite-index subgroup $\Gamma_i'$ of~$\Gamma_i$, containing $\Delta$, such that $\Gamma_i' \cap F_i = \varnothing$.
Thus the $\Gamma_i'$-orbit of~$x_0$ is entirely contained in~$\Omega_{1-i}$, and after replacing $\Gamma_i$ by~$\Gamma'_i$ we can take for $\C_i$ the convex hull of $\C_i^{\varepsilon}$ as above and of the $\Gamma'_i$-orbit of~$x_0$.
Note that $\C_i$ is contained in $\C_i^{\varepsilon'}$ for some $\varepsilon' \geq \varepsilon$, and so $\partiali \C_i = \partiali\Ccore_{\Omega_i}(\Gamma_i)$ again by assumption~\ref{item:hopeful-cases} and Lemma~\ref{lem:partiali-Ccore-union-of-faces}.
\end{proof}

\subsection{Step 2: cocompactness of the ``putty''}

We now fix some $\Gamma_i$-invariant closed convex neighborhoods $\C_i$ of $\Ccore_{\Omega_i}(\Gamma_i)$ in~$\Omega_i$ as in Lemma~\ref{lem:ias2-new}.
Note the action of $\Delta$ on $\overline{\C_i} \smallsetminus \Lambdao_{\Omega_{01}}(\Delta)$ is properly discontinuous and cocompact.

\begin{remark} \label{rem:C0-cap-C1-cocompact}
We have $\C_0 \cap \C_1 = (\overline{\C_0} \smallsetminus \Lambdao_{\Omega_{01}}(\Delta)) \cap (\overline{\C_1} \smallsetminus \Lambdao_{\Omega_{01}}(\Delta))$.
Indeed, the left-hand side is clearly contained in the right-hand side.
Conversely, since $\overline{\C_i} \smallsetminus \Lambdao_{\Omega_{01}}(\Delta) \subset \Omega_{1-i}$, the right-hand side is contained in $\overline{\C_i} \cap \Omega_i = \C_i$ for both $i=0,1$.
We deduce that the action of $\Delta$ on $\C_0 \cap \C_1$ is cocompact.
\end{remark}

The following lemma will be used twice in the proofs of Proposition~\ref{prop:hopeful-putty} and Lemma~\ref{lem:virtual-occultation} below.

\begin{lemma} \label{lem:ias3}
For $i\in\{0,1\}$, let $\xi \in \partiali\C_i$ and $y \in \overline{\C_{1-i}}$ satisfy $[\xi, y] \cap \C_i^\circ = \nolinebreak \varnothing$.
Then $[\xi, y] \subset \Lambdao_{\Omega_{01}}(\Delta)$.
\end{lemma}

\begin{proof}
We note that the segment $[y,\xi) \subset \overline{\Omega_i}$ cannot contain a point $x$ of~$\Omega_i$.
Indeed, otherwise, by assumption~\ref{item:hopeful-cases} and the fact that $\partiali \C_i = \partiali \Ccore_{\Omega_i}(\Gamma_i)$ (Lemma~\ref{lem:ias2-new}), the ray $[x,\xi)$ would eventually enter any uniform neighborhood of $\Ccore_{\Omega_i}(\Gamma_i)$ in $(\Omega_i,d_{\Omega_i})$, in particular $\C_i^\circ$, contradicting the assumption that $[\xi,y] \cap \C_i^\circ = \varnothing$.

Since $y\in\overline{\C_{1-i}}$ and since we have chosen $\C_{1-i}$ in such a way that $\overline{\C_{1-i}} \smallsetminus \Lambdao_{\Omega_{01}}(\Delta) \subset \Omega_i$, we deduce that $y \in \Lambdao_{\Omega_{01}}(\Delta)$.
In particular, $\Lambdao_{\Omega_{01}}(\Delta) \neq \varnothing$, hence $\Delta$ is infinite.
Moreover, $[y,\xi] \subset \overline{\C_i}$ since $\Lambdao_{\Omega_{01}}(\Delta) \subset \Lambdao_{\Omega_i}(\Gamma_i) \subset \overline{\C_i}$ by construction.

By Fact~\ref{fact:cc-subsets}, the group $\Delta$ acts convex cocompactly on~$\Omega_i$.
By Fact~\ref{fact:cc-nbhd}, the convex set $\C_i = \overline{\C_i} \smallsetminus \Lambdao_{\Omega_{01}}(\Delta)$ has bisaturated boundary.
Therefore $[y,\xi] \subset \Lambdao_{\Omega_{01}}(\Delta)$ by Definition~\ref{def:boundaries} of bisaturated boundary.
\end{proof}

The convex hull $\Conv{\C_0  \cup \C_1}$ (taken inside of $\overline{\Omega_0}$ or $\overline{\Omega_1}$) is the union of $\C_0 \cup \C_1$ together with some additional points.
We now investigate the behavior of those additional points.
For this we consider the set
$$P := \{ x_0 \times x_1 \in \partialn \C_0 \times \partialn \C_1 : \ [x_0,x_1]  \cap \C_0^\circ  =  [x_0,x_1]  \cap \C_1^\circ = \varnothing \}.$$
Here and in the rest of this section, we denote an ordered pair of points by $x_0 \times x_1$ in order to avoid any ambiguity with the notation for open segments $(x_0, x_1)$.
We introduce the set
\begin{equation} \label{eqn:putty}
\mathscr P := \bigcup_{x_0\times x_1 \in P} [x_0,x_1]
\end{equation}
(which we call the ``putty'').

\begin{proposition} \label{prop:hopeful-putty}
\begin{enumerate}
  \item\label{item:putty-1} We have $\Conv{\C_0  \cup \C_1} = \C_0 \cup \C_1 \cup \mathscr P$.
  \item\label{item:putty-2} The action of $\Delta$ on~$\mathscr P$ is cocompact.
  In particular, $\mathscr P$ is closed in~$\Omega_{01}$.
  \item\label{item:putty-3} We have $\Conv{\C_0  \cup \C_1}^\circ \subset \C_0^\circ \cup \C_1^\circ \cup \mathscr P$.
\end{enumerate}
\end{proposition}

The proof of Proposition~\ref{prop:hopeful-putty}.\eqref{item:putty-2} will rely on the following lemma.

\begin{lemma} \label{lem:ias4}
The set $\overline P \smallsetminus P$ is contained in
$$\big\{ x_0 \times x_1 \subset \Lambdao_{\Omega_{01}}(\Delta) \times \Lambdao_{\Omega_{01}}(\Delta) : [x_0, x_1] \subset \Lambdao_{\Omega_{01}}(\Delta)\big\}.$$
\end{lemma}

\begin{proof}[Proof of Lemma~\ref{lem:ias4}]
The set $P$ is the intersection of $\partialn \C_0 \times \partialn \C_1$ with the closed subset of $\overline{\C_0} \times \overline{\C_1}$ consisting of those pairs $x_0 \times x_1 \in \overline{\C_0} \times \overline{\C_1}$ such that $[x_0,x_1]  \cap \C_0^\circ  = [x_0,x_1]  \cap \C_1^\circ = \varnothing$.
Therefore, if $x_0 \times x_1$ lies in $\overline P \smallsetminus P$, then $x_i \in \partiali \C_i$ for some $i \in \{0,1\}$ and we still have $[x_0,x_1]  \cap \C_0^\circ  = [x_0,x_1]  \cap \C_1^\circ = \varnothing$.
Lemma~\ref{lem:ias3} with $(\xi,y) = (x_0,x_1)$ then implies $[x_0, x_1] \subset \Lambdao_{\Omega_{01}}(\Delta)$.
\end{proof}

\begin{proof}[Proof of Proposition~\ref{prop:hopeful-putty}]
\eqref{item:putty-1} Clearly $\C_0 \cup \C_1 \cup \mathscr P \subset \Conv{\C_0  \cup \C_1}$.
Let us prove the reverse inclusion.
Consider a point $z \in \Conv{\C_0  \cup \C_1} \smallsetminus (\C_0\cup\C_1)$.
Since $\C_0$ and $\C_1$ are convex, $z$ lies in the convex hull $[x_0, x_1]$ (taken inside of $\overline{\Omega_0}$ or $\overline{\Omega_1}$) of two points $x_0 \in \C_0$ and $x_1 \in \C_1$.
Since $z \notin \C_0\cup\C_1$, there exist $x_0' \in \overline{\C_0}$ and $x_1' \in \overline{\C_1}$ such that $z \in [x_0', x_1'] \subset [x_0, x_1]$ and $(x_0',x_1')$ is disjoint from both $\C_0$ and~$\C_1$.
In order to show that $z \in \mathscr P$, it is sufficient to check that $x_i' \in \partialn \C_i$ for both $i=0,1$.
Suppose by contradiction that this is not the case: there exists $i\in\{0,1\}$ such that $x_i' \in \partiali \C_i$.
If $x_0' = x_1'$, then $z = x_0' = x_1' \in \partiali\C_i \cap \overline{\C_{1-i}} \subset \partial\Omega_i \cap \overline{\C_{1-i}}$; since by construction $\overline{\C_{1-i}} \smallsetminus \Lambdao_{\Omega_{01}}(\Delta) \subset \Omega_{i}$ (Lemma~\ref{lem:ias2-new}), we deduce $z \in \Lambdao_{\Omega_{01}}(\Delta)$.
If $x_0' \neq x_1'$, then Lemma~\ref{lem:ias3} with $(\xi,y) = (x_0',x_1')$ implies $[x_0', x_1'] \subset \Lambdao_{\Omega_{01}}(\Delta)$, hence $z \in \Lambdao_{\Omega_{01}}(\Delta)$.
In either case, this is impossible since $\Conv{\C_0 \cup \C_1} \subset \Omega_{01}$, while $\Lambdao_{\Omega_{01}}(\Delta) \subset \partial \Omega_{01}$. 
So in fact $x_i' \in \partialn \C_i$ for both $i=0,1$, and therefore $z \in \mathscr P$.

\eqref{item:putty-2} In order to prove that the action of $\Delta$ on $\mathscr P$ is cocompact, it is sufficient to show that the action of $\Delta$ on $P$ is cocompact.
Let us prove this.
By Lemma~\ref{lem:ias2-new}, the action of $\Delta$ on $\overline{\C_0} \smallsetminus \Lambdao_{\Omega_{01}}(\Delta)$ is properly discontinuous and cocompact.
Let $\mathcal{K}_0$ be a compact fundamental domain for this action. 
Then $P \subset \bigcup_{\gamma\in\Delta} \gamma\cdot (\mathcal{K}_0 \times \overline{\C_1})$.
Observe that $\mathcal{K}_0 \times \overline{\C_1}$ is compact and $P$ is closed in $(\overline{\C_0} \smallsetminus \Lambdao_{\Omega_{01}}(\Delta)) \times \overline{\C_1}$ by Lemma~\ref{lem:ias4}.
Hence $(\mathcal{K}_0 \times \overline{\C_1}) \cap P$ is compact.
This shows that the action of $\Delta$ on~$P$ is cocompact, hence the action of $\Delta$ on~$\mathscr P$ is too.

\eqref{item:putty-3} Consider a point $z \in \Conv{\C_0  \cup \C_1}^\circ \smallsetminus (\C_0^\circ \cup \C_1^\circ)$.
If $z \notin \C_0 \cup \C_1$, then $z \in \mathscr{P}$ by~\eqref{item:putty-1}.
Otherwise, $z \in \partialn \C_0 \cup \partialn \C_1$.
Consider a point $w$ in the set $\C_0^\circ \cap \C_1^\circ$, which is nonempty by Lemma~\ref{lem:ias2-new}.
Then the ray starting from $w$ and passing through $z$ exits both $\C_0$ and $\C_1$ at~$z$.
By approaching $z$ along this ray on the side opposite~$w$, we may write $z$ as a limit of points $z_n \notin (\C_0 \cup \C_1)$.
Since $z \in \Conv{\C_0  \cup \C_1}^\circ$, we have $z_n \in \Conv{\C_0  \cup \C_1}$ for all large enough~$n$, hence $z_n \in \mathscr{P}$ by~\eqref{item:putty-1}.
Since $\mathscr{P}$ is closed in~$\Omega_{01}$ by~\eqref{item:putty-2}, we deduce $z \in \mathscr{P}$.
\end{proof}

\subsection{Step~3: use separability to get into occultation position}

By Remark~\ref{rem:C0-cap-C1-cocompact} and Proposition~\ref{prop:hopeful-putty}, the sets $\C_0 \cap \C_1$ and $\mathscr P$ are both cocompact under the action of~$\Delta$.
We now use a standard separability argument to remove the elements of $\Gamma_0 \smallsetminus \Delta$ and $\Gamma_1 \smallsetminus \Delta$ which take $(\C_0 \cap \C_1) \cup \mathscr P$ back to itself. 

\begin{lemma} \label{lem:precise}
For each $i \in \{0,1\}$, there exists a finite-index subgroup $\Gamma_i'$ of~$\Gamma_i$, containing~$\Delta$, such that $((\C_0 \cap \C_1) \cup \mathscr P) \cap \gamma \cdot ((\C_0 \cap \C_1) \cup \mathscr P) = \varnothing$ for all $\gamma \in \Gamma_i' \smallsetminus \Delta$.
\end{lemma}

\begin{proof}
By Remark~\ref{rem:C0-cap-C1-cocompact} and Proposition~\ref{prop:hopeful-putty}, the action of $\Delta$ on $(\C_0 \cap \C_1) \cup \mathscr P$ is cocompact.
Let $\mathcal{K} \subset (\C_0 \cap \C_1) \cup \mathscr P$ be a compact fundamental domain for this action.
By proper discontinuity of the action of $\Gamma_i$ on $\Omega_i$, the set $F_i := \{\gamma \in \Gamma_i \smallsetminus \Delta \,|\, \mathcal{K} \cap \gamma\cdot\mathcal{K} \neq \varnothing\}$ is finite.
By our separability assumption~\ref{item:hopeful-separable}, there exists a finite-index subgroup $\Gamma_i'$ of~$\Gamma_i$, containing $\Delta$, such that $\Gamma_i' \cap F_i = \varnothing$.
For $\gamma \in \Gamma_i'$, we then have $((\C_0 \cap \C_1) \cup \mathscr P) \cap \gamma \cdot ((\C_0 \cap \C_1) \cup \mathscr P) \neq \varnothing$ if and only $\gamma \in \Delta$. 
\end{proof}

We now fix $\Gamma_i'$ as in Lemma~\ref{lem:precise}.
It is an immediate consequence of Definition~\ref{def:limcore} (see \cite[Lem.\,4.7]{dgk-proj-cc}) that $\Lambdao_{\Omega_i}(\Gamma_i') = \Lambdao_{\Omega_i}(\Gamma_i)$, hence $\Ccore_{\Omega_i}(\Gamma_i') = \Ccore_{\Omega_i}(\Gamma_i)$.

\begin{lemma}\label{lem:precise-LambdaoDelta}
For each $i\in\{0,1\}$ and each $\gamma \in \Gamma_i' \smallsetminus \Delta$, we have $\Lambdao_{\Omega_{01}}(\Delta) \cap \gamma \cdot \Lambdao_{\Omega_{01}}(\Delta) = \varnothing$.
\end{lemma}

\begin{proof}
If $\Delta$ is finite, then $\Lambdao_{\Omega_{01}}(\Delta) = \varnothing$ and there is nothing to prove.
So we assume that $\Delta$ is infinite, in which case the assumption~\ref{item:hopeful-CC-Omega01} of Theorem~\ref{thm:hopeful} is that $\Delta$ acts convex cocompactly on~$\Omega_{01}$.
Then $\Delta$ also acts convex cocompactly on~$\Omega_i$ by Fact~\ref{fact:cc-subsets}.
Consider an element $\gamma \in \Gamma'_i$ and a point $x \in \PP(V)$ such that $x$ and $\gamma\cdot x$ both belong to~$\Lambdao_{\Omega_{01}}(\Delta)$.
Let $[a,x)$ be a ray in $\Ccore_{\Omega_{01}}(\Delta)$.
By Lemma~\ref{lem:enter-nbhd}, the ray $\gamma\cdot [a,x) = [\gamma\cdot a, \gamma\cdot x)$ eventually enters any uniform neighborhood of $\Ccore_{\Omega_i}(\Delta)$ in $(\Omega_i,d_{\Omega_i})$.
Since $\C_0 \cap \C_1$ contains such a uniform neighborhood (as arranged by Lemma~\ref{lem:ias2-new}), we have that $\gamma\cdot [a,x)$ enters $\C_0 \cap \C_1$.
By Lemma~\ref{lem:precise}, this implies $\gamma \in \Delta$. 
\end{proof}

\begin{lemma} \label{lem:virtual-occultation}
The groups $\Gamma_i'$ acting on the interiors $\C_i^\circ$ satisfy the conditions of Theorem~\ref{thm:amalgam}: for each $i \in \{0,1\}$ and each $\gamma \in \Gamma_i' \smallsetminus \Delta$, the triple $(\C_{1-i}^\circ, \C_i^\circ, \gamma \cdot \C_{1-i}^\circ)$ is in occultation position. 
\end{lemma}

\begin{proof}
Fix $i \in \{0,1\}$ and $\gamma \in \Gamma_i' \smallsetminus \Delta$, and let us check that the triple $(\C_{1-i}^\circ, \C_i^\circ, \gamma \cdot \C_{1-i}^\circ)$ is in occultation position.

We first observe that $\C_{1-i}^\circ \not\subset \C_i^\circ$:
indeed, otherwise we would have $\C_0^\circ \cap \C_1^\circ = \C_{1-i}^{\circ}$, hence $\Gamma_{1-i}' \cdot (\C_0^\circ \cap \C_1^\circ) = \C_0^\circ \cap \C_1^\circ$, contradicting the definition of $\Gamma_{1-i}'$ from Lemma~\ref{lem:precise}.
(Note that $\Gamma_{1-i}' \smallsetminus \Delta \neq \varnothing$ by our assumption that $\Delta$ has infinite index in~$\Gamma_{1-i}$.)
Applying $\gamma$, we further find $\gamma \cdot \C_{1-i}^\circ \not\subset \gamma \cdot \C_i^\circ = \C_i^\circ$.

By construction (see Lemma~\ref{lem:ias2-new}) we have $\C_{1-i}^\circ \cap \C_i^\circ \neq \varnothing$.
Therefore $\gamma \cdot \C_{1-i}^\circ \cap \C_i^\circ =\linebreak \gamma \cdot (\C_{1-i}^\circ \cap \C_i^\circ) \neq \varnothing$.

Let us check that $\C_{1-i}^\circ \cap \gamma \cdot \C_{1-i}^\circ = \varnothing$.
To this end, consider the sets $\Conv{\C_i \cup \C_{1-i}}^\circ$ and $\gamma \cdot \Conv{\C_i \cup \C_{1-i}}^\circ = \Conv{\C_i \cup \gamma \cdot \C_{1-i}}^\circ$.
The intersection of these two sets clearly contains $\C_i^\circ$; we claim the intersection is equal to $\C_i^\circ$.
Indeed, suppose not.
Then the intersection contains a point $x \in \Fr(\C_i)$.
In fact, since both sets $\Conv{\C_i \cup \C_{1-i}}^\circ$ and $\gamma \cdot \Conv{\C_i \cup \C_{1-i}}^\circ = \Conv{\C_i \cup \gamma \cdot \C_{1-i}}^\circ$ are contained in $\Omega_i$ by construction (see Lemma~\ref{lem:ias2-new}), we must have $x \in \partialn \C_i$.
Then, either $x \in \C_{1-i}^\circ$, in which case $x \in \C_0 \cap \C_1$, or $x \notin \C_{1-i}^\circ$, in which case $x \in \mathscr P$ by Proposition~\ref{prop:hopeful-putty}.\eqref{item:putty-3}.
Either way, $x \in (\C_0 \cap \C_1) \cup \mathscr P$.
Similarly, $x \in \gamma \cdot ((\C_0 \cap \C_1) \cup \mathscr P)$.
This contradicts the definition of $\Gamma_i'$ in Lemma~\ref{lem:precise}.
Therefore 
\begin{align}\label{eqn:key-thing}
\Conv{\C_i \cup \C_{1-i}}^\circ \cap \Conv{\C_i \cup \gamma \cdot \C_{1-i}}^\circ = \C_i^\circ.
\end{align}
This implies in particular that $\C_{1-i}^\circ \cap \gamma \cdot \C_{1-i}^\circ \subset \C_i^\circ$.
Then $\C_{1-i}^\circ \cap \gamma \cdot \C_{1-i}^\circ =\linebreak (\C_{1-i}^\circ \cap \C_i^\circ) \cap \gamma \cdot (\C_{1-i}^\circ \cap \C_i^\circ) = \varnothing$, again by Lemma~\ref{lem:precise}. 
 
Finally, let us check that 
\begin{equation}
\label{eq:aardvaark}
\overline{(\C_i^\circ)^*} \subset (\C_{1-i}^\circ)^* \cup (\gamma \cdot \C_{1-i}^\circ)^*.
\end{equation}
Suppose by contradiction that this is not the case.
Then there is a projective line $\ell$ passing through both $\overline{\C_{1-i}^\circ} = \overline{\C_{1-i}}$ and $\overline{\gamma \cdot \C_{1-i}^\circ} = \gamma \cdot \overline{\C_{1-i}}$ but not intersecting $\C_i^\circ$.
Since both sets $\overline{\C_{1-i}}$ and $\gamma \cdot \overline{\C_{1-i}}$ intersect $\C_i^\circ$, we may move $\ell$ closer and closer to $\C_i^\circ$ until it just touches $\Fr(\C_i)$ at a point~$x$.
Let $z \in \ell \cap \overline{\C_{1-i}}$ and $y \in  \ell \cap \gamma \cdot \overline{\C_{1-i}}$. 

First consider the case $x \in \partiali \C_i$.
Since $[x,z] \cap \C_i^\circ = \varnothing$, Lemma~\ref{lem:ias3} implies $[x,z] \subset \Lambdao_{\Omega_{01}}(\Delta)$.
Similarly, since $\gamma^{-1}\cdot x \in \partiali \C_i$ and $[\gamma^{-1}\cdot x,\gamma^{-1}\cdot y] \cap \C_i^\circ = \varnothing$, Lemma~\ref{lem:ias3} implies $[\gamma^{-1}\cdot x,\gamma^{-1}\cdot y] \subset \Lambdao_{\Omega_{01}}(\Delta)$.
Thus $x \in \Lambdao_{\Omega_{01}}(\Delta) \cap \gamma \cdot \Lambdao_{\Omega_{01}}(\Delta)$: contradiction with Lemma~\ref{lem:precise-LambdaoDelta}.

Next consider the case $x \in \partialn \C_i$. 
Possibly, $x \in \C_{1-i}$, in which case $x \in \C_0 \cap \C_1$.
Otherwise, $x \notin \C_{1-i}$.
Then, let $z' \in [x,z]$ be such that $[x,z'] \subset [x,z]$ is maximal with the property $[x,z'] \cap \C_{1-i}^\circ = \varnothing$; we have $z' \in \Fr(\C_{1-i}) = \partiali \C_{1-i} \cup \partialn \C_{1-i}$.
We cannot have $z' \in \partiali \C_{1-i}$, otherwise Lemma~\ref{lem:ias3} with $(i,\xi,y)$ replaced by $(1-i,z',x)$ would imply $x \in \Lambdao_{\Omega_{01}}(\Delta) \subset \partiali \C_i$, which we have assumed is not the case; therefore $z' \in \partialn \C_{1-i}$, showing that $[x, z'] \subset \mathscr P$.
Either way, $x \in (\C_0 \cap \C_1) \cup \mathscr P$.
Similarly, we have $x \in \gamma \cdot ((\C_0 \cap \C_1) \cup \mathscr P)$.
This contradicts the definition of $\Gamma_i'$ from Lemma~\ref{lem:precise}.
Thus~\eqref{eq:aardvaark} holds, finishing the proof.
\end{proof}

\subsection{Proof of conclusions \eqref{item:hopeful-1} and~\eqref{item:hopeful-3} of Theorem~\ref{thm:hopeful}} \label{subsec:proof-hopeful-1-3}
 
\begin{proof}[Proof of conclusion~\eqref{item:hopeful-1}]
Let $\C_0 \subset \Omega_0$ and $\C_1 \subset \Omega_1$ be as in Lemma~\ref{lem:ias2-new}, and let $\Gamma_0'$ and~$\Gamma_1'$ be finite-index subgroups of $\Gamma_0$ and~$\Gamma_1$ as in Lemma~\ref{lem:precise}.
By Lemma~\ref{lem:virtual-occultation}, for each $i \in \{0,1\}$ and $\gamma \in \Gamma_i' \smallsetminus \Delta$, the triple $(\C_{1-i}^\circ, \C_i^\circ, \gamma\cdot\C_{1-i}^\circ)$ is in occultation position.
Applying Theorem~\ref{thm:amalgam} with $(\Gamma_0,\Omega_0,\Gamma_1,\Omega_1)$ replaced by $(\Gamma'_0,\C_0^{\circ},\Gamma'_1,\C_1^{\circ})$, we obtain that the representation $\rho : \Gamma'_0 *_{\Delta} \Gamma'_1 \to \PGL(V)$ induced by the inclusions of $\Gamma'_0$ and $\Gamma'_1$ is discrete and faithful, and that its image preserves a nonempty properly convex open subset of $\PP(V)$, namely
\begin{equation} \label{eqn:Omega-hopeful}
\Omega := \bigcup_{\gamma \in \Gamma'_0 *_{\Delta} \Gamma'_1} \rho(\gamma) \cdot \Conv{\C_0^{\circ} \cup \C_1^{\circ}}.
\end{equation}
The set $\Omega$ contains $\Omega_{0,1}\cup \Ccore_{\Omega_0}(\Gamma_0) \cup \Ccore_{\Omega_1}(\Gamma_1)$ by construction. Note that the action of $\Delta$ on~$\Omega$ via~$\rho$ is still convex cocompact by Fact~\ref{fact:cc-subsets}.
\end{proof}

\begin{remark} \label{rem:adj-convex-hopeful}
Recall from the proof of Theorem~\ref{thm:amalgam} in Section~\ref{subsec:proof-amalgam-HNN} that the sets $\rho(\gamma) \cdot \Conv{\C_0^{\circ} \cup \C_1^{\circ}}$ in the definition \eqref{eqn:Omega-hopeful} of $\Omega$ are the convex sets $\Omega(e)$ associated to the edges $e$ of the Bass--Serre tree for $\Gamma'_0 *_{\Delta} \Gamma'_1$ in a tree-of-convex-sets construction as in Theorem~\ref{thm:tree}.
Conclusion~\eqref{treethm:3} of Theorem~\ref{thm:tree} states that two such convex sets $\Omega(e),\Omega(e')$ overlap if and only if the corresponding edges $e,e'$ in the Bass-Serre tree share a vertex; this means that $\Conv{\C_0^{\circ} \cup \C_1^{\circ}} \cap \rho(\gamma) \cdot \Conv{\C_0^{\circ} \cup \C_1^{\circ}} \neq \varnothing$ if and only if $\gamma$ belongs to $\Gamma'_0$ or to~$\Gamma'_1$.
\end{remark}
 
\begin{proof}[Proof of conclusion~\eqref{item:hopeful-3}]
Suppose that the action of $\Gamma_i$ on~$\Omega_i$ is convex cocompact for both $i=0,1$.
By Lemma~\ref{lem:ias2-new} and Fact~\ref{fact:strict-C1-nbhd}, up to possibly replacing each $\Gamma_i$ by a finite-index subgroup, there exist closed $\Gamma_i$-invariant convex subsets $\C'_i \subset \C_i$ of~$\Omega_i$, for $i=0,1$, such that
\begin{itemize}
  \item $\C_i$ is contained in a uniform neighborhood of $\Ccore_{\Omega_i}(\Gamma_i)$ in $(\Omega_i,d_{\Omega_i})$,
  \item $\C'_i$ is a closed $\Gamma_i$-invariant convex subset of~$\C_i^{\circ}$ with strictly convex nonideal boundary,
  \item $\overline{\C'_i} \smallsetminus \Lambdao_{\Omega_{01}}(\Delta) \subset \overline{\C_i} \smallsetminus \Lambdao_{\Omega_{01}}(\Delta) \subset \Omega_{1-i}$,
  \item ${\C'_0}^{\circ} \cap {\C'_1}^{\circ}$ is nonempty and contains a neighborhood of $\Ccore_{\Omega_{01}}(\Delta)$ in~$\Omega_{01}$.
\end{itemize}
Our assumption that the action of $\Gamma_i$ on~$\Omega_i$ is convex cocompact means that $\Ccore_{\Omega_i}(\Gamma_i)$ is cocompact under~$\Gamma_i$; since $\C'_i$ is closed in~$\Omega_i$ and contained in a uniform neighborhood of $\Ccore_{\Omega_i}(\Gamma_i)$ in $(\Omega_i,d_{\Omega_i})$, it is still cocompact under~$\Gamma_i$.
  
Let $\mathscr P$ be the ``putty'' of \eqref{eqn:putty} for $\C_0$ and~$\C_1$, and let $\mathscr P'$ be defined similarly with respect to $\C'_0$ and~$\C'_1$.
By Lemma~\ref{lem:precise}, for each $i \in \{0,1\}$, there exists a finite-index subgroup $\Gamma_i'$ of~$\Gamma_i$, containing~$\Delta$, such that
$$((\C_0 \cap \C_1) \cup \mathscr P) \cap \gamma \cdot ((\C_0 \cap \C_1) \cup \mathscr P) = ((\C_0' \cap \C_1') \cup \mathscr P') \cap \gamma \cdot ((\C_0' \cap \C_1') \cup \mathscr P') = \varnothing$$
for all $\gamma \in \Gamma_i' \smallsetminus \Delta$.
By Lemma~\ref{lem:virtual-occultation}, for each $i \in \{0,1\}$ and each $\gamma \in \Gamma_i' \smallsetminus \Delta$, the triples $(\C_{1-i}^\circ, \C_i^\circ, \gamma\cdot\C_{1-i}^\circ)$ and $({\C^{\prime \, \circ}_{1-i}}, {\C^{\prime \, \circ}_i}, \gamma\cdot{\C^{\prime \, \circ}_{1-i}})$ are in occultation position.
By Proposition~\ref{prop:hopeful-putty}.\eqref{item:putty-1}, the closure of $\Conv{\C'_0\cup\C'_1} \smallsetminus (\C'_0\cup\C'_1)$ in $\Conv{\C_0^\circ \cup \C_1^\circ}$ is a closed subset of $\mathscr P'$; since $\mathscr P'$ is cocompact under the action of~$\Delta$ (Proposition~\ref{prop:hopeful-putty}.\eqref{item:putty-2}), so is the closure of $\Conv{\C'_0 \cup \C'_1} \smallsetminus (\C'_0 \cup \C'_1)$ in $\Conv{\C_0^\circ \cup \C_1^\circ}$.
We then apply Theorem~\ref{thm:amalgam-cc} with $(\Gamma_0,\Omega_0,\C_0,\Gamma_1,\Omega_1,\C_1)$ replaced by $(\Gamma'_0,\C_0^{\circ},\C'_0,\Gamma'_1,\C_1^{\circ},\C'_1)$.
\end{proof}

\subsection{Step~4: control the full orbital limit sets of $\Gamma'_0$ and~$\Gamma'_1$}

\begin{lemma} \label{lem:Ci-unif-neighb-Ccore-Omega}
For both $i=0,1$, the convex set $\C_i$ contains a uniform neighborhood of $\Ccore_{\Omega_i}(\Gamma_i)$ in $(\Omega,d_{\Omega})$.
\end{lemma}

\begin{proof}
Fix $i\in\{0,1\}$.
Suppose by contradiction that $\C_i$ does \emph{not} contain a uniform neighborhood of $\Ccore_{\Omega_i}(\Gamma_i)$ in $(\Omega,d_{\Omega})$: this means that for any $n\geq 1$ there exists $x_n \in \partialn\C_i$ such that $d_{\Omega}(x_n,\Ccore_{\Omega_i}(\Gamma_i)) \leq 1/n$.
By definition \eqref{eqn:Omega-hopeful} of~$\Omega$, for each~$n$ there exists $\gamma_n \in \Gamma'_0 *_{\Delta} \Gamma'_1$ such that $y_n := \rho(\gamma_n) \cdot x_n$ belongs to $\Conv{\C_0^{\circ} \cup \C_1^{\circ}}$.

We claim that $\gamma_n \in \Gamma'_i$ for all~$n$.
Indeed, by perturbing slightly $x_n$ we see that\linebreak $\C_i^{\circ} \cap \rho(\gamma_n)^{-1} \cdot \Conv{\C_0^{\circ} \cup \C_1^{\circ}} \neq \varnothing$, hence
$$\gamma \cdot \Conv{\C_0^{\circ} \cup \C_1^{\circ}} \cap \rho(\gamma_n)^{-1} \cdot \Conv{\C_0^{\circ} \cup \C_1^{\circ}} \neq \varnothing$$
for all $\gamma \in \Gamma'_i$.
By Remark~\ref{rem:adj-convex-hopeful}, for any $g \in \Gamma'_0 *_{\Delta} \Gamma'_1$, we have $\Conv{\C_0^{\circ} \cup \C_1^{\circ}} \cap \rho(g) \cdot \Conv{\C_0^{\circ} \cup \C_1^{\circ}} \neq \varnothing$ if and only if $g$ belongs to $\Gamma'_0$ or to~$\Gamma'_1$.
Therefore, the fact that $\gamma \cdot \Conv{\C_0^{\circ} \cup \C_1^{\circ}} \cap \rho(\gamma_n)^{-1} \cdot \Conv{\C_0^{\circ} \cup \C_1^{\circ}} \neq \varnothing$ for all $\gamma \in \Gamma'_i$ implies that $\gamma_n$ belongs to~$\Gamma'_i$.

Thus for any~$n$ we have found a point $y_n = \gamma_n \cdot x_n \in \partialn\C_i \cap \Conv{\C_0^{\circ} \cup \C_1^{\circ}}$ such that $d_{\Omega}(y_n,\Ccore_{\Omega_i}(\Gamma_i)) \leq 1/n$.
Since the action of $\Delta$ on $\overline{\C_i} \smallsetminus \Lambdao_{\Omega_{01}}(\Delta)$ is cocompact (Lemma~\ref{lem:ias2-new}), up to replacing each $y_n$ by $\delta_n \cdot y_n$ for some $\delta_n \in \Delta$, and up to passing to a subsequence, we may assume that $(y_n)_{n\in\NN}$ converges to some point $z \in \overline{\C_i} \smallsetminus \Lambdao_{\Omega_{01}}(\Delta)$.
This point $z$ cannot belong to~$\C_i$ because $d_{\Omega}(y_n,\Ccore_{\Omega_i}(\Gamma_i)) \to 0$ while $\C_i$ contains a neighborhood of $\Ccore_{\Omega_i}(\Gamma_i)$.
Thus $z \in \partiali \C_i \smallsetminus \Lambdao_{\Omega_{01}}(\Delta)$.

By Proposition~\ref{prop:hopeful-putty}.\eqref{item:putty-3}, we have $y_n \in \C_{1-i}^\circ \cup \mathscr P$, where $\mathscr P$ is the ``putty'' of \eqref{eqn:putty}.
Therefore, up to passing to a subsequence, we may assume that $y_n \in \C_{1-i}^{\circ}$ for all~$n$, or $y_n \in \mathscr P$ for all~$n$.
It is not possible that $y_n \in \C_{1-i}^{\circ}$ for all~$n$, because then $z$ would belong to the set $\overline{\C_{1-i}} \cap ( \partiali \C_i \smallsetminus \Lambdao_{\Omega_{01}}(\Delta)) \subset  (\overline{\C_{1-i}} \smallsetminus \Lambdao_{\Omega_{01}}(\Delta)) \cap \partial \Omega_i$, which is empty by assumption~\ref{item:hopeful-lim-set}. 
Therefore we may assume that $y_n \in \mathscr P$ for all~$n$.
Since the action of $\Delta$ on $\mathscr P$ is cocompact (Proposition~\ref{prop:hopeful-putty}.\eqref{item:putty-2}), up to replacing each $y_n$ by $\delta_n \cdot y_n$ for some $\delta_n \in \Delta$, and up to passing to a subsequence, we may assume that $z$ belongs to~$\mathscr P$: impossible since $\partiali \C_i \cap \mathscr P = \varnothing$ by Lemma~\ref{lem:ias3}.
\end{proof}

\begin{corollary} \label{cor:Lambda-orb-Gamma-i}
For both $i=0,1$, we have $\Lambdao_{\Omega}(\Gamma'_i) \subset \partiali\Ccore_{\Omega_i}(\Gamma_i)$.
\end{corollary}

\begin{proof}
Each point $\xi \in \Lambdao_{\Omega}(\Gamma'_i)$ is contained in some open face $F$ of $\partial \Omega$.
By Fact~\ref{fact:unif-neighb-face}.\eqref{item:distance-goes-down}, the set $F \cap \Lambdao_{\Omega_i}(\Gamma'_i)$ is nonempty, and by assumption~\ref{item:hopeful-cases} it is an open face of $\partial \Omega_i$.
Since $\Ccore_{\Omega_i}(\Gamma_i) \subset \C_i \subset \Omega_i$, we deduce $\partiali \C_i \cap F = \partiali \Ccore_{\Omega_i}(\Gamma_i) \cap F$.
On the other hand, by Lemma~\ref{lem:Ci-unif-neighb-Ccore-Omega}, the set $\C_i$ contains a uniform neighborhood of $\Ccore_{\Omega_i}(\Gamma_i)$ in $(\Omega,d_{\Omega})$, and so Fact~\ref{fact:unif-neighb-face}.\eqref{item:R-nbhd-faces} implies that $\partiali \C_i \cap F$ contains a uniform neighborhood of $\partiali \Ccore_{\Omega_i}(\Gamma_i) \cap F$ in $(F,d_F)$.
Thus the $R$-neighborhood of $\partiali \Ccore_{\Omega_i}(\Gamma_i) \cap F$ in $F$ is equal to itself, which is only possible if $F = \partiali \Ccore_{\Omega_i}(\Gamma_i) \cap F$. Hence $\xi \in \partiali \Ccore_{\Omega_i}(\Gamma_i)$. 
\end{proof}

\subsection{Step~5: use separability to get into a small neighborhood of $\partiali\Ccore_{\Omega_0}(\Gamma_0) \cup \partiali\Ccore_{\Omega_1}(\Gamma_1)$}

Let us fix an open neighborhood $\mathcal{U}$ of $\partiali\Ccore_{\Omega_0}(\Gamma_0) \cup \partiali\Ccore_{\Omega_1}(\Gamma_1)$ as in the ``moreover'' statement of Theorem~\ref{thm:hopeful}.
By Lemma~\ref{lem:ias1}, the group $\Delta$ acts cocompactly on $\overline{\Ccore_{\Omega_i}(\Gamma_i)} \smallsetminus \Lambdao_{\Omega_{01}}(\Delta)$, hence also cocompactly on $\partiali\Ccore_{\Omega_i}(\Gamma_i) \smallsetminus \Lambdao_{\Omega_{01}}(\Delta)$.
Let $\mathcal{U}_i \subset \Omega_{1-i}$ be an open neighborhood of a compact fundamental domain for the action of $\Delta$ on $\partiali\Ccore_{\Omega_i}(\Gamma_i) \smallsetminus \Lambdao_{\Omega_{01}}(\Delta)$, such that $\overline{\mathcal{U}_i} \subset \Omega_{1-i}$.
Note that the accumulation set of the orbit $\Delta \cdot \overline{\mathcal{U}_i}$ is contained in $\Lambdao_{\Omega_{1-i}}(\Delta)$, which is equal to $\Lambdao_{\Omega_{01}}(\Delta)$ by Fact~\ref{fact:cc-subsets}.
Hence by taking $\mathcal{U}_i$ sufficiently small for each $i \in \{0,1\}$, we may arrange that the union of orbits $\Delta \cdot \overline{\mathcal{U}_0} \cup \Delta \cdot \overline{\mathcal{U}_{1}}$ is contained in $\mathcal{U}$.

\begin{lemma} \label{lem:Si}
For each $i \in \{0,1\}$ and each compact subset $\mathcal{K}$ of~$\Omega_i$, there exists a finite-index subgroup $\Gamma_i''$ of~$\Gamma_i'$, containing~$\Delta$, such that $\gamma \cdot \mathcal{K} \subset \Delta\cdot \mathcal{U}_i$ for all $\gamma \in \Gamma_i'' \smallsetminus \Delta$.
\end{lemma}

\begin{proof}
The compact set $\mathcal{K}$ is contained in some closed uniform neighborhood $\C_i^R$ of $\Ccore_{\Omega_i}(\Gamma_i)$ in $(\Omega_i,d_{\Omega_i})$.
We have $\partiali \C_i^R = \partiali\Ccore_{\Omega_i}(\Gamma_i)$ by assumption~\ref{item:hopeful-cases} and Lemma~\ref{lem:partiali-Ccore-union-of-faces}.
As in the proof of Lemma~\ref{lem:ias2-new}, the action of $\Delta$ on $\overline{\C_i^R} \smallsetminus \Lambdao_{\Omega_{01}}(\Delta)$ is properly discontinuous and cocompact.
Let $\mathcal{D}$ be a compact fundamental domain for this action.
The set $\mathcal{D} \smallsetminus (\Delta\cdot \mathcal{U}_i)$ is compact.
It is also contained in $\Omega_i$, since $\mathcal{D} \cap \partial \Omega_i$ is contained in $\overline{\C_i^R} \cap \partial \Omega_i \smallsetminus \Lambdao_{\Omega_{01}}(\Delta) = \partiali \C_i^R \smallsetminus \Lambdao_{\Omega_{01}}(\Delta)$, which is equal to $\partiali \Ccore_{\Omega_i}(\Gamma_i) \smallsetminus \Lambdao_{\Omega_{01}}(\Delta)$, hence contained in $\Delta \cdot \mathcal{U}_i$. 

Let $\gamma \in \Gamma_i$ satisfy $\gamma \cdot \mathcal{K} \not\subset \Delta \cdot \mathcal{U}_i$: there exists $p \in \mathcal{K}$ such that $\gamma \cdot p \notin \Delta \cdot \mathcal{U}_i$.
Let $\delta \in \Delta$ be such that $\delta \gamma \cdot p \in \mathcal{D}$.
We have $\delta \gamma \cdot p \notin \Delta\cdot \mathcal{U}_i$, hence $\delta \gamma$ belongs to the set 
$$F := \{ \gamma' \in \Gamma_i \,|\, (\mathcal{K} \cup (\mathcal{D} \smallsetminus (\Delta\cdot \mathcal{U}_i))) \cap \gamma' \cdot (\mathcal{K} \cup  (\mathcal{D} \smallsetminus (\Delta\cdot \mathcal{U}_i))) \neq \varnothing\},$$ 
which is finite since $\mathcal{K} \cup  (\mathcal{D} \smallsetminus (\Delta\cdot \mathcal{U}_i))$ is a compact subset of $\Omega_i$ and the action of $\Gamma_i$ on $\Omega_i$ is properly discontinuous.

By our separability assumption~\ref{item:hopeful-separable}, there exists a finite-index subgroup $\Gamma_i''$ of~$\Gamma_i'$, containing $\Delta$, such that $\Gamma_i'' \cap F = \varnothing$.
It follows that for $\gamma \in \Gamma_i'' \smallsetminus \Delta$ we have $\gamma \cdot \mathcal{K} \subset \Delta \cdot \mathcal{U}_i$. 
\end{proof}

\subsection{Step~6: use normal forms} \label{subsec:normal-forms}

Let us review some basic group theory for amalgamated free products (see \eg \cite[\S\,VII.A]{mas88}).
Any element $\gamma \in \Gamma_0' *_\Delta \Gamma_1'$ is either in $\Delta$ or otherwise may be written as a product
\begin{align}\label{eqn:normal-form}
\gamma &= a_m \ldots a_1
\end{align}
of $m \geq 1$ elements where each $a_k$ belongs to $\Gamma_0' \smallsetminus \Delta$ or $\Gamma_1' \smallsetminus \Delta$, and $a_k \in \Gamma_i'$ if and only if $a_{k+1} \in \Gamma_{1-i}'$.
This decomposition is not unique; however, any two such decompositions differ by a sequence of basic moves of the following type: replace $a_k$ with $a_k \delta$ and $a_{k-1}$ with $\delta^{-1} a_{k-1}$ for some $k \in \{2,\ldots, m\}$ and $\delta \in \Delta$. In particular, the length $m = m(\gamma)$ of the decomposition of $\gamma$ is well defined.
We refer to the decomposition~\eqref{eqn:normal-form} as a \emph{normal form} for~$\gamma$.

Now, fix a point $p \in \Omega_{01}$.
For each $i \in \{0,1\}$, let $\mathcal{V}_i$ be a convex open neighborhood of $\overline{\Conv{\mathcal{U}_i}}$ in~$\Omega_{1-i}$, such that $\overline{\mathcal{V}_i} \subset \Omega_{1-i}$.
By Lemma~\ref{lem:Si}, up to replacing $\Gamma_i'$ by a finite-index subgroup containing~$\Delta$, we may assume that
\begin{equation} \label{eqn:V->U}
\gamma \cdot (\overline{\mathcal{V}_{1-i}} \cup \{p\}) \subset \Delta\cdot \mathcal{U}_i \quad\quad \text{for all }\gamma \in \Gamma_i' \smallsetminus \Delta.
\end{equation}
A simple ping-pong argument then yields the following.

\begin{lemma} \label{lem:ping-pong}
Let $\gamma = a_m\ldots a_1 \in (\Gamma_0' *_\Delta \Gamma_1') \smallsetminus \Delta$ be in normal form as in \eqref{eqn:normal-form}, where each $a_k$ belongs to $\Gamma'_{\epsilon(k)} \smallsetminus \Delta$ for some $\epsilon(k) \in \{0,1\}$, and $\epsilon(k+1) = 1-\epsilon(k)$ for all $1\leq k<m$.
Then $\gamma \Delta \cdot (\overline{\mathcal{V}_{1-\epsilon(1)}} \cup \{p\}) \subset \Delta \cdot \mathcal{U}_{\epsilon(m)} \subset \mathcal{U}$.
\end{lemma}

\begin{proof}
Let us check by induction that $a_k \ldots a_1 \Delta \cdot (\overline{\mathcal{V}_{1-\epsilon(1)}} \cup \{p\}) \subset \Delta \cdot \mathcal{U}_{\epsilon(k)}$ for all $1\leq k\leq m$.
The case $k=1$ follows immediately from \eqref{eqn:V->U} applied to $\gamma \delta_1$ instead of $\gamma$, where $\delta_1 \in \Delta$.
Suppose that $a_{k-1} \ldots a_1 \Delta \cdot (\overline{\mathcal{V}_{\epsilon(1)}} \cup \{p\}) \subset \Delta \cdot \mathcal{U}_{\epsilon(k-1)}$, and let us check that $a_k \ldots a_1 \Delta \cdot (\overline{\mathcal{V}_{\epsilon(1)}} \cup \{p\}) \subset \Delta \cdot \mathcal{U}_{\epsilon(k)}$.
Consider an element $\delta \in \Delta$ and a point $x \in \overline{\mathcal{V}_{\epsilon(1)}} \cup \{p\}$.
By assumption, $a_{k-1} \ldots a_1 \delta \cdot x = \delta' \cdot x'$ for some $\delta' \in \Delta$ and some $x' \in \mathcal{U}_{\epsilon(k-1)} = \mathcal{U}_{1-\epsilon(k)}$.
Then $a_k \ldots a_1 \delta \cdot x = a_k \delta' \cdot x' \in \Delta \cdot \mathcal{U}_{\epsilon(k)}$ by \eqref{eqn:V->U} applied to $a_k \delta'$ instead~of~$\gamma$.
\end{proof}

Recall from Section~\ref{subsec:P1-div} that a sequence $(g_n)_{n\in\NN}$ of elements of $\PGL(V)$ is \emph{$P_1$-divergent} if the ratio of the first and second singular values of (a lift to $\GL(V)$ of) $g_n$ tends to $+\infty$ as $n\to +\infty$.

\begin{lemma}[$P_1$-divergence]\label{lem:P1-divergence}
Consider a sequence $(\gamma_n)_{n\in\NN}$ of elements of $\Gamma_0' *_\Delta \Gamma_1' \smallsetminus \Delta$, such that the length $m(n) := m(\gamma_n)$ of the normal form of $\gamma_n$ tends to infinity.
Then $(\gamma_n)$ is $P_1$-divergent.
\end{lemma}

\begin{proof}
For $n\in\NN$, write $\gamma_n = a_{m(n)}^{(n)}\ldots a_1^{(n)}$ in normal form as in \eqref{eqn:normal-form}, where $a_k^{(n)}$ belongs to $\Gamma_0' \smallsetminus \Delta$ or $\Gamma_1' \smallsetminus \Delta$, and $a_k^{(n)} \in \Gamma_i'$ if and only if $a_{k+1}^{(n)} \in \Gamma_{1-i}'$ for all $k \in \{1, \ldots, m(n)\}$.

We first assume that there exists $i\in\{0,1\}$ such that $a_1^{(n)} \in \Gamma_i'$ for all~$n$ and that $m(n)$ is even for all~$n$, so that the last term $a_{m(n)}^{(n)}$ always lies in~$\Gamma_{1-i}'$.
Let $\epsilon(k)$ be defined to be $i$ if $k$ is odd, and to be $1-i$ if $k$ is even.
We follow the ping-pong argument from the proof of Lemma~\ref{lem:ping-pong}: since by construction $\gamma \cdot \overline{\mathcal{V}_{1-i}} \subset \Delta\cdot \mathcal{U}_i$ for all $\gamma \in \Gamma_i' \smallsetminus \Delta$, we may adjust the normal form of $\gamma_n$ by inserting an expression of the form $\delta \delta^{-1}$, for $\delta \in \Delta$, between consecutive terms, and then regrouping so that the $k$-th term $a^{(n)}_k$ now maps $\mathcal{V}_{\epsilon(k)}$ into $\mathcal{U}_{\epsilon(k+1)} = \mathcal{U}_{1-\epsilon(k)} \subset \Conv{\mathcal{U}_{1-\epsilon(k)}}$ for all $1 \leq k \leq m(n) - 1$.
By induction, the image of $\mathcal{V}_i$ under $\gamma_n$ is contained in $m(n)-1$ nested convex sets:
\begin{align*}
a_n^{(m(n))}\ldots a_n^{(1)} \cdot \mathcal{V}_i & \subset a_n^{(m(n))} \ldots a_n^{(2)} \cdot \Conv{\mathcal{U}_{1-i}} \subset a_n^{(m(n))} \ldots a_n^{(2)} \cdot \mathcal{V}_{1-i} \\
& \subset \dots \subset a_n^{(m(n))} \cdot \Conv{\mathcal{U}_{1-i}} \subset a_n^{(m(n))} \cdot \mathcal{V}_{1-i}.
\end{align*}
The right-hand set is contained in $\Delta\cdot \mathcal{U}_i$, hence we may apply some element $\delta_n \in \Delta$ so that
\begin{align*}
\delta_n a_n^{(m(n))}\ldots a_n^{(1)} \cdot \mathcal{V}_i & \subset \delta_n a_n^{(m(n))} \ldots a_n^{(2)} \cdot \Conv{\mathcal{U}_{1-i}} \subset \delta_n a_n^{(m(n))} \ldots a_n^{(2)} \cdot \mathcal{V}_{1-i} \\
& \subset \dots \subset \delta_n a_n^{(m(n))} \cdot \Conv{\mathcal{U}_{1-i}} \subset \delta_n a_n^{(m(n))} \cdot \mathcal{V}_{1-i} \subset \Conv{\mathcal{U}_i} .
\end{align*}
Note that for every other consecutive pair of nested convex sets in the sequence above, the Hilbert diameter of the inner one with respect to the outer one is bounded below by\linebreak $\min(\mathrm{diam}_{\mathcal{V}_0}\Conv{\mathcal{U}_0}, \mathrm{diam}_{\mathcal{V}_1}\Conv{\mathcal{U}_1})$, where $\mathrm{diam}_{\mathcal{V}_i}\Conv{\mathcal{U}_i}$ denotes the diameter of $\Conv{\mathcal{U}_i}$ in $(\mathcal{V}_i,d_{\mathcal{V}_i})$. It follows that the Hilbert diameter of $\gamma_n \cdot \Conv{\mathcal{U}_i}$ in $\mathcal{V}_{\epsilon(m(n))} = \mathcal{V}_{1-i}$ is exponentially small in $m(n)$.
Since $m(n) \to +\infty$ by assumption, we obtain that the diameter of $\delta_n \gamma_n \cdot \mathcal{V}_i$ goes to zero with respect to the Hilbert metric for $\Conv{\mathcal{U}_i}$, and hence for the Hilbert metric for the larger set $\Omega_{1-i}$ as well (see Remark~\ref{rem:Hilb-metric-include}).
Hence the diameter of $\gamma_n \cdot \mathcal{V}_i$ with respect to $d_{\Omega_{1-i}}$ goes to zero and it follows that, after passing to a subsequence, $\gamma_n \cdot \mathcal{V}_i$ converges to a point.
By Fact~\ref{fact:P1-div}, the sequence $(\gamma_n)_{n\in\NN}$ is $P_1$-divergent.

Similarly, if there exists $i\in\{0,1\}$ such that $a_1^{(n)} \in \Gamma_i'$ for all~$n$ and that $m(n)$ is odd for all~$n$, then $(\gamma_n)_{n\in\NN}$ is $P_1$-divergent.

In general, we consider four subsequences of $(\gamma_n)_{n\in\NN}$, determined by the parity of $m(n)$ and by the integer $i\in\{0,1\}$ such that $a_1^{(n)} \in \Gamma_i' \smallsetminus \Delta$; the reasoning above shows that each of the four subsequences is $P_1$-divergent, hence $(\gamma_n)_{n\in\NN}$ is $P_1$-divergent too.
\end{proof}

\subsection{Proof of conclusion \eqref{item:hopeful-4} of Theorem~\ref{thm:hopeful}} \label{subsec:proof-hopeful-2}

By conclusion~\eqref{item:hopeful-1} proved in Section~\ref{subsec:proof-hopeful-1-3} above, the group $\Gamma'_0 *_{\Delta} \Gamma'_1$ preserves a nonempty properly convex open subset $\Omega$ of $\PP(V)$.
Consider a point $p' \in \Omega$ and a sequence $(\gamma_n)_{n\in\NN}$ of pairwise distinct elements of $\Gamma_0' *_\Delta \Gamma_1'$ such that $(\gamma_n \cdot p')_{n\in\NN}$ converges to some $p'' \in \partial\Omega$, and let us prove that $p'' \in \mathcal{U}$.

Up to passing to a subsequence, we may assume that either $\gamma_n \in \Delta$ for all $n\in\NN$ or otherwise $\gamma_n \notin \Delta$ for all $n\in\NN$.
If $\gamma_n \in \Delta$ for all $n$, then $p'' = \lim_n \gamma_n\cdot p'$ belongs to $\Lambda_\Omega(\Delta)$, which is equal to $\Lambda_{\Omega_{01}}(\Delta)$ by Fact~\ref{fact:cc-subsets}, hence contained in~$\mathcal{U}$ by assumption.
So we now assume that $\gamma_n \notin \Delta$ for all~$n$.
There are two cases to consider.

\smallskip

\emph{First case:} The length $m(\gamma_n)$ of the normal form $\gamma_n = a_{m(\gamma_n)}^{(n)} \dots a_1^{(n)}$ does not remain bounded.
Up to passing to a subsequence, we may assume that $m(\gamma_n) \to +\infty$.
Up to passing to a further subsequence, we may assume that the sequence $(\gamma_n\cdot p)_{n\in\NN}$ converges to some $p_{\infty} \in \partial\Omega$, where $p \in \Omega_{01}$ is the point that we have fixed 
before Lemma~\ref{lem:ping-pong}.
We have $p_{\infty} \in \Delta\cdot\overline{\mathcal{U}_0} \cup \Delta\cdot\overline{\mathcal{U}_1} \subset \mathcal{U}$ by Lemma~\ref{lem:ping-pong}.
By Lemma~\ref{lem:P1-divergence}, the sequence $(\gamma_n)_{n\in\NN}$ is $P_1$-divergent.
By Fact~\ref{fact:P1-div}, the point $p'' = \lim_n \gamma_n\cdot p'$ is equal to~$p_{\infty}$, hence belongs to~$\mathcal{U}$.

\smallskip

\emph{Second case:} The length $m(\gamma_n)$ of the normal form $\gamma_n = a_{m(\gamma_n)}^{(n)} \dots a_1^{(n)}$ remains bounded.
Up to passing to a subsequence, we may assume that $m(\gamma_n) = m$ is constant and that there exists $i\in\{0,1\}$ such that $a_1^{(n)} \in \Gamma_i' \smallsetminus \Delta$ for all~$n$.
For $1\leq k\leq m$, let $\epsilon(k)$ be equal to $i$ if $k$ is odd, and to $1-i$ if $k$ is even, so that $a_k^{(n)} \in \Gamma'_{\epsilon(k)} \smallsetminus \Delta$ for all~$n$.
Consider the sequence $(a_1^{(n)})_n$.
If it is bounded in $\Gamma_{\epsilon(1)}' \smallsetminus \Delta$, then pass to a subsequence so that $a_1^{(n)} = a_1$ is constant.
More generally, if the sequence of cosets $\Delta a_1^{(n)}$ is bounded, then replace each $a_1^{(n)}$ by some $\delta_1^{(n)} a_1^{(n)}$, for $\delta_1^{(n)} \in \Delta$, so that this new sequence remains bounded, then pass to a subsequence so that $a_1^{(n)} = a_1$ is constant.
To compensate, replace $a_2^{(n)}$ with $a_2^{(n)} (\delta_1^{(n)})^{-1}$. Now move to the second term of the normal form and repeat this process.
We may continue inductively term by term of the normal form in the same way until one of two possibilities is reached.

\smallskip\noindent
$\bullet$ Possibility one: $a_1^{(n)},\ldots, a_{m-1}^{(n)}$ are all constant and $a_m^{(n)} = \delta_m^{(n)} a_m$ for some $a_m \in \Gamma_{\epsilon(m)}' \smallsetminus \Delta$, with $\delta_{m}^{(n)}$ going to infinity in $\Delta$.
In this case, the limit $p''$ of $\gamma_n \cdot p' = \delta_m^{(n)}\cdot(a_m \dots a_1 \cdot p')$ belongs to $\Lambdao_{\Omega}(\Delta)$, which is equal to $\Lambdao_{\Omega_{01}}(\Delta)$ by Fact~\ref{fact:cc-subsets}, hence contained in~$\mathcal{U}$ by assumption.

\smallskip\noindent
$\bullet$ Possibility two: there exists $1\leq k_0 \leq m$ such that $a_1^{(n)} = a_1, \ldots, a_{k_0-1}^{(n)} = a_{k_0-1}$ are all constant, and the coset $\Delta a_{k_0}^{(n)}$ is unbounded.
After passing to a subsequence, $\Delta a_{k_0}^{(n)}$ goes to infinity in $\Delta \backslash \Gamma_{\epsilon(k_0)}'$, and $a_{k_0}^{(n)}\cdot (a_{k_0-1} \dots a_1 \cdot p')$ converges to a point $\xi \in \Lambdao_{\Omega}(\Gamma'_{\epsilon(k_0)})$.
By Lemma~\ref{lem:even-more-pure}, up to replacing each $a_{k_0}^{(n)}$ by $\delta_{k_0}^{(n)} a_{k_0}^{(n)}$ for some $\delta_{k_0}^{(n)} \in \Delta$ (making the compensating adjustment in the $(k_0+1)$-th term), we may assume that $\xi \notin \Lambdao_{\Omega}(\Delta) = \Lambdao_{\Omega_{01}}(\Delta)$.
Thus $\xi \in \partiali\Ccore_{\Omega_{\epsilon(k_0)}}(\Gamma_{\epsilon(k_0)}) \smallsetminus \Lambdao_{\Omega_{01}}(\Delta)$ by Corollary~\ref{cor:Lambda-orb-Gamma-i}, and so for all sufficiently large~$n$ the point $q_{k_0}^n := a_{k_0}^n \cdot (a_{k_0-1}\dots a_1 \cdot p')$ lies in $\Delta \cdot \mathcal{U}_{\epsilon(k_0)}$.
By Lemma~\ref{lem:ping-pong}, for all sufficiently large~$n$, the point $\gamma_n \cdot p' = a_m^{(n)} \dots a_{k_0+1}^{(n)} \cdot q_{k_0}^n$ lies in $\Delta \cdot \mathcal{U}_{\epsilon(m)}$, and so $p'' = \lim_n \gamma_n \cdot p'$ lies in $\Delta \cdot \overline{\mathcal{U}_{\epsilon(m)}} \subset \mathcal{U}$.

This completes the proof of conclusion \eqref{item:hopeful-4} of Theorem~\ref{thm:hopeful}.

\subsection{Proof of conclusion~\eqref{item:hopeful-2} of Theorem~\ref{thm:hopeful}}
 
We now check the $P_1$-divergence of $\Gamma_0' *_\Delta \Gamma_1'$, assuming $\Gamma_0, \Gamma_1$ (and hence $\Gamma'_0,\Gamma'_1$) are $P_1$-divergent.
Following the proof in the previous Section~\ref{subsec:proof-hopeful-2}, we consider a sequence $(\gamma_n)$ of pairwise distinct elements and consider the two possible cases. In the first case, for which the length of the normal form goes to infinity, we have $P_1$-divergence by Lemma~\ref{lem:P1-divergence}. Examining the second case, for which the length of normal form is constant, we obtain $P_1$-divergence of $(\gamma_n)$ from the ping-pong argument of Lemma~\ref{lem:ping-pong} and $P_1$-divergence of the leftmost divergent letter in the normal form.

\section{Applications of the virtual amalgamation theorem} \label{sec:virt-amalg-applic}

In this section we derive a consequence of Theorem~\ref{thm:hopeful} for discrete groups $\Gamma_i$ which are not necessarily Gromov hyperbolic (Proposition~\ref{prop:amalgam-cc-general}).
We deduce Proposition~\ref{prop:amalgam-word-hyperbolic}, and apply it to Hitchin representations into $\SL(3,\RR)$ (Example~\ref{ex:Hitchin}).

\subsection{A consequence of Theorem~\ref{thm:hopeful}}

Given a properly convex open subset $\mathcal{O}$ of $\PP(V)$ which is invariant under an infinite discrete subgroup $\Gamma$ of $\PGL(V)$, we denote by $\mathcal{O}^{\max}$ the connected component of 
$$\PP(V) \smallsetminus \bigcup_{H' \in \Lambdao_{\mathcal{O}^*}(\Gamma)} H'$$
that contains $\mathcal{O}$; it is a convex (but not necessarily properly convex) open subset of $\PP(V)$.
The following is a consequence of Theorem~\ref{thm:hopeful} and Propositions \ref{prop:thicken-add-dim}--\ref{prop:thicken-not-proper}.

\begin{proposition} \label{prop:amalgam-cc-general}
Suppose $V$ is the direct sum of three nonzero linear subspaces $V_0, W, V_1$.
For $i\in\{0,1\}$, let $\Gamma_i$ be an infinite discrete subgroup of $\GL(V)$.
Suppose that
\begin{itemize}
  \item for each $i\in\{0,1\}$, the group $\Gamma_i$ acts trivially on~$V_{1-i}$ and acts convex cocompactly on some properly convex open subset $\mathcal{O}_i$ of $\PP(V_i\oplus W)$;
  \item the set $\mathcal{O} := \mathcal{O}_0 \cap \mathcal{O}_1$ is a nonempty properly convex open subset of $\PP(W)$ and $\Delta := \Gamma_0\cap\Gamma_1$ is finite or acts convex cocompactly on~$\mathcal O$;
  \item $\Delta$ is separable in both $\Gamma_0$ and~$\Gamma_1$,
   \item for each $i\in\{0,1\}$, we have $\overline{\Ccore_{\mathcal{O}_i}(\Gamma_i)} \cap \PP(V_i) = \varnothing$ and the image of $\overline{\Ccore_{\mathcal{O}_i}(\Gamma_i)} \smallsetminus \Lambdao_{\mathcal{O}}(\Delta)$ under the projection $\PP(V_i \oplus W) \smallsetminus \PP(V_i) \to \PP(W)$ lies inside the $\Gamma_{1-i}$-invariant convex set $\mathcal{O}_{1-i}^{\max} \subset \PP(V_{1-i} \oplus W)$.
\end{itemize}
Then there exist finite-index subgroups $\Gamma'_0$ of~$\Gamma_0$ and $\Gamma'_1$ of~$\Gamma_1$, each containing~$\Delta$, such that the representation $\Gamma'_0 *_\Delta \Gamma'_1 \to \PGL(V)$ induced by the inclusions $\Gamma'_0, \Gamma'_1 \hookrightarrow \PGL(V)$ is discrete and faithful, and its image is convex cocompact in $\PP(V)$.
\end{proposition}

\begin{remark} \label{rem:O-1-i-max}
By Lemma~\ref{lem:Omega-max}, any $\Gamma_{1-i}$-invariant properly convex open set of $\PP(V_{1-i} \oplus W)$ is contained in $\mathcal{O}_{1-i}^{\max}$, hence the fourth hypothesis in Proposition~\ref{prop:amalgam-cc-general} may be replaced with the stronger (but in fact equivalent --- see the proof below) condition that the projection of $\overline{\Ccore_{\mathcal{O}_i}(\Gamma_i)} \smallsetminus \Lambdao_{\mathcal{O}}(\Delta)$ lie inside some $\Gamma_{1-i}$-invariant properly convex open subset of $\PP(V_{1-i} \oplus W)$.
\end{remark}

\begin{proof}[Proof of Proposition~\ref{prop:amalgam-cc-general}]
We first observe that it is sufficient to find, for each $i \in \{0,1\}$, a properly convex open subset $\Omega_i$ of $\PP(V)$ containing~$\mathcal{O}_i$ such that $\overline{\Ccore_{\mathcal{O}_{1-i}}(\Gamma_{1-i})} \smallsetminus \Lambdao_{\mathcal{O}}(\Delta) \subset \Omega_i$.
Indeed, Proposition~\ref{prop:thicken-add-dim} then implies that $\Ccore_{\Omega_{1-i}}(\Gamma_{1-i}) = \Ccore_{\mathcal{O}_{1-i}}(\Gamma_{1-i})$, that $\Ccore_{\Omega_0\cap\Omega_1}(\Delta) = \Ccore_{\mathcal{O}}(\Delta)$, and that the actions of $\Gamma_i$ on~$\Omega_i$ and of $\Delta$ on $\Omega_0 \cap \Omega_1$ (if $\Delta$ is infinite) are convex cocompact.
Using also Lemma~\ref{lem:hopeful-cases-satisfied}, we see that the assumptions \ref{item:hopeful-separable}, \ref{item:hopeful-CC-Omega01}, \ref{item:hopeful-lim-set}, \ref{item:hopeful-cases} of Theorem~\ref{thm:hopeful} are satisfied, and we can conclude by applying Theorem~\ref{thm:hopeful}.

Fix $i\in\{0,1\}$, and let us construct a properly convex open subset $\Omega_i$ of $\PP(V)$ containing~$\mathcal{O}_i$ such that $\overline{\Ccore_{\mathcal{O}_{1-i}}(\Gamma_{1-i})} \smallsetminus \Ccore_{\mathcal{O}}(\Delta) \subset \Omega_i$.
The set $\Omega_i^{\max} := \mathcal{O}_i^{\max} \times V_{1-i}$ (Notation~\ref{notn:fat}) is a $\Gamma_i$-invariant convex (but not properly convex) open subset of $\PP(V)$.
By assumption, we have $\overline{\Ccore_{\mathcal{O}_{1-i}}(\Gamma_{1-i})} \smallsetminus \Lambdao_{\mathcal{O}}(\Delta) \subset \mathcal{O}_i^{\max} \subset \Omega_i^{\max}$.
By Proposition~\ref{prop:thicken-add-dim}, there is a $\Gamma_i$-invariant properly convex open subset $\hat{\Omega}_i$ of $\PP(V)$ such that $\mathcal{O}_i \subset \hat{\Omega}_i \subset \mathcal{O}_i \times V_{1-i} \subset \Omega_i^{\max}$.
By Proposition~\ref{prop:thicken-not-proper}, there is an increasing sequence of $\Gamma_i$-invariant properly convex open subsets $\hat{\Omega}_i = \Omega_i^{(1)} \subset \Omega_i^{(2)} \subset \dots \subset \Omega_i^{(n)} \subset \dots$ of $\PP(V)$ such that the union over all $n$ of the $\Omega_i^{(n)}$ is $\Omega_i^{\max}$.
Our goal is to show that for any large enough~$n$, the set $\Omega_i^{(n)}$ contains $\overline{\Ccore_{\mathcal{O}_{1-i}}(\Gamma_{1-i})} \smallsetminus \Lambdao_{\mathcal{O}}(\Delta)$, so that we can take $\Omega_i = \Omega_i^{(n)}$.
First, this is true if $\Delta$ is finite, since then $\Lambdao_{\mathcal O}(\Delta)$ is empty, so $\overline{\Ccore_{\mathcal{O}_{1-i}}(\Gamma_{1-i})} \smallsetminus \Lambdao_{\mathcal{O}}(\Delta)$ is compact. 
In the case that $\Delta$ is infinite, the action of $\Delta$ on $\mathcal{O}_{1-i}$ is convex cocompact, and hence it follows from Lemma~\ref{lem:even-more-pure}, that $\Delta$ acts cocompactly on  $\overline{\Ccore_{\mathcal{O}_{1-i}}(\Gamma_{1-i})} \smallsetminus \Lambdao_{\mathcal{O}}(\Delta)$. A compact fundamental domain for the action of $\Delta$ on $\overline{\Ccore_{\mathcal{O}_{1-i}}(\Gamma_{1-i})} \smallsetminus \Lambdao_{\mathcal{O}}(\Delta)$ must be contained in $\Omega_i^{(n)}$ for some $n$, and so it follows that the entire set $\overline{\Ccore_{\mathcal{O}_{1-i}}(\Gamma_{1-i})} \smallsetminus \Lambdao_{\mathcal{O}}(\Delta)$ is contained in $\Omega_i^{(n)}$ as well.
This completes the proof.
\end{proof}

\subsection{A consequence of Proposition~\ref{prop:amalgam-cc-general}}

We use Proposition~\ref{prop:amalgam-cc-general} to prove Proposition~\ref{prop:amalgam-word-hyperbolic}.

\begin{proof}[Proof of Proposition~\ref{prop:amalgam-word-hyperbolic}]
For each $i \in \{0,1\}$, let $\mathcal{O}_i$ be a properly convex open subset of $\PP(\RR^{d_i})$ on which $\Gamma_i$ acts convex cocompactly.
We may take $\mathcal{O}_i$ to be strictly convex by \cite[Th.\,1.15]{dgk-proj-cc}. 
Since $\Delta_i$ divides some nonempty properly convex open subset of $\PP(W_i)$ and acts trivially on the complement $V_i$ of $W_i$ in $\RR^{d_i}$, the full orbital limit set $\Lambdao_{\mathcal{O}_i}(\Delta_i)$ is contained in $\PP(W_i)$. 
Since $\mathcal{O}_i$ is strictly convex, it follows that $\mathcal{O}_{W_i} := \mathcal{O}_i \cap \PP(W_i)$ is nonempty and hence, by a standard cohomological dimension argument (see Remark~\ref{rem:vcd}), it is divided by $\Delta_i$.
We may arrange that the linear isomorphism $\varphi: V_0 \to V_1$ takes $\mathcal{O}_0 \cap \PP(V_0)$ to $\mathcal{O}_1 \cap \PP(V_1)$: if $\Delta_0$ acts irreducibly on~$V_0$, this happens automatically, but even if not there is a linear automorphism of $V_i$ taking any properly convex open set divided by $\Delta_i$ to any other (see \cite[Th.\,5]{vey70} and \cite[Prop.\,3.1]{ben00}). 

Now, since $\Delta_i$ divides $\mathcal{O}_{W_i}$ and since $\Delta_i$ acts trivially on $V_i$, it follows that every supporting hyperplane to $\mathcal{O}_i$ at a point of $\partial \mathcal{O}_{W_i} = \partial \mathcal{O}_i \cap \PP(W_i)$ contains $\PP(V_i)$. It then follows from strict convexity that $\partial \mathcal{O}_i \cap \PP(V_i) = \varnothing$, and that the projection from $\PP(\RR^{d_i}) \smallsetminus \PP(V_i)$ to $\PP(W_i)$ maps $\partial \mathcal{O}_i \smallsetminus \PP(W_i)$ into $\mathcal{O}_{W_i}$.
In particular, the set $\Lambdao_{\mathcal{O}_i}(\Gamma_i) \smallsetminus \Lambdao_{\mathcal{O}_i}(\Delta_i)$ maps into $\mathcal{O}_{W_i}$, and so does $\overline{\Ccore_{\mathcal{O}_i}(\Gamma_i)} \smallsetminus  \Lambdao_{\mathcal{O}_i}(\Delta_i)$.

Now, using the isomorphism $\varphi$ to identify $W_0 = W_1 =: W$, defining $V = V_0 \oplus W \oplus V_1$, and extending the action of $\Gamma_i$ to be trivial on $V_{1-i}$, the hypotheses of Proposition~\ref{prop:amalgam-cc-general} are satisfied, and the conclusion gives the desired result.
\end{proof}

\begin{example} \label{ex:Hitchin}
For $i\in\{ 0,1\}$, let $S_i$ be a closed orientable surface of negative Euler characteristic and $\rho_i : \pi_1(S_i) \to \SL(3,\RR)$ a Hitchin representation.
By \cite{cg93}, the group $\rho_i(\pi_1(S_i))$ divides a properly convex open subset $\Omega_i$ of $\PP(\RR^3)$.
Suppose $\alpha_0 \in \pi_1(S_0)$ and $\alpha_1 \in \pi_1(S_1)$ are each primitive elements in their respective groups such that $\rho_1(\alpha_1) = \rho_2(\alpha_2) =: \delta$.
Further assume that the eigenvalues of this element are symmetric, equal to $\lambda, 1, \lambda^{-1}$ for some $\lambda > 1$.
Let $(e^+, e^0, e^-)$ be the corresponding eigenbasis of~$\RR^3$.
Observe that the cyclic group generated by $\delta$ divides an open interval of $\PP(\RR e^+ + \RR e^-)$ and acts trivially on $\RR e_0$.
Note that primitive cyclic subgroups of surface subgroups are separable (see \eg \cite[Lem.\,princ.]{ber00} or \cite[Prop.\,3.2]{tt25}).

Let $\Gamma_i := \rho(\pi_1(S_i))$ and $\Delta_0 = \Delta_1 := \langle \delta \rangle$.
Let $\tau_0, \tau_1 : \SL(3,\RR)\hookrightarrow\SL(4,\RR)$ be defined by $\tau_0(e^+,e^0,e^-) := (e_2,e_1,e_3)$ and $\tau_0(e^+,e^0,e^-) := (e_2,e_4,e_3)$ where $(e_1,e_2,e_3,e_4)$ is the canonical basis of~$\RR^4$.
By Proposition~\ref{prop:amalgam-word-hyperbolic} and its proof, there exist finite-index subgroups $\Gamma'_0$ of~$\Gamma_0$ and $\Gamma'_1$ of~$\Gamma_1$, both containing~$\Delta$, such that the representation $\Gamma'_0 *_\Delta \Gamma'_1 \to \PSL(4,\RR)$ induced by the inclusions $\tau_0 : \Gamma'_0 \hookrightarrow \PSL(4,\RR)$ and $\tau_1 : \Gamma'_1 \hookrightarrow \PSL(4,\RR)$ is faithful with image a discrete subgroup $\Gamma$ of $\PSL(4,\RR)$ which is convex cocompact in $\PP(\RR^4)$. 

The group $\Gamma_0' *_\Delta \Gamma_1'$ is the free product of two surface groups amalgamated over a copy of~$\ZZ$.
In particular, it is not virtually a surface group nor a free group. 
We remark that if $\rho_0$ and $\rho_1$ have image in $\SO(2,1)$, then the resulting subgroup $\Gamma$ of $\PSL(4,\RR)$ lies in a copy of $\SO(3,1)$.
Such examples are already known from the combination theorem by Baker--Cooper~\cite{bc08}.
However, we may choose the Hitchin representations $\rho_0$ and $\rho_1$ so that the groups $\Gamma_0$ and $\Gamma_1$ are far from any copy of $\SO(2,1)$; indeed, taking $\alpha_i$ to be a simple curve on $S_i$, it follows from~\cite{gol90} that we may choose the eigenvalues of $\rho_0(\alpha_0) = \rho_1(\alpha_1)$ to be symmetric as above while also choosing independently the eigenvalues of the curves of complementary pants decompositions on $S_0$ and $S_1$ to be as asymmetric as we wish.
In this case, the combined group $\Gamma$ is far from any copy of $\SO(3,1)$.
Such examples do not come from any small deformation of known Kleinian combination examples in $\SO(3,1)$.
\end{example}

We note that in Example~\ref{ex:Hitchin} the amalgamated free product $\Gamma_0' *_\Delta \Gamma_1'$ is Gromov hyperbolic by the following combination theorem for Gromov hyperbolic groups, due to Bestvina--Feighn (see \cite[Th.\,2.6]{tt23}).

\begin{fact}[\cite{bf92}] \label{fact:amalgam-Gromov-hyp}
For $i\in\{0,1\}$, let $\Gamma_i$ be a Gromov hyperbolic group and $\Delta_i$ be a quasiconvex subgroup which is malnormal in~$\Gamma_i$ (\ie $\gamma\Delta_i\gamma^{-1} \cap \Delta_i = \{ 1\}$ for all $\gamma\in\Gamma_i\smallsetminus\Delta_i$).
Let $\varphi : \Delta_0\overset{\sim}{\longrightarrow}\Delta_1$ be a group isomorphism.
Then the amalgamated free product $\Gamma_0 *_{\varphi} \Gamma_1$ is Gromov hyperbolic.
\end{fact}

\section{Anosov representations} \label{sec:Anosov}

We finish the paper by giving proofs of Theorem~\ref{thm:Anosov-free-product-G/P} and Corollaries \ref{cor:free-prod-Ano}, \ref{cor:free-prod-Ano-general}, and~\ref{cor:amalgam-cyclic-Ano}.

\subsection{Proof of Corollaries \ref{cor:free-prod-Ano}, \ref{cor:free-prod-Ano-general}, and~\ref{cor:amalgam-cyclic-Ano}}

\begin{proof}[Proof of Corollary~\ref{cor:free-prod-Ano}]
By Fact~\ref{fact:Anosov}, each $\Gamma_i$ acts convex cocompactly on some properly convex open subset $\Omega_i$ of $\PP(V)$, but does not divide it by assumption.
By Corollary~\ref{cor:free-prod-cc}, there exists $g\in\PGL(V)$ such that the group generated by $\Gamma_0$ and $g\Gamma_1g^{-1}$ is isomorphic to the free product $\Gamma_0\ast\Gamma_1$ and is convex cocompact in $\PP(V)$.
On the other hand, a free product of two Gromov hyperbolic groups is Gromov hyperbolic (see \eg \cite[Ch.\,1, Ex.\,34]{gh90}).
By Fact~\ref{fact:Anosov}, the representation $\rho:  \Gamma_0 * g\Gamma_1 g^{-1} \to \PGL(V)$ induced by the natural inclusions of $\Gamma_0$ and $g\Gamma_1 g^{-1}$ is $P_1$-Anosov.
\end{proof}

\begin{proof}[Proof of Corollary~\ref{cor:free-prod-Ano-general}]
By \cite[Lem.\,3.2 \& Prop.\,3.3]{ggkw17}, for each $i\in\{0,1\}$ there exist a finite-dimensional real vector space $V_i$ and a representation $\sigma_i : G_i\to\GL(V_i)$ with the following property: for any Gromov hyperbolic group $\Gamma_i$ and any $P_{\theta_i}$-Anosov representation $\rho_i : \Gamma\to G_i$, the composed representation $\sigma_i\circ\rho_i : \Gamma_i\to\GL(V_i)$ is $P_1$-Anosov.

Let $\sigma'_i : \GL(V_i)\to\GL(S^2V_i)$ be the second symmetric power of the standard representation of $\GL(V_i)$ on~$V_i$.
Let $V := S^2V_0 \oplus S^2V_1$, and let $\tau_0 := \sigma'_0\circ\sigma_0 \oplus \mathbf{1}_{G_0} : G_0\to\GL(V)$ and $\tau_1 := \mathbf{1}_{G_1} \oplus \sigma'_1\circ\sigma_1 : G_1\to\GL(V)$ where $\mathbf{1}_{G_i} : G_i\to\GL(S^2V_{1-i})$ is the constant representation.

We claim that $V$, $\tau_0$, and $\tau_1$ satisfy the conclusions of Corollary~\ref{cor:free-prod-Ano-general}.
Indeed, consider Gromov hyperbolic groups $\Gamma_i$ and $P_{\theta_i}$-Anosov representations $\rho_i : \Gamma\to G_i$.
By Fact~\ref{fact:Omega-sym}, the image of $\sigma'_i\circ\sigma_i$ preserves a nonempty properly convex open subset of $\PP(S^2V_i)$.
Therefore, by Fact~\ref{fact:Anosov}, each $\sigma'_i\circ\sigma_i\circ\rho_i(\Gamma_i)$ is convex cocompact in $\PP(S^2V_i)$.
By \cite[Th.\,1.16.(E)]{dgk-proj-cc} (or Proposition~\ref{prop:thicken-add-dim}), each $\tau_i\circ\rho_i(\Gamma_i)$ is convex cocompact in $\PP(V)$.
By Remark~\ref{rem:vcd}, we have $\mathrm{vcd}(\Gamma_i) \leq \dim \PP(S^2V_i) < \dim \PP(V)$, hence $\tau_i\circ\rho_i(\Gamma_i)$ does not divide any properly convex open set in $\PP(V)$.
By Corollary~\ref{cor:free-prod-cc}, there exists $g\in\PGL(V)$ such that the group generated by $\tau_0\circ\rho_0(\Gamma_0)$ and $g(\tau_1\circ\rho_1(\Gamma_1))g^{-1}$ is isomorphic to the free product $\Gamma_0\ast\Gamma_1$ and is convex cocompact in $\PP(V)$.
On the other hand, a free product of two Gromov hyperbolic groups is Gromov hyperbolic (see \eg \cite[Ch.\,1, Ex.\,34]{gh90}).
By Fact~\ref{fact:Anosov}, the representation $\rho:  \Gamma_0 * g\Gamma_1 g^{-1} \to \PGL(V)$ induced by $\tau_0\circ\rho_0$ and $\tau_1\circ\rho_1$ is $P_1$-Anosov.
\end{proof}

\begin{proof}[Proof of Corollary~\ref{cor:amalgam-cyclic-Ano}]
Since the maximal cyclic subgroup $\langle\gamma_i\rangle$ of the Gromov hyperbolic group $\Gamma_i$ is its own centralizer, it is separable in~$\Gamma_i$ (see \eg \cite[Prop.\,3.2]{tt25}).
By Fact~\ref{fact:Anosov}, the group $\Gamma_i$ acts convex cocompactly on some properly convex open subset $\Omega_i$ of $\PP(\RR^{d_i})$.
By Proposition~\ref{prop:amalgam-word-hyperbolic}, there exist, for each $i \in \{0,1\}$, a finite-index subgroup $\Gamma'_i$ of~$\Gamma_i$, containing $\langle\gamma_i\rangle$, and a discrete and faithful representation $\rho : \Gamma'_0 *_{\langle\gamma_0\rangle = \langle\gamma_1\rangle} \Gamma'_1 \to \SL(d_0+d_1-2,\RR)$ whose image is convex cocompact in $\PP(\RR^{d_0+d_1-2})$.
This representation is $P_1$-Anosov by Fact~\ref{fact:Anosov}.
\end{proof}

\subsection{A preparatory lemma on proximal elements}

Before proving Theorem~\ref{thm:Anosov-free-product-G/P} in Section~\ref{sec:appli-anosov}, we establish the following Lie-theoretic lemma.

\begin{lemma} \label{lem:prox-move-transverse}
Let $G$ be a connected real linear semisimple Lie group, $P$ a self-opposite parabolic subgroup of~$G$, and $g$ an element of~$G$.
If $g$ is proximal in $G/P$, then the set $\mathcal{U}_g$ of points $x\in G/P$ transverse to $g\cdot x$ contains a nonempty Zariski-open subset of $G/P$.
\end{lemma}

\begin{proof}
Let $\aaa$ be a Cartan subspace of the Lie algebra $\g$ of~$G$, and let $\Sigma$ be the set of restricted roots of $\aaa$ in~$\g$.
For $\alpha\in\Sigma$, let $\g_\alpha := \{ x\in\g \,|\, \mathrm{ad}(a)x = \alpha(a)x ~~ \forall a\in\aaa \}$ be the root space associated to~$\alpha$.
Let $\Delta \subset \Sigma$ be a choice of simple restricted roots, and $\Sigma^+ = \Sigma \cap (\RR_{\geq 0}\text{-span}(\Delta))$ the corresponding set of positive restricted roots.
There exists a subset $\theta$ of~$\Delta$ such that, up to conjugation, $P = P_{\theta}$ is the closed algebraic subgroup of~$G$ with Lie algebra
$$\g_0 \oplus \bigoplus_{\alpha\in\Sigma^+} \g_{\alpha} \oplus \bigoplus_{\alpha\in\Sigma^+\smallsetminus \Sigma_\theta^+} \g_{-\alpha},$$
where $\Sigma_{\theta}^+ := \Sigma^+ \smallsetminus \mathrm{span}(\Delta \smallsetminus \theta)$.
We also consider the opposite parabolic subgroup $P^-$ to~$P$ with Lie algebra
$$\g_0 \oplus \bigoplus_{\alpha\in\Sigma^+} \g_{-\alpha} \oplus \bigoplus_{\alpha\in\Sigma^+\smallsetminus \Sigma_\theta^+} \g_{\alpha}.$$
Its unipotent radical has Lie algebra
\begin{equation} \label{eqn:Lie-alg-u}
\uu^- = \bigoplus_{\alpha\in\Sigma_{\theta}^+} \g_{-\alpha}.
\end{equation}
We view $P^-$ as an element of $G/P$ via the identification of $G/P$ with the set of $G$-conjugates of~$P$ (since we have assumed $P$ to be self-opposite, $\theta$ is invariant under the opposition involution).

Let $W = N_G(\mathfrak{a})/Z_G(\mathfrak{a})$ be the restricted Weyl group, and let $W_{\theta}$ be the subgroup of~$W$ generated by the reflections in the simple restricted roots in $\Delta\smallsetminus\theta$.
Then the generalized Bruhat decomposition
$$G = \bigsqcup_{[w] \in W_{\theta}\backslash W/W_{\theta}} P w P$$
holds (see \cite[Th.\,5.15 \& Cor.\,5.20]{bt65}).
There is a unique double class $P w P$, or ``big Bruhat cell'', which is a nonempty Zariski-open subset of~$G$; it corresponds to the longest element $w_0$ of~$W$.
The elements of $G/P$ transverse to $P$ are exactly the cosets in $P w_0 P$ (\eg the coset $w_0 P$ corresponding to~$P^-$).

For $g\in G$, let $C_g$ be the conjugacy class of $g$ in~$G$ and $\overline{C_g}$ its Zariski closure in~$G$.
We would like to prove that if $g$ is proximal in $G/P$, then the set $\mathcal{U}_g$ is nonempty, \ie there exists $h\in G$ such that $ghP$ is transverse to $hP$ in $G/P$; equivalently, $h^{-1}ghP$ is transverse to~$P$; equivalently, $C_g \cap P w_0 P \neq \varnothing$.
If this is the case, then $C_g \cap P w_0 P$ is a nonempty Zariski-open subset of~$C_g$, and so $\mathcal{U}_g$ contains a nonempty Zariski-open subset of $G/P$.

Let us now check that $C_g \cap P w_0 P \neq \varnothing$.
Since $P w_0 P$ is open, we only need to check that $\overline{C_g} \cap P w_0 P \neq \varnothing$.
If we write the Jordan decomposition of~$g$ as $g = g_e g_h g_u$ where $g_e,g_h,g_u$ commute and are respectively elliptic, hyperbolic, and unipotent, then $C_{g_e g_h} \subset \overline{C_g}$; moreover, $g$ is proximal in $G/P$ if and only if $g_e g_h$ is.
Therefore, it is sufficient to prove that if $g$ is proximal in $G/P$ and $g = g_e g_h$ (\ie $g$ is semisimple), then $C_g \cap P w_0 P \neq \varnothing$.

Up to replacing $g$ by a conjugate in~$G$ (which does not change $C_g$), we may assume that $g$ has attracting fixed point $P$ and repelling fixed point~$P^-$.
In particular, $g$ belongs to the stabilizer of $P$ and~$P^-$ in~$G$, which is the reductive group $L := P\cap P^-$.
Up to further conjugating in~$L$, we may additionally assume $g_h \in \exp(\aaa)$.
For any $y\in \g$, we then have 
\begin{equation} \label{eq:conjtrans}
\exp(y) g \exp(y)^{-1} \in P w_0 P \iff \exp(y) \exp(-\mathrm{Ad}(g)y) \in P w_0 P.
\end{equation}

We now consider the embedding of $\varphi: \uu^- \to G/P$ given by $\varphi(y) = \exp(y)P$; its image (the set of flags transverse to~$P^-$) is Zariski-open in $G/P$, hence dense.
Let $Y$ be the set of elements $y \in\uu^-$ such that $\varphi(y)$ is nontransverse to~$P$, \ie such that $\exp(y) \notin P w_0 P$: it is a proper closed algebraic subvariety of~$\uu^-$ passing through (and usually singular at) the origin $0\in \uu^-$.

Since $g$ is proximal in $G/P$ with attracting fixed point~$P$ and repelling fixed point~$P^-$, we have $\langle -\alpha,\log(g_h)\rangle < 0$ for all $\alpha\in\theta$ (see \cite[Prop.\,3.3.(c)]{ggkw17}), hence for all $\alpha\in\Sigma_{\theta}^+$.
By \eqref{eqn:Lie-alg-u}, the Lie algebra $\uu^-$ is a sum of eigenspaces of $\mathrm{ad}(\log(g_h))$ for eigenvalues which are all~$<0$.
Therefore, $\uu^-$ is a sum of eigenspaces of $\mathrm{Ad}(g_h)$ for (positive) eigenvalues which are all~$<1$.
Let $\Vert\cdot\Vert$ be a norm on~$\uu^-$ for which the eigenspaces of $\mathrm{Ad}(g_h)$ are pairwise orthogonal.
We may assume that $\Vert\cdot\Vert$ is invariant under $\mathrm{Ad}(g_e)$ since $g_e$ is elliptic and commutes with~$g_h$.
For all nonzero $y\in\uu^-$ we have $\Vert\Ad(g_h)y\Vert < \Vert y\Vert$, and therefore $\Vert\Ad(g)y\Vert = \Vert\Ad(g_e g_h)y\Vert < \Vert y\Vert$.
Hence, $\mathrm{id}-\mathrm{Ad}(g)$ is an invertible linear transformation of~$\uu^-$.
In particular, there exists $y_0 \in \uu^- \smallsetminus \{0\}$ such that $(\mathrm{id}-\Ad(g))y_0$ is \emph{not} tangent to the closed algebraic variety $Y$ at~$0$, in the sense that there exists an open cone $\mathcal{O} \subset \uu^-$ containing the open ray through $(\mathrm{id}-\Ad(g))y_0$, and disjoint from $Y$ in a neighborhood of the origin.
As $t\to 0$, we have
$$\exp(ty_0)\exp(-\mathrm{Ad}(g)ty_0) = \exp\big(t (\mathrm{id}-\Ad(g))y_0 + O(t^2)\big) \in \exp(\mathcal{O}).$$
The right-hand side lies in $P w_0 P$  by definition of $Y$.
By~\eqref{eq:conjtrans}, this shows $\exp(ty_0) g \exp(ty_0)^{-1} \in P w_0 P$, hence\ $C_g \cap P w_0 P \neq\nolinebreak\varnothing$ as desired.
\end{proof}

\subsection{Application to Anosov representations} \label{sec:appli-anosov}

\begin{proposition} \label{prop:Ano-transv}
Let $G$ be a connected real linear semisimple Lie group, $P$ a self-opposite parabolic subgroup of~$G$, and $\Gamma_0$ a torsion-free $P$-Anosov subgroup of~$G$.
Let $\Gamma$ be a Zariski-dense subgroup of~$G$.
Suppose that there is a point $x \in \Lambda_{\Gamma}^{G/P}$ which is transverse to all points of $\Lambda_{\Gamma_0}^{G/P}$.
Then any neighborhood of $x$ in $G/P$ contains a point $y \in \Lambda_{\Gamma}^{G/P}$ which is uniformly transverse to $(\Gamma_0 \smallsetminus \{1\}) \cdot y$.
\end{proposition}

\begin{proof}
Since $\Gamma_0$ is $P$-Anosov and $x$ is transverse to $\Lambda_{\Gamma_0}^{G/P}$, the set of accumulation points of the $\Gamma_0$-orbit of~$x$ is $\Lambda_{\Gamma_0}^{G/P}$ and there is a finite subset $F_0$ of~$\Gamma_0$ such that $x$ is uniformly transverse to $\gamma_0\cdot x$ for all $\gamma_0 \in \Gamma_0 \smallsetminus F_0$.
Since $\Gamma_0$ is $P$-Anosov and torsion-free, every nontrivial element $\gamma_0$ of~$\Gamma_0$ is proximal in $G/P$; by Lemma~\ref{lem:prox-move-transverse}, the set $\mathcal{U}_{\gamma_0}$ of points $y\in G/P$ transverse to $\gamma_0\cdot y$ contains a nonempty Zariski-open subset of $G/P$.
In particular, $\mathcal{U} := \bigcap_{\gamma_0 \in F_0 \smallsetminus \{ 1\}} \mathcal{U}_{\gamma_0}$ contains a nonempty Zariski-open subset of $G/P$, and so $G/P \smallsetminus \mathcal{U}$ is contained in an algebraic subvariety of $G/P$ of positive codimension.
Since $\Lambda_{\Gamma}^{G/P}$ is not locally contained in an algebraic subvariety of $G/P$ of positive codimension (see \cite[Lem.\,2.11]{elo23}), any neighborhood of $x$ in $G/P$ meets $\Lambda_{\Gamma}^{G/P} \cap \mathcal{U}$.
\end{proof}

\begin{proof}[Proof of Theorem~\ref{thm:Anosov-free-product-G/P}]
We argue as in the proof of Theorem~\ref{thm:free-product-G/P} in Section~\ref{subsec:free-product-in-G}.
By \cite[Lem.\,3.2 \& Prop.\,3.3--3.5]{ggkw17}, there is a finite-dimensional irreducible linear representation $\tau : G\to\GL(V)$ of~$G$ such that
\begin{itemize}
  \item for any $g\in G$, the element $g$ is proximal in $G/P$ if and only if $\tau(g)$ is biproximal in $\PP(V)$,
  \item for any subgroup $\Gamma'$ of~$G$ which is Gromov hyperbolic, the natural inclusion $\Gamma' \hookrightarrow G$ is $P$-Anosov if and only if the restriction of $\tau$ to~$\Gamma'$ is $P_1$-Anosov;
  \item there is a $\tau$-equivariant embedding $\iota : G/P \hookrightarrow \mathcal{F}$, with image $\Lambda_G^{\mathcal{F}}$, such that a pair of points of $G/P$ is transverse if and only if its image under~$\iota$ is transverse in~$\mathcal{F}$,
\end{itemize}
where $\mathcal{F}$ is the space of partial flags $(x,X)$ with $x \in \PP(V)$ and $X \in \PP(V^*)$.
Up to replacing $\tau$ by $\mathrm{Sym}^2(\tau)$ and $V$ by $\mathrm{Sym}^2(V)$ (see Fact~\ref{fact:Omega-sym}), we may furthermore assume that $\tau(G)$ preserves a nonempty properly convex open subset $\Omega$ of $\PP(V)$.

By Proposition~\ref{prop:Ano-transv}, for each $i\in\{0,1\}$ there is a point $y_i \in G/P$ which is uniformly transverse to $(\Gamma_i \smallsetminus \{1\}) \cdot y_i$; moreover, if $x_i$ belongs to $\Lambda_{\Gamma}^{G/P}$ for some Zariski-dense subgroup $\Gamma$ of~$G$, then we can also take $y_i$ in~$\Lambda_{\Gamma}^{G/P}$.
Uniform transversality ensures that $\Gamma_i$ acts faithfully on $G/P$; therefore $\Gamma_i$ also acts faithfully on $\Lambda_G^{\PP(V)} = \iota(G/P)$ via~$\tau$, and the restriction of $\tau$ to~$\Gamma_i$ is injective.
Moreover, $\iota(y_i) \in \Lambda_{\tau(\Gamma)}^{\mathcal{F}}$ is uniformly transverse to $(\tau(\Gamma_i)\smallsetminus\{1\}) \cdot y_i$.

By Fact~\ref{fact:Anosov}, for each $i\in\{0,1\}$ the group $\Gamma_i$ acts convex cocompactly via~$\tau$ on some nonempty properly convex open subset $\Omega'_i$ of $\PP(V)$.
All $\Gamma_i$-invariant properly convex open subsets of $\PP(V)$ are contained in
$$\mathcal{U} := \PP(V) \smallsetminus \bigcup_{H\in\Lambda_{\Gamma_i}^{\PP(V^*)}} H.$$
Since the restriction of $\tau$ to~$\Gamma_i$ is $P_1$-Anosov, the action of $\Gamma_i$ on $\Lambda_{\Gamma_i}^{\PP(V^*)}$ is minimal, and so $\mathcal{U}$ has only one $\Gamma_i$-invariant connected component.
From this we deduce that $\Omega \cap \Omega'_i$ is nonempty (it contains the convex hull of $\Lambda_{\Gamma_i}^{\PP(V)}$ in~$\mathcal{U}$).
We conclude by applying Theorem~\ref{thm:free-product-cc-in-G}, and then Fact~\ref{fact:Anosov} again.
\end{proof}

\subsection{Ensuring the existence of a point transverse to the limit set}

Recall that a \emph{proper domain} of $G/P$ is a connected nonempty open subset $\Omega$ such that there is a point of $G/P$ which is transverse to all points of~$\overline{\Omega}$.

\begin{lemma} \label{lem:proper-domain-transv}
Let $G$ be a connected real linear semisimple Lie group, $P$ a self-opposite parabolic subgroup of~$G$, and $\Gamma_0$ a subgroup of~$G$ preserving a proper domain $\Omega$ of $G/P$.
Then any point of~$\Omega$ is transverse to any point of the proximal limit set $\Lambda_{\Gamma_0}^{G/P}$.
\end{lemma}

\begin{proof}
Since $\Omega$ is a proper domain, it has a nonempty dual $\Omega^*$, which is by definition the set of elements of $G/P$ which are transverse to all elements of~$\overline{\Omega}$.
Let $\gamma \in \Gamma_0$ be proximal in $G/P$.
Since $P$ is self-opposite, $\gamma$ has both an attracting fixed point $x_{\gamma}^+$ and a repelling fixed point $x_{\gamma}^-$ in $G/P$, and any point of $G/P$ which is transverse to $x_{\gamma}^-$ is attracted towards $x_{\gamma}^+$ by positive powers of~$\gamma$; in particular, there exists $y \in \Omega^*$ such that $\gamma^n\cdot y \to x_{\gamma}^+$, and so $x_{\gamma}^+ \in \partial \Omega^*$.
By taking the closure, we have $\Lambda_{\Gamma_0}^{G/P} \subset \partial\Omega^*$, hence any point of~$\Omega$ is transverse to $\Lambda_{\Gamma_0}^{G/P}$.
\end{proof}

Here is an immediate consequence of Theorem~\ref{thm:Anosov-free-product-G/P} and Lemma~\ref{lem:proper-domain-transv}.

\begin{corollary} \label{cor:Anosov-free-product-proper-domain}
Let $G$ be a noncompact connected real linear semisimple Lie group and $P$ a self-opposite parabolic subgroup of~$G$.
For $i\in\{0,1\}$, let $\Gamma_i$ be a torsion-free discrete subgroup of~$G$ which is Gromov hyperbolic and such that the natural inclusion $\Gamma_i \hookrightarrow G$ is $P$-Anosov; suppose that $\Gamma_i$ preserves a proper domain $\Omega_i$ in $G/P$.
Then there exists $g\in\Gamma$ such that the representation $\rho : \Gamma_0 * g\Gamma_1 g^{-1} \to G$ induced by the inclusions $\Gamma_0, g\Gamma_1 g^{-1} \hookrightarrow G$ is $P$-Anosov.
\end{corollary}

We note that the set of $P$-Anosov representations preserving a proper domain in $G/P$ is open in $\mathrm{Hom}(\Gamma_i,G)$: see \cite[Cor.\,4.4.1]{gal-PhD}.

Corollary~\ref{cor:Anosov-free-product-proper-domain} applies for instance in the setting of \emph{Nagano spaces} (also known as \emph{extrinsic symmetric spaces} or \emph{symmetric R-spaces}).
By definition, these are the flags manifolds $G/P$ that identify with a Riemannian symmetric space of a maximal compact subgroup of~$G$.
Nagano \cite{nag65} proved that in that case the proper parabolic subgroup $P$ of~$G$ is necessarily maximal, given by one simple restricted root~$\alpha$, and that there is a noncompact reductive subgroup $H$ of~$G$ preserving a proper domain $\Omega_H$ in $G/P$, which identifies with the Riemannian symmetric space of~$H$.
The full list of triples $(G,H,\alpha)$ with $G$ simple and $\alpha$ invariant under the opposition involution (\ie $P$ self-opposite) is given in Table~\ref{table1} below.
In all these cases the group $H$ contains elements which are proximal in $G/P$, hence (see Theorem~\ref{thm:Anosov-free-product-G/P}) nonabelian free groups which are $P$-Anosov.
In some cases there exist $P$-Anosov representations of closed surface groups with image in~$H$ which can be continuously deformed into $P$-Anosov representations whose image is Zariski-dense in~$G$ and still preserves a proper domain in $G/P$: in particular, this is proved by Galiay \cite[Prop.\,4.4.2]{gal-PhD} for cases (iii), (ix), (x) of Table~\ref{table1} with even~$n$, and for case~(vii) with $q=1$.
By Corollary~\ref{cor:Anosov-free-product-proper-domain}, we can combine the images of these Anosov representations into a $P$-Anosov free product. 

\begin{center}
\begin{table}[!h]
\centering
\begin{tabular}{|p{0.8cm}|p{2.6cm}|p{3.3cm}|p{3.5cm}|p{2.2cm}|}
\hline
& \centering $G$ & \centering $H$ & \centering restricted root system of $G$ & \centering $\alpha$\tabularnewline
\hline
\centering (i) & \centering $\SO(2n,2n)$ & \centering $\SO(2n,\CC)$ & \centering $D_{2n}$ & \centering $\alpha_{2n-1}$ or $\alpha_{2n}$\tabularnewline
\centering (ii) & \centering $\Sp(n,n)$ & \centering $\Sp(2n,\CC)$ & \centering $C_n$ & \centering $\alpha_n$\tabularnewline
\centering (iii) & \centering $\SU(n,n)$ & \centering $\SL(n,\CC) \times \RR$ & \centering $C_n$ & \centering $\alpha_n$\tabularnewline
\centering (iv) & \centering $\SL(2n,\RR)$ & \centering $\SO(n,n)$ & \centering $A_{2n-1}$ & \centering $\alpha_n$\tabularnewline
\centering (v) & \centering $\SL(2n,\CC)$ & \centering $\SU(n,n)$ & \centering $A_{2n-1}$ & \centering $\alpha_n$\tabularnewline
\centering (vi) & \centering $\SL(2n,\HH)$ & \centering $\Sp(n,n)$ & \centering $A_{2n-1}$ & \centering $\alpha_n$\tabularnewline
\centering (vii) & \centering $\SO(p+1,q+1)$ & \centering $\SO(p,1) \times \SO(1,q)$ & \centering $\left \{ \!\!\! \begin{array}{l} B_{q+1} \text{ if } p>q \\ D_{q+1} \text{ if } p=q>1\end{array} \right .$
& \centering $\alpha_1$\tabularnewline
\centering (viii) & \centering $\SO(n,1)$ & \centering $\SO(n-1,1)$ & \centering $A_1$ & \centering $\alpha_1$\tabularnewline
\centering (ix) & \centering $\SO^*(4n)$ & \centering $\SL(n,\HH) \times \RR$ & \centering $C_n$ & \centering $\alpha_n$\tabularnewline
\centering (x) & \centering $\Sp(2n,\RR)$ & \centering $\SL(n,\RR) \times \RR$ & \centering $C_n$ & \centering $\alpha_n$\tabularnewline
\centering (xi) & \centering $E_{7(-25)}$ & \centering $E_{6(-26)} \times \RR$ & \centering $C_3$ & \centering $\alpha_3$\tabularnewline
\centering (xii) & \centering $\SO(n+2,\CC)$ & \centering $\SO(n,2)$ & \centering $\left \{ \!\!\! \begin{array}{ll} B_{(n+1)/2} & \text{if $n$ odd} \\ D_{(n+2)/2} &\text{if $n$ even} \end{array} \right .$
& \centering $\alpha_1$\tabularnewline
\centering (xiii) & \centering $\Sp(2n,\CC)$ & \centering $\Sp(2n,\RR)$ & \centering $C_n$ & \centering $\alpha_n$\tabularnewline
\centering (xiv) & \centering $\SO(4n,\CC)$ & \centering $\SO^*(4n)$ & \centering $D_{2n}$ & \centering $\alpha_{2n-1}$ or $\alpha_{2n}$\tabularnewline
\centering (xv) & \centering $E_{7,\CC}$ & \centering $E_{7(-25)}$ & \centering $E_7$ & \centering $\alpha_7$\tabularnewline
\centering (xvi) & \centering $E_{7(7)}$ & \centering $\SL(4,\HH)$ & \centering $E_7$ & \centering $\alpha_7$\tabularnewline
\hline
\end{tabular}
\vspace{0.15cm}
\caption{Complete list (up to local isomorphism) of triples $(G,H,\alpha)$ such that $G$ is simple, $\alpha$ is a simple restricted root determining a maximal proper parabolic subgroup $P$ of~$G$ such that $G/P$ is a self-opposite Nagano space, and $H$ is a noncompact reductive subgroup of~$G$ whose Riemannian symmetric space embeds as a proper domain in $G/P$. Here $n,p,q$ are any positive integers with $p\geq q$. In cases (i) and~(xiv) we take $n\geq 2$, in case (vii) we take $\max(p,q)\geq 2$, and in case (xii) we take $n\geq 3$. This list is obtained from \cite{nag65,mak73}.}
\label{table1}
\end{table}
\end{center}


\end{document}